\documentclass[a4paper,11pt]{article}
 
 \usepackage{multicol}
\usepackage{amsmath,amssymb,graphics,epsfig,color,cite}
\usepackage{mathrsfs,bbm}
\usepackage{amsthm} 
\usepackage{bm}
\usepackage[toc,page]{appendix}
\usepackage{tikz}
\usetikzlibrary{arrows}
\usetikzlibrary{decorations.pathreplacing,decorations.markings}
\tikzset{arrow data/.style 2 args={%
      decoration={%
         markings,
         mark=at position #1 with \arrow{#2}},
         postaction=decorate}
      }%

 \usepackage{enumitem}

\usepackage{hyperref}
 
\usepackage{authblk}

\usepackage{geometry}
\DeclareMathAlphabet{\mathpzc}{OT1}{pzc}{m}{it}

\newcommand{\mc}[1]{\mathcal{#1}}

\newcommand{\mbf}[1]{\mathbf{#1}}

\DeclareMathOperator{\range}{Ran}

\newcommand{\lp}{\langle}
\newcommand{\rp}{\rangle}

\newcommand{\ve}{\varepsilon}

\DeclareMathOperator{\hess}{Hess}

\DeclareMathOperator*{\argmin}{arg\,min}

\def\div{{\rm div}\,}
\def\Hess{{\rm Hess\,}}

\def\supp{\mathop{\rm supp} \nolimits} 
\def\Ran{{\rm Ran}}

\def \Sp{{\rm  Sp\;}}

\def\and {{\rm  and \;}}

\def\Ker {{\rm  Ker  \;}}
\def\Im {{\rm  Im\;}}

\def\dim {{\rm  dim  \;}}

\newcommand{\ft}[1]{\mathsf{#1}} 
\newcommand {\pa}{\partial}

\newtheorem{theorem}{Theorem}
\newtheorem*{theorem*}{Theorem}
\newtheorem{proposition}{Proposition}
\newtheorem{definition}[proposition]{Definition}
\newtheorem*{definition*}{Definition}
\newtheorem{lemma}[proposition]{Lemma}
\newtheorem{corollary}[proposition]{Corollary}
\newtheorem{remark}[proposition]{Remark}

\newtheorem*{assumption*}{Assumption}
\newtheorem{manualassumptioninner}{Assumption}
\newenvironment{manualassumption}[1]{%
  \renewcommand\themanualassumptioninner{#1}%
  \manualassumptioninner
}{\endmanualassumptioninner}

\date{}

\usepackage{titlesec}
\titleformat*{\section}{\Large\bfseries}
\titleformat*{\subsection}{\large\bfseries}
\titleformat*{\subsubsection}{\normalsize\bfseries}
\titleformat*{\paragraph}{\normalsize\bfseries}
\titleformat*{\subparagraph}{\normalsize\bfseries}

\usepackage[hang,font={small,it},labelfont=sf]{caption}

 \usepackage{titlesec}
\titleformat*{\section}{\large\bfseries}
\titleformat*{\subsection}{\normalsize\bfseries}
\titleformat*{\subsubsection}{\normalsize\bfseries}
\titleformat*{\paragraph}{\normalsize\bfseries}
\titleformat*{\subparagraph}{\normalsize\bfseries}

\title{Eyring-Kramers exit rates for the overdamped Langevin dynamics:  the case with  saddle points on the boundary}
 
\author[1]{Tony Lelièvre}
\author[2]{Dorian Le Peutrec}
\author[3]{Boris Nectoux}

\affil[1]{\footnotesize Ecole des Ponts ParisTech, CERMICS and Inria, France.  Email address: \texttt{tony.lelievre@enpc.fr}}
\affil[2]{\footnotesize  Institut Denis Poisson, Université d'Orléans,
  France. Email address: \texttt{dorian.le-peutrec@univ-orleans.fr}}
\affil[3]{\footnotesize Université Clermont Auvergne, LMBP, France. Email address: \texttt{boris.nectoux@uca.fr}}

\begin{document} 

%

 \maketitle

\abstract{ Let $(X_t)_{t\ge 0}$ be the stochastic process solution to
  the overdamped
  Langevin dynamics
  $$dX_t=-\nabla f(X_t) \, dt +\sqrt h \, dB_t$$
and let $\Omega \subset  \mathbb R^d $ be the basin of attraction of a local
minimum of $f: \mathbb R^d \to \mathbb R$. Up to a small perturbation of $\Omega$ to make it
smooth, we prove that the exit rates of $(X_t)_{t\ge 0}$ from $\Omega$
through each of the saddle points of $f$ on $\partial \Omega$
can be parametrized by the celebrated Eyring-Kramers laws, in the
limit $h \to 0$. 
This result provides firm mathematical grounds to jump Markov
models which are used to model the evolution of molecular systems, as
well as to some numerical methods which use these underlying jump
Markov models to efficiently sample metastable trajectories of the  overdamped
  Langevin dynamics.

\medskip

}



\tableofcontents

\section{Motivation and statements of the main results}\label{sec:intro}

\subsection{An informal presentation of the results}\label{sec:informal}

Let us first present in this section the motivation for this work,
namely the modelling and the efficient simulation of metastable
stochastic dynamics which are used in molecular dynamics, as well as
an informal statement of the main results.

\medskip
\noindent
\textbf{Overdamped Langevin dynamics and metastable exit.}  Let us consider a
potential energy function  
 $$f:\mathbb R^d \to \mathbb R,$$
which is assumed to be smooth and with non-degenerate critical points.
A prototypical dynamics to describe the evolution of a molecular
system in the energy landscape~$f$ at a fixed temperature is the
overdamped Langevin dynamics:
\begin{equation}\label{eq.langevin}
dX_t=-\nabla f(X_t) \, dt +\sqrt h \, dB_t,
\end{equation}
where $(X_t)_{t\ge 0}$ gives the positions of the atoms as a function
of time, $h>0$ is (proportional to) the temperature,
 and $(B_t)_{t\ge 0}$ is a $d$-dimensional standard Brownian
 motion. Let us consider $\Omega \subset \mathbb R^d$ a basin of
 attraction\footnote{Actually, as will be discussed below, since we
   require $\Omega$ to be a smooth bounded domain, one may need to
   consider a small perturbation of a basin of attraction of $f$ to
   apply our results, see Remark~\ref{rem:smooth_basin}.} of a local
 minimum of $f$. In many cases of interest, the process spends a lot
 of time within $\Omega$ before leaving it, typically because
 the temperature $h$ is
 small compared to the energy barriers which have to be overcome to
 leave $\Omega$: this phenomenon is called metastability, and an exit
 which occurs after a long relaxation time within $\Omega$ is called a
 metastable exit (this will be formalized below using the notion of
 quasi-stationary distribution). We are interested in the so-called exit
 problem~\cite{FrWe} which consists in
 precisely describing the exit event from $\Omega$ in such a
 situation, namely the law of the couple of random
 variables\footnote{In all this work, $\Omega$ is a fixed domain, and
   we therefore do not indicate explicitely the dependency of $\tau$
   on $\Omega$.} $(\tau,X_{\tau})$
where 
\begin{equation}\label{eq:tau}
\tau=\inf\{t\ge 0, \, X_t\notin \Omega\}
\end{equation}
is the first exit time from $\Omega$, and $X_\tau$ is thus the first
exit point.
More precisely, we will show that for a metastable exit, in the limit
$h \to 0$, the law of $(\tau,X_{\tau})$ can be
approximated using a simple jump Markov model with exit rates from
$\Omega$ parametrized by the celebrated Eyring-Kramers laws, a model
which is sometimes called kinetic Monte Carlo in the physics
literature~\cite{voter-05}. These exit rates are associated with the
local minima of $f$ on $\partial \Omega$, which are saddle points of
$f$ (namely critical points of $f$ of index $1$) since $\Omega$ is a
basin of attraction. These points are on the most probable exit pathways from $\Omega$.

Before providing more details on this kinetic Monte Carlo model in the
next paragraph, let us emphasize that this question is both important in terms
of modelling, and in terms of numerical simulation of~\eqref{eq.langevin}. In terms of
modelling, it gives a rigorous framework to prove that a
coarse-grained version of the overdamped Langevin dynamics is indeed
the kinetic Monte Carlo dynamics (a.k.a. Markov State Model)
parametrized by Eyring-Kramers laws. Actually, if the states form a partition of~$\mathbb R^d$ (which is indeed the case, up to a null set, if one defines the states as
the basin of attractions of the local minima of $f$) and if all the exits are assumed to be metastable, one can even use a kinetic Monte Carlo model not only to
sample the exit from a metastable state, but to actually describe the
full evolution of the system, see for
example~\cite{cameron-14b,schuette-98, schuette-sarich-13,
  voter-05,wales-03, pande2010everything}. 
In tems of numerical simulations, metastability implies that the direct numerical simulation
of~\eqref{eq.langevin} is prohibitive, because a lot of computational time is wasted
in metastable states: using the simpler underlying kinetic Monte Carlo
model, one can then accelerate the sampling of the exit event when the
process $(X_t)_{t \ge 0}$ remains trapped in a metastable state. This
is the cornerstone of the so-called accelerated dynamics algorithms
such as {temperature accelerated
  dynamics}~\cite{sorensen-voter-00} or
{hyperdynamics}~\cite{voter-97,tiwary2013metadynamics},
see~\cite{di-gesu-lelievre-le-peutrec-nectoux-17} for more details. These  algorithms are widely
used in practice with applications in material science,
see for instance
\cite{bai2010efficient,PhysRevLett87126101,PhysRevB66205415,fan-yip-yildiz-14}. In
this context, the states are very often defined as basins of attractions of local
minima of $f$: this is indeed numerically convenient since a simple
steepest decent algorithm can be used to identify in which state the
system is.

\medskip
\noindent
\textbf{Kinetic Monte Carlo and Eyring Kramers law.}
Let us recall that $\Omega$ is a basin of attraction of a local
minimum of $f$. Thus, $f$ contains a unique critical point
 in $\Omega$, which is also the global minimum of $f$ in $\Omega$,
 denoted by $x_0$. Moreover, the local minima of $f$ on $\partial
 \Omega$ are saddle points of $f$, that we denote by $\{z_1, \ldots, z_n\} \subset \partial \Omega$. The kinetic
Monte Carlo algorithm models the exit event from $\Omega$ through a
couple of random variables $(\tau_{\mathsf {kMC}},  \mathsf Y_{\mathsf
  {kMC}})$, where $\tau_{\mathsf{kMC}}$ is the exit time and $\mathsf
Y_{\mathsf {kMC}} \in \{z_1, \ldots, z_n\}$ is equal to $z_i$ if the
process exits $\Omega$ through a neighborhood of $z_i$ in $\partial
\Omega$. The law of $(\tau_{\mathsf {kMC}},  \mathsf Y_{\mathsf
  {kMC}})$ requires a collection of rates $( \mathsf k_{z})_{z \in \{z_1,
  \ldots, z_n\}}$ associated with
the saddle points, and is defined by the following three properties:
(i) the time $ \tau_{\mathsf {kMC}}$  is exponentially distributed
with parameter  $\sum_{z \in \{z_1,
  \ldots, z_n\}}  \mathsf k_{z}$
\begin{equation}\label{eq.k_temps}
 \tau_{\mathsf {kMC}} \sim\mathcal E\Big( \sum_{z \in \{z_1,
  \ldots, z_n\}}  \mathsf k_{z}\Big);
\end{equation}
(ii) $\tau_{\mathsf {kMC}}$ is  independent of $\mathsf Y_{\mathsf
  {kMC}}$ ; and (iii) for all $z \in \{z_1, \ldots z_n\}$,
\begin{equation}\label{eq.YY}
\mathbb P[\mathsf Y_{\mathsf {kMC}}=z]= \frac{\mathsf k_{z}}{\sum_{z \in \{z_1,
  \ldots, z_n\}} \mathsf  k_{z}}.
\end{equation}
Moreover, in the setting of the so-called harmonic transition state theory, the
rates are defined using the famous Eyring-Kramers
formula~\cite{hanggi-talkner-barkovec-90,voter-05}: for any $z \in
\{z_1, \ldots,z_n\}$,
\begin{equation}\label{eq.ek}
 \mathsf k_{z}=\mathsf P_{z}\, e^{-\frac 2h   (f(z)-f(x_0) )}, 
\end{equation}
where, we recall, $x_0\in \Omega$ is the global minimum of $f$ in
$\Omega$ and the prefactor $\mathsf P_{z}$ is 
\begin{equation}\label{eq.prefactor}
\mathsf P_{z}=\frac{\vert \mu_{z}\vert}{\pi }\frac{  \sqrt{ {\rm det\,  Hess}\, f (x_0) }    }{   \sqrt{\vert {\rm det \,Hess}\, f(z)  \vert  }   },
 \end{equation} 
where  $\mu_{z}$ is the   negative eigenvalue of $\hess f(z)$.

\begin{remark}  The Eyring-Kramers formulas are sometimes defined with
  a prefactor which is the half of the right-hand-side in~\eqref{eq.prefactor}. This depends
  whether one considers exit rates (as in this work) or transition
  rates (as for example in the works~\cite{BEGK,BGK} where eigenvalues
  of the infinitesimal generator of the process $(X_t)_{t \ge 0}$ are identified with transition rates). The transition rates are half of the exit rates since
once the process reaches a saddle point $z$, it has a probability
$1/2$ to immediately come back to $\Omega$, and a probability $1/2$ to actually
transition to the neighbooring state (see e.g. \cite[Remark~8]{IHPLLN}
for further discussions).  
\end{remark}


The objective of this work is to show that, for a metastable exit, in the limit $h \to 0$,
the law of $(\tau_{\mathsf {kMC}},  \mathsf Y_{\mathsf
  {kMC}})$ indeed approximates the law of $(\tau,  X_\tau)$, in a
sense that will be made precise in the next paragraph. We will use the
quasi-stationary distribution approach to metastability, which appears
to be very useful to study the exit problem~\cite{le2012mathematical,DLLN,bianchi2020soft}.

\medskip
\noindent
{\bf The quasi-stationary distribution approach to metastability.}
As explained above, we will study metastable exits, namely exits which
occur after the stochastic process $(X_t)_{t \ge 0}$ solution
to~\eqref{eq.langevin} relaxes within $\Omega$. The notion of
quasi-stationary distribution gives a way to formalize mathematically
this idea. Let us recall standard facts on the existence and uniqueness of a
quasi-stationary distribution for a diffusion process (see for
example~\cite{cattiaux-collet-lambert-martinez-meleard-san-martin-09,collet-martinez-san-martin-13}
for more details).
\begin{definition}  Let us denote by $\mathcal P(\Omega)$ the set of
  probability measures supported in $\Omega$. A quasi-stationary
   distribution  in $\Omega \subset \mathbb R^d$
   for a Markov process $(X_t)_{t \ge 0}$ with values in $\mathbb R^d$
   is a probability measure $\mu \in \mathcal P(\Omega)$
such that:
$$
\forall t\ge 0, \forall \text{ measurable }A\subset\Omega, \quad
\mu (A)=\frac{ \mathbb P_\mu \left[X_t \in A,   t<\tau\right]  }{  \mathbb P_\mu \left[t<\tau \right]},
$$
where $\tau = \inf\{t >0, \, X_t \not \in \Omega\}$, and the subscript
$\mu$ in $\mathbb P_\mu$ indicates that $X_0 \sim \mu$.
\end{definition}
It is well-known (see for
example~\cite{le2012mathematical,ChampagnatCoulibaly}) that for a
smooth potential $f$ and a bounded smooth domain $\Omega$,  the process
$(X_t)_{t \ge0}$ soution to~\eqref{eq.langevin} admits a unique
quasi-stationary distribution on $\Omega$, denoted by
$\nu_h$ in the following. Moroever, the
previously cited works also show the following exponential convergence
result: $ \exists c>0, \forall \mu \in \mathcal P(\Omega), \exists C(\mu)>0, \exists t(\mu)>0$, 
\begin{equation}\label{eq.con-c}
\forall t\ge t(\mu), \forall \text{ measurable }A\subset\Omega, \quad \big \vert \mathbb P_\mu\big[X_t\in A\big | t<\tau   \big] -\nu_h(A)      \big \vert \le  C(\mu)e^{-ct}.
\end{equation}
Therefore, if the process $(X_t)_{t \ge0}$ remains trapped in $\Omega$
for a long-time, then $X_t$ is approximately distributed according to the
quasi-stationary distribution $\nu_h$, which can thus be seen as a
local equillibrium within $\Omega$. A metastable exit is then an exit
which occurs after this local equilibrium has been reached, namely
(using the Markov property)
an exit for  the process $(X_t)_{t \ge0}$ with initial condition $X_0
\sim \nu_h$. 

If $X_0\sim \nu_h$, the exit event satisfies the two fundamental properties (see
for example \cite[Proposition 2.4]{le2012mathematical}): 
\begin{equation}\label{eq.-indep} 
\tau\sim\mathcal E(\lambda_h) \text{ and } \tau \text{ is independent of  } X_{\tau}.
\end{equation}
With these two properties, one can use a kinetic Monte Carlo model to
exactly sample the exit event. Indeed, assume again for simplicity that $\Omega$ is
the basin of attraction of a local minimum of~$f$, and let us denote
by $\mathsf W_+^z \subset \partial \Omega$ the stable manifold of the
saddle point $z \in \{z_1, \dots, z_n\}$ (see~\eqref{eq.stablem} below
for a
precise definition). Up to a null set, the sets
$(W_+^z)_{z \in \{z_1, \ldots, z_n\}}$ form a partition of the
boundary $\partial \Omega$ of the  basin
of attraction. Let us now introduce the rates: for any $z \in
\{z_1, \ldots, z_n\}$,
\begin{equation}\label{eq:rates}
\mathsf k_{z}^{o\ell} :=\frac{\mathbb P_{\nu_{h}}\big [
  X_{\tau} \in W^+_{z} \big]}{\mathbb E_{\nu_{h}}\big [ \tau \big]}.
\end{equation}
where the superscript $o\ell$ indicates that we consider the
overdamped Langevin dynamics~\eqref{eq.langevin}. Then the kinetic
Monte Carlo model parametrized with these rates generate an exit event
$(\tau_{\mathsf {kMC}},  \mathsf Y_{\mathsf  {kMC}})$ which is exactly
consistent with the exit event $(\tau,X_\tau)$ of the original
dynamics~\eqref{eq.langevin}. Indeed, using~\eqref{eq.k_temps}--\eqref{eq.YY}
and~\eqref{eq.-indep}, one has: (i) $\tau_{\mathsf {kMC}}$ has the same law as $\tau$,
(ii) $\tau_{\mathsf {kMC}}$ and $\mathsf Y_{\mathsf  {kMC}}$ are
independent, which is also the case for $\tau$ and  $X_\tau$,  and finally (iii)
$\mathbb P(Y_{\mathsf  {kMC}}=z)=\mathbb P(X_{\tau} \in W^+_{z})$. The
mathematical question, which is the focus of this work, is now to
prove that the rates $\mathsf k_{z}^{o\ell}$ can indeed be accurately
approximated by the Eyring-Kramers formulas~\eqref{eq.ek}. 

As already mentioned above (see footnote 1 and
Remark~\ref{rem:smooth_basin} below), we will need to assume that $\Omega$ is
smooth and bounded. The smoothness assumption may require to slightly modify the basin of
attraction in the neighborhoods of the boundaries of $W^+_z$ where
$\partial \Omega$ is not necessarily smooth (these are anyway typically high
energy points which are thus visited with an exponentially small
probability when $h \to 0$). Therefore, we will
not consider exactly $\mathsf k_{z}^{o\ell}$ but the
following rates: for any $z \in \{z_1, \ldots, z_n\}$,
 \begin{equation}\label{kijl}
\mathsf k_{z}^{o\ell}(\Sigma_{z}) :=\frac{\mathbb P_{\nu_{h}}\big [
  X_{\tau} \in \Sigma_{z} \big]}{\mathbb E_{\nu_{h}}\big [ \tau \big]},
\end{equation}
where $\Sigma_{z}$ is an open  neighborhood of $z$ in $\partial
\Omega$ which is positively stable for the gradient dynamics $\dot{x}
= -\nabla f(x)$ and can be chosen arbitrarily large in $ \partial \Omega \cap W^+_z$. We will prove that, under
some geometric assumptions, these rates can indeed be accurately approximated by
the Eyring-Kramers formulas in the small temperature regime $h \to 0$, see Corollary~\ref{co.ek} below.  This requires sharp estimates  of the
probabilities that $(X_t)_{t \ge 0}$ exits $\Omega$ through the
neighborhoods $\Sigma_z$
of the saddle points $z \in \{z_1, \ldots, z_n\}$. These
precise approximations of the exit rates are used in particular in the
temperature accelerated dynamics algorithm~\cite{sorensen-voter-00} to extrapolate exit events
observed at high temperature to low temperature
(see Remark~\ref{rem:TAD} for a discussion underlying the similarities between our mathematical
analysis and this algorithm). Let us now
leave this informal presentation and present the precise setting and
the main mathematical results of this work.

\subsection{Mathematical setting and statements of the main results}

\subsubsection{Notation and definition}

In the following, $\Omega$ is a smooth bounded  domain of $\mathbb R^d$. The function $f:\overline \Omega\to \mathbb
R$ is assumed to be a $\mathcal C^\infty$  function, i.e. it is the restriction
to~$\overline \Omega$ of a smooth function  defined on $\mathbb R^d$. We still
denote by $f$ a smooth   extension of $f:\overline \Omega\to \mathbb
R$ to~$\mathbb R^d$. Since the quantities of interest in this work
only depends on the values of $f$ in the bounded set
$\overline{\Omega}$,  we assume  throughout this work without loss of
generality  that the extension of $f$ is such that: 
\begin{equation}\label{eq.DB}
\sup_{x\in \mathbb R^d} |\nabla f(x)|+ \sup_{x\in \mathbb R^d} |\text{Hess } f(x)|<+\infty,
\end{equation}
where $\hess f(x)$ denotes the Hessian matrix of $f$ at $x\in \mathbb R^d$.

\medskip
\noindent
\textbf{Basic notation.}  The open ball of radius $r>0$ centred at $x\in \mathbb R^d$ is denoted by $\mathsf B(x,r)$. The unit outward normal to $\Omega$ at  $z\in \partial \Omega$ is denoted by $\mathsf n_\Omega(z)$. The normal derivative on $\partial \Omega$ of a smooth function $f:\overline \Omega\to \mathbb R$ is denoted by $\partial_{\mathsf n_\Omega}f$. Its tangental gradient on $\partial \Omega$ is denoted by $\nabla_{\mbf T}f$. We will simply write  $\{f<a\}$ for the set $\{x\in \overline \Omega, f(x)<a\}$.  

\medskip
\noindent
\textbf{Index of a critical point.} A point $x \in \overline{\Omega}$ is a critical
point of $f$ if $|\nabla f(x)|=0$. The critical point $x$ is non-degenerate
if furthermore  $\hess f(x)$ is invertible.  The function $f$ is a Morse function if all its critical points in $\overline \Omega$ are non degenerate.   The non-degenerate critical
point $x$ is of index $p\in \{0,\ldots,d\}$ if $\hess f(x)$ admits $p$
negative eigenvalues. A saddle point is a non
 degenerate critical point with index 1. Notice that the
 index of a critical point on $\partial \Omega$ does not depend on the
 extension of $f$ outside $\overline{\Omega}$.

\medskip
\noindent
 \textbf{Stable and unstable manifolds.} 
  Let $x\in \mathbb R^d$ and denote by~$\varphi_x(t)$ the maximal
  solution    to the ordinary differential equation (which is defined
  for all $t \in \mathbb R$ by~\eqref{eq.DB}):
\begin{equation}\label{eq.hbb}
  \frac{d}{dt}\varphi_x(t)=-\nabla f(\varphi_x(t)) \text{ with } \varphi_x(0)=x.
  \end{equation}     
When~$z\in\mathbb R^d$ is a saddle point of~$f$,   we denote by~$\mathsf W_{+}^z$ and~$\mathsf W_{-}^z$   respectively the stable and unstable manifolds of~$z$ for  the dynamics~\eqref{eq.hbb}, i.e.  
\begin{equation}\label{eq.stablem}
\mathsf W_{\pm}^z=\Big \{x\in \mathbb R^d,\, \lim_{t\to \pm \infty}\varphi_x(t)= z  \Big \}.
\end{equation}
Let us recall the stable manifold theorem (see~\cite[Corollary 6.4.1]{jost2008riemannian}). 
 
 \begin{theorem*}[Stable Manifold Theorem]
Let   $f: \mathbb R^d\to \mathbb R$  satisfying \eqref{eq.DB}, and let $z$
be a saddle point of~$f$. Then,    $\mathsf W_{+}^z $ and  $\mathsf
W_{-}^z $ are   $\mathcal C^\infty$ embedded manifolds, with
dimensions $d-1$ and $1$ respectively. 
Moreover,  the tangent spaces of  $\mathsf W_{+}^z $ and  $\mathsf
W_{-}^z $ at point $z$ satisfy
 $$
  T_z\mathsf W_{+}^z ={\rm Span} (\mathsf e_1\ldots,\mathsf e_{d-1})
  \text{ and }  T_z \mathsf W_{-}^z =  {\rm Span}(\mathsf e_d), 
  $$
 where  $(\mathsf e_1,\ldots,\mathsf e_{d-1})$ is a basis of eigenvectors associated with the $d-1$ positive eigenvalues of $\hess f(z)$ and $\mathsf e_d$ is an eigenvector associated with the negative eigenvalue of   $\hess f(z)$. 
 \end{theorem*}


\noindent
\textbf{Agmon distance.} Let us introduce the Agmon distance on $\overline \Omega$ which will be used to state our main results below. 

\begin{definition} \label{de.da} Let $f:\overline \Omega\to \mathbb R$ be a $\mathcal C^\infty$  function.  
The Agmon pseudo-distance between two points $x\in \overline \Omega$ and $y \in \overline \Omega$ is defined by:
$$
\mathsf d_a\left(x,y\right)=\inf _{\gamma\in \mathcal C^1(x,y)}  \int_0^1\vert \nabla f\vert    (\gamma(t)) \left\vert \gamma'(t) \right\vert  dt,
$$
where $\mathcal C^1\left(x,y\right)$ is the set of  curve $\gamma :[0,1] \to \overline \Omega$ which are $\mathcal C^1$ with $\gamma(0)=x$, $\gamma(1)=y$.  
\end{definition}
\noindent Since $f$ has a finite number of critical points in
$\overline \Omega$ (which is indeed the case if $f$ is a Morse function on
$\overline \Omega$), $\mathsf d_a$ is a distance since for all $x,y\in
\overline \Omega$, $\mathsf d_a(x,y)=0$ if and only if $x=y$.

\subsubsection{Assumptions}

 Let us now gather
in the following assumption all the geometric requirements on $\Omega$
and $f$.

\begin{manualassumption}{\bf($\Omega$-$f$)}\label{A}
The set $\Omega$ is a  $\mathcal C^\infty$   bounded  domain of~$\mathbb R^d$. The functions $f:\overline \Omega\to \mathbb R$ and  $f|_{\partial \Omega}$ 
 are $\mathcal C^\infty$ Morse functions.  Moreover:
 \begin{enumerate}
 \item The domain $\Omega$ is positively stable for~\eqref{eq.hbb}:
   $\forall x\in \Omega, \, \forall t\ge 0, \, \varphi_x(t)\in
   \Omega$. Moreover, there exists $x_0\in \Omega$ such that for all $x\in \Omega$,  $\lim_{t\to +\infty}\varphi_x(t)=x_0$. 

 \item  For any critical point $z\in \partial \Omega$ of
   $f$, there exists an open 
subset $\Gamma_z$ of $\partial
   \Omega$ containing $z$ and satisfying the following: 
 \begin{itemize}
\item[a.]  If $z$ is a saddle point of $f$, then
\begin{equation}\label{eq.incluWW}
 \overline{\Gamma_{z} }\subset \mathsf W_+^z,
  \end{equation}
  and  $\Gamma_z$ is positively stable for the
  dynamics~\eqref{eq.hbb}: $\forall x\in \Gamma_{z}, \forall t\ge 0,
  \, \varphi_x(t)\in \Gamma_{z}$. 

\item[b.] If $z$ is not a saddle point of $f$, then $\partial_{\mathsf n_\Omega} f=0$ on $\Gamma_z$. 
 
 \end{itemize} 
 
\item All the local minima of $f|_{\partial \Omega}$ are saddle points of $f$. 
   \end{enumerate}
\end{manualassumption}
\noindent
Assumption \autoref{A} has   simple consequences that will be used
many times in the following (the proofs are standard, and provided in
Section~\ref{sec:le.start} for completeness).
\begin{lemma}\label{le.start}
The following holds:\begin{enumerate}
\item Assume that item 1 in \autoref{A} is satisfied. Then $\partial_{\mathsf n_\Omega}f\ge 0$ on $\partial \Omega$ and  $x_0$ is the only critical point of  the function $f$ in $\Omega$. There is no local minimum of $f$ on $\partial \Omega$. Furthermore, 
 $f(x_0)=\min_{\overline \Omega}f<\min_{\partial \Omega}f$,  $\{f< \min_{\partial \Omega}f\}$ is connected and $\pa \{f< \min_{\partial \Omega}f\}\cap \partial \Omega =\argmin_{\partial \Omega}f$.
\item Assume that \autoref{A} is satisfied. For all $z\in \partial \Omega$ such that $\vert \nabla f\vert (z)=0$, $\mathsf n_\Omega(z)$ is an eigenvector of $\hess f(z)$ associated with a negative eigenvalue. 
 \end{enumerate}
\end{lemma}
Other simple consequences of Assumption \autoref{A}  are the following. If $z\neq x$  are saddle
points of~$f$, then $\Gamma_z\cap \Gamma_x=\emptyset$. For  $z$ a
saddle point of $f$, the  existence of a set $\Gamma_z$   whose
closure is arbitrarily large in $\mathsf W_+^z$  and which  is
positively stable for the dynamics~\eqref{eq.hbb} is ensured
by~\cite[Proposition~80]{DLLN}, and $\Omega$ can then be defined such
that $\Gamma_z \subset \partial \Omega$, see Remark~\ref{rem:smooth_basin}.  If   $z$ is a saddle point of $f$ one can check that
 $\partial_{\mathsf n_\Omega} f=0\text{ on } \Gamma_z$  (since $\overline{\Gamma_z}\subset \mathsf W_+^z\cap \partial \Omega$). 
 Finally, all the saddle points of $f$ in
$\overline{\Omega}$ necessarily  belong to $\partial \Omega$, and
coincide with the local minima of $f$ on $\partial \Omega$.

\begin{remark}\label{rem:smooth_basin}  Let $\mathcal A(x_0)$ be the basin of attraction of $x_0$ for the
 dynamics \eqref{eq.hbb}. As explained in the
  introduction, practitioners typically use as a definition
  of a bounded metastable domain the whole basin of attraction $\mathcal
  A(x_0)$, which indeed naturally satisfies all the Assumptions~\autoref{A},
  except in some cases the smoothness assumption ($\Omega$ is indeed assumed
  to be $\mathcal C^\infty$ in~\autoref{A}). More precisley, $\pa \mathcal A(x_0)$ is
  smooth on $\mathsf W_+^z$ for all $z \in
    \{z_1,\ldots,z_n\}$, $\pa \mathcal A(x_0)=\cup_{z \in
    \{z_1,\ldots,z_n\}} \overline{\mathsf W_+^z}$ (see~\cite[Theorem~B.13]{menz-schlichting-14}), but singularities may occur on the
  boundaries of $\mathsf W_+^z$.
 In such a case, a domain $\Omega \subset \mathcal
  A(x_0)$ satisfying  \autoref{A} can typically be obtained from    $\mathcal
  A(x_0)$ by slightly modifying  it in neighborhoods
  of the points of the boundary $\pa \mathcal A(x_0)$ where $\pa \mathcal
  A(x_0)$ is not smooth, see for example Figure~\ref{fig:ex_basin} for
  a schematic illustration in dimension 2. This modification typically only
  concerns high energy points, which are anyway visited with an
  exponentially small probability by the dynamics~\eqref{eq.langevin}
  in the regime $h \to 0$.  
 
 \begin{figure}
\begin{center}
\begin{tikzpicture}[scale=0.6]
\tikzstyle{vertex1}=[draw,circle,fill=black,minimum size=5pt,inner sep=0pt]
\tikzstyle{vertex}=[draw,circle,fill=black,minimum size=5pt,inner sep=0pt]
\tikzstyle{ball}=[circle, dashed, minimum size=1cm, draw]
\tikzstyle{point}=[circle, fill, minimum size=.01cm, draw]

 \draw[dashed] [rounded corners=5pt] (-3,0)   -- (0,3) -- (3,0) -- (0,-3)   --cycle;
 \draw[dashed] [rounded corners=11pt] (-2.3,0)  -- (0,2.3) -- (2.3,0) -- (0,-2.3)   --cycle;
\draw (-3,3)-- (3,3);
\draw (-3,-3)-- (3,-3);
\draw (-3,3)-- (-3,-3);
\draw (3,3)-- (3,-3);

  \draw (-3,-3) node[vertex,label= south west : {$m_4$}](v){};
  \draw (-3,3) node[vertex,label= north west: {$m_1$}](v){};
  \draw (3,-3) node[vertex,label= south east: {$m_3$}](v){};
  \draw (3,3) node[vertex,label= north east: {$m_2$}](v){};
  
\draw[ultra thick] (-1.45,3)-- (1.45,3);
\draw[ultra thick] (-1.45,-3)-- (2,-3);
\draw[ultra thick]  (-3,1.45)-- (-3,-1.45);
\draw[ultra thick]  (3,1.45)-- (3,-1.45);

     \draw  (2,-2.3) node[]{$\Omega$};
  
  \draw (-3,0) node[vertex,label= west: {$z_1$}](v){};
  \draw (3,0) node[vertex,label= east: {$z_3$}](v){};
  \draw (0,3) node[vertex,label= north: {$z_2$}](v){};
  \draw (0,-3) node[vertex,label= south: {$z_4$}](v){};
  
\draw (0,0) node[vertex1,label={$x_0$}](v){};
  \draw[ultra thick] (-1.45,3)  ..controls   (-2.8,2.8)  .. (-3,1.45) ;
    \draw[ultra thick] (1.45,3)  ..controls   (2.8,2.8)  .. (3,1.45) ;
      \draw[ultra thick] (-1.45,-3)  ..controls   (-2.8,-2.8)  .. (-3,-1.45) ;
      \draw[ultra thick] (1.87,-3)  ..controls   (2.8,-2.9)  .. (3,-1.45) ;

 \draw[dashed]   (0,0) circle (0.4cm);

 \draw[dashed]   (0,0) circle (0.9cm);
  \draw[dashed]   (0,0) circle (1.3cm);
 
\end{tikzpicture}

\caption{The basin of attraction $\mathcal A(0)=(-1,1)^2$ (for the dynamics~\eqref{eq.hbb}) of the local
  minimum  $0\in \mathbb R^2$ of the Morse function  $f(x,y)=-\cos(\pi
  x)-\cos(\pi y)$. There are  8 critical points on $\pa \mathcal
  A(0)$:  four saddle points ($z_1,z_2,z_3,z_4$) and  four local
  maxima ($m_1,m_2,m_3,m_4$). Each edge of the square $(-1,1)^2$ is  the stable manifold of the saddle point it contains. In thick lines, a  domain $\Omega$ satisfying~\autoref{A}. In dashed lines,     the level sets of $f$. 
}

\label{fig:ex_basin}

 \end{center}
\end{figure}
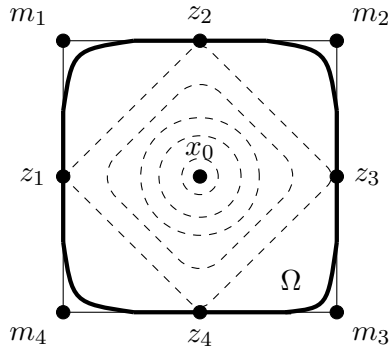
\end{remark}

\begin{definition}\label{eq.de-introd}
When \autoref{A} holds, the  saddle points of $f$ in $\overline{\Omega}$ are denoted by $\{z_1,\ldots,z_n\} \subset
\partial \Omega$  and ordered such that 
\begin{equation}\label{eq.n0}
\min_{\partial \Omega}f=f(z_1)=\ldots
= f(z_{n_0}) < f(z_{n_0+1}) \le \ldots \leq f(z_n).
\end{equation}
The cardinal of $\arg\min f|_{\partial \Omega}$ is thus $n_0 \in \{1,
\ldots, n\}$. For all $k\in \left\{1,\ldots,n\right\}$,  $\mu_{z_k}$ is the negative eigenvalue of  ${\rm Hess }f (z_k)$. 
For all $k\in \left\{1,\ldots,n\right\}$, we denote by  
$\Sigma_{z_k} \subset \partial \Omega$ an open set such that 
\begin{equation}\label{Sigma_k}
z_k \in  \Sigma_{z_k}  \ \text{ and } \ \overline{\Sigma_{z_k} }\subset  \Gamma_{z_k}.
\end{equation} 
\end{definition}
 \noindent
A schematic representation of $\Omega$, $x_0$, $\{z_1,\ldots,z_n\}$, and $\{\Sigma_{z_1},\ldots, \Sigma_{z_n}\}$ is given in Figure~\ref{fig:dom_omega} when $n=4$.
 
\subsubsection{From probability to partial differential equations}  
\label{sec.diri}

In order to give sharp asymptotic estimates of the rates~\eqref{kijl}
when $h \to 0$, we
will rewrite the law of the random variable $(\tau,X_\tau)$ using the
first eigenvalue and eigenvector of the infinitesimal generator of the
process~\eqref{eq.langevin} with homogeneous Dirichlet boundary
conditions on $\Omega$. The small temperature regime then consists in
analyzing the semi-classical limit of this eigenstate.

Let us   denote by  
\begin{equation}\label{eq:inf_gen} 
 \mathsf L_{f,h}^{(0)} =-\frac h2 \Delta+\nabla f\cdot \nabla
\end{equation}
the opposite of the infinitesimal generator of the process~\eqref{eq.langevin}.
Let $ H^1_0(\Omega,  e^{-\frac 2hf}dx)$ be the set of functions $g\in  H^1(\Omega,  e^{-\frac 2hf}dx)$ such that $g=0$ on $\partial \Omega$. 
The operator $ \mathsf L_{f,h}^{(0)}$    on $L^2(\Omega, e^{-\frac 2h f}dx)$ with domain 
$$H^2(\Omega, e^{-\frac 2h f}dx)\cap H^1_0(\Omega, e^{-\frac 2h f}dx)=\big \{w\in H^2(\Omega, e^{-\frac 2h f}dx), \, w=0 \text{ on } \partial \Omega   \big\}.$$ 
  is denoted by $ \mathsf L^{\mathsf{Di},(0)}_{f,h}(\Omega)$. The
  superscripts $\mathsf{Di}$ and $(0)$ respectively indicate that the
  operator is supplemented with Dirichlet
  boundary conditions, and acts on functions, namely $0$-forms (operators on $1$-forms will be also considered, see Section~\ref{sec.seclienwitten}). 
The operator $ \mathsf L^{\mathsf{Di},(0)}_{f,h}(\Omega) $ is the Friedrichs extension (see for instance~\cite[Section 4.3]{helffer2013spectral}) on $L^2(\Omega,  e^{-\frac 2hf}dx)$ of the closed quadratic form    
 \begin{equation}\label{eq.Q}
\psi \in H^1_0(\Omega,  e^{-\frac 2hf}dx)\mapsto \frac h2 \int_\Omega \vert \nabla \psi  \vert^2\, e^{-\frac 2hf}.
 \end{equation}
The operator $ \mathsf L^{\mathsf{Di},(0)}_{f,h}(\Omega)$ is thus a
positive self-adjoint operator  on $L^2(\Omega, e^{-\frac 2h
  f}dx)$. In addition, it has a compact resolvent (from the compact injection $H^1_0(\Omega, e^{-\frac 2h f}dx)\subset  L^2(\Omega, e^{-\frac 2h f}dx)$). 
Then, by standard results on elliptic operators, its smallest eigenvalue   $ \lambda_{h}$ is simple and any associated eigenfunction $u_{h}$ is $\mathcal C^\infty$ on $\overline \Omega$ and has a sign on $\Omega$ (see for instance~\cite[Sections~6.3 and~6.5]{Eva}). 
Without  \label{page.lambdah}
 loss of
 generality, let us assume that:
\begin{equation}\label{eq.u_norma0}
u_h > 0 \text{ on } \Omega \ \text{ and }  \int_{\Omega} u_h^2\ e^{-\frac{2}{h}f}  =1.
\end{equation}
Then, by the   Hopf Lemma (see for instance~\cite[Section
6.4.2]{Eva}), one has $\partial_{\mathsf n_\Omega}u_h>0$ on $\partial
\Omega$.

Let us now go back to the probabilistic setting introduced
in Section~\eqref{sec:informal} and rewrite the rate~\eqref{kijl} in
terms of $(\lambda_h,u_h)$ (see for example~\cite{le2012mathematical}
for proofs of these results). The unique quasi-stationary distribution $\nu_h$ of the
process $(X_t)_{t \ge 0}$ in $\Omega$ can be written in terms of $u_h$
as follows:
\begin{equation} \label{eq.expQSD}
\nu_h(dx)=\frac{  u_{h}(x) e^{-\frac{2}{h}  f(x)}}{  \int_{\Omega}u_{h}(y) e^{-\frac{2}{h}  f(y)}dy}\,  dx.
\end{equation}
Moreover, if $X_0 \sim \nu_h$ the parameter of the exponential random
variable $\tau$ is $\lambda_h$ (in particular $\mathbb E_{\nu_h}(\tau)=\lambda_h^{-1}$), and the law of $X_\tau$  can be
written in terms of $(\lambda_h,u_h)$ as follows: for any bounded measurable test function
$\varphi: \partial \Omega \to \mathbb R$,
\begin{equation}\label{eq.dens}
\mathbb E [\varphi (X_\tau)]=- \frac{h}{2\lambda_{h}}
\frac{\int_{\partial \Omega}\varphi (x)  \partial_{\mathsf n_\Omega}
  u_{h}(x) e^{-\frac{2}{h} f(x)} \sigma(dx) }{ \int_{\Omega} u_{h}(y) e^{-\frac{2}{h} f(y)} dy}
\end{equation}
where $\sigma$ is the Lebesgue measure on $\partial \Omega$.
Using these properties, the rate~\eqref{kijl} can thus be written in terms
of $u_h$: for all $z \in \{z_1, \ldots, z_n\}$,
 \begin{equation}\label{eq.expk}
\mathsf k_{z}^{o\ell}(\Sigma_{z})=-\frac {h}{2} \frac{ \int_{\Sigma_{z}}   \partial_{\mathsf n_\Omega}
  u_h e^{-\frac{2}{h} f} d\sigma }{  \int_\Omega u_h e^{-\frac{2}{h} f}}.
\end{equation}
Proving that the transition rates~\eqref{kijl} are accurately
approximated by the Eyring-Kramers laws~\eqref{eq.ek} in the limit $h
\to 0$ thus requires in particular to get precise estimates of $\partial_{\mathsf n_\Omega}
  u_h$ on each $\Sigma_{z}$. 

\subsubsection{Main results}
\label{ssec.mainresu}
We are now in position to precisely state our main results. Theorem~\ref{thm1} and  Proposition~\ref{pr.LP-N}  give precise
asymptotic estimates on $(\lambda_h,u_h)$ in the limit $h \to 0$.
\begin{theorem}\label{thm1}
 Let us assume that the assumption \autoref{A} is satisfied. Then, for all $k\in \left\{1,\ldots,n_0\right\}$  it holds in the limit $h\to 0$
\begin{equation}\label{eq.dnuh_asymptot}
\int_{\Sigma_{z_k}}   \partial_{\mathsf n_\Omega} u_h \,   e^{- \frac{2}{h}  f} d\sigma=  \frac{ 2\vert \mu_{z_k}\vert \big( {\rm det \ Hess } f   (x_0)   \big)^{\frac14}              }{  \pi ^{\frac{3d}{4}}\big \vert   {\rm det  \,  Hess }f (z_k) \big \vert   ^{\frac 12}}  \, h^{\frac d4 -1} \,e^{-\frac{1}{h}(2f(z_1)-f(x_0))}\big( 1+ O({\sqrt h}) \big),
\end{equation}
 where $u_h$ is the principal eigenfunction of $\mathsf L^{\mathsf{Di}
   ,(0)}_{f,h}(\Omega)$ with the normalization~\eqref{eq.u_norma0}.  In addition, there exits $c>0$ such that, when $h\to 0$
 \begin{equation}\label{eq.dnuh_asymptot2}
\int_{\partial \Omega \setminus \cup_{k=1}^{n_0}\Sigma_{z_k}}   \partial_{\mathsf n_\Omega} u_h \,   e^{- \frac{2}{h}  f} d\sigma= O\big(e^{-\frac{1}{h}(2f(z_1)-f(x_0)+c)} \big).
\end{equation}
 Moreover, assume that:  
 \begin{equation} 
 \label{hypo1}
 \forall k\in \left\{1,\ldots,n\right\}, \ \inf_{z\in \partial \Omega\setminus \Gamma_{z_k}} \mathsf d_a(z,z_k) >\max[f(z_n)-f(z_k),f(z_k)-f(z_1)],
 \end{equation}
 and 
 \begin{equation} 
 \label{hypo2} 
 f(z_1)-f(x_0)>f(z_n)-f(z_1).  
 \end{equation}
Then, for all $k\in \left\{n_0+1,\ldots,n\right\}$,  it holds in the limit $h\to 0$:
\begin{equation}\label{eq.dnuh_asymptot3}
\int_{\Sigma_{z_k}}   \partial_{\mathsf n_\Omega} u_h \,   e^{- \frac{2}{h}  f} d\sigma=   \frac{ 2\vert \mu_{z_k}\vert \big( {\rm det \ Hess } f   (x_0)   \big)^{\frac14}              }{  \pi ^{\frac{3d}{4}}\big \vert   {\rm det  \,  Hess }f (z_k) \big \vert   ^{\frac 12}}  \, h^{\frac d4 -1}  \,e^{-\frac{1}{h}(2f(z_k)-f(x_0))}\big( 1+ O( {\sqrt h}) \big).
\end{equation}
 \end{theorem}
  \begin{proposition}\label{pr.LP-N}
 Let us assume that the assumption \autoref{A} is satisfied.
\begin{equation}\label{eq.moyenn-uh}
\int_\Omega u_h\, e^{-\frac 2h f}=  (\pi h)^{\frac d4}   \big(  {\rm det  \,  Hess }f(x_0)\big)^{-\frac 14}  \,  e^{-\frac{1}{h}\min_{\overline{\Omega}} f} (1+O(h)) .
\end{equation}
Moreover, it holds in the limit $h\to 0$:
\begin{equation}\label{eq.lambda_h-Equiv}
\lambda_h=   \sum\limits _{ \ell=1 }^{n_0}     \frac{ \vert \mu_{z_\ell}\vert \big( {\rm det \ Hess } f   (x_0)   \big)^{\frac12}              }{  \pi \big \vert   {\rm det  \,  Hess }f (z_\ell) \big \vert   ^{\frac 12}        } e^{-\frac 2h (f(z_1)-f(x_0))} (1+O(\sqrt h)).
\end{equation}
\end{proposition}

 
Using the expression~\eqref{eq.dens} for the law of the exit point $X_\tau$, Theorem~\ref{thm1}  and Proposition~\ref{pr.LP-N} yield the following
sharp estimate of this law:
 \begin{theorem}\label{thm2}
  Let us assume that the assumption \autoref{A} is satisfied. Then, for all $k\in \left\{1,\ldots,n_0\right\}$,   it holds in the limit $h\to 0$:
\begin{equation}\label{eq.puh1}
\mathbb P_{\nu_h}[X_{\tau}\in \Sigma_{z_k}]=   \frac{ \vert \mu_{z_k}\vert }{\sqrt{\vert  {\rm det \, Hess } f   (z_k) \vert } }  
\left ( \sum_{\ell=1}^{n_0} \frac{ \vert \mu_{z_\ell}\vert }{\sqrt{ \vert {\rm det \, Hess } f   (z_\ell) \vert } } \right)^{-1}    + O({\sqrt h}),
\end{equation}
In addition, there exists $c>0$ such that in the limit $h\to 0$:
 \begin{equation}\label{eq.puh2}
\mathbb P_{\nu_h}[X_{\tau}\in \partial \Omega \setminus  \cup_{k=1}^{n_0}\Sigma_{z_k}]\le e^{-\frac{c}{h}}.
\end{equation}
Finally, if~\eqref{hypo1} and~\eqref{hypo2} are satisfied, it holds  for all $k\in \left\{n_0+1,\ldots,n\right\}$,    in the limit $h\to 0$:
\begin{equation}\label{eq.puh3}
\mathbb P_{\nu_h}[X_{\tau}\in  \Sigma_{z_k}]=\frac{ \vert \mu_{z_k}\vert }{\sqrt{\vert  {\rm det \, Hess } f   (z_k) \vert } }  
\left ( \sum_{\ell=1}^{n_0} \frac{ \vert \mu_{z_\ell}\vert }{\sqrt{ \vert {\rm det \, Hess } f   (z_\ell) \vert } } \right)^{-1} \,e^{-\frac{2}{h}(f(z_k)-f(z_1))}\big( 1+ O( {\sqrt h}) \big).
\end{equation}
  \end{theorem}

 As a corollary of  Theorem~\ref{thm2} and  Proposition~\ref{pr.LP-N},
one immediately gets the following sharp estimates of the exit rates defined in~\eqref{kijl}:
     \begin{corollary}\label{co.ek}
      Let us assume that the assumption \autoref{A} is satisfied. 
     Then, for all $k\in \left\{1,\ldots,n_0\right\} $,   it holds in the limit $h\to 0$:
\begin{equation}\label{eq.k*1}
 \mathsf k_{z_k}^{o\ell}(\Sigma_k)= \frac{\vert \mu_{z_k}\vert}{\pi }\frac{  \sqrt{ {\rm det\,  Hess}\, f (x_0) }    }{   \sqrt{\vert {\rm det \,Hess}\, f(z_k)  \vert  }   }   \,e^{-\frac{2}{h}(f(z_1)-f(x_0))}\big( 1+ O( {\sqrt h}) \big).
\end{equation}
In addition, if~\eqref{hypo1} and~\eqref{hypo2} are satisfied, it holds  for all $k\in \left\{n_0+1,\ldots,n\right\}$,    in the limit $h\to 0$:
\begin{equation}\label{eq.k*2}
 \mathsf k_{z_k}^{o\ell}(\Sigma_k)= \frac{\vert \mu_{z_k}\vert}{\pi }\frac{  \sqrt{ {\rm det\,  Hess}\, f (x_0) }    }{   \sqrt{\vert {\rm det \,Hess}\, f(z_k)  \vert  }   }    \,e^{-\frac{2}{h}(f(z_k)-f(x_0))}\big( 1+ O( {\sqrt h}) \big).
\end{equation}
\end{corollary} 
As discussed in Section~\ref{sec:informal},  Corollary~\ref{co.ek}
thus justifies the approximation of metastable exits of the overdamped Langevin
  dynamics~\eqref{eq.langevin} by a kinetic Monte Carlo model
  parametrized with the Eyring-Kramers formulas.

 We will also show that Theorem~\ref{thm2} extends to deterministic initial conditions $x\in
  \Omega$ as follows (the subscript $x$ in $\mathbb P_{x}$ indicates
  that $X_0=x$).
   \begin{theorem}\label{thm3}
  Let us assume that the assumption \autoref{A} is satisfied. Let $\mathsf K$ be a compact subset of~$\Omega$. Then, for all $k\in \left\{1,\ldots,n_0\right\}$,   it holds in the limit $h\to 0$:
\begin{equation}\label{eq.px1}
\mathbb P_{x}[X_{\tau}\in \Sigma_{z_k}]=   \frac{ \vert \mu_{z_k}\vert }{\sqrt{\vert  {\rm det \, Hess } f   (z_k) \vert } }  
\left ( \sum_{\ell=1}^{n_0} \frac{ \vert \mu_{z_\ell}\vert }{\sqrt{ \vert {\rm det \, Hess } f   (z_\ell) \vert } } \right)^{-1}   + O({\sqrt h}), 
\end{equation}
 uniformly in $x\in \mathsf K$. 
In addition, there exists $c>0$ such that in the limit $h\to 0$:
 \begin{equation}\label{eq.px2}
\sup_{x\in \mathsf K}\mathbb P_{x}[X_{\tau}\in \partial \Omega \setminus\cup_{k=1}^{n_0}\Sigma_{z_k}]\le e^{-\frac{c}{h}}.
\end{equation}
Let us assume that~\eqref{hypo1} and~\eqref{hypo2} are satisfied.  Assume in addition there exists $\ell_0\in \{n_0+1,\ldots,n\}$ such
that 
\begin{equation}\label{eq:hypo2_bis}
2( f(z_{\ell_0})-f(z_1))<f(z_1)-f(x_0).
\end{equation}
 Let $k_0\in \{ n_0+1,\ldots,\ell_0\}$ and $ \alpha_* \in \mathbb R$ be such
that  $f(x_0) <  \alpha_*< 2 f(z_1)-f(z_{k_0})$  (notice that
necessarily $\alpha_*< f(z_1)=\min_{\partial \Omega}f$). 
Then, it holds for $k\in
\{  n_0+1,\ldots,k_0\}$    in the limit $h\to 0$:
\begin{equation}\label{eq.px3}
 \mathbb P_{x}[X_{\tau}\in  \Sigma_{z_k}]=      \frac{ \vert \mu_{z_k}\vert }{\sqrt{\vert  {\rm det \, Hess } f   (z_k) \vert } }  
\left ( \sum_{\ell=1}^{n_0} \frac{ \vert \mu_{z_\ell}\vert }{\sqrt{ \vert {\rm det \, Hess } f   (z_\ell) \vert } } \right)^{-1}      \,e^{-\frac{2}{h}(f(z_k)-f(z_1))}\big( 1+ O( {\sqrt h}) \big),
\end{equation}
 uniformly in $x   \in \, {\{ f\le \alpha_*\}}$.
  \end{theorem}
   \noindent

Before precisely discussing related results in the literature, let us
provide some preliminary comments on the statements presented in this
section. First,  Equations~\eqref{eq.puh1}--\eqref{eq.puh2} and
\eqref{eq.px1}--\eqref{eq.px2} show that  the most probable places of
exit from $\Omega$ as $h\to 0$ are $\{z_1,\ldots,z_{n_0}\}$, and they
provide the relative probabilities of exiting
through (neighborhoods of) these points. Moreover,
Equations~\eqref{eq.puh2} and~\eqref{eq.px3} give precise asymptotic
estimate of the probability to leave through higher energy saddle points.
All these results can be seen as generalizations of those
previously obtained in~\cite{DLLN}
and of some results in \cite{DLLN-saddle1}, where it is
assumed that $\partial_{\mathsf n_\Omega}f>0$ on $\partial \Omega$. In
this case, the local minima of $f$ on $\partial \Omega$ play the role
of saddle points, and different prefactors than~\eqref{eq.prefactor} appear in the
asymptotic rates, for example. Let us finally emphasize that, as will become clear from the proofs,
all the  error terms $O( {\sqrt h})$ follow from the
Laplace's method applied to integrals on $\mathbb R_-^d$ and are optimal
(see also~\cite[Remarks 25 and 39]{DoNe2} for more details).

 \begin{figure}
\begin{center}
\begin{tikzpicture}[scale=0.6]
\tikzstyle{vertex1}=[draw,circle,fill=black,minimum size=5pt,inner sep=0pt]
\tikzstyle{vertex}=[draw,circle,fill=black,minimum size=5pt,inner sep=0pt]
\tikzstyle{ball}=[circle, dashed, minimum size=1cm, draw]
\tikzstyle{point}=[circle, fill, minimum size=.01cm, draw]
\draw [rounded corners=10pt] (1,0.5) -- (-0.25,2.5) -- (1,5) -- (5,6.5) -- (7.6,3.75) -- (6,1) -- (4,0) -- (2,0) --cycle;

     \draw  (5,2.5) node[]{$\Omega$};
   \draw[ultra thick] (-0.05,2.7)  ..controls  (0,1.8) and  (0.38,1.2)  .. (1.4,0.3)  ;
    \draw[]  (0.1,0.3) node[]{$\Sigma_{z_1}$};
    \draw[ultra thick] (2.5,0)..controls (3.9,-0.1) and (5.5,0.2)   .. (6.5,1.9) ;
    \draw[]  (5.5,-0.3)  node[]{$\Sigma_{z_2}$};
    \draw[ultra thick] (5.36,6.1) to (7.16,4.2)  ;
     \draw[]  (7.7,4.9)   node[]{$\Sigma_{z_3}$};
         \draw[ultra thick] (0.5,4) ..controls  (1.5,5.34) and  (0.3,4.8) .. (3.8,6)   ;
         \draw[] (0.2,5.2) node[]{$\Sigma_{z_4}$};
         
 \draw (1.7,5.3) node[vertex,label={$z_4$}](v){};
\draw (3.4 ,3.2) node[vertex1,label={$x_0$}](v){};
\draw (4.2,0.1) node[vertex,label= {$z_2$}](v){};
\draw (0.38,1.45) node[vertex,label={}](v){};
 \draw  (-0.5,1.45) node[]{$z_{1}$};
\draw (6.2,5.2) node[vertex,label= {$z_3$}](v){};
%
%
\end{tikzpicture}

\caption{Example of a domain  $\Omega$ with $n=4$ saddle points $\{z_1,\ldots,z_4\}$.}

\label{fig:dom_omega}

 \end{center}
\end{figure}
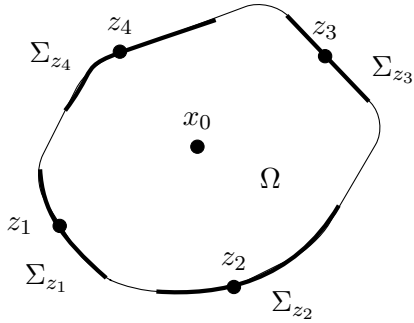

\subsubsection{A short review on mathematical approaches to metastability}

In this section, we succinctly present two aspects of metastability
which have received attention from the mathematical community: the
exit problem (which is the focus of this work) and the spectral
analysis of the infinitesimal generator.

\medskip
\noindent
\textbf{On the exit problem.}    
Even though the exit problem from a basin of attraction of a local
minimum of $f$ is a very natural question, this setting has not been
considered up to now in the literature, at least to the best of our
knowledge. This is essentially because of the
mathematical difficulties induced by the presence of critical points
of $f$ on the boundary. Let us recall the main results which have been
obtained.

Let us first mention that early inspiring formal computations were conducted
by Z. Schuss and co-workers~\cite{MS77,schuss90, MaSc}. In terms of
rigorous proofs, two techniques have then been developed, based on large
deviations or the analysis of partial differential equations
associated to the stochastic process.

From a probabilistic viewpoint, the exit problem has been studied a
lot using large deviation techiques, pioneered by M.I Friedlin and
A.D. Wentzell~\cite{FrWe}. 
 Typically, results are only obtained on $h$-log
limits of the  mean exit time $\tau$ and of the law of the exit
location  $X_{\tau}$, under the assumption that~$f$ does not have
critical point on $\partial \Omega$, see also the developments by
M.V. Day and
M. Sugiura~\cite{Day2,Day4,day1987r,day1984a,sugiura2001,sugiura1996}).
A
noteworthy exception is the work~\cite{Day1} by M.V. Day where large
deviations principles are given for some conormally  reflected
processes with attractors on the boundary. 

Techniques based on parabolic or elliptic partial diffferential
equations associated with the stochastic process
have also been developed in particular by S. Kamin~\cite{kamin-78,kamin1979elliptic}, H.~Ishii and
P.E.~Souganidis~\cite{souga1,souga2},
B.~Perthame~\cite{perthame-90}, and more recently D. Borisov and
D. Sultanov~\cite{borisov2}. In particular,  these articles study
the concentration of the  law of $X_{\tau}$ on the global minima of $f$ on
$\partial \Omega$ in the limit  $h\to 0$. In these
works, it is again assumed
that $f$ does not have critical points on the boundary.
Let us mention~\cite{mathieu1995spectra} for early results on the $h$-log limits of the  smallest eigenvalues of  and~\cite{NePC} for sharp asymptotic equivalents on the mean
exit time $\mathbb E[\tau]$ when~$f$ has critical points on $\partial
\Omega$.


Notice that $h$-log limits cannot be used to compute
the relative probabilities of exits through the lowest saddle points
$\{z_1, \ldots, z_{n_0}\}$.  Moreover, Equations~\eqref{eq.px1}
and~\eqref{eq.px2} extend the results of~\cite[Theorem 2.1]{FrWe}
and~\cite{kamin1979elliptic,kamin-78,perthame-90,day1984a,day1987r,DLLN-saddle2} to
the case when $f$ has critical points on $\partial \Omega$.  Let us however acknowledge
that even if the techniques mentioned above seem inherently limited to $h$-log limits, some of them are robust enough to apply to non
reversible elliptic processes, or quasilinear parabolic equations
(see~\cite{FrWe,Day2,Day4,day1987r,day1984a,souga1,souga2}), whereas we
only consider reversible dynamics here.


\medskip
\noindent
\textbf{Spectral problem and the Eyring-Kramers laws.} 
The focus of the present work is on the exit problem (exit time, and
first exit points), and we prove that
Eyring-Kramers laws precisely describe the exit rates from a basin of
attraction of the potential energy function. In the mathematical literature,
Eyring-Kramers laws have also been obtained in a different context,
namely when studying the smallest eigenvalues of the infinitesimal
generator (seen as an
operator on $\mathbb R^d$) of the overdamped
Langevin process, see the defintion~\eqref{eq:inf_gen} of $\mathsf L_{f,h}^{(0)}$. Two variational techniques have been used, based either on
tools from potential theory or from spectral theory (see~\cite{Ber} for a
nice review). 


Let us first mention that sharp lower and upper bounds on the small
eigenvalues were obtained in the pioneering
works~\cite{miclo-95,holley-kusuoka-stroock-89}. Then,
A. Bovier and collaborators
 developed in~\cite{BEGK,BGK} a potential theoretic approach~\cite{bovier2016metastability} to
obtain precise equivalents of the $n_p$ smallest eigenvalues of
$\mathsf L_{f,h}^{(0)}$, $n_p$ being the number of local minima of $f$
in $\mathbb R^d$. It is also proven that the non-zero eigenvalues coincide with
the inverses of mean transition times to go from one local minimum of $f$
to any of the other local minima with smaller energies. This
potential theoretic approach have then been further developed by
N. Beglund and co-workers~\cite{Berglund_Gentz_MPRF,BD15}, and by
C. Landim and I. Seo~\cite{landim2017dirichlet, lee2020non}, in
particular to non-reversible diffusions.

Using tools developed to analyze the semi-classical limit of the
Schr\"odinger operator, similar results on the low-lying spectrum have been derived by
B. Helffer, M. Klein and F. Nier in~\cite{HKN}. See also the recent
works~\cite{dorianmichel,michel2017small,peutrec2020bar,le-peutrec_nier_viterbo-13}
for generalizations, and~\cite{herau-hitrick-sjostrand-11} for
asymptotic equivalents of the smallest eigenvalues of the kinetic
Langevin operator. Let us mention the nice
work~\cite{menz-schlichting-14} where it s proven that Poincaré and
Logarithmic-Sobolev inequalities constants asymptotically satisfy an
Eyring-Kramers law in the limit $h \to 0$.
 
Let us finally emphasize that the two problems we have discussed up to
now in this section (the exit problem, and the low-lying spectrum of
$\mathsf L_{f,h}^{(0)}$ in $\mathbb R^d$) are different in nature. In
particular, the exit problem requires to precisely study the law of the
first exit point in order to estimate all the the exit rates.


%

 \subsubsection{Strategy of the proofs  and mathematical novelties}
  \label{sec.strategy1}
 
\medskip
\noindent
\textbf{Strategy of the proofs and organization of the article.}
Let us provide a concise presentation of the strategy of the proofs,
together with an outline of this work.
 In view of Theorem~\ref{thm1} and~\eqref {eq.expk}, one needs precise
 asymptotic estimates of of~$\nabla u_h\cdot \mathsf
 n_\Omega$  on~$\pa\Omega$, as $h\to 0$. Recall that $u_h$ is the
 principal eigenfunction of $\mathsf
 L^{\mathsf{Di},(0)}_{f,h}(\Omega)$: $\mathsf L^{\mathsf{Di},(0)}_{f,h}(\Omega)\, u_h  =  \lambda_h u_h$. The cornerstone of the proof of
 Theorem~\ref{thm1} is that  $\nabla u_h$ also satisfies an eigenvalue
 problem (with the same exponentially small eigenvalue $\lambda_h$), obtained by differentiating the previous equation: 
\begin{equation}\label{eq:L1-intro}
\left\{
\begin{aligned}
\mathsf L^{(1)}_{f,h} \nabla u_h &= \lambda_h \nabla u_h \text{ in $\Omega$,}\\
\nabla_{\mbf T} u_h& = 0 \text{ on $\partial \Omega$,}\\
\Big(-\frac{h}{2} {\rm div} + \nabla f \cdot \Big) \nabla u_h & = 0 \text{ on $\partial \Omega$,}\\
\end{aligned}
\right.
\end{equation}
where  $
\mathsf L^{(1)}_{f,h}=- \frac{h}{2} \Delta + \nabla f \cdot \nabla + {\rm
  Hess} \, f$ 
is an operator acting on vector fields, i.e. $1$-forms. In the
following, the operator $\mathsf L^{(1)}_{f,h}$ with tangential Dirichlet
boundary conditions, as introduced in~\eqref{eq:L1-intro}, is denoted
by~$\mathsf L^{\mathsf{Di},(1)}_{f,h}(\Omega)$.
For $q\in\{0,1\}$, let us denote, by  $ \pi_h^{(q)}$  the orthogonal
projector  of $\mathsf L^{\mathsf{Di},(q)}_{f,h}(\Omega)$ on the
eigenspace associated with the eigenvalues of $\mathsf
L^{\mathsf{Di},(q)}_{f,h}(\Omega)$ smaller than a constant $\mathsf
c_0$ independent of~$h$. From~\eqref{eq:L1-intro}, it holds, in the
limit $h \to 0$,
\begin{equation}\label{eq.nu-in}
\nabla u_h\in \Ran\, \pi_h^{(1)}. 
\end{equation}
The first step of the analysis thus consists in studying the spectrum
of the operators $\mathsf
L^{\mathsf{Di},(q)}_{f,h}(\Omega)$, $q\in\{0,1\}$. This is done in
Section~\ref{sec.comptevp} in a rather general setting (in particular
without assuming that all the local minima of $f$ on the boundary are
necessary saddle points of $f$), since this
study has its own interest.  We will prove in
particular 
(see Theorem~\ref{thm.main1} and Corollary~\ref{thm-pc}) that  for
some~$\mathsf
c_0$, and for all $h$ sufficiently small,
\begin{equation}\label{eq.dim-no}
\Ran \, \pi_h^{(0)} =  {\rm Span}\, u_h  \text{ and}  \   {\rm dim} \,  \Ran\, \pi_h^{(1)} =n.
\end{equation} 

Then, in order to study the asymptotic behaviour of   $u_h$ and
$\nabla u_h$ when $h\to 0$, we construct in Section~\ref{sec.zk-QM} a suitable
orthonormal basis of  $\Ran\, \pi_h^{(1)} $    using so-called
quasi-modes $\{\mathsf f_1^{(1)} , \ldots, \mathsf f_n^{(1)}\} $ (see
in particular
Propositions~\ref{ESTIME} and~\ref{ESTIME2}). These quasi-modes $\{\mathsf f_1^{(1)} , \ldots, \mathsf
f_n^{(1)}\} $ are built such that  for each $k\in \{1,\ldots,n\}$,
$\mathsf f_k^{(1)} $ is  essentially the principal eigenform of the
operator $\mathsf L^{(1)}_{f,h}$ defined on a domain 
 $\Omega_k^{\textsc{M}} \subset \Omega$,
with mixed Dirichlet-Neumann boundary conditions (the superscript
${\textsc{M}}$ refers to the fact that mixed  Dirichlet-Neumann
boundary conditions will be considered on $\partial
\Omega_k^{\textsc{M}}$). The only critical points of $f$ in
$\overline{\Omega_k^{\textsc{M}}}$  are $x_0$ and  $z_k$, so that
$\mathsf f_k^{(1)}$ gather information on the exit through~$z_k$.
In particular, we define $\{\mathsf f_1^{(1)} , \ldots, \mathsf f_n^{(1)}\} $ such that for all $k\in \{1,\ldots,n\}$:
\begin{equation}\label{eq.B-ff}
\int_{\Sigma_{z_k}}   \partial_{\mathsf n_\Omega}
  u_h \ e^{-\frac{2}{h} f} d\sigma 
  \sim \int_{\Sigma_{z_k}}    \mathsf f_k^{(1)} \cdot \mathsf n_\Omega  \   e^{- \frac{2}{h} f}  d\sigma, \  \text{ as $h\to 0$,}
\end{equation}
and such that $\int_{\Sigma_{z_k}}    \mathsf f_k^{(1)} \cdot \mathsf
n_\Omega  \   e^{- \frac{2}{h} f}  d\sigma$ has the expected
asymptotic equivalent leading to Theorem~\ref{thm1}. More precisely, we construct $\Omega_k^{\textsc{M}}$ in a way that allows us to   compute the asymptotic equivalent of the principal eigenvalue $\lambda(\mathsf
\Omega_k^{\textsc{M}})$ of the operator $\mathsf L^{(1)}_{f,h}$ defined with mixed Dirichlet-Neumann boundary conditions on  
 $\Omega_k^{\textsc{M}}$, with techniques recently used in~\cite{DoNe2}. 
We then show that the  asymptotic equivalent of $\lambda(\mathsf
\Omega_k^{\textsc{M}})$  provides the required asymptotic equivalent
of the right-hand side in~\eqref{eq.B-ff}, for each $k\in
\{1,\ldots,n\}$. This method to estimate $\int_{\Sigma_{z_k}}   \partial_{\mathsf n_\Omega}
  u_h \ e^{-\frac{2}{h} f} d\sigma$
 is probably the main difference with
 the approach used previously in~\cite{DLLN}. 

Finally, Section~\ref{sec.mainresu} builds on the two previous sections
to prove the main results stated in Section~\ref{ssec.mainresu}: Section~\ref{sec.proof0} is devoted to the
proofs of Theorem~\ref{thm1}, Proposition~\ref{pr.LP-N},
Theorem~\ref{thm2}, and Corollary~\ref{co.ek}; Section~\ref{sec.proof}
contains the proof of Theorem~\ref{thm3}. 

 The appendix gathers various technical results and additional comments.

\begin{remark}\label{rem:TAD}
Interestingly enough, in
 the Temperature Accelerated Dynamics
 algorithm~\cite{sorensen-voter-00,aristoff-lelievre-14,di-gesu-lelievre-le-peutrec-nectoux-17}, the numerical method 
 consists in sampling successive exits through the saddle points
 $(z_k)_{1 \le k \le n}$ at high temperature by imposing reflecting boundary conditions on
 the already visited transition pathways, and then to infer the
 exit event that would have been observed at low temperature using the
 Eyring-Kramers laws (see Corollary~\ref{co.ek}). Imposing reflecting boundary conditions on
 the dynamics is
 equivalent to introducing Neumann boundary conditions on the
 infinitesimal generator, and the sampled exits are thus very much
 related to the principal eigenforms $(\mathsf f_k^{(1)})_{1 \le k \le n}$ that we use
 as quasi-modes. For example, in the procedure outlined above, the exit through $z_k$ is exponentially
 distributed with parameter $\lambda(\mathsf
\Omega_k^{\textsc{M}})$ (in the regime~$h \to 0$).
\end{remark}

\medskip
\noindent
{\bf Mathematical novelties.} Let us finally emphasize the main
mathematical novelties and difficulties of the present work, which
is the first to precisely analyze the exit problem from a domain $\Omega$ when the
local minima of $f$ on $\partial \Omega$ are saddle points of $f$.
We actually
studied a similar problem in~\cite{DLLN}, but
under the less natural assumption that  $\partial_{\mathsf
  n_\Omega}f>0$ on~$\partial \Omega$. 
 The presence of critical points
of $f$ on $\partial \Omega$ implies substantial difficulties from a
mathematical viewpoint. First,  to prove ${\rm dim} \,  \Ran\,
  \pi_h^{(1)} =n$, we extend the analysis
of~\cite{helffer-nier-06} (see Remark~\ref{rem:numb-small} for more details), which is of independent interest. This is the purpose of 
Section~\ref{sec.comptevp}   on the Witten complex, see more precisely
Theorem~\ref{thm.main1}. Second, we develop a new
approach to compute the asymptotic equivalents as $h\to 0$ of  the
right-hand side of~\eqref{eq.B-ff}
    without relying on
WKB approximations which were used for example in~\cite{DLLN}. Though WKB approximations
are  very powerful and central tools    on which rely many works in
semi-classical analysis  (see for
instance~\cite{helffer-sjostrand-85,He2,DiSj,HKN,helffer-nier-06}), the fact that both $z_k\in
\partial \Omega_k^{\textsc{M}} $ and  $z_k$ is a critical point of $f$
prevent us from using   previously constructed WKB approximations for
Witten Laplacians~\cite{helffer-sjostrand-85,helffer-nier-06} (this is explained in more
details in Section~\ref{sec:WKB}). Third,  the proof of Theorem~\ref{thm3} uses other
arguments than the one made to prove~\cite[Corollary 16]{DLLN}
especially because the results of~\cite{eizenberg-90} (based on
techniques from the large deviation theory) do not hold when $f$ has
critical points on $\partial \Omega$ (see  the
discussion after Corollary~\ref{co.Px}).

 \section{Number of small eigenvalues of the Witten Laplacian}
 \label{sec.comptevp}

In all this section, the following general setting is assumed:
  \begin{manualassumption}{\bf($\mathsf M$-$f$)}\label{B}
Let $\overline{\mathsf M}$ be a $\mathcal C^{\infty}$ oriented compact
and  connected Riemannian manifold of dimension $d$, with boundary
$\partial  \mathsf M$ and interior $\mathsf M$.   The metric tensor on
$\overline{\mathsf M}$ is denoted by~$\mbf g_{\mathsf M}$. Let $f :
\overline{\mathsf M} \to \mathbb R$ be a $\mathcal C^{\infty}$
function. The functions~$f: \overline{\mathsf M} \to \mathbb R$
and~$f|_{\partial \mathsf M}$ are assumed to be Morse functions. Finally, for all~$x\in \pa \mathsf M$ such that $\vert \nabla f(x)\vert =0$, there exists a neighborhood $\mathsf V_x^{\pa \mathsf M}$ of~$x$ in $\pa \mathsf M$ such that:
$$\forall y\in \mathsf V_x^{\pa \mathsf M},  \  \partial_{\mathsf n_{ \mathsf M}}f(y)=0.$$
  \end{manualassumption}
 \noindent
Notice that~\autoref{B}  implies that $f$ and~$f|_{\partial \mathsf M}$  have a finite number of
critical points. Since the normal derivative $\partial_{\mathsf n_{ \mathsf M}}f$ is zero around critical
points on  $\pa \mathsf M$,   $\pa \mathsf M$ is said
to be characteristic (for the function~$f$) in these regions. Let us
recall that this condition is in particular natural when $\mathsf M\subset \mathbb R^d$ is the basin of attraction of a local minimum of $f$.

The objective of this section is to relate the number of critical
points of index $p$ of $f$, to the number of small
eigenvalues of  the Witten Laplacian acting on $p$-forms with
tangential Dirichlet boundary conditions on $\pa \mathsf M$, see
Theorem~\ref{thm.main1} below. This result is standard for manifolds
without
boundary~\cite{witten-82,helffer-sjostrand-85,simon1983semiclassical,herau-hitrick-sjostrand-11},
and has been proven in~\cite[Theorem 3.2.3]{helffer-nier-06} for manifolds with
boundaries but when $f$ does not have critical points on $\pa \mathsf M$ (see
also~\cite{laudenbach2011morse,le-peutrec-10}). This section is organized
as follows. The Witten Lapacian is introduced in
Section~\ref{sec:Witt}. The main result is stated in
Section~\ref{sec.mainresult_petitevp} and proved in
Section~\ref{sec.numb-small}, after the study of model problems on the
half space $\mathbb R^d_-$ in Section~\ref{sec.Rdmoins}. Finally, consequences of these results
to the particular problem of interest in this work are detailed in
Section~\ref{sec.seclienwitten}, with in particular the proof of~\eqref{eq.dim-no}.

\subsection{Witten Laplacian with tangential Dirichlet boundary conditions}\label{sec:Witt}

 \subsubsection{Notation for Sobolev spaces}
 
Let us introduce standard notation for Sobolev spaces on manifolds
with boundaries (see~\cite{GSchw} for details).
For $q\in\{0,\ldots,d\}$, one denotes by
$\Lambda^q\mathcal C^{\infty}(\overline{\mathsf M})$ (respectively $\Lambda^q\mathcal C^\infty_c( \mathsf M)$) the space of $\mathcal C^{\infty}$ $q$-forms on
$\overline{\mathsf M}$ (respectively on $\mathsf M$ and with compact support in $\mathsf M$). 
Moreover, the set~$\Lambda^q\mathcal C^{\infty}_T(\overline{\mathsf M})$ is the set of
$\mathcal C^{\infty}$ $q$-forms $v$ such that $\mbf tv=0$ on $\partial
\mathsf M$, where $\mathbf t$ denotes the tangential trace on forms. 
For $q\in \{0, \ldots,d\}$, $\Lambda^qL^2(\mathsf M,\mbf g_{\mathsf M})$ is the completion of the space $\Lambda^q\mathcal C^{\infty}(\overline{\mathsf M} )$ for the norm 
$$w\in \Lambda^q\mathcal C^{\infty}(\overline{\mathsf M})\mapsto
\left( \int_{\mathsf M} \vert w\vert^2 \right)^{1/2}.$$ For $m\ge 0$, one denotes by
$\Lambda^qH^m(\mathsf M,\mbf g_{\mathsf M})$ the Sobolev spaces of $q$-forms with
regularity index~$m$: $v\in \Lambda^qH^m(\mathsf M,\mbf g_{\mathsf
  M})$ if and only if for all multi-index $\alpha$ with $\vert \alpha
\vert \le m$, the $\alpha$ derivative of $v$ is
in~$\Lambda^qL^2(\mathsf M,\mbf g_{\mathsf M})$. Let us recall for a
multi-index $\alpha=(\alpha_1,\ldots,\alpha_d) \in \mathbb N^d$, $\vert \alpha
\vert= \sum_{i=1}^d \alpha_i$ and $\partial^\alpha v= \,
^t(\partial_{x_1}^{\alpha_1} v, \ldots, \partial_{x_d}^{\alpha_d} v)$. We will denote by~$\Vert . \Vert_{H^m(\mathsf M,\mbf g_{\mathsf M})}$  the norm on the  space $\Lambda^qH^m
(\mathsf M)$. Moreover $\langle
\cdot , \cdot\rangle_{L^2(\mathsf M,\mbf g_{\mathsf M})}$ denotes the scalar product in~$\Lambda^qL^2 (\mathsf M,\mbf g_{\mathsf M})$.
For $q\in\{0,\ldots,d\}$ and $m> \frac{1}{2}$, the set $\Lambda^qH^m_{\mbf T}(\mathsf M,\mbf g_{\mathsf M})$ is defined by 
$$\Lambda^qH^m_{\mbf T}(\mathsf M,\mbf g_{\mathsf M}):= \left\{v\in \Lambda^qH^m(\mathsf M,\mbf g_{\mathsf M})
  \,|\,  \mathbf  tv=0 \ {\rm on} \ \partial \mathsf M\right\}.$$
We will always explicitely indicate the dependency
on the metric $\mbf g_{\mathsf M}$ in  the notation of the Witten
Laplacians or associated quadratic forms, but often omit it in  the notation of the Sobolev
spaces and associated norms, to ease the notation.

\subsubsection{Tangential Dirichlet boundary conditions}
\label{sec:LD}
In this section, we  introduce the tangential Dirichlet Witten Laplacian and recall some of its properties. 
For $q\in\{0,\dots,d\}$,   one defines  the so-called distorted  exterior
derivative \textit{\`a la Witten}   $\mathsf d_{f,h}^{(q)}:\Lambda^{q}
\mathcal C^{\infty}(\mathsf M)\to \Lambda^{q+1}  \mathcal
C^{\infty}(\mathsf M)$  and   its formal adjoint 
$\mathsf d_{f,h}^{(q)} \, ^{*}:\Lambda^{q+1}  \mathcal
C^{\infty}(\mathsf M)\to \Lambda^{q}  \mathcal C^{\infty}(\mathsf M)$ by 
$$\mathsf d_{f,h}^{(q)}  := e^{-\frac{1}{h}f}\, h\, \mathsf d^{(q)} \, e^{\frac{1}{h}f} \, \text{ and }  \mathsf d_{f,h}^{(q)}\, ^{*} := e^{\frac{1}{h}f}\,h\,\mathsf d^{(q)}\, ^{*}\,e^{-\frac{1}{h}f},$$
where $\mathsf d^{(q)}$ is the differential  operator on $\mathsf M$
and $\mathsf d^{(q)}\, ^{*}$ is the co-differential operator on the
manifold~$\mathsf M$  equipped with the metric tensor $\mbf g_{\mathsf
  M}$. We may drop the superscript $(q)$ when the index of the form is
explicit from the context. 
The Witten Laplacian, firstly introduced in \cite{witten-82},
is then defined similarly as the Hodge Laplacian $\Delta^{(q)}_{\mbf
  H}(\mathsf M,\mbf g_{\mathsf M}):=(\mathsf d+\mathsf
d^*)^2:\Lambda^{q}  \mathcal C^{\infty}(\mathsf M )\to \Lambda^{q}
\mathcal C^{\infty}(\mathsf M )$ by
$$
\Delta_{f,h}^{(q)} (\mathsf M,\mbf g_{\mathsf M}):= (\mathsf d_{f,h}+ \mathsf d_{f,h}^*)^2= \mathsf d_{f,h}\mathsf d^{*}_{f,h}+\mathsf d^{*}_{f,h}\mathsf d_{f,h}\ :\ \Lambda^{q}  \mathcal C^{\infty}(\mathsf M )\to \Lambda^{q}  \mathcal C^{\infty}(\mathsf M).
$$
Equivalently, one has
\begin{equation}\label{dec.delta}
\Delta_{f,h}^{(q)} (\mathsf M,\mbf g_{\mathsf M})=h^2\Delta_{\mbf H}^{(q)}(\mathsf M,\mbf g_{\mathsf M})+\vert \nabla f \vert_{\mbf g_{\mathsf M}} ^2+h(\mathcal   L_{\nabla f} + \mathcal    L_{\nabla f}^*),
\end{equation}
where $\mathcal   L_{\nabla f}$ is the Lie derivative associated with
the vector field ${\nabla f}$. Here and in the following $\vert
. \vert_{\mbf g_{\mathsf M}} $   stands for the  norm in the
  tangent space associated with the metric tensor $\mbf g_{\mathsf M}$. 
Let us now introduce the   Dirichlet   realization of
$\Delta_{f,h}^{(q)}(\mathsf M,\mbf g_{\mathsf M})$ on $\Lambda^{q}
L^{2}(\mathsf M)$, following~\cite[Section 2.4]{helffer-nier-06}. 
\begin{proposition}\label{pr.defDfhD}
Let us assume that~\autoref{B} is satisfied. Let $q\in \{0,\ldots,d\}$ and $h>0$. The Friedrichs extension of the quadratic form 
  $$Q_{f,h}^{\mathsf{Di},(q)}( \mathsf M,\mbf g_{\mathsf M} ): w\in  \Lambda^qH^1_{\mbf T}(\mathsf M)\mapsto  \Vert  \mathsf d_{f,h}w \Vert_{L^2( \mathsf M)} ^2  +   \Vert\mathsf   d_{f,h}^*w  \Vert_{L^2( \mathsf M)}^2$$
  on $\Lambda^qL^2( \mathsf M )$
is denoted by $\Delta_{f,h}^{\mathsf{Di},(q)}(\mathsf M,\mbf g_{\mathsf M})$. Its domain is  
$$\mathcal D \big(\Delta_{f,h}^{\mathsf{Di},(q)}( \mathsf M ,\mbf g_{\mathsf M} )\big)=\big\{ w\in \Lambda^qH^1_{\mbf T}(\mathsf M)\cap\Lambda^qH^2( \mathsf M), \    \mbf
td^*_{f,h} w=0 \text{ on } \partial \mathsf M  \big\}.$$ 
Moreover,  $\Delta^{\mathsf{Di},(q)}_{f,h}(\mathsf M,\mbf g_{\mathsf
  M})$ is a self-adjoint operator, with compact resolvent. Finally it holds, for all Borel set $E\subset \mathbb R$ and $u\in \Lambda^qH^1_{\mbf T}(\mathsf M)$, 
 \begin{equation}\label{eq.complexeMM1}
 \pi_{E}\big (\Delta_{f,h}^{\mathsf{Di},(q+1)}(\mathsf M,\mbf g_{\mathsf M})\big )\, \mathsf  d_{f,h}u=\mathsf  d_{f,h}\, \pi_{E}\big (\Delta_{f,h}^{\mathsf{Di},(q)}(\mathsf M,\mbf g_{\mathsf M})\big ) \, u
   \end{equation}
and
 \begin{equation}\label{eq.complexeMM2}
 \pi_{E}\big (\Delta_{f,h}^{\mathsf{Di},(q-1)}(\mathsf M,\mbf g_{\mathsf M})\big )\, \mathsf  d_{f,h}^*u=\mathsf  d_{f,h}^*\, \pi_{E}\big (\Delta_{f,h}^{\mathsf{Di},(q)}(\mathsf M,\mbf g_{\mathsf M}) \big )\, u.
   \end{equation}
\end{proposition}
\noindent
Here and in the following, for a Borel set $E\subset \mathbb R$ and
$\mathsf T$ a non negative self-adjoint operator on a Hilbert
Space,~$\pi_{E}(\mathsf T)$ denotes the spectral projector associated
with $\mathsf T$ and $E$.
\noindent
The following standard lemma will be    used several times throughout this work.
\begin{lemma} \label{quadra}
Let $(\mathsf T,D\left (\mathsf T\right ))$ be a non negative self-adjoint operator on a Hilbert Space $\left(\mc H, \Vert\cdot\Vert\right)$ with associated quadratic form $q_{\mathsf T}(x)=(x,\mathsf Tx)$ whose domain is $Q\left (\mathsf T\right )$. It then holds:  
$$\forall b>0, \, \forall u\in Q\left (\mathsf T\right ), \quad \left\Vert \pi_{[b,+\infty)} (\mathsf T) \, u\right\Vert^2 \leq \frac{q_{\mathsf T}(u)}{b}.$$
\end{lemma}
%
Generally speaking, a $\mathcal H$-normalized element $u\in D\left (\mathsf T\right )$
such that $\Vert \pi_{[b,+\infty)} (\mathsf T) \, u \Vert$ is small is
called a quasi-mode for the spectrum   in $[0,b]$ of $\mathsf T$.

 The objective of this section is to count
the number of  eigenvalues smaller than $ch$ (for some $c>0$)  of
$\Delta_{f,h}^{\mathsf{Di},(q)}( \mathsf M,\mbf g_{\mathsf M} )$,
namely to identify the dimension of the range of $\pi_{[0,\mathsf
  ch]}\big (\Delta_{f,h}^{\mathsf{Di},(q)}( \mathsf M,\mbf g_{\mathsf
  M} ) \big)$, for $h$ sufficiently small.

\subsection{Number of small eigenvalues of
  $\Delta_{f,h}^{\mathsf{Di},(q)}( \mathsf M,\mbf g_{\mathsf M} )$}
\label{sec.mainresult_petitevp}

Before stating the main result of Section~\ref{sec.comptevp}, let us
introduce a few more notation.  Let us assume that \autoref{B} holds. Let $z\in \pa \mathsf M$ be a critical point of $f$ (i.e. $\vert \nabla f(z)\vert =0$). Then, $z$ is a critical point of   $f|_{\pa \mathsf M}$ and  
  the unit outward normal~$\mathsf n_{\mathsf M}(z)$  to $ \mathsf M$
  at $z$  (see item 2 in  Lemma~\ref{le.start}) is an eigenvector with
  the associated eigenvalue:
  \begin{equation}\label{eq.lambdad}
  \mu_z = \, ^t \mathsf n_{\mathsf M}(z) \text{ Hess } f(z) \, \mathsf n_{\mathsf M}(z).
    \end{equation}
Let us now  introduce  the  set of so-called  generalized critical
points of~$f$ for the operator $\Delta_{f,h}^{\mathsf{Di}}( \mathsf
M,\mbf g_{\mathsf M} )$, which can be seen intuitively as critical
points for the function $f$ extended by $-\infty$ outside
$\overline{\mathsf M}$. 
For $q\in \{0,\ldots,d\}$, the standard critical points with index $q$
in $\mathsf M$ are:
 $$
 \mathsf U_q^{\mathsf M}= \big \{x \in  \mathsf M, \,  x \text{ is a
   critical point of  } f   \text{ of index } q \big \}
$$
with cardinal $\mathsf m_{q}^{\mathsf M}={\rm Card}\big (\mathsf
U_q^{\mathsf M}\big)$. Two additional sets of generalized critical
points with index $q$ on~$\partial \mathsf M$ should be considered. First, let us introduce
\begin{equation}\label{eq.PSGENE}
 \mathsf U_q^{\pa  \mathsf M,1}= \big \{z \in \pa \mathsf M, \,  z \text{ is a critical point of  } f|_{\pa \mathsf M} \text{ of index } q-1 \, \text{ and }  \,  \partial_{\mathsf n_{ \mathsf M}}f(z)>0\big \},
\end{equation}
with cardinal $
 \mathsf m_q^{\pa  \mathsf M,1}={\rm Card}\big (  \mathsf U_q^{\pa
   \mathsf M,1} \big)$, and
 with the convention that $
\mathsf  U_0^{\pa \mathsf M,1}=\emptyset$
 for $q=0$. Second, one defines,
\begin{equation}\label{eq.PS-nouveau}
 \mathsf U_q^{\pa  \mathsf M,2}=\big \{z \in \pa \mathsf M, \, \vert \nabla f(z)\vert =0,  \,  z \text{ is a critical point of  } f|_{\pa \mathsf M} \text{ of index } q-1 \text{ and }  \mu_z  <0\big \},
\end{equation}
with cardinal $\mathsf  m_q^{\pa  \mathsf M,2}={\rm Card}\big (
\mathsf U_q^{\pa  \mathsf M,2}\big)$, and  with again the convention that $
 \mathsf U_0^{\pa  \mathsf M,2}=\emptyset$ for $q=0$. 
 Finally, one defines the total number
 of generalized critical points with index $q$:
\begin{equation}\label{numb-generalize}
 \mathsf m_q=\mathsf  m_q^{ \mathsf M}+\mathsf m_q^{\pa \mathsf M,1}+\mathsf m_q^{\pa  \mathsf M,2}.
\end{equation}

\noindent
Let us now state the  main result of this section. 
\begin{theorem}\label{thm.main1}
Let us assume that~\autoref{B} holds. Then, for all $q\in \{0,\ldots,d\}$, there exists $\mathsf c>0$ and $h_0>0$ such that for all $h\in (0,h_0)$: 
$$\dim \Ran \, \pi_{[0,\mathsf ch]}\big (\Delta_{f,h}^{\mathsf{Di},(q)}( \mathsf M,\mbf g_{\mathsf M} )\big )=\mathsf  m_q  \ \text{ where $\mathsf m_q$ is defined by~\eqref{numb-generalize}}.$$
\end{theorem}
\noindent
Let us mention that this result is proved in~\cite{DoNe2} for $q=0$ under a weaker assumption than~\autoref{B}. 
%



The proof of Theorem~\ref{thm.main1}, inspired from~\cite[Section
3]{helffer-nier-06} and~\cite{CFKS}, consists in finding where the $L^2(\mathsf
M,\mbf g_{\mathsf M})$-norms of  eigenforms associated with eigenvalues
of order $o(h)$  concentrate in $\overline{\mathsf M}$,  and in
determining a finite dimensional linear space close to them. We first
study in
Section~\ref{sec.Rdmoins} model problems on $\mathbb R^{d}_-$ where $$\mathbb
R^d_-=\{x=(x',x_d),\  \ x'=(x_1,\ldots,x_{d-1})\in \mathbb R^{d-1}, \,
x_d\in \mathbb R, \, x_d<0\},$$
 before providing the proof of Theorem~\ref{thm.main1} in Section~\ref{sec.numb-small}.  

\begin{remark}\label{rem:numb-small}
Let us mention that the main difference with~\cite[Chapter 3]{helffer-nier-06}
is that we cannot use a block-diagonalization of the metric $\mbf
g_{\mathsf M}$ and of the function $f$ near the critical points 
in~$\pa \mathsf M$, which would lead to an exact tensorization into a Witten Laplacian in a variable $x'\in \pa \mathsf M$ and  a Witten
Laplacian in a  variable $x_d\in \mathbb R_-$.  
We actually only decompose the metric $\mbf g_{\mathsf
  M}$  in a  local system of coordinates near the critical point,
constructed with the   geodesic distance to the boundary. Then, using the
fact that $\partial_{\mathsf n_{\mathsf
    M}}f=0$ near critical points on $\partial \mathsf M$,  it appears that a local
asymptotic expansion of $f$ in these coordinates is precise enough to  count the number of small eigenvalues.
%
 \end{remark}

\begin{remark}\label{rem:Morse}
A simple consequence of the above results is the following finite
dimensional Dirichlet complex strutures for Witten Laplacians on
bounded domains under the assumption~\autoref{B}:
$$\{0\} \longrightarrow
\Ran\,\pi_{[0,ch]}(\Delta^{\mathsf{Di},(0)}_{f,h}(\mathsf M,\mbf
g_{\mathsf M}))  \xrightarrow{ \mathsf d_{f,h}  } \,\cdots\,
\xrightarrow{ \mathsf d_{f,h}  } \Ran\,\pi_{[0,ch]}(\Delta^{\mathsf{Di},(d)}_{f,h}(\mathsf M,\mbf g_{\mathsf M})) \xrightarrow{ \mathsf d_{f,h} }\{0\},$$
$$\{0\} \xleftarrow{\mathsf d_{f,h}^* }
\Ran\,\pi_{[0,ch]}(\Delta^{\mathsf{Di},(0)}_{f,h}(\mathsf M,\mbf g_{\mathsf
  M}))  \xleftarrow{\mathsf d_{f,h} ^*} \cdots\xleftarrow{\mathsf
  d_{f,h} ^*} \Ran\,\pi_{[0,ch]}(\Delta^{\mathsf{Di},(d)}_{f,h}(\mathsf
M,\mbf g_{\mathsf M})) \longleftarrow \ \{0\},$$
which, combined with Theorem~\ref{thm.main1}, yields strong Morse
inequalities. This generalizes standard results for the Witten Laplacians in the
full domain~\cite{helffer-sjostrand-85,CFKS,witten-82} or on bounded domain without critical points on
the
boundary~\cite{helffer-nier-06,le-peutrec-10} (see also~\cite{laudenbach2011morse,CL}).
\end{remark}

 \subsection{Number of small eigenvalues of Witten Laplacians in  $\mathbb R^d_-$}\label{sec.Rdmoins}

The goal  of this section is to
 count the number of small eigenvalues of $\Delta_{f,h}^{\mathsf{Di},(q)} (
 \mathbb R^d_-,\mbf g )$ in a simple geometric setting (in particular
 $f$ has a single critical point, located at $0$). The main result
 (Proposition~\ref{pr.prop-loca1-case1}) is
 stated in Section~\ref{sec.Rdmoins_main}. The proof is done in three
 steps: we first recall well-known results for Witten Laplacians in
 $\mathbb R^{d-1}$ in Section~\ref{sec.Rdmoins_d-1}; then we prove Proposition~\ref{pr.prop-loca1-case1}
in a simplified setting in Sections~\ref{sec.Rdmoins-simplified}
and~\ref{sec:prop-loca1-proof}; and we finally conclude with the proof
of Proposition~\ref{pr.prop-loca1-case1} in Section~\ref{sec.prop-loca1-case1}.


\subsubsection{Witten Laplacian in $\mathbb R^d_-$ with tangential
  Dirichlet boundary conditions}\label{sec.Rdmoins_main}
Let us first introduce  the tangential
  Dirichlet Witten
Laplacian~$\Delta_{f,h}^{(q)} ( \mathbb R^d_- ,\mbf g)$ in $ \mathbb
R^d_-$, under two sets of assumptions. 
 \begin{manualassumption}{\bf(Metric-$\mathbb R^{d}_-$)}\label{MetricRd-}
 The space $\overline{\mathbb R^d_-}$  is endowed with  a  metric
 tensor $\mbf g$  satisfying the following:
 \begin{enumerate}
\item [(i)] $\mbf g$ writes, for some  $\mathcal C^\infty$ function
  $\mbf G$ on  $\overline{\mathbb R^d_-}$,
\begin{equation}\label{eq.outc2}
\mbf g(x) = {\mbf G}(x',x_d)dx'^2+dx_d^2 \,,
\end{equation}
with ${\mbf G}(0,0)$ the identity matrix.
 \item [(ii)] ${\mbf G}$ and all its derivatives are bounded over $\overline{\mathbb R^d_-}$.
\item [(iii)] ${\mbf G}$ is uniformly elliptic over $\overline{\mathbb R^d_-}$.  
\end{enumerate}

 \end{manualassumption}
\noindent
To ease the notation, we will not indicate explicitely the metric ${\mbf G}$ in the
functional spaces nor in the associated norm: we will simply write $ \Lambda^qH^k( \mathbb
R^{d}_-)$ (resp. $\Lambda^qH^1_{\mbf T}( \mathbb R^d_-)$) for
$\Lambda^qH^k( \mathbb R^{d}_-,\mbf g )$ (resp. $\Lambda^qH^1_{\mbf
  T}( \mathbb R^d_-,\mbf g )$), and denote by  $ \Vert  . \Vert_{H^k(
  \mathbb R^{d}_-)}$  the associated norm.

Notice that under \autoref{MetricRd-},  the norm on  $(\mathbb R^d_-,
\mbf g)$ is uniformly equivalent to the norm on $(\mathbb R^d_-,
\mathsf I_{d}  \,  dx^2)$  (where
$\mathsf I_{d}$ is the identity matrix of size $d$), which is simply denoted by $\vert x \vert$: $\vert x\vert^2=\sum_{i=1}^{d}x_i^2$. 
Moreover, for all $q\in \{0,\ldots,d\}$ and $k\ge 0$,  the norm on $\Lambda^qH^k( \mathbb R^d_-,\mbf g )$ is equivalent to the norm on $\Lambda^qH^k( \mathbb R^d_-, \mathsf I_d dx^2)$, and 
$\Lambda^qH^1_{\mbf T}( \mathbb R^d_-,\mbf g )=\Lambda^qH^1_{\mbf T}( \mathbb R^d_-, \mathsf I_{d}  \,  dx^2).$
 \begin{manualassumption}{\bf(Potential-$\mathbb R^{d}_-$)}\label{PotentialRd-}
 The function $f:\overline{\mathbb R^{d}_-}\to  \mathbb R$ satisfies:  \begin{enumerate}
 \item [(i)] $f$ is a $\mathcal C^\infty$ function such that for all multi-index $\alpha \in \mathbb N^d$ with  $\vert \alpha\vert \ge 1$,  $\sup_{\overline{\mathbb R^d_-}}\vert \partial_x^\alpha f\vert<+\infty$. 
 \item [(ii)]  The point $0$ is the only critical point of $f$ in $\overline{\mathbb R^d_-}$ and    is a non degenerate critical point of $f$  (this condition is independent of the metric tensor  on $\mathbb R^{d}_-$). Moreover,  there exist $R>0$ and $c>0$ such that:
 \begin{equation}\label{eq.grad>0}
\forall x \in \mathbb R^d_-, \, \vert x \vert \ge R \Longrightarrow  \vert \nabla f\vert (x)\ge c.
      \end{equation}

\item[(iii)] It holds:
 \begin{equation}\label{eq.stable}
 \forall x'\in \mathbb R^{d-1}, \  \partial_{\mathsf n_{\mathbb R^d_-}}f(x',0)=0.  
  \end{equation}
\end{enumerate}

  \end{manualassumption}
  \noindent
Notice that thanks to~\eqref{eq.outc2}, for any $\phi \in
\Lambda^0\mathcal C^1(\overline{\mathbb R^d_-})$, one has:
\begin{equation}\label{eq.dec-gradient-phi}
\forall  x'\in   \mathbb R^{d-1},\       \partial_{\mathsf n_{\mathbb R_-^d} }\phi (x',0)=\partial_{x_d}\phi(x',0).
\end{equation} 
Moreover, under the above assumptions, up to an orthogonal transformation
on $x'$ (which preserves the fact that $\mbf G(0,0)$ is the identity matrix), one can assume that the Hessian
matrix of $f|_{\pa \mathbb R^{d}_-}$ at $0\in \mathbb R^{d-1}$ is
diagonal. As a consequence,  there exists   a neighborhood $\mathsf V_0$  of $0$ in $\overline{\mathbb R^d_-}$ and  $(\mu_1,\ldots,\mu_{d})\in (\mathbb R^*)^d$ such that:
\begin{equation}\label{eq.decf}
 \forall x=(x_1,\ldots,x_d)\in \mathsf V_0,\ \,  f(x)= f(0)+ 
 \sum_{i=1}^d \frac{\mu_i}{2} x_i^2 +O(|x|^3)
\end{equation}
where $(\mu_1, \ldots, \mu_d)$ are the eigenvalues of $\hess
f(0)$. More precisely, $\mu_d=\partial_{x_d,x_d} f(0)$, and
$(\mu_1,\ldots,\mu_{d-1})$ are the eigenvalues of $\hess f|_{\pa \mathbb R^{d}_-}(0)$.

We will need the following standard results on the operator
$\Delta_{f,h}^{\mathsf{Di},(q)}( \mathbb R^d_- ,\mbf g)$.
 \begin{proposition}\label{pr.pr1-d} 
  Let us assume that~\autoref{MetricRd-}  and item $(i)$ in
  \autoref{PotentialRd-} are  satisfied.  Let $q\in \{0,\ldots,d\}$
  and $h>0$ be fixed.  The
Friedrichs extension of the quadratic form  
\begin{equation}\label{eq.fried-e}
Q_{f,h}^{\mathsf{Di},(q)}( \mathbb R^d_-,\mbf g ): w\in \Lambda^qH^1_{\mbf T}( \mathbb R^d_- )
 \mapsto    \Vert  \mathsf  d_{f,h}w \Vert_{L^2( \mathbb R^d_-  )} ^2  +    \Vert  \mathsf  d_{f,h}^*w \Vert_{L^2( \mathbb R^d_-  )}^2.
 \end{equation}
on $\Lambda^qL^2( \mathbb R^d_-)$ is denoted by
$\Delta_{f,h}^{\mathsf{Di},(q)}( \mathbb R^d_- ,\mbf g)$. It is a
self-adjoint operator with domain
  $$\mathcal D \big(\Delta_{f,h}^{\mathsf{Di},(q)}( \mathbb R^d_- ,\mbf
g) \big)=\big\{w\in  \Lambda^qH^1_{\mbf T}( \mathbb R^d_- )\cap  \Lambda^qH^2( \mathbb R^d_- ), \ \mbf t \mathsf d_{f,h}^*w=0 \text{ on }  \pa \mathbb R^d_-\big\}.$$
Moreover, it holds, for all Borel set $E\subset \mathbb R$ and $u\in
\Lambda^qH^1_{\mbf T}( \mathbb R^d_- )$,
 \begin{equation}\label{eq.complexe1}
 \pi_{E}\big (\Delta_{f,h}^{\mathsf{Di},(q+1)}( \mathbb R^d_- ,\mbf g )\big )\, \mathsf  d_{f,h}u=\mathsf  d_{f,h}\, \pi_{E}\big (\Delta_{f,h}^{\mathsf{Di},(q)}( \mathbb R^d_- ,\mbf g )\big ) \, u
   \end{equation}
and
 \begin{equation}\label{eq.complexe2}
 \pi_{E}\big (\Delta_{f,h}^{\mathsf{Di},(q-1)}( \mathbb R^d_- ,\mbf g )\big )\, \mathsf  d_{f,h}^*u=\mathsf  d_{f,h}^*\, \pi_{E}\big (\Delta_{f,h}^{\mathsf{Di},(q)}( \mathbb R^d_- ,\mbf g )\big ) \, u.
   \end{equation}
 \end{proposition}

The following Green
formula will be used many times in the sequel (it can be proven as in the compact case~\cite[Lemma
2.3.2]{helffer-nier-06} by density of $\Lambda^q\mathcal C^\infty_c( \mathbb R^d_-)$ in
 $\Lambda^qH^1(\mathbb R^d_-)$).
\begin{lemma}\label{le.Green}
Let $q\in \{0,\ldots,d\}$. Let us assume that~\autoref{MetricRd-} is satisfied. 
Then, for all $w\in \Lambda^qH^1_{\mbf T}(\mathbb R^d_-)$, it holds:
\begin{align*}
  \Vert  \mathsf  d_{f,h}w  \Vert_{L^2( \mathbb R^d_-   )} ^2  +    \Vert  \mathsf  d_{f,h}^*w  \Vert_{L^2( \mathbb R^d_- )}^2&=h^2  \Vert \mathsf  d \,  w   \Vert_{L^2( \mathbb R^d_-  )} ^2  +  h^2  \Vert  \mathsf d ^*\, w   \Vert_{L^2( \mathbb R^d_-   )}^2 \\
    &\quad+\big  \lp  w, \big(\vert \nabla f\vert_{\mbf g}^2 +h (\mathcal L_{\nabla f}+\mathcal L_{\nabla f}^* )\big)\, w\big \rp_{L^2( \mathbb R^d_-  )}\\
    &\quad -h\int_{\pa \mathbb R^d_-}\langle w,w\rangle _{T_{(x',0)}\pa \mathbb R^d_-}\partial_{\mathsf n_{\mathbb R^d_-}}f(x',0) \lambda (dx') 
    \end{align*}
    where   $\lambda (dx')$ is of course the volume form on $\pa \mathbb R^d_-$ induced by the  metric tensor $\mbf g$. 
    \end{lemma}

\noindent
Let us now state the main result of this section (recall the
definition~\eqref{eq.decf} of $\mu_d$). 
   \begin{proposition}\label{pr.prop-loca1-case1}
  Let us assume that~\autoref{MetricRd-} and \autoref{PotentialRd-}
  hold.  Let  $q\in \{0,\ldots,d\}$.  Then, there exist $C>0$, $c>0$,
  and $h_0>0$ such that for all $h\in (0,h_0)$, the following holds:
  \begin{enumerate}
  \item[(i)] 
  If $q=0$, then:
  \begin{equation}\label{eq.last1-case1}
\forall w\in \Lambda^0 H^1_{\mbf T}( \mathbb R^d_-), \ \,  Q_{f,h}^{\mathsf{Di},(0)}( \mathbb R^d_- ,\mbf g)( w)\ge Ch\,  \Vert w  \Vert_{L^2(\mathbb R^d_-)}^2.
\end{equation}
Let  $q\in \{1,\ldots,d\}$. If the index of $0$ as a critical point of
$f|_{\pa \mathbb R^d_-}$ is not $q-1$ or if  $\mu_{d}>0$, then:
 \begin{equation}\label{eq.last2-case1}
  \forall w\in \Lambda^q H^1_{\mbf T}( \mathbb R^d_- ), \ \,  Q_{f,h}^{\mathsf{Di},(q)}( \mathbb R^d_- ,\mbf g)( w)\ge Ch\,  \Vert w  \Vert_{L^2(\mathbb R^d_-)}^2.
 \end{equation}
If the index of $0$ as a critical point of $f|_{\pa   \mathbb R^d_-}$
is $q-1$ and $\mu_{d}<0$, then:
 \begin{equation}\label{eq.last3-case1}
  \Ran \, \pi_{[0,ch]}\big (\Delta_{f,h}^{\mathsf{Di},(q)}( \mathbb R_-^d,\mbf g )\big )=\Ker \Delta_{f,h}^{\mathsf{Di},(q)}( \mathbb R_-^d,\mbf g ) \text{ has dimension  } 1.
 \end{equation}
   \item[(ii)]  Assume that the index of $0$ as a critical point of $f|_{\pa   \mathbb R^d_-}$
is $q-1$ and $\mu_{d}<0$. Let  $\chi: \overline{\mathbb R^{d}_-}\to [0,1]$ be a $\mathcal C^\infty$ function supported in a neighborhood of $0$ which equals~$1$ in a neighborhood of $0$. 
Let  $\Psi_h\in \Ker \Delta_{f,h}^{\mathsf{Di},(q)}( \mathbb R_-^d,\mbf g )$ such that $ \Vert \Psi_h \Vert_{L^2(\mathbb R^d_-)}=1$. Then, 
 in the limit $h\to 0$, it holds:
 \begin{equation}\label{eq.norme-case1}
  \ \Vert \chi \Psi_h  \Vert_{L^2(\mathbb R^d_-)} =1+O  (h^{2}  ) \ \text{ and } \  Q_{f,h}^{\mathsf{Di},(q)}( \mathbb R^d_- ,\mbf g)( \chi \Psi_h)=O(h^{2}).
  \end{equation}
  \end{enumerate}
 \end{proposition}
%
%

The next Lemma shows that is it enough to
prove~\eqref{eq.last1-case1}--\eqref{eq.last2-case1}   in
Proposition~\ref{pr.prop-loca1-case1} for   forms $w$ supported in a
ball $\mathsf B(0,h^{2/5})$.
 \begin{lemma}\label{le.IMS-loca}
 Let us assume that~\autoref{MetricRd-} and \autoref{PotentialRd-} are
 satisfied. Let us assume that there  exist $C>0$ and $h_0>0$ such
 that for all $h\in (0,h_0)$ and for all $v\in \Lambda^q H^1_{\mbf T}( \mathbb R^d_-  )$ supported in $\mathsf B(0,h^{2/5})$, 
  \begin{equation}\label{eq.cond1Q}
  Q_{f,h}^{\mathsf{Di},(q)}( \mathbb R^d_-,\mbf g )( v)\ge Ch\,    \Vert v \Vert_{L^2(\mathbb R^d_- )}^2.
  \end{equation}
    Then, there exist $c>0$ and $h_0>0$ such that for all $h\in
    (0,h_0)$ and for all  $w\in \Lambda^q H^1_{\mbf T}( \mathbb R^d_- )$,
        \begin{align*} 
    Q_{f,h}^{\mathsf{Di},(q)}(\mathbb R^d_-,\mbf g)(   w )\ge Ch  \Vert   w  \Vert_{L^2( \mathbb R_-^d )}^2.
        \end{align*}
 
 \end{lemma}

 \begin{proof}
  Let us   consider a partition of unity $(\chi_1,\chi_2)$  such that
  $\chi_1\in \mathcal C_c^\infty (\overline {\mathbb R_-^d})$,
  $\chi_1=1$ on $\mathsf B(0,1/2)$, supp$\chi_1\subset  \mathsf
  B(0,1)$ and $\chi_1^2+\chi_2^2=1$. The IMS
  formula~\cite{CFKS,helffer-nier-06} yields: for all $  w\in \Lambda^qH_{\mbf T}^1(\mathbb R^d_-)$,
    \begin{align}
\label{eq.IMS}
Q_{f,h}^{\mathsf{Di},(q)}(\mathbb R^d_-,\mbf g )(   w ) &=\sum_{k=1}^2 Q_{f,h}^{\mathsf{Di},(q)}( \mathbb R^d_- ,\mbf g )(\chi_k(h^{-2/5  }.)w)    -   h^2\,\big \Vert \nabla\big[ \chi_k(h^{-2/5  }.) \big]w\big \Vert_{L^2( \mathbb R_-^d)}^2.
\end{align}
Using Lemma~\ref{le.Green}, $\partial_{\mathsf n_{\mathbb
    R^{d}_-}}f=0$ on $\partial \mathbb R^{d}_-$ (see~\eqref{eq.stable}) and $\mbf t   (\chi_k(h^{-2/5  }.)w)(x',0)=0$, one has:
    \begin{align*}
    Q_{f,h}^{\mathsf{Di},(q)}( \mathbb R^d_-,\mbf g  )(\chi_2(h^{-2/5  }.)w) &=h^2\big \Vert   \mathsf  d \big(\chi_2(h^{-2/5  }.)w\big)\big \Vert_{L^2( \mathbb R^d_-)} ^2  +  h^2\big \Vert  \mathsf d ^*\big(\chi_2(h^{-2/5  }.)w\big)\big \Vert_{L^2( \mathbb R^d_-  )}^2 \\
    &\quad+\big  \lp \chi_2(h^{-2/5  }.)w, \Big(\vert \nabla f\vert^2   -h\big (\mathcal L_{\nabla f}+\mathcal L_{\nabla f}^*\big)\Big)\chi_2(h^{-2/5  }.)w\big \rp_{L^2( \mathbb R^d_-  )}.
        \end{align*} 
Moreover, using \autoref{MetricRd-} and item $(ii)$ in~\autoref{PotentialRd-},  there exists $C>0$ such that,
\begin{equation}\label{eq.gradientup}
\forall x\in \mathbb R^d_-\setminus \mathsf B(0,h^{2/5  }/2), \ \vert \nabla f(x)\vert^2_{{\mbf g}}   \ge Ch^{\frac 45},
\end{equation}
 for some $C>0$ independent of $h$. Thus, using the fact that $\mathcal L_{\nabla f}+\mathcal L_{\nabla f}^*$ is a $0$th order operator and ${\rm supp}\, \chi_2(h^{-2/5  }.)\subset \mathbb R^d\setminus \mathsf  B(0,h^{2/5  }/2)$, one obtains that there exist $c>0$ and $h_0>0$ such that for all $h\in (0,h_0)$: 
\begin{equation}\label{eq.chi_2-Q}
Q_{f,h}^{\mathsf{Di},(q)}(\mathbb R^d_-,\mbf g )(  \chi_2(h^{-2/5  }.)w)\ge Ch^{4/5}    \Vert \chi_2(h^{-2/5  }.) w  \Vert_{L^2( \mathbb R_-^d)}^2.
\end{equation}
  This implies that there exist $c>0$ and  $C>0$ such that 
    \begin{align*}
    Q_{f,h}^{\mathsf{Di},(q)}(\mathbb R^d_-,\mbf g )(   w )&\ge Q_{f,h}^{\mathsf{Di},(q)}(\mathbb R^d_-,\mbf g )(  \chi_1(h^{-2/5  }.)w)+Ch^{4/5}    \Vert \chi_2(h^{-2/5  }.) w  \Vert_{L^2( \mathbb R_-^d)}^2  \\
    &\quad-Ch^{6/5}   \Vert   w  \Vert_{L^2( \mathbb R_-^d)}^2.
    \end{align*}
    If~\eqref{eq.cond1Q} holds, 
    one obtains (taking $v=\chi_1(h^{-2/5  }.)  w$ in \eqref{eq.cond1Q}) for all $h$ small enough:
    \begin{align*} 
    Q_{f,h}^{\mathsf{Di},(q)}(\mathbb R^d_-,\mbf g )(   w )&\ge C\Big(  h \Vert \chi_1(h^{-2/5  }.)  w  \Vert_{L^2( \mathbb R_-^d )}^2\!\!+ h^{4/5}   \Vert   \chi_2(h^{-2/5  }.)  w  \Vert_{L^2( \mathbb R_-^d  )}^2 \!\!-  h^{6/5}   \Vert   w  \Vert_{L^2( \mathbb R_-^d  )}^2\Big).
        \end{align*}
        Thus, $Q_{f,h}^{\mathsf{Di},(q)}(\mathbb R^d_-,\mbf g )(   w )\ge Ch  \Vert   w  \Vert_{L^2( \mathbb R_-^d)}^2$. This ends the proof of Lemma~\ref{le.IMS-loca}.
 \end{proof}
The proof of Proposition~\ref{pr.prop-loca1-case1} will be done in
Section~\ref{pr.prop-loca1-case1}, after considering successively
model problems on $\mathbb R^{d-1}$, and on $\mathbb R^d_-$ in a
simplified setting.

\subsubsection{Witten Laplacian in $\mathbb R^{d-1}$}\label{sec.Rdmoins_d-1}
Let us first recall standard results   on the number of  eigenvalues of
order $o(h)$ for the  Witten Laplacian on  $\mathbb R^{d-1}$ associated with a function~$f_+:\mathbb R^{d-1}\to \mathbb R$ which has only one critical point in~$\mathbb R^{d-1}$. 
Let us introduce the two sets of assumptions used to state this result.
\begin{manualassumption}{\bf(Metric-$\mathbb R^{d-1}$)}\label{MetricRd-1}
The space  ${\mathbb R^{d-1}}$ is endowed with a~$\mathcal C^\infty$   metric tensor  denoted by $x'\in \mathbb R^{d-1} \mapsto    \tilde {\mbf G}(x')\, dx'^2$. In addition,
\begin{enumerate}
\item [(i)] $\tilde {\mbf G}$ and all its derivatives are bounded over ${\mathbb R^{d-1}}$. 
\item [(ii)] $ \tilde {\mbf G}$ is uniformly elliptic over $ {\mathbb
    R^{d-1}}$, i.e. $ \tilde {\mbf G}\ge c$  over $\mathbb R^{d-1}$, for some $c>0$.
\end{enumerate}
 \end{manualassumption}
 \noindent
Again, we will not indicate explicitely the metric $\tilde {\mbf G}$ in the
functional spaces nor in the associated norm: we will simply write $\Lambda^q H^k( \mathbb
R^{d-1})$ for $\Lambda^qH^k( \mathbb R^{d-1}, \tilde {\mbf G} \, dx'^2
)$ and denote $ \Vert  \cdot  \Vert_{H^k(  \mathbb R^{d-1})}$ the
associated norm.

Notice that, as above, under \autoref{MetricRd-1},  the norm on  $(\mathbb
R^{d-1},   \tilde {\mbf G} \, dx'^2)$ is uniformly equivalent to the
norm on $(\mathbb R^{d-1},   \mathsf I_{d-1}  \,  dx'^2)$, the latter
being simply denoted $\vert x'\vert$: $\vert x'\vert^2=\sum_{i=1}^{d-1}x_i^2$.
In addition,  for all $q\in \{0,\ldots,d-1\}$ and $k\ge 0$,  the norm
on  $\Lambda^qH^k( \mathbb R^{d-1},  \tilde {\mbf G} \, dx'^2 )$ is
equivalent to the norm on   $\Lambda^qH^k( \mathbb R^{d-1}, \mathsf I_{d-1}  \,  dx'^2)$.  
 \begin{manualassumption}{\bf(Potential-$\mathbb R^{d-1}$)}\label{PotentialRd-1}
 The function $f_+:  {\mathbb R^{d-1}}\to \mathbb R$ satisfies: 
 \begin{enumerate}
 \item[(i)] $f_+$ is a $\mathcal C^\infty$ function such that for all multi-index
 $\beta\in \mathbb N^{d-1}$ with  $\vert \beta\vert \ge 1$,   $\sup_{\mathbb R^{d-1}}\vert \partial_x^\beta f_+\vert<+\infty$.
 \item[(ii)] 
The point  $0$ is the only critical point of $f_+$ in ${\mathbb
  R^{d-1}}$ and is a non degenerate critical point, with an index
denoted by $p\in
\{0,\ldots,d-1\}$ (the non-degeneracy and the index do not depend on the metric tensor
on $\mathbb R^{d-1}$). Moreover, there exist  $R>0$ and $c>0$ such that:
 $$ 
\forall x'\in  \mathbb R^{d-1},  \, \vert x'\vert \ge R \Longrightarrow \vert \nabla f_+(x')\vert  \ge c .$$
 \end{enumerate}

  \end{manualassumption}

Under these assumptions, the following result holds (see~\cite[Propositions 3.3.2 and
 3.3.3]{helffer-nier-06} and\cite[Proposition 2.2]{HKN} for proofs of very similar results).
 \begin{proposition}\label{pr.case-d-1}
  Let $d\ge 2$ and   assume that~\autoref{MetricRd-1} and
  \autoref{PotentialRd-1} hold.     Let $q\in \{0,\ldots,d-1\}$ and $h
  > 0$. The Friedrichs extension of the quadratic form
$$Q_{f_+,h}^{(q)}( \mathbb R^{d-1},\tilde {\mbf G}\, dx'^2 ):  w\in
\mathcal D\big(\Delta_{f_+,h}^{(q)}( \mathbb R^{d-1} , \tilde {\mbf
  G}\, dx'^2) \big ) \mapsto    \Vert  \mathsf  d_{f,h}w \Vert_{L^2(
  \mathbb R^{d-1} )} ^2  +    \Vert  \mathsf  d_{f,h}^*w \Vert_{L^2(
  \mathbb R^{d-1})}^2$$
on $\Lambda^q L^2( \mathbb R^{d-1})$ is
denoted by $\Delta_{f_+,h}^{(q)}( \mathbb R^{d-1},  \tilde {\mbf
  G}dx'^2 )$. It is a self-adjoint operator with domain 
$$\mathcal D\big(\Delta_{f_+,h}^{(q)}( \mathbb R^{d-1},  \tilde {\mbf
  G}dx'^2 )\big)=\Lambda^q H^2(\mathbb R^{d-1}).$$
  Moreover,   there exist $C>0$, $c>0$  and $h_0>0$ such that for all $h\in (0,h_0)$:
      \begin{itemize}  
   \item[(i)] $\inf\sigma_{\text{ess}}\big(\Delta_{f_+,h}^{(q)}( \mathbb R^{d-1},  \tilde {\mbf G}\, dx'^2 )\big)\ge C$. 
    \item[(ii)] When $p\neq q$, 
$\dim \Ran \, \pi_{[0,ch]}\big (\Delta_{f_+,h}^{(q)}( \mathbb R^{d-1},  \tilde {\mbf G}\, dx'^2 )\big )=0.$\\
When $p=q$, $\Ran \, \pi_{[0,ch]}\big (\Delta_{f_+,h}^{(q)}( \mathbb R^{d-1},  \tilde {\mbf G}\, dx'^2 )\big )=\Ker \Delta_{f_+,h}^{(q)}( \mathbb R^{d-1},  \tilde {\mbf G}\, dx'^2 )$ has dimension~$1$. 
    \end{itemize}
 \end{proposition}

\subsubsection{A  simplified model  in $\mathbb R^d_-$}\label{sec.Rdmoins-simplified}
We will first prove item $(i)$ of Proposition~\ref{pr.prop-loca1-case1} in the special
case when $\mbf G(x',x_d)$ in item $(i)$ of \autoref{MetricRd-}  is independent of the variable $x_d$.
%
  \begin{proposition}\label{pr.prop-loca1}
  Assume that~\autoref{MetricRd-} and \autoref{PotentialRd-}  are satisfied. 
  Assume in addition that $\mbf G$ is independent of $x_d$:
\begin{equation}\label{eq.metric-indep}
\forall (x',x_d) \in \overline{ \mathbb R^d_-},\,  \mbf G(x',x_d)=\tilde {\mbf G}(x').
  \end{equation}
for some $\mathcal C^\infty$ function $\tilde{\mbf G}$ defined on $\mathbb R^{d-1}$.
  Then, item (i) in Proposition~\ref{pr.prop-loca1-case1}   is satisfied. 
 \end{proposition}
Before providing the proof of this proposition in
Section~\ref{sec:prop-loca1-proof}, let us conclude
this section with a few preliminary results. Notice first that when \autoref{MetricRd-} and~\eqref{eq.metric-indep}, are
satisfied, $\tilde{\mbf G}(x') dx'^2$ satisfies
\autoref{MetricRd-1}. Moreover, we will need the following decomposition of the
function $f$.
 \begin{definition}\label{de.pre1}
Assume that \autoref{PotentialRd-} is satisfied, and recall the
expansion~\eqref{eq.decf} of $f$ around~$0$.
  Let us define $f_+$ and $f_-$  by:
\begin{equation}\label{eq.plusmoins}
\forall x=(x',x_d)\in \mathsf V_0,\, f_+(x')=  \sum_{i=1}^{  d-1}\frac{\mu_i}{2} \, x_i^2 \ \text{ and } \ f_-(x_d)=-\frac{\mu_{d}}{2} x_d^2.
\end{equation}
Let us then extend the function $f_+$ (resp. $f_-$) to a $C^\infty$
function over $ {\mathbb R^{d-1}}$ (resp. $\overline{\mathbb R_-}$)
such that:
\begin{enumerate}
\item All the derivatives of $f_+$ (resp. $f_-$) of order at least $1$
 are  bounded over   $ {\mathbb R^{d-1}}$  (resp. $\overline{\mathbb R_-}$) .
 \item  The point  $0$ is the only critical point of $f_+$ (resp. $f_-$) on $  {\mathbb R^{d-1}}$ (resp. $\overline{\mathbb R_-}$), and for some $c>0$, $\vert \nabla f_+\vert \ge c$ (resp. $\vert \nabla f_-\vert \ge c$) outside a compact set of   $ {\mathbb R^{d-1}}$  (resp. of $\overline{\mathbb R_-}$) .
 \end{enumerate}
 In other words, $f_+$ satisfies \autoref{PotentialRd-1}, and $f_-$ satisfies \autoref{PotentialRd-} for $d=1$. 

\end{definition}
\noindent
It is easy to check that,  when $\mu_{d}<0$, $f_-$  satisfies 
\begin{equation}\label{eq.f-ext}
\forall h>0,\ e^{-\frac 2h f_-}\in \Lambda^0H^2(\mathbb R_-). 
\end{equation}
The following result  is the key point to prove Proposition~\ref{pr.prop-loca1}. 
It allows us to separate the variables $x'$ and $x_d$ in  the Witten
Laplacian $\Delta_{f,h}^{\mathsf{Di},(q)}(\mathbb R^d_-, \mbf g)$, up
to remainder terms of order $h^{6/5}$.

 \begin{lemma}\label{le.splitting} 
Assume that~\autoref{MetricRd-},\autoref{PotentialRd-}
and~\eqref{eq.metric-indep} are satisfied. Let u consider the
functions $f_+$ and $f_-$ as  introduced in Definition~\ref{de.pre1}.
Let  $q\in \{0,\ldots,d\}$  and $w\in \mathcal  D(\Delta^{\mathsf{Di},(q)}_{f,h}(\mathbb R_-^d, \mbf g))$.
Write $w=\mathsf a \wedge dx_d+ \mathsf b$ 
where
$$\mathsf b  = \sum \limits_{\substack{\mathsf J=\{j_1,\ldots,j_q\},\\ \, j_1<\ldots<j_q ,  \, d\notin \mathsf  J} } \mathsf b_{\mathsf J}\, dx_{\mathsf J}  \  \text{ and }\  \mathsf a =\sum \limits_{\substack{\mathsf I =\{i_1,\ldots,i_{q-1}\},\\ \, i_1<\ldots<i_{q-1},\,d\notin \mathsf  I }} \mathsf a_{\mathsf I}\, dx_{\mathsf I}.$$
It then holds, for some $\mathsf c_1>0$ and  $\mathsf c_2>0$
independent of $h>0$ and of $w$, 
\begin{align}
\nonumber
 Q^{\mathsf{Di},(q)}_{f,h}(\mathbb R_-^d, \mbf g)(w) &\ge \mathsf c_1 \sum \limits_{\mathsf I }   \int_{x'\in \mathbb R^{d-1} } Q^{\mathsf{Di},(1)}_{-f_-,h}(\mathbb R_-, dx_d^2)\big (\mathsf a_{\mathsf I}(x',.)dx_d  \big ) \mu(dx')\\
\nonumber
&\quad + \mathsf c _1\sum \limits_{\mathsf J  } \int_{x'\in \mathbb R^{d-1} }  Q^{\mathsf{Di},(0)}_{-f_-,h}(\mathbb R_-, dx_d^2)\big (\mathsf b_{\mathsf J}(x',.) \big ) \mu(dx')\\
\nonumber
&\quad +  \int_{x_d\in \mathbb R_- }   Q^{(q-1)}_{f_+,h}(\mathbb R^{d-1},  \tilde {\mbf G}dx'^2)\big (\mathsf a(.,x_d)\big ) dx_d \\
\nonumber
&\quad +  \int_{x_d\in \mathbb R_- }   Q^{(q)}_{f_+,h}(\mathbb
            R^{d-1},  \tilde {\mbf G}dx'^2)\big (\mathsf b(.,x_d) \big
            ) dx_d - \mathsf  e(h,w)
 \end{align}
 where $\vert \mathsf   e(h,w)\vert \le \mathsf c_2h^{6/5}\,  \Vert w
 \Vert_{L^2(\mathbb R^d_-  )}^2$ if supp $w\subset \mathsf B(0,h^{2/5})$. 
 The measure $\mu(dx')$   is   the measure $\sqrt{ {\rm det}\  \tilde {\mbf G}(x')} \, dx'$, where $dx'$ is  the Lebesgue measure on $\mathbb R^{d-1}$, and the measure $dx_d$  is the Lebesgue measure on $\mathbb R_-$.
 \end{lemma}
 
 \begin{proof}
\noindent 
One has  from~\eqref{eq.decf} and~\eqref{eq.plusmoins},
in a neighborhood   $ \mathsf V_0$ of $0$ in $\overline{\mathbb R^d_-}$,  \begin{equation}\label{eq.decf2}
 \forall x=(x',x_d)\in \mathsf V_0,\ \,  f(x)=f(0)+  f_+(x')- f_-(x_d)  +O(\vert x\vert ^3)
\end{equation}
and (using~\eqref{eq.metric-indep}), 
\begin{equation}\label{eq.decf3}
 \vert \nabla f(x)\vert_{{\mbf g}}^2=  \vert \nabla_{x'}  f_+(x')\vert_{\tilde {\mbf G}dx'^2}^2 + \vert \partial_{x_d}  f_-(x_d)\vert^2  +O(\vert x\vert ^3).
 \end{equation}
Moreover, one has:
 \begin{equation}\label{eq.decf4}
  \mathcal   L_{\nabla f(x)}+ \mathcal    L_{\nabla f(x)}^*=   \mathcal    L_{\nabla (f_+(x)-f_-(x))}+  \mathcal    L_{\nabla  (f_+(x)-f_-(x))}^*+ O(\vert x \vert).
\end{equation}
Let $w\in \mathcal D\big( \Delta_{f,h}^{\mathsf{Di},(q)}( \mathbb R_-^d, \mbf g ) \big)$. 
One has  
\begin{align}
Q_{f,h}^{\mathsf{Di},(q)}(\mathbb R^d_-, \mbf g )(w)&=\lp w,
                                                      \Delta_{f,h}^{(q)}(\mathbb
                                                      R^d_-, \mbf g )
                                                      w\rp_{L^2(\mathbb
                                                      R^d_-)}
=\lp w, \Delta_{f^+-f^-,h}^{(q)}(\mathbb R^d_-, \mbf g ) w\rp_{L^2(\mathbb R^d_-)} + \mathsf  e(h,w), \label{eq.ref1} 
\end{align}
where,  owing
to~\eqref{dec.delta},~\eqref{eq.decf2},~\eqref{eq.decf3},
and~\eqref{eq.decf4}, the remainder term $\mathsf e(h,w)$ satisfies:  if  $w$ is supported in $\mathsf B(0,h^{2/5  })$, 
\begin{equation}\label{taux-erreur}
\vert  \mathsf e(h,w)\vert \le C ( h^{6/5}+ h\times h^{2/5}) \Vert w \Vert_{L^2(\mathbb R^d_-)}^2\le Ch^{6/5}\,    \Vert w \Vert_{L^2(\mathbb R^d_-)}^2.
\end{equation}
Let us now give a lower bound on  $ \lp w,\Delta^{(q)}_{f_+-f_-,h}(\mathbb R^d_-, \mbf g)w\rp_{L^2(\mathbb R^d_-)}$. 
 Algebraically, using~\eqref{eq.metric-indep}, one has
 (see~\cite[Equation (3.17)]{helffer-nier-06} or~\cite[Equation (4.3.16)]{le-peutrec-10}
 for similar computations):
\begin{align}
\nonumber
\lp w,\Delta^{(q)}_{f_+-f_-,h}(\mathbb R^d_-, \mbf g)w\rp_{L^2(\mathbb R^d_-)}
&=  \Big \lp  \sum \limits_{\mathsf I}    dx_{\mathsf I}\wedge (\mathsf a_{\mathsf I}\, dx_d), \, \sum \limits_{ \mathsf I } dx_{\mathsf I}\wedge\Delta^{(1)}_{-f_-,h}(\mathbb R_-,dx_d^2)(\mathsf a_{\mathsf I} \,   dx_d)\Big\rp_{L^2(\mathbb R^d_-)} \\
\nonumber
&\quad +  \Big \lp \sum \limits_{ \mathsf J  }   \mathsf b_{\mathsf J}\, dx_{\mathsf J} ,\, \sum \limits_{ \mathsf J }   \Delta^{(0)}_{-f_-,h}(\mathbb R_-,dx_d^2)( \mathsf b_{\mathsf J})\, dx_{\mathsf J}\Big\rp_{L^2(\mathbb R^d_-)}\\
\nonumber
&\quad +\big \lp \mathsf a \wedge dx_d ,\Delta^{(q-1)}_{f_+,h}(\mathbb R^{d-1}, \tilde{\mbf G}dx'^2)(\mathsf a )\wedge dx_d \big\rp_{L^2(\mathbb R^d_-)} \\
\label{eq.dec55}
&\quad +\lp \mathsf b,\Delta^{(q)}_{f_+,h}(\mathbb R^{d-1}, \tilde{\mbf G}dx'^2)\mathsf b\rp_{L^2(\mathbb R^d_-)}.
 \end{align} 
Since  $\mbf t w =0$ on $\pa \mathbb R^d_-$,  it holds, for all
$\mathsf J$ and for a.e.  $x'\in \mathbb R^{d-1}$, $ \mathsf
b_{\mathsf J}(x',0)=0$. Thus (see Proposition~\ref{pr.pr1-d} for the
domain of $\Delta_{-f_-,h}^{\mathsf{Di},(0)}(\mathbb R_-, dx_d^2)$ and
item 1 of Definition~\ref{de.pre1}),
for all  $\mathsf J$  and a.e.  $x'\in \mathbb R^{d-1}$,
$$
  \mathsf b_{\mathsf J}(x',.)\in  \Lambda^0H^2 (\mathbb R_-)\cap
  \Lambda^0H^1_{\mbf T} (\mathbb R_-) = \mathcal  D \big(\Delta_{-f_-,h}^{\mathsf{Di},(0)}(\mathbb R_-, dx_d^2)   \big).$$
From~\eqref{eq.metric-indep},  $\mbf t \mathsf d_{f,h}^*w =0$ on $\pa \mathbb R^d_-$ (for the metric tensor $ \mbf g$) writes:  for a.e.  $x'\in \mathbb R^{d-1}$ and all $\mathsf I$, 
\begin{equation}\label{eq.domd*}
\partial_{x_d}(e^{-\frac 1hf}\mathsf a_{\mathsf I})(x',0)=0.
\end{equation}
 Because  $\partial_{x_d} f(x',0)=0$ for all $x'\in \mathbb R^{d-1}$, see~\eqref{eq.stable} and~\eqref{eq.dec-gradient-phi}, this condition thus writes  $\partial_{x_d}\mathsf a_{\mathsf I}(x',0)=0$. On the other hand, $f_-'(0)=0$ and hence, $\partial_{x_d}(e^{-\frac 1h f_-}\mathsf a_{\mathsf I})(x',0)=0$ for a.e. $x'\in \mathbb R^{d-1}$, i.e. $\mbf t \mathsf d_{f_-,h}^* (\mathsf a_{\mathsf I}(x',x_d)dx_d)  =0$ on  $\pa \mathbb R_-$ for the metric tensor $  dx_d^2$ (recall that $\mathsf d^*(\phi dx_d)=-\partial_{x_d}\phi$ for the metric tensor $  dx_d^2$).   Thus,  because in addition $\mathsf a_{\mathsf I}(x',.) dx_d \in  \Lambda^1H^2 (\mathbb R_-)$, one has  (see Proposition~\ref{pr.pr1-d} for the domain of $\Delta_{-f_-,h}^{\mathsf{Di},(1)}(\mathbb R_-, dx_d^2)$), for all  $\mathsf I$ and a.e.  $x'\in \mathbb R^{d-1}$:
$$\mathsf a_{\mathsf I}(x',.) dx_d \in   \mathcal  D\big(\Delta_{-f_-,h}^{\mathsf{Di},(1)}(\mathbb R_-, dx_d^2)   \big).$$
Furthermore,  one has (see Proposition~\ref{pr.case-d-1} for the domain of $ \Delta_{f_+,h}^{(q-1)}(\mathbb R^{d-1},  \tilde{\mbf G}dx'^2)$), for a.e.  $x_d< 0$:
$$  \mathsf a(.,x_d) \in \Lambda^{q-1}H^2 (\mathbb R^{d-1} ) = \mathcal  D\big(\Delta_{f_+,h}^{(q-1)}(\mathbb R^{d-1},  \tilde{\mbf G}dx'^2)\big), 
$$
$$ \mathsf b(.,x_d) \in  \Lambda^{q}H^2  (\mathbb R^{d-1},  \tilde{\mbf G}dx'^2 )= \mathcal  D\big(\Delta_{f_+,h}^{(q)}(\mathbb R^{d-1},  \tilde{\mbf G}dx'^2)\big).$$
Lemma~\ref{le.splitting} then follows from~\eqref{taux-erreur} and \eqref{eq.ref1} together with  two  integration by parts in $\mathbb R^{d-1}$ and two integrations by parts in $\mathbb R_-$  in~\eqref{eq.dec55}, the constant $\mathsf c_1 >0$ being the minimum of the  smallest eigenvalues of the matrices $(\Pi_{k=1}^{q-1} \tilde{\mbf G}_{i_k,i_k'})_{\mathsf I,\mathsf I'}$ and $ (\Pi_{k=1}^{q} \tilde{\mbf G}_{j_k,j_k'})_{\mathsf J,\mathsf J'}$ on $\mathbb R^{d-1}$.
 \end{proof}

\subsubsection{Proof of Proposition~\ref{pr.prop-loca1}}\label{sec:prop-loca1-proof}
We are now in position to prove Proposition~\ref{pr.prop-loca1}.
 Let us assume that~\autoref{MetricRd-}--\autoref{PotentialRd-} and~\eqref{eq.metric-indep} are satisfied.  Let us recall that according to Lemma~\ref{le.IMS-loca},   it is enough to prove Proposition~\ref{pr.prop-loca1} for all 
 $w\in \Lambda^q H^1_{\mbf T}( \mathbb R^d_-)$ supported in $\mathsf B(0,h^{2/5  })$.  All along the proof, the constants $C>0$ and $c>0$  can change from one occurrence to another but  do not depend on $h$ and on the test function $w$. 
 The proof of Proposition~\ref{pr.prop-loca1} is divided into three
 steps: the case $d=1$, the proof of~\eqref{eq.last1-case1}
 and~\eqref{eq.last2-case1} when $d>1$, and finally the proof of~\eqref{eq.last3-case1} when $d>1$.
 \medskip
 
 \noindent
 \textbf{Step 1: The case $d=1$ (i.e.  $\mathbb R^d_-=\mathbb R_-$).}
Let us recall that according to item $(i)$ in~\autoref{MetricRd-}, the
space $\mathbb R_-=\{x_d\in \mathbb R,\ x_d<0\}$ is endowed with the
metric tensor $\mbf g(x_d)=dx_d^2$.  
From~\eqref{eq.decf}, in a neighborhood $\mathsf V_0$  of $0$ in $\mathbb R_-$, one has
 $$
 \forall x_d\in \mathsf V_0,\ \,  f(x_d)= f(0)+ \frac{\mu_{d}}{2} x_d^2 +O(|x|^3).
 $$ 
 Notice that for $w\in \Lambda H^1_{\mbf T}( \mathbb R_- )$ according to the decomposition $w=\mathsf a\wedge dx_d+\mathsf b$ in Lemma~\ref{le.splitting},  $w=\mathsf b$ when $w$ is a $0$-form and  $w=\mathsf a dx_d$ when  $w$ is a $1$-form ($\mathsf a$ is a function,  see~\eqref{eq.1-form-CL} below). 
For all $\mathsf b \in  \Lambda^0 H^1_{\mbf T}( \mathbb R_-  )$, one has from  Lemma~\ref{le.Green} and since $\mathsf b(0)=0$,
\begin{equation}\label{eq.1bb-<}
Q_{f,h}^{\mathsf{Di},(0)}( \mathbb R_- ,dx_d^2 )( \mathsf b)= h^2  \Vert \partial_{x_d} \mathsf b  \Vert_{L^2(\mathbb R_-)}^2+   \Vert \mathsf b\, \partial_{x_d} f \Vert_{L^2(\mathbb R_-)}^2-h\lp \mathsf b \, \partial_{x_d}^2 f, \, \mathsf b \rp_{L^2(\mathbb R_-)}.
\end{equation}
For all $\mathsf a \, dx_d\in  \Lambda^1 H^1_{\mbf T}( \mathbb R_- )$ where we recall that 
\begin{equation}\label{eq.1-form-CL}
\Lambda^1 H^1_{\mbf T}( \mathbb R_-   ) =\Lambda^1 H^1( \mathbb R_-  )=\big \{\mathsf a \, dx_d,\, \mathsf a \in  \Lambda^0 H^1( \mathbb R_- )\big \},
\end{equation}
 one has, since $\partial_{x_d} f(0)=0$ (the boundary term vanishes in Lemma~\ref{le.Green}):
\begin{align*}
Q_{f,h}^{\mathsf{Di},(1)}( \mathbb R_-, dx_d^2 )(\mathsf a \, dx_d )
&= h^2\big \Vert \partial_{x_d} \mathsf a \big \Vert_{L^2(\mathbb R_-)}^2+ \big \Vert \mathsf a\, \partial_{x_d} f\big \Vert_{L^2(\mathbb R_-)}^2+ h\lp \mathsf a \, \partial_{x_d}^2 f, \, \mathsf a \rp_{L^2(\mathbb R_-)}.
\end{align*}
Let us now consider the two possibilities: $\mu_d>0$ or $\mu_d < 0$.

\noindent
\textbf{Step 1a: The case $d=1$ and $\mu_{d}>0$ (i.e. $ \partial_{x_d}^2 f(0)>0$).}
Then, there exists $C>0$ such that $\partial_{x_d}^2 f\ge C$ in a neighborhood  of $0$ in $\mathbb R_-$.  
Thus, for all $\mathsf a \, dx_d\in  \Lambda^1 H^1_{\mbf T}( \mathbb R_-  )$ such that  $\mathsf a$ is supported in $\mathsf B(0,h^{2/5  })$, one has  $
Q_{f,h}^{\mathsf{Di},(1)}(\mathbb R_-, dx_d^2)(\mathsf a  \, dx_d )\ge  Ch\,    \Vert \mathsf a  \Vert_{L^2(\mathbb R_-)}^2$. 
Thanks to Lemma~\ref{le.IMS-loca}, this inequality extends for all $\mathsf a \, d x_d\in  \Lambda^1 H^1_{\mbf T}( \mathbb R_-  )$: there exists $C>0$ such that for $h$ small enough,
 \begin{equation}\label{eq.cas2a-1}
\forall \mathsf a\,  d x_d\in  \Lambda^1 H^1_{\mbf T}( \mathbb R_-), \ Q_{f,h}^{\mathsf{Di},(1)}( \mathbb R_-, dx_d^2 )(\mathsf a  \, dx_d )\ge  Ch\,   \Vert \mathsf a   \Vert_{L^2(\mathbb R_-)}^2.
\end{equation}
Let us now prove that there exists $c>0$ such that for $h$ small enough:  
 \begin{equation}\label{eq.cas2a-2}
\forall \mathsf b\in  \Lambda^0 H^1_{\mbf T}( \mathbb R_-  ),\ Q_{f,h}^{\mathsf{Di},(0)}( \mathbb R_- , dx_d^2)( \mathsf b)\ge  c \, h\, \Vert \mathsf b  \Vert_{L^2(\mathbb R_-)}^2.
\end{equation}
 It is clear that $\Ker \Delta_{f,h}^{\mathsf{Di},(0)}( \mathbb R_- , dx_d^2)=\{0\}$ since $e^{-\frac 1hf}$ is not in the domain of $\Delta_{f,h}^{\mathsf{Di},(0)}( \mathbb R_-, dx_d^2 )$. Let us  now consider   $\psi_h\in \Ran\, \pi_{[0,Ch/2]}\big (\Delta_{f,h}^{\mathsf{Di},(0)}( \mathbb R_- ,dx_d^2 )\big )$
   (where~$C$ is the constant appearing in~\eqref{eq.cas2a-1}). Then,
   $\mathsf  d_{f,h} \psi_h \in \Ran\, \pi_{[0,Ch/2]}
  \big( \Delta_{f,h}^{\mathsf{Di},(1)}( \mathbb R_- ,dx_d^2 ) \big) $  (thanks to~\eqref{eq.complexe1}). From~\eqref{eq.cas2a-1}, this implies that $\mathsf  d_{f,h}\psi_h=0$. Thus, $\Delta_{f,h}^{(0)}( \mathbb R_- , dx_d^2)\psi_h=0$ and hence,  $\psi_h=0$. This proves~\eqref{eq.cas2a-2}.
 \medskip
 
\noindent
 \textbf{Step 1b: The case $d=1$ and $\mu_{d}<0$ (i.e. $ \partial_{x_d}^2 f(0)<0$).}
Then, there exists $C>0$ such that $\partial_{x_d}^2 f\le -C$ in a neighborhood of $0$ in $\overline{\mathbb R_-}$. Thus, from~\eqref{eq.1bb-<}, for $h$ small enough, one has for all $\mathsf b\in  \Lambda^0 H^1_{\mbf T}( \mathbb R_- )$ such that  $\mathsf b$ is supported in $\mathsf B(0,h^{2/5  })$: $
Q_{f,h}^{\mathsf{Di},(0)}( \mathbb R_-, dx_d^2 )( \mathsf b)\ge  C h\,   \Vert \mathsf b \Vert_{L^2(\mathbb R_-)}^2$. 
Using Lemma~\ref{le.IMS-loca}, this inequality extends for all $\mathsf b\in  \Lambda^0 H^1_{\mbf T}( \mathbb R_- )$, i.e.  for $h$ small enough:
 \begin{equation}\label{eq.cas1a-1}
\forall \mathsf b\in  \Lambda^0 H^1_{\mbf T}( \mathbb R_-, dx_d^2 ),\ Q_{f,h}^{\mathsf{Di},(0)}( \mathbb R_- , dx_d^2)( \mathsf b)\ge  Ch\,  \Vert \mathsf b \Vert_{L^2(\mathbb R_-)}^2.
\end{equation}
Let us now prove that there exists $c>0$ such that  for $h$ small enough
\begin{equation}\label{eq.cas1a-2}
\Ran \, \pi_{[0,ch]}\big (\Delta_{f,h}^{\mathsf{Di},(1)}( \mathbb R_-, dx_d^2 )\big )=\Ker \Delta_{f,h}^{\mathsf{Di},(1)}( \mathbb R_-, dx_d^2 ) = {\rm Span}\, (e^{\frac fh}dx_d).
\end{equation}  
\noindent
From item $(ii)$ in \autoref{PotentialRd-}  and using  the same
arguments  as those to check~\eqref{eq.f-ext}, one has $f'>c$ on $[-\infty, -\ve]$ for some $\ve>0$. Hence, for $h>0$, $e^{\frac fh}\in \Lambda^0L^2( \mathbb R_-  )$ and from item $(i)$ in \autoref{PotentialRd-},  $e^{\frac fh}\in \Lambda^0H^2( \mathbb R_-  )$. Consequently  (see Proposition~\ref{pr.pr1-d}),   $e^{\frac fh}dx_d\in  \mathcal D(\Delta_{f,h}^{\mathsf{Di},(1)}( \mathbb R_-, dx_d^2 ) )$. 
Therefore, since for all $\mathsf a \, dx_d\in  \Lambda^1 H^1_{\mbf T}( \mathbb R_-  )$, $Q_{f,h}^{\mathsf{Di},(1)}( \mathbb R_-, dx_d^2 )(\mathsf a  \, dx_d )=   \Vert \mathsf d^*_{f,h} \mathsf a    \Vert_{L^2(\mathbb R_-)}^2$, it holds:
$$\Ker \Delta_{f,h}^{\mathsf{Di},(1)}( \mathbb R_-,dx_d^2 ) = {\rm Span}\, \big (e^{\frac fh}dx_d\big ).$$ 
\noindent
Let us now consider  an eigenform  $\psi_h\in \Ran\, \pi_{[0,Ch/2]}
\Delta_{f,h}^{\mathsf{Di},(1)}( \mathbb R _- ,dx_d^2 )  $  (where~$C$
is the constant appearing in~\eqref{eq.cas1a-1}). Then,  $\mathsf
d_{f,h}^* \psi_h \in \Ran\, \pi_{[0,Ch/2]}
\Delta_{f,h}^{\mathsf{Di},(0)}( \mathbb R_- ,dx_d^2 )  $  (thanks
to~\eqref{eq.complexe2}). From~\eqref{eq.cas1a-1}, this implies that
for $h$ small enough,  $\mathsf d^*_{f,h} \psi_h=0$. Thus $\psi_h\in
{\rm Span}\, (e^{\frac fh} dx_d)$. This proves~\eqref{eq.cas1a-2}.
 \medskip
 
 \noindent
 \textbf{Step 2: The case $d>1$,  proofs of
   Equations~\eqref{eq.last1-case1}  and~\eqref{eq.last2-case1}.}
 Remember that $\mathbb R^d_-$ is endowed with a metric tensor  $ \mbf
 g$ satisfying~\eqref{eq.metric-indep}.  Thanks to
 Lemma~\ref{le.IMS-loca}, it is enough to consider
 $$w\in \mathcal  D \big(\Delta^{\mathsf{Di},(q)}_{f,h}(\mathbb R_-^d,
 \mbf g)\big) \text{ with supp }  w\subset \mathsf B(0,h^{2/5}).$$
Following Lemma~\ref{le.splitting}, $w=\mathsf b + \mathsf a \wedge dx_d,$ where:$$\mathsf b =\sum \limits_{\substack{\mathsf J=\{j_1,\ldots,j_q\},\\ \, j_1<\ldots<j_q ,  \, d\notin \mathsf  J} } \mathsf b_{\mathsf J}\, dx_{\mathsf J} \text{ and } \mathsf a =\sum \limits_{\substack{\mathsf I=\{i_1,\ldots,i_{q-1}\},\\ \, i_1<\ldots<i_{q-1}, \, d\notin \mathsf I }} \mathsf a_{\mathsf I}\, dx_{\mathsf I}.$$
We will use many times that,  from \autoref{MetricRd-} and \eqref{eq.metric-indep}, 
  $\Vert w  \Vert_{L^2(\mathbb R^d_-  )}^2= \Vert \mathsf b  \Vert_{L^2(\mathbb R^d_-  )}^2+ \Vert \mathsf a \wedge dx_d  \Vert_{L^2(\mathbb R^d_-  )}^2$ (because $\mathsf b $ is orthogonal to $\mathsf a \wedge dx_d$) with  $\Vert \mathsf b  \Vert_{L^2(\mathbb R^d_-  )}^2\ge \mathsf c_1   \sum_{\mathsf I} \Vert \mathsf b_{\mathsf J}  \Vert_{L^2(\mathbb R^d_-  )}^2$ and $\Vert \mathsf a \wedge dx_d  \Vert_{L^2(\mathbb R^d_-  )}^2=\Vert \mathsf a  \Vert_{L^2(\mathbb R^d_-  )}^2  \ge \mathsf c_1 \sum_{\mathsf J} \Vert \mathsf a_{\mathsf I}  \Vert_{L^2(\mathbb R^d_-  )}^2$ (where  $\mathsf c_1 >0$  is as in Lemma~\ref{le.splitting}). 
 \medskip

\noindent
 \textbf{Step 2a: The case $d>1$ and $q= 0$, proof of~\eqref{eq.last1-case1}.}
Assume that $q= 0$ (i.e. $w=\mathsf b$ is a function). Then, using Lemma~\ref{le.splitting}, one has:
 \begin{align}\label{dec-Q2}
Q^{\mathsf{Di},(q)}_{f,h}(\mathbb R^{d}_-, \mbf g)(w)\ge \mathsf c_1\int_{x'\in \mathbb R^{d-1} }  Q^{\mathsf{Di},(0)}_{-f_-,h}(\mathbb R_-,dx_d^2)\big (\mathsf b(x',.)\big ) \mu(dx')-\mathsf c_2h^{6/5}  \Vert w \Vert_{L^2(\mathbb R^d_-  )}^2.
 \end{align}
 Equations~\eqref{eq.cas2a-2} and~\eqref{eq.cas1a-1} imply that that there exists $C>0$ (independent of $x'$) such that for all $h$ small enough and a.e. $x'\in \mathbb R^{d-1}$:
 $$ Q^{\mathsf{Di},(0)}_{-f_-,h}(\mathbb R_-,dx_d^2)\big (\mathsf b (x',.)\big ) \ge Ch  \Vert \mathsf b (x',.)  \Vert_{L^2(\mathbb R_- )}^2.$$
 Thus, using~\eqref{dec-Q2}, one obtains for all $w\in D\big( \Delta_{f,h}^{\mathsf{Di},(q)}( \mathbb R_-^d ) \big)$ supported in $\mathsf B(0,h^{2/5  })$:
 $$Q^{\mathsf{Di},(q)}_{f,h}(\mathbb R_-^d, \mbf g )(w)\ge Ch   \Vert w  \Vert_{L^2(\mathbb R^d_-  )}^2-\mathsf c_2h^{6/5}\,  \Vert w \Vert_{L^2(\mathbb R^d_-  )}^2\ge ch \Vert w  \Vert_{L^2(\mathbb R^d_-  )}^2.$$
 Together with Lemma~\ref{le.IMS-loca}, this proves~\eqref{eq.last1-case1}.
\medskip
 
\noindent
 \textbf{Step 2b: The case $d>1$, $q \ge 1$ and $\mu_{d}>0$, proof of~\eqref{eq.last2-case1}.}
The analysis above in dimension~1 (see~\eqref{eq.cas2a-1} and
\eqref{eq.cas2a-2})  implies that there exists $C>0$ (again,
independent of $x'$) such that for $h$ small enough, for all $\mathsf I$ and a.e. $x'\in \mathbb R^{d-1}$,
$$Q^{\mathsf{Di},(1)}_{-f_-,h}(\mathbb R_-,dx_d^2)  (\mathsf a_{\mathsf I}(x',.)dx_d  ) \ge Ch   \Vert \mathsf a_{\mathsf I}(x',.)   \Vert_{L^2(\mathbb R_-)}^2,$$
and for $h$ small enough, for all $\mathsf J$ and a.e. $x'\in \mathbb R^{d-1}$,
$$ Q^{\mathsf{Di},(0)}_{-f_-,h}(\mathbb R_-,dx_d^2)  (\mathsf b_{\mathsf J}(x',.)  ) \ge Ch   \Vert \mathsf b_{\mathsf J}(x',.)   \Vert_{L^2(\mathbb R_-)}^2.$$
Thus, using  Lemma~\ref{le.splitting},  for all $w\in D\big( \Delta_{f,h}^{\mathsf{Di},(q)}( \mathbb R_-^d, \mbf g ) \big)$ supported in $\mathsf B(0,h^{2/5  })$, one has:
$$Q^{\mathsf{Di},(q)}_{f,h}(\mathbb R_-^d, \mbf g )(w)\ge Ch  \Vert w   \Vert_{L^2(\mathbb R^d_-)}^2-\mathsf c_2h^{6/5}\,   \Vert w  \Vert_{L^2(\mathbb R^d_-  )}^2\ge ch \Vert w  \Vert_{L^2(\mathbb R^d_-  )}^2 .$$
Using~Lemma~\ref{le.IMS-loca}, this proves~\eqref{eq.last2-case1} when $q\ge 1$ and $\mu_{d}>0$.
\medskip
 
\noindent
 \textbf{Step 2c: The case $d>1$, $q \ge 1$, $\mu_{d}<0$ and the index
   of $0$ as a critical point of $f|_{\pa \mathbb R^{d}_-}$ is not $q-1$, proof of~\eqref{eq.last2-case1}.}
Using~\eqref{eq.cas1a-1}, there exists $C$ (again, independent of $x'$) such that for $h$ small enough, for all $\mathsf J$ and a.e. $x'\in \mathbb R^{d-1}$,
\begin{equation}\label{cas-lambda<0}
 Q^{\mathsf{Di},(0)}_{-f_-,h}(\mathbb R_-,dx_d^2)\big (\mathsf b_{\mathsf J}(x',.)\big ) \ge Ch   \Vert \mathsf b_{\mathsf J}(x',.)  \Vert_{L^2(\mathbb R_-)}^2.
 \end{equation}
Thus, using Lemma~\ref{le.splitting}, one has:
\begin{align}
\nonumber
Q^{\mathsf{Di},(q)}_{f,h}(\mathbb R_-^d, \mbf g)(w)&\ge Ch   \Vert \mathsf b    \Vert_{L^2(\mathbb R^d_-)}^2 +\mathsf c_1\int_{x_d\in \mathbb R_- }   Q^{(q-1)}_{f_+,h}(\mathbb R^{d-1}, \mbf gdx'^2)\big (\mathsf a(.,x_d)\big ) dx_d  \\
  \label{eq.cas-2}
&\quad -\mathsf c_2h^{6/5}\,   \Vert w  \Vert_{L^2(\mathbb R^d_-)}^2.
 \end{align}
Recall that $0$ is not a critical point of index $q-1$ of $f|_{\pa \mathbb R^{d}_-}$. Then, $0$ is not a critical point of index $q-1$ for $f_+$ (see~\eqref{eq.decf} and Definition~\ref{de.pre1}).  Since $\mathsf a$ is a $q-1$ form, this implies from~Proposition~\ref{pr.case-d-1} (applied  with  the metric tensor  $ \tilde {\mbf G}dx'\, ^2$), that there exists~$C$ (independent of $x_d$)  such that for $h$ small enough,
$$ Q^{(q-1)}_{f_+,h}(\mathbb R^{d-1}, \tilde {\mbf G}dx'^2)\big (\mathsf a(.,x_d)\big )  \ge Ch    \Vert \mathsf a(.,x_d)    \Vert_{L^2(\mathbb R^{d-1})}^2.$$
Therefore, using~\eqref{eq.cas-2}, for $h$ small enough, one has:
$$Q^{\mathsf{Di},(q)}_{f,h}(\mathbb R^{d}_-, \mbf g)(w)\ge Ch(    \Vert \mathsf b     \Vert_{L^2(\mathbb R^d_-)}^2+   \Vert \mathsf a     \Vert_{L^2(\mathbb R^{d}_-)}^2)-\mathsf c_2h^{6/5}\,   \Vert w  \Vert_{L^2(\mathbb R^d_-  )}^2  \ge ch   \Vert w    \Vert_{L^2(\mathbb R^d_-)}^2 .$$
Using~Lemma~\ref{le.IMS-loca}, this proves~\eqref{eq.last2-case1} when
$q\ge 1$, $\mu_{d}<0$ and the index of $0$ as a critical point of
$f|_{\pa \mathbb R^{d}_-}$ is not $q-1$.
\medskip
 
\noindent
 \textbf{Step 3: The case $d>1$, $q \ge 1$, $\mu_{d}<0$ and the index
   of  $0$ as a critical point of $f|_{\pa \mathbb R^{d}_-}$ is $q-1$, proof of~\eqref{eq.last3-case1}.}
Notice that in this case, the point~$0$ is a critical point of  $f_+$ of index $q-1$ (see Definition~\ref{de.pre1}). 

  \medskip
 
 \noindent
 \textbf{Step 3a: Proof of~\eqref{eq.last3-case1} when  $f=f_+-f_-$.}
Let us first prove~\eqref{eq.last3-case1}  for the  potential (see Definition~\ref{de.pre1}):
$$x=(x',x_d)\in \overline{\mathbb R^d_-}\mapsto  f_+(x') -f_-(x_d).$$
In view of Definition~\ref{de.pre1}  and~\eqref{eq.dec-gradient-phi}, $f_+ -f_-$   satisfies~\autoref{PotentialRd-}. Thus,  
Proposition~\ref{pr.pr1-d}, \eqref{eq.last1-case1},
and~\eqref{eq.last2-case1} are  valid for $f_+ -f_-$ and  $   \mbf
g$. Let us   consider  $\Psi_h \in \Ker \Delta_{f_+,h}^{(q-1)}( \mathbb
R^{d-1}, \tilde{\mbf G}dx'^2)$ with
$\Psi_h\neq 0$ (which exists thanks to item $(ii)$ in Proposition~\ref{pr.case-d-1}).   
Let us prove that there exist $c>0$ and $h_0>0$ such that for all $h\in (0,h_0)$,   \begin{align}
\label{f+f-}
 \Ran \, \pi_{[0,ch]}\big (\Delta_{f_+ -f_-,h}^{\mathsf{Di},(q)}( \mathbb R_-^d, { \mbf g} )\big )&=\Ker \Delta_{f_+ -f_-,h}^{\mathsf{Di},(q)}( \mathbb R_-^d, { \mbf g}  ) ={\rm Span} \,\big( \Psi_h \wedge e^{-\frac 1h f_- }dx_d\big).
\end{align}
 Let $c_0>0$ and $\phi_h\in \Ran \, \pi_{[0,c_0h]}\big(
 \Delta_{f_+-f_-,h}^{\mathsf{Di},(q)}( \mathbb R_-^d,    \mbf g)
 \big)$ where the constant $c_0$ is  strictly  smaller than the
 constants $C>0$ in~\eqref{eq.last1-case1} and~\eqref{eq.last2-case1}
 applied to $f=f_+ -f_-$. Hence, using~\eqref{eq.complexe1} and~\eqref{eq.complexe2}, one has for $h$ small enough
  $$\mathsf  d_{f,h}\phi_h= 0 \text{ and } \mathsf  d_{f,h}^*\phi_h=0.$$
 Thus, $Q_{f_+-f_-,h}^{\mathsf{Di},(q)}( \mathbb R_-^d,  \mbf g)(
 \phi_h)=0$. Using Lemma~\ref{le.splitting} with $f=f_+-f_-$ (in which
 case $\mathsf   e(h,\phi_h)=0$)  together with item $(ii)$ in
 Proposition~\ref{pr.case-d-1} and~\eqref{eq.cas1a-2} with $f=-f_-$ ?, one obtains   
\begin{equation}\label{eq.imply-Phi}
\phi_h\in {\rm Span} \,\big( \Psi_h \wedge e^{-\frac 1h f_- }dx_d\big).
\end{equation}
To prove \eqref{f+f-}, it   thus remains to show that:
\begin{equation}\label{=psi-1}
\Psi_h \wedge e^{-\frac 1h f_- }dx_d\in  \Ker \Delta_{f_+ -f_-,h}^{\mathsf{Di},(q)}( \mathbb R_-^d,   { \mbf g}  ).
\end{equation} 
It first holds, from Propositions~\ref{pr.case-d-1} and~\ref{pr.pr1-d},  and~\eqref{eq.f-ext},  $\Psi_h \wedge e^{-\frac 1h f_- }dx_d\in \mathcal  D(\Delta_{f_+ -f_-,h}^{\mathsf{Di},(q)}( \mathbb R_-^d,    \mbf g ))$ (recall that the boundary condition $\mbf t \mathsf d_{f_+-f-,h}^*w =0$ is equivalent to $\partial_{x_d}(e^{-\frac 1h (f_+-f_-)}\mathsf a_{\mathsf I})(x',0)=0$, see indeed~\eqref{eq.domd*}). 
Besided, one has:
$$\mathsf d_{f_+-f_-,h}\big( \Psi_h\wedge e^{-\frac 1h f_-}dx_d\big )=    \mathsf  d_{f_+,h}( \Psi_h)  \wedge e^{-\frac 1h f_-}dx_d=0. $$
 Moreover, from~\eqref{eq.metric-indep} (see also item $(i)$
 in~\autoref{MetricRd-}), it holds
$$\mathsf  d^{*, { \mbf g}}\big(
\Psi_h\wedge e^{-\frac 1h f_-}dx_d\big )= \mathsf
  d^{*,\tilde{ \mbf G}dx'^2}\big ( \Psi_h\big ) \wedge
e^{-\frac 1h f_-(x_d)}dx_d+  \Psi_h   \mathsf
  d^{*, dx_d^2}\big( e^{-\frac 1h f_-}dx_d\big).$$ 
where the superscript indicates in which metric the operator $d^*$ is built. 
And one can check that
\begin{align*}
\mbf i_{\nabla (f_+-f_-)} \big( \Psi_h\wedge e^{-\frac 1h f_-}dx_d\big )&= \mbf i_{\nabla_{x'} f_+} \big(\Psi_h\big)\wedge e^{-\frac 1h f_-(x_d)}dx_d-\Psi_h\wedge \mbf i_{\nabla_{x_d} f_-}\big( e^{-\frac 1h f_-}dx_d\big)\\
&=\mbf i_{\nabla_{x'} f_+} \big(\Psi_h\big)\wedge e^{-\frac 1h f_-(x_d)}dx_d+h\, \Psi_h \,  \partial_{x_d}\big( e^{-\frac 1h f_-(x_d)}\big).
   \end{align*}
   Therefore,  $  \mathsf d^*_{f_+-f_-,h}\big( \Psi_h\wedge e^{-\frac 1h f_-}dx_d\big )=0$. 
This proves~\eqref{=psi-1} and then~\eqref{f+f-}. This concludes  the proof of~\eqref{eq.last3-case1} when $f=f_+-f_-$.   
  \medskip
 
\noindent
 \textbf{Step 3b: Proof of~\eqref{eq.last3-case1} for a general
   function $f$.} 
%
Let  $c_0>0$ be  strictly smaller than the constants $C>0$ in~\eqref{eq.last1-case1} and~\eqref{eq.last2-case1}.
Assume that $\Ran \, \pi _{[0,c_0h]}\big( \Delta_{f
  ,h}^{\mathsf{Di},(q)}( \mathbb R_-^d, \mbf g  ) \big) \neq \{0\}$ and let us consider a $L^2( \mathbb R_-^d ) $-normalized   form $$\psi_h\in \Ran \, \pi _{[0,c_0h]}\big( \Delta_{f ,h}^{\mathsf{Di},(q)}( \mathbb R_-^d, \mbf g  ) \big).$$
 Then, using~\eqref{eq.complexe1} and~\eqref{eq.complexe2}, one has for $h$ small enough,
  $\mathsf  d_{f,h}\psi_h= 0$, $\mathsf  d_{f,h}^*\psi_h=0$, and thus $Q_{f ,h}^{\mathsf{Di},(q)}( \mathbb R_-^d, \mbf g  ) \psi_h=0$. 
 This proves that  for $h$ small enough:
 $$\Ran \, \pi_{[0,c_0h]}\big (\Delta_{f,h}^{\mathsf{Di},(q)}( \mathbb R_-^d,   {\mbf g}  )\big )=\Ker \Delta_{f,h}^{\mathsf{Di},(q)}( \mathbb R_-^d,   {\mbf g}   ).$$
 Let us now consider a partition of unity $(\chi_1,\chi_2)$  such that $\chi_1\in\mathcal  C_c^\infty (\overline {\mathbb R_-^d)}$, $\chi_1=1$ on $\mathsf B(0,1/2)$, supp$\chi_1\subset  \mathsf B(0,1)$,  and $\chi_1^2+\chi_2^2=1$. 
The IMS formula~\eqref{eq.IMS} 
implies that  there exists $C>0$ such that :
\begin{align}
\label{eq.ast1}
0&\ge   Q_{f,h}^{\mathsf{Di},(q)}( \mathbb R_-^d, \mbf g ) (\chi_1(h^{-\frac 25  }.)  \psi_h )+Q_{f,h}^{\mathsf{Di},(q)}(\mathbb R^d_-, \mbf g)(  \chi_2(h^{-\frac 25  }.) \psi_h ) -C h^{\frac  65}\,  \Vert   \psi_h   \Vert_{L^2( \mathbb R_-^d  )}^2.
\end{align}
From~\eqref{eq.chi_2-Q}, one has:
  $$Q_{f,h}^{\mathsf{Di},(q)}(\mathbb R^d_-, \mbf g)(  \chi_2(h^{-2/5  }.) \psi_h )\ge Ch^{4/5}  \Vert \chi_2(h^{-2/5  }.)  \psi_h  \Vert_{L^2( \mathbb R_-^d  )}^2.$$
   Thus,   it holds for $h>0$ small enough: 
\begin{equation}\label{norme-chi}
 \Vert \chi_2(h^{-2/5  }.)  \psi_h   \Vert_{L^2( \mathbb R_-^d )}^2 =O(h^{2/5  }) \text{ and }  \Vert \chi_1(h^{-2/5  }.)  \psi_h  \Vert_{L^2( \mathbb R_-^d)}^2=1+O  (h^{2/5  }  ),
\end{equation}
and then:
\begin{equation}\label{quad-chi1}
Q_{f,h}^{\mathsf{Di},(q)}(\mathbb R^d_-,   { \mbf g})  (\chi_1(h^{-2/5  }.) \psi_h )\le Ch^{\frac 65}    \Vert \chi_1(h^{-1/5  }.) \psi_h \Vert_{L^2( \mathbb R_-^d)}^2.
\end{equation}
Using~\eqref{eq.fried-e},~\eqref{eq.decf3} and~\eqref{eq.decf4}, and using twice Lemma~\ref{le.Green} (once for $f$ and once for $f_+-f_-$) together with the fact that $ \partial_{\mathsf n_{\mathbb R^d_-}}f= \partial_{\mathsf n_{\mathbb R^d_-}}(f_+-f_-)=0$ on $\pa \mathbb R^d_-$,  one has for $h$ small enough and for all $v\in \Lambda^qH^1_{\mbf T}(\mathbb R^d_-, \mbf g) $ supported in $\mathsf B(0,h^{2/5})$:
 \begin{align}
 \label{eq.=forme}
 Q_{f,h}^{\mathsf{Di},(q)}(\mathbb R^d_-,   { \mbf g})  (v)&=Q_{f_+-f_-,h}^{\mathsf{Di},(q)}(\mathbb R^d_-, { \mbf g})  (v) + O(h^{6/5})   \Vert v  \Vert_{L^2(\mathbb R^d_-)}^2.
 \end{align}
 Thus, one gets for~$h$ small enough:
\begin{align*}
Q_{f,h}^{\mathsf{Di},(q)}(\mathbb R^d_-,   { \mbf g})  ( \chi_1(h^{-1/5  }.)  \psi_h )&=Q_{f_+-f_-,h}^{\mathsf{Di},(q)}(\mathbb R^d_-,   { \mbf g})  ( \chi_1(h^{-1/5  }.)  \psi_h )  + O(h^{  6/5}) \big \Vert  \chi_1(h^{-1/5  }.)  \psi_h  \big \Vert_{L^2(\mathbb R^d_-)}^2.
\end{align*}
Then, using~\eqref{quad-chi1}, it holds for $h$ small enough:
$$Q_{f_+-f_-,h}^{\mathsf{Di},(q)}(\mathbb R^d_-, { \mbf g})  (\chi_1(h^{-2/5  }.) \psi_h  )=O(h^{6/5})\, \big\Vert \chi_1(h^{-1/5  }.)  \psi_h \big\Vert_{L^2( \mathbb R_-^d)}^2.$$
For all $(x',x_d)\in \mathbb R^d_-$, let us define (see~\eqref{f+f-}),
$$\Theta_h(x',x_d)=\kappa_h\, \Psi_h(x')\wedge e^{-\frac 1h f_-(x_d)}dx_d, \text{ where $\kappa_h= \Vert \Psi_h\wedge e^{-\frac 1h f_-}dx_d\Vert_{L^2(\mathbb R^d_-)}^{-1}$. }$$
Using Lemma~\ref{quadra} and~\eqref{f+f-} (choosing $c_0$ smaller than $c>0$ appearing in~\eqref{f+f-}), one has for $h$ small enough, 
\begin{align*} 
{\rm dist}_{L^2( \mathbb R_-^d  )}\, \big( \chi_1(h^{-2/5  }.) \psi_h ,{\rm Span}\, \Theta_h\big)&=\big  \Vert \pi_{[0,ch]}\big (\Delta_{f_+-f_-,h}^{\mathsf{Di},(q)}( \mathbb R_-^d ,   { \mbf g})\big )\big( \chi_1(h^{-2/5  }.)\phi_h \big)\big \Vert_{L^2(\mathbb R^d_-)}\\
&\le \frac{ Q_{f,h}^{\mathsf{Di},(q)}(\mathbb R^d_-,   { \mbf g})  ( \chi_1(h^{-1/5  }.)  \psi_h )^{\frac 12}}{\sqrt{ch}} \le Ch^{{1}/{10}}.
\end{align*}
Using in addition~\eqref{norme-chi}, one obtains   for $h$ small enough:
$${\rm dist}_{L^2( \mathbb R_-^d )}\, \big(  \psi_h ,{\rm Span}\, \Theta_h\big)\le   Ch^{ {1}/{10}} + C \Vert \chi_2(h^{-2/5  }.)  \psi_h   \Vert_{L^2( \mathbb R_-^d  )}^2 \le  2Ch^{ {1}/{10}}.$$ 
 Therefore, since we assumed that   $\Ran \, \pi _{[0,c_0h]}\big( \Delta_{f ,h}^{\mathsf{Di},(q)}( \mathbb R_-^d, \mbf g  ) \big) \neq \{0\}$, it holds for $h$ small enough:
$$
\dim \Ran \, \pi_{[0,c_0h]}\big (\Delta_{f,h}^{\mathsf{Di},(q)}( \mathbb R_-^d, \mbf g )\big )=1.
$$

It thus remains to prove   that $\Ran \, \pi _{[0,c_0h]}\big( \Delta_{f ,h}^{\mathsf{Di},(q)}( \mathbb R_-^d, \mbf g  ) \big) \neq \{0\}$. To this end, let us show that  $\Delta_{f,h}^{\mathsf{Di},(q)}( \mathbb R_-^d, \mbf g )$ admits an eigenvalue which is $o(h)$ when $h\to 0$. Using the IMS formula~\eqref{eq.IMS} together with the fact that 
$$Q_{f_+-f_-,h}^{\mathsf{Di},(q)}( \mathbb R_-^d , { \mbf g}) (   \Theta_h)=0,$$
 and 
 $$Q_{f_+-f_-,h}^{\mathsf{Di},(q)}( \mathbb R_-^d , { \mbf g}) ( \chi_2(h^{-2/5  }.)  \Theta_h) \ge ch^{\frac 45}\, \big\Vert  \chi_2(h^{-2/5  }.)  \Theta_h\big\Vert_{L^2( \mathbb R_-^d )}^2 ,$$
one obtains,   when $h\to 0$, 
$$Q_{f_+-f_-,h}^{\mathsf{Di},(q)}( \mathbb R_-^d , { \mbf g}) ( \chi_1(h^{-2/5  }.) \Theta_h)=O(h^{\frac 65}) \text{ and }  \big\Vert  \chi_1(h^{-2/5  }.)  \Theta_h\big\Vert_{L^2( \mathbb R_-^d)}^2=1+O(h^{2/5  }).$$
Using~\eqref{eq.=forme} and  the Min-Max principle,   $\Delta_{f,h}^{\mathsf{Di},(q)}( \mathbb R_-^d, \mbf g )$ admits an eigenvalue  of order $O(h^{\frac 65})$ when $h\to 0$. 
Therefore,   for $h$ small enough, $\dim \Ran \, \pi_{[0,c_0h]}\big (\Delta_{f,h}^{\mathsf{Di},(q)}( \mathbb R_-^d, \mbf g )\big )=1$. 
This proves~\eqref{eq.last3-case1} and concludes  the proof of Proposition~\ref{pr.prop-loca1}. 

\subsubsection{Proof of Proposition~\ref{pr.prop-loca1-case1}}\label{sec.prop-loca1-case1}


We are now in
position to prove Proposition~\ref{pr.prop-loca1-case1}.   
Let us first state a preliminary result.
 \begin{lemma}\label{le-3.3.7}
Let us assume that the space $\overline{\mathbb R^d_-}$ is endowed with a metric tensor   $ \mbf{g}$ satisfying~\autoref{MetricRd-}. Assume that $f$ satisfies item (i) in~\autoref{PotentialRd-}.  
Define for all $x'\in \mathbb R^{d-1}$, $ \tilde{\mbf G}(x')= {\mbf
  G}(x',0)$ and let us introduce the metric on  $\mathbb R^d_-$:
\begin{equation}\label{eq.m-tilde}
 \forall x=(x',x_d)\in \overline{\mathbb R^d_-},  \ \,   \tilde{\mbf g}(x)= \tilde{\mbf G}(x')dx'\, ^2+dx_d^2.
\end{equation} 
Let $ (\mbf g_1,\mbf g_2)=( \mbf g, \tilde{\mbf g})$ or $ (\mbf g_1,\mbf g_2)=(\tilde{\mbf g}, \mbf g)$. 
Then, there exist $C>0$, $c>0$,   $h_0>0$, $\eta:[0,h_0]\to\mathbb
R_+$,  such that for $h\in (0,h_0)$, $\eta(h)=O(h^{2/5  })$ and  for all $w\in \Lambda^q H^1_{\mbf T}( \mathbb R^d_- )$ such that ${\rm supp}\, w\subset \mathsf B(0,h^{2/5})$, it holds,
\begin{equation}\label{eq.3.3.7-L2}
 \Vert w  \Vert_{L^2(\mathbb R^d_-, \mbf g_2)} =    \Vert w  \Vert_{L^2(\mathbb R^d_-, \mbf g_1)} \big (1+\eta(h)\big ),
\end{equation}
and for all $q \in \{0, \ldots, d\}$,
\begin{equation}\label{eq.3.3.7-Q}
  Q_{f,h}^{\mathsf{Di},(q)}( \mathbb R^d_- ,  \mbf  g_1)( w)\ge C\,  Q_{f,h}^{\mathsf{Di},(q)}( \mathbb R^d_-, \mbf g_2 )( w)- Ch^{ 7/5}\,   \Vert w  \Vert_{L^2(\mathbb R^d_-, \mbf g_2)}^2.
\end{equation}
 \end{lemma}
Equation~\eqref{eq.3.3.7-L2} is a simple consequence of the two metric
tensors are smooth and coincide at
$x_d=0$. Equation~\eqref{eq.3.3.7-Q} is easily obtained following the
proof of~\cite[Lemma 3.3.7]{helffer-nier-06}. 
 \noindent
Let us now prove Proposition~\ref{pr.prop-loca1-case1}.
\begin{proof}
Let us assume that~\autoref{MetricRd-} and \autoref{PotentialRd-} are  satisfied.
The proof is divided into three steps. 
\medskip

\noindent
 \textbf{Step 1: Proofs of~\eqref{eq.last1-case1} and~\eqref{eq.last2-case1}.} 
Let us recall that according to Lemma~\ref{le.IMS-loca}, it is sufficient to prove~\eqref{eq.last1-case1} and~\eqref{eq.last2-case1} for  forms $w\in  \Lambda^q H^1_{\mbf T}( \mathbb R^d_-,  \mbf{g} )$ supported in $\mathsf B(0,h^{2/5  })$. 
Because the metric tensor~$\tilde{\mbf g}$ defined
in~\eqref{eq.m-tilde} satisfies~\autoref{MetricRd-}
Proposition~\ref{pr.prop-loca1} implies that~\eqref{eq.last1-case1}
and \eqref{eq.last2-case1} hold
for~$\tilde{\mbf g}$ and~$f$. From those estimates
and~\eqref{eq.3.3.7-L2} and~\eqref{eq.3.3.7-Q}, one gets~\eqref{eq.last1-case1} and~\eqref{eq.last2-case1} for
 ${\mbf g}$ and $f$. 

\medskip
 
\noindent
 \textbf{Step 2: Proof of~\eqref{eq.last3-case1}.}
  Let us   assume that $0$ is  a critical point of index $q-1$ of
  $f|_{\pa \mathsf M}$ and $\mu_{d}<0$. Let $c>0$   be  strictly
  smaller than  the constants $C>0$ in~\eqref{eq.last1-case1}
  and~\eqref{eq.last2-case1}. Assume that $\Ran \, \pi_{[0,ch]}\big(
  \Delta_{f,h}^{\mathsf{Di},(q)}( \mathbb R_-^d,  \mbf g ) \big)\neq
  \{0\}$ and let us   consider a $L^2( \mathbb R_-^d,  \mbf g ) $-normalized form $\phi_h\in \Ran \, \pi_{[0,ch]}\big( \Delta_{f,h}^{\mathsf{Di},(q)}( \mathbb R_-^d,  \mbf g ) \big)$. This implies, using~\eqref{eq.complexe1},~\eqref{eq.complexe2},  and the results of {Step 1}, that
  $\mathsf  d_{f,h}\phi_h= 0 \text{ and } \mathsf d_{f,h}^*\phi_h=0$.
  Thus, it holds $ Q_{f,h}^{\mathsf{Di},(0)}( \mathbb R^d_- ,\mbf g)( \phi_h)=0$.
Consequently,   for  $h>0$ small enough, 
  $$\Ran \, \pi_{[0,ch]}\big (\Delta_{f,h}^{\mathsf{Di},(q)}( \mathbb R^d_- ,\mbf g )\big )=\Ker \Delta_{f,h}^{\mathsf{Di},(q)}( \mathbb R^d_- ,\mbf g ).$$ 
Now, let   $(\chi_1,\chi_2)$  be a partition of unity such that $\chi_1\in \mathcal C_c^\infty (\overline {\mathbb R_-^d)}$, $\chi_1=1$ on $\mathsf B(0,1/2)$, supp$\chi_1\subset  \mathsf B(0,1)$, and  $\chi_1^2+\chi_2^2=1$.
Using the IMS formula~\eqref{eq.IMS},  there exists $C>0$ such that :
$$ 
0 \ge   Q_{f,h}^{\mathsf{Di},(q)}( \mathbb R^d_- ,\mbf g) (\chi_1(h^{-2/5  }.)   \phi_h)+Q_{f,h}^{\mathsf{Di},(q)}(\mathbb R^d_- ,\mbf g)(  \chi_2(h^{-2/5  }.)  \phi_h) -C h^{6/5}\, \big \Vert   \phi_h\big \Vert_{L^2( \mathbb R^d_- ,\mbf g)}^2.
$$
 In addition, let us recall that (see indeed~\eqref{eq.chi_2-Q}),
   $$Q_{f,h}^{\mathsf{Di},(q)}(\mathbb R^d_- ,\mbf g)(  \chi_2(h^{-2/5  }.) \phi_h)\ge Ch^{4/5}  \big \Vert \chi_2(h^{-2/5  }.)  \phi_h\big \Vert_{L^2( \mathbb R^d_- ,\mbf g )}^2.$$
Therefore, one obtains in the limit $h\to 0$
\begin{equation}\label{norme-chi-b}
\big\Vert \chi_2(h^{-2/5  }.)  \phi_h\big \Vert_{L^2( \mathbb R^d_- ,\mbf g )}^2 =O (h^{2/5  }) \text{ and } \big\Vert \chi_1(h^{-2/5  }.)  \phi_h\big\Vert_{L^2( \mathbb R^d_- ,\mbf g )}^2=1+O  (h^{2/5  }  ),
\end{equation} 
and
\begin{equation}\label{quad-chi1-b}
Q_{f,h}^{\mathsf{Di},(q)}(\mathbb R^d_- ,\mbf g )  (\chi_1(h^{-2/5  }.)\phi_h)= O  (h^{6/5}  ).
\end{equation}
Then, using~\eqref{eq.3.3.7-Q} with $\mbf g_1=\mbf g$ and $\mbf g_2=\tilde{\mbf g}$, one gets for all~$h$  small enough:
$$
 h^{\frac 65} \ge   C\,  Q_{f,h}^{\mathsf{Di},(q)}( \mathbb R^d_-,\tilde{\mbf g} )(\chi_1(h^{-2/5  }.)\phi_h)- Ch^{\frac 75}\,  \big \Vert \chi_1(h^{-2/5  }.)\phi_h\big \Vert_{L^2(\mathbb R^d_-,\tilde{\mbf g})}^2.
$$
Notice that  from~\eqref{eq.3.3.7-L2} and~\eqref{norme-chi-b}, one has for $h$ small enough
$  \Vert \chi_1(h^{-2/5  }.)\phi_h\big \Vert_{L^2(\mathbb R^d_-,\tilde{\mbf g})}^2=1+O\big (h^{2/5  }\big ).$
Therefore, one obtains :
\begin{align}\label{eq.fin-preuveQ}
  Q_{f,h}^{\mathsf{Di},(q)}( \mathbb R^d_-,\tilde{\mbf g} )(\chi_1(h^{-2/5  }.)\phi_h) \le C \, h^{\frac 65}  \,  \big \Vert \chi_1(h^{-2/5  }.)\phi_h\big \Vert_{L^2(\mathbb R^d_-,\tilde{\mbf g})}^2.
\end{align} 
Recall (since $f$ and $\tilde{\mbf g}$ satisfy~\autoref{PotentialRd-} and~\autoref{MetricRd-}) that according to~Proposition~\ref{pr.prop-loca1}, there exist $c_0>0$ and $h_0>0$ such that for all $h\in (0,h_0)$, there exists a $L^2(\mathbb R^d_-,\tilde{\mbf g})$-normalized   $q$-form~$\Phi_h$ such that 
\begin{equation}\label{eq.PHIH}
\Ran \, \pi_{[0,c_0h]}\big (\Delta_{f,h}^{\mathsf{Di},(q)}( \mathbb R_-^d,\tilde{\mbf g} )\big )={\rm Span}\,  ( \Phi_h )=\Ker \Delta_{f,h}^{\mathsf{Di},(q)}( \mathbb R_-^d,\tilde{\mbf g} ).
\end{equation}
 Using Lemma~\ref{quadra} and~\eqref{eq.fin-preuveQ}, one obtains that for $h$ small enough: 
\begin{align*} 
{\rm dist}_{L^2( \mathbb R_-^d,\tilde{\mbf g} )}\, \big( \chi_1(h^{-2/5  }.)\phi_h  ,\, {\rm Span}\,  \Phi_h \big)= O(h^{\frac{1}{10}}).
\end{align*}
This implies together with~\eqref{norme-chi-b} and~\eqref{eq.3.3.7-L2}, and since we assume that $\Ran \, \pi_{[0,ch]}\big( \Delta_{f,h}^{\mathsf{Di},(q)}( \mathbb R_-^d,  \mbf g ) \big)\neq \{0\}$,   that   for $h$ small enough:
$$\dim \Ran \, \pi_{[0,ch]}\big (\Delta_{f,h}^{\mathsf{Di},(q)}( \mathbb R^d_- ,\mbf g )\big )= 1.$$

It remains to prove that  $\Ran \, \pi_{[0,ch]}\big( \Delta_{f,h}^{\mathsf{Di},(q)}( \mathbb R_-^d,  \mbf g ) \big)\neq \{0\}$. To this end, let us prove that  $
\Delta_{f,h}^{\mathsf{Di},(q)}( \mathbb R^d_- ,\mbf g )$ admits an eigenvalue of order $o(h)$ when $h\to 0$. 
Let us consider a $L^2(\mathbb R^d_-,\tilde{\mbf g})$-normalized   $q$-form $\Phi_h$ which satisfies~\eqref{eq.PHIH}.  
Recall that from~\eqref{norme-chi} and~\eqref{quad-chi1}, one has when $h\to 0$:
$$
\big\Vert \chi_2(h^{-2/5  }.) \Phi_h\big \Vert_{L^2( \mathbb R_-^d,\tilde{\mbf g} )}^2 =O(h^{2/5  }) \text{ and } \big\Vert \chi_1(h^{-2/5  }.)  \Phi_h\big\Vert_{L^2( \mathbb R_-^d,\tilde{\mbf g} )}^2=1+O\big (h^{2/5  }\big ).
$$
and 
$$Q_{f,h}^{\mathsf{Di},(q)}( \mathbb R^d_-,\tilde{\mbf g} )(\chi_1(h^{-2/5  }.)\Phi_h)\le C h^{ 6/5} \, \big \Vert \chi_1(h^{-2/5  }.)\Phi_h\big \Vert_{L^2(\mathbb R^d_-,\tilde{\mbf g})}^2.$$
From~\eqref{eq.3.3.7-L2} and~\eqref{eq.3.3.7-Q} (applied with $\mbf g_1=\tilde{\mbf g}$ and $\mbf g_2=\mbf g$), one deduces that: 
\begin{equation}\label{bes2}
  Q_{f,h}^{\mathsf{Di},(q)}( \mathbb R^d_- ,\mbf g )(\chi_1(h^{-2/5  }.)\Phi_h)\le C\,   h^{6/5}\, \big \Vert \chi_1(h^{-2/5  }.)\Phi_h\big \Vert_{L^2(\mathbb R^d_- ,\mbf g)}^2.
\end{equation}
Then, using the Min-Max principle, for $h$ small enough,  $
\Delta_{f,h}^{\mathsf{Di},(q)}( \mathbb R^d_- ,\mbf g )$ admits an eigenvalue of order $ h^{ 6/5}$ when $h\to 0$. Thus, $\Ran \, \pi_{[0,ch]}\big( \Delta_{f,h}^{\mathsf{Di},(q)}( \mathbb R_-^d,  \mbf g ) \big)\neq \{0\}$. 
This ends the proof of~\eqref{eq.last3-case1}. 
\medskip

\noindent
 \textbf{Step 3: Proof  of~\eqref{eq.norme-case1}.}  
Let  $\Psi_h\in \Ker \Delta_{f,h}^{\mathsf{Di},(q)}( \mathbb R^d_- ,\mbf g )$ such that $ \Vert \Psi_h  \Vert_{L^2(\mathbb R^d_- ,\mbf g)}=1$. Let   $\chi: \overline{\mathbb R^{d}_-}\to [0,1]$ be a $\mathcal C^\infty$ function  supported in a neighborhood of $0$ which equals~$1$ in a neighborhood of $0$ in $\overline{\mathbb R^{d}_-}$.  Let us  define $\tilde \chi=\sqrt{1-\chi^2}$. Then,  using Lemma~\ref{le.Green} (and the fact that $\partial_{\mathsf n_{\mathbb R^d_-}}f(x',0)=0$ for all $x'\in \mathbb R^{d-1}$), since there exists $c_1>0$ such that $\inf_{\text{supp}\tilde \chi}\vert \nabla f\vert \ge c_1 $,  it holds  
$$Q_{f,h}^{\mathsf{Di},(q)}( \mathbb R^d_- ,\mbf g )(\tilde \chi \Psi_h)\ge C \big \Vert \tilde \chi \Psi_h\big \Vert_{L^2(\mathbb R^d_- ,\mbf g)}^2$$
Using in addition the fact that  $Q_{f,h}^{\mathsf{Di},(q)}( \mathbb R^d_-,\mbf g )( \Psi_h)=0$ together with  the IMS formula~\eqref{eq.IMS}, one obtains~\eqref{eq.norme-case1} using a similar reasoning as in~\eqref{norme-chi-b} and~\eqref{quad-chi1-b}. This ends the proof of Proposition~\ref{pr.prop-loca1-case1}.
\end{proof}


 \subsection{Proof of Theorem~\ref{thm.main1}}\label{sec.numb-small}
Let us assume that~\autoref{B} holds. 
For a fixed $q\in \{0,\ldots,d\}$, let us  consider the operator $\Delta_{f,h}^{\mathsf{Di},(q)}(\mathsf M,\mbf g_{\mathsf M} )$. We will identify the number of   eigenvalues  smaller than $\mathsf c h$ for 
 this operator,  for some $\mathsf c>0$ and for all
 sufficiently small $h$. 
 
 According to the analysis made in~\cite[Chapter~3]{helffer-nier-06}
 and~\cite{helffer-sjostrand-85}, it is already known that one can build linearly
 independent quasi-modes associated with 
 the  (generalized) critical points in $\mathsf U_q^{\mathsf M}\cup  \mathsf U^{\pa \mathsf M,1}_q$ which thus yield 
 at least  $\mathsf m_q^{\mathsf M}+\mathsf m_q^{\pa  \mathsf M,1}$
 small   eigenvalues.  The main novelty compared
 to~\cite[Chapter~3]{helffer-nier-06}  is that we also have to consider critical
 points of $f$ located on $\partial \mathsf M$:
\begin{equation}\label{eq.defB}
 \mathsf B^{\pa \mathsf M,2}:=\big \{z \in \pa \mathsf M, \, \vert \nabla f(z)\vert =0 \}.
\end{equation}  
 In the proof, we will thus consider all the critical points in
$$ \mathsf P_q= \mathsf U_q^{\mathsf M}\cup  \mathsf U^{\pa \mathsf M,1}_q\cup  \mathsf B^{\pa \mathsf M,2}$$ 
 as potential candidates to generate small eigenvalues, and we will prove 
 that only those critical points in $\mathsf Q_q \subset \mathsf P_q$ where
$$\mathsf Q_q= \mathsf U_q^{\mathsf M}\cup  \mathsf U^{\pa \mathsf M,1}_q\cup  \mathsf U^{\pa \mathsf M,2}_q$$
 will actually contribute to the spectrum of
 $\Delta_{f,h}^{\mathsf{Di},(q)}(\mathsf M,\mbf g_{\mathsf M} )$ in
 $[0,\mathsf ch]$.
%

By assumption~\autoref{B}, for all $z\in  \mathsf B^{\pa \mathsf M,2}$, 
$ \partial_{\mathsf n_{\mathsf M}}f=0$ in a neighborhood of $z$
in~$\pa \mathsf M$. Let us thus introduce a family $(\mathsf
V_y)_{y\in \mathsf P_q}$ of neighborhoods  in
$\overline{\mathsf M}$ of $y\in  \mathsf P_q$ such that:
 \begin{itemize}
\item For all $y\in  \mathsf U^{\mathsf M}_q$, $\overline{\mathsf V_y}\subset \mathsf M$ and $y$ is the only critical point of $f$ in $\overline{\mathsf V_y}$.
\item  For all $y\in  \mathsf U^{\pa \mathsf M,1}_q$, $
  \partial_{\mathsf n_{\mathsf M}}f>0$ on $\pa \mathsf M \cap
  \overline{\mathsf V_y}$ and $y$ is the only critical point of  $f|_{\pa \mathsf M}$ in $\pa \mathsf M \cap  \overline{\mathsf V_y}$. 
\item  For all $y\in  \mathsf B^{\pa \mathsf M,2}$, $ \partial_{\mathsf n_{\mathsf M}}f=0$ on $\pa \mathsf M \cap\overline {\mathsf V_y}$ 
and $y$ is the only critical point of  $f$ in $\overline {\mathsf V_y}$. 
\item The  sets $(\overline{\mathsf V_y})_{y\in \mathsf P_q}$ are pairwise disjoint. 
\end{itemize}
The neighborhoods $(\mathsf V_y)_{y\in \mathsf P_q}$ may be shrunk in the
following in order to introduce local coordinates on $\mathsf V_y$, this will be made precise below.  
In order to use an  IMS localization formula, one now introduces a partition of unity $(\chi_y)_{y\in \mathsf P_q}\cup \tilde \chi$ which is adapted with the neighborhoods $(\mathsf V_y)_{y\in \mathsf P_q}$: for all $y\in \mathsf P_q$, $\chi_y: \overline{\mathsf M}\to [0,1]$ is $\mathcal C^\infty$, supported in $\mathsf V_y$, and $\chi_y=1$ in a neighborhood of $y$ in $\overline{\mathsf M} $. 
Then, one defines: 
$$\tilde \chi :=\sqrt{ 1-\sum_{y\in \mathsf P_q} \chi_y^2},$$
so that on $\overline{\mathsf M}$, $\tilde \chi ^2+\sum_{y\in \mathsf P_q} \chi_y^2=1.$
 Let $w\in  \Lambda^qH^1_{\mbf T}(\mathsf M,\mbf g_{\mathsf M} )$. 
The IMS formula~\cite{CFKS,helffer-nier-06}  reads:
  \begin{align}
\nonumber
Q_{f,h}^{\mathsf{Di},(q)}(\mathsf M,\mbf g_{\mathsf M} )(   w ) &=\sum_{y\in \mathsf P_q} Q_{f,h}^{\mathsf{Di},(q)}( \mathsf M ,\mbf g_{\mathsf M} )(\chi_y w)    -\sum_{y\in \mathsf P_q}  h^2\,\big \Vert w\,  \nabla \chi_y   \big \Vert_{L^2( \mathbb R_-^d,\mbf g_{\mathsf M}  )}^2\\
\nonumber
&\quad + Q_{f,h}^{\mathsf{Di},(q)}( \mathsf M ,\mbf g_{\mathsf M} )(\tilde \chi  w)    - h^2\,\big \Vert w\, \nabla  \tilde \chi   \big \Vert_{L^2( \mathbb R_-^d,\mbf g_{\mathsf M}  )}^2.
\end{align}
Thus, there exists $C>0$ such that 
\begin{align}
\label{eq.IMSfin}
Q_{f,h}^{\mathsf{Di},(q)}(\mathsf M,\mbf g_{\mathsf M} )(   w ) &\ge Q_{f,h}^{\mathsf{Di},(q)}( \mathsf M ,\mbf g_{\mathsf M} )(\tilde \chi  w)   -C\,   h^2\, \Vert w    \Vert_{L^2( \mathbb R_-^d,\mbf g_{\mathsf M}  )}^2 +\sum_{y\in \mathsf P_q} Q_{f,h}^{\mathsf{Di},(q)}( \mathsf M ,\mbf g_{\mathsf M} )(\chi_y w) .
\end{align}
\noindent
To prove Theorem~\ref{thm.main1}, we will  study   separately the quantities $Q_{f,h}^{\mathsf{Di},(q)}( \mathsf M ,\mbf g_{\mathsf M} )(\tilde \chi  w)$ and $Q_{f,h}^{\mathsf{Di},(q)}( \mathsf M ,\mbf g_{\mathsf M} )(\chi_y w)$ for $y\in \mathsf P_q$. 
The latter will be estimated using
Proposition~\ref{pr.prop-loca1-case1}, after having introduced
coordinates on $\mathsf V_y$ in which the metric has the block structure
assumed in item $(i)$ of \autoref{MetricRd-}, and $f$ satisfies~\eqref{eq.decf}. The proof of Theorem~\ref{thm.main1} is divided into four steps. 

\medskip

\noindent
\textbf{Step 1: Results from~\cite[Chapter 3]{helffer-nier-06} and~\cite{helffer-sjostrand-85}. }

 \noindent
\textbf{Step 1a: Quasi-modes associated with points in $  \mathsf
  U_q^{\mathsf M}\cup  \mathsf U^{\pa \mathsf M,1}_q$.}  Let     $y\in
\mathsf U_q^{\mathsf M}\cup  \mathsf U^{\pa \mathsf M,1}_q$. Let us
introduce the set $\mathsf E$ defined as follows:
 $$\text{$\mathsf E=\mathbb R^d$ if $ y\in  \mathsf U_q^{\mathsf
     M} \quad$ and  $\quad\mathsf E=\mathbb R^d_-$   if $y\in \mathsf U^{\pa \mathsf
     M,1}_q$  .}$$
Up to reducing the neighborhood   $\mathsf V_y$ of $y$ in $\overline{\mathsf M}$, the following results hold according to the analysis in~\cite[Chapter 3]{helffer-nier-06} (see also~\cite{helffer-sjostrand-85} and~\cite{HKN} for the case when $\mathsf E=\mathbb R^d$). There exists a   $\mathcal C^\infty$ system of coordinates  
$$v\in \mathsf V_y\mapsto   x(v) \in  \overline{\mathsf E},$$
and a metric tensor $\mbf g_y$ and a function $f_y$ on $
\overline{\mathsf E}$ which coincide on $x( \mathsf V_y)$ respectively with $\mbf
g_{\mathsf M}$ and $f$ expressed in the $x$-coordinates, 
such that the following holds:
\begin{equation}\label{eq.span=12}
\exists  \mathsf c_y>0, \exists h_0 >0, \forall h \in (0,h_0), \Ran \, \pi_{[0,\mathsf c_yh]}\big (\mathsf T_y \big )=\Ker \mathsf T_y \text{ has dimension  } 1,
\end{equation}
where $\mathsf T_y=\Delta_{f_y,h}^{(q)}( \mathbb R^d,\mbf g_y )$ if
$y\in \mathsf U_q^{\mathsf
     M}$ and $\mathsf T_y=\Delta_{f_y,h}^{\mathsf{Di},(q)}( \mathbb R^d_-,\mbf g_y )$ if $y\in \mathsf U^{\pa \mathsf
     M,1}_q$. 
 Moreover, let  $\chi: \overline{\mathsf E} \to [0,1]$ be a $\mathcal C^\infty$ function supported in  $x(\mathsf V_y)$  which equals~$1$ in a neighborhood of $0$ in $x(\mathsf V_y)$.
Let  $\Psi_h^y\in \Ker\mathsf T_y$ such that $  \Vert \Psi_h^y\Vert_{L^2(\mathsf E,\mbf g_y)}=1$, then, 
 in the limit $h\to 0$, it holds:
\begin{equation}\label{eq.omega-z2}
  \Vert \chi \Psi_h^y  \Vert_{L^2(\mathsf M,\mbf g_{\mathsf M})} =1+O  (h^{2})  \text{ and } Q_{f,h}^{\mathsf{Di},(q)}( \mathsf M ,\mbf g_{\mathsf M})( \chi \Psi_h^y)=O(h^2).
  \end{equation}  

%
%
%
%
\medskip

 \noindent
\textbf{Step 1b: Lower bound on $Q_{f,h}^{\mathsf{Di},(q)}( \mathsf M
  ,\mbf g_{\mathsf M} )(\tilde \chi  w)$.} Moreover, it is proved
in~\cite[Section 3.4]{helffer-nier-06} that there exist $\tilde C>0$  and $h_0>0$ such that for all $h\in (0,h_0)$ and all  $w\in  \Lambda^qH^1_{\mbf T}(\mathsf M,\mbf g_{\mathsf M} )$:
\begin{equation}\label{eq.tildeC}
 Q_{f,h}^{\mathsf{Di},(q)}( \mathsf M ,\mbf g_{\mathsf M} )(\tilde \chi  w)\ge  \tilde Ch\,  \Vert\tilde \chi w \Vert_{L^2(\mathsf M,\mbf g_{\mathsf M})}^2.
\end{equation}
More precisely,  by~\cite[Section 3.4]{helffer-nier-06}, one has for any $x\in \overline{\mathsf M}\setminus \mathsf P_q$:
\begin{enumerate}
\item  Either $x\in \mathsf M$ with $\vert \nabla f(x)\vert \neq 0$, in which case there exist $c>0$ and  a neighborhood $\mathsf V_x$ of $x$ in $\mathsf M$ such that $Q_{f,h}^{\mathsf{Di},(q)}( \mathsf M ,\mbf g_{\mathsf M} )( \chi_x  w)\ge c\,  \Vert \chi_x w \Vert_{L^2(\mathsf M,\mbf g_{\mathsf M})}^2$ for any smooth function $\chi_x$   supported in $\mathsf V_x$;
\item Or $x\in \mathsf M$ with  $\vert \nabla f(x)\vert = 0$  and $x\notin \mathsf U_q^{\mathsf M}$, 
 in which case there exist $c>0$ and  a neighborhood $\mathsf V_x$ of $x$ in $\mathsf M$ such that $Q_{f,h}^{\mathsf{Di},(q)}( \mathsf M ,\mbf g_{\mathsf M} )( \chi_x  w)\ge ch\,   \Vert \chi_x w \Vert_{L^2(\mathsf M,\mbf g_{\mathsf M})}^2$ for any smooth function $\chi_x$   supported in $\mathsf V_x$;
 \item Or $x\in \pa \mathsf M$ with $\vert \nabla f(x)\vert \neq 0$  and $x\notin  \mathsf U^{\pa \mathsf M,1}_q$, in which case there exist $c>0$ and  a neighborhood $\mathsf V_x$ of $x$ in $\overline{\mathsf M}$ such that  $Q_{f,h}^{\mathsf{Di},(q)}( \mathsf M ,\mbf g_{\mathsf M} )( \chi_x  w)\ge ch\,   \Vert \chi_x w \Vert_{L^2(\mathsf M,\mbf g_{\mathsf M})}^2$ for any smooth function $\chi_x$   supported in $\mathsf V_x$.
\end{enumerate}
Equation \eqref{eq.tildeC} then follows from the fact that $\tilde \chi=0$ in a neighborhood of  
all the points in  $\mathsf P_q$.
\medskip

\noindent
\textbf{Step 2: Change of coordinates near $y\in \mathsf B^{\pa \mathsf M,2}$.}
\medskip

\noindent
For $\ve>0$ small enough,  for all $v\in  \overline { \mathsf M}$ such that $\mathsf d_{\overline { \mathsf M}}(v,\pa  \mathsf M)<\ve$, there exists a unique 
point $\mathsf z(v)\in \pa \mathsf M$ such that  $$x_d(v):=-\mathsf d_{\, \overline {\mathsf M}}(v,\pa \mathsf M)=-\mathsf d_{  \overline {\mathsf M}}(v,\mathsf z(v)),$$
 where we recall  $\mathsf d_{\, \overline{\mathsf M}}$ denotes   the
 geodesic distance in $ \overline  {\mathsf M}$. Moreover the function
 $v\mapsto \mathsf d_{ \overline  {\mathsf M}}(v,\partial \mathsf M)$ is smooth on the set $\{v\in  \overline {\mathsf M}, \mathsf d_{\overline{ \mathsf M}}(v,\pa {\mathsf M})<\ve\}$. 

 Let us now consider a fixed $y\in  \mathsf B^{\pa \mathsf M,2}\subset \pa \mathsf M$ and let  $x'$ be a local system of coordinates in $\pa \mathsf M$ centered at $y$.  Then there exists a neighborhood $\mathsf U_y$ of $y$ in $\overline {\mathsf M}$ such that the mapping\begin{equation}\label{eq.norm-teng}
v\in \mathsf U_y \mapsto x(v):=(x'(\mathsf z(v)),x_d(v))\in \mathbb R^{d-1}\times \overline{\mathbb R_-}
\end{equation}
 is a system of coordinates near $y\in \pa \mathsf M$, centered at
 $y$: this is the so-called tangential-normal system of coordinates.  Then, up to choosing $\mathsf V_y$  smaller, one can assume that:
$$\mathsf U_y=\mathsf V_y.$$
It   holds, by construction of $v\mapsto x(v)$:  
$$
x(y)=0\,,\ \ \{v\in \mathsf V_y, \, x_d(v)<0\}=  \mathsf M\cap  \mathsf V_y \,,\ \ 
\{v\in \mathsf V_y,\,  x_d(v)=0\}=\pa  \mathsf M \cap \mathsf V_y,
$$
and for all $(x',0)\in x(\mathsf V_y)$, 
$$
 \pa _{x_d}v(x',0)  = \mathsf n_{\mathsf M}(v(x',0))\,.
$$
Moreover, by construction,     the metric tensor $\mbf g_{\mathsf M}$ in the $x$-coordinates   has the desired  block structure~of item $(i)$ in \autoref{MetricRd-}, i.e.: 
\begin{equation}\label{g_y}
 \forall  (x',x_d)\in x(\mathsf V_y), \ \, \mbf  g_{\mathsf M}(x',x_d)=\mbf G_{\mathsf M}(x',x_d)dx'\, ^2+dx_d^2
\end{equation}
where it is assumed, without loss of generality, that $\mbf G_{\mathsf
  M} (0,0)$ is the identity matrix.
In the following and with a slight abuse of notation, one still
denotes by $f$  the function~$f$ in the $x$-coordinates. Since $\mbf
g_{\mathsf M}(0,0)= dx'\, ^2+dx_d^2$, the
Hessian matrix of $f$ (resp. the Hessian matrix of
$f|_{\partial \mathbb R_-^d}$) at $0$ in this new coordinates is unitarily equivalent
to $\hess f(y)$ (resp.
$\hess f|_{\partial \mathsf M}(y)$). In
particular, they have the same eigenvalues.
Let us recall that   according to  \autoref{B},   $\mathsf n_{\mathsf M}(y)$ is an eigenvector of $\Hess f(y)$ for the eigenvalue $\mu_y$, see~\eqref{eq.lambdad}, also denoted by $\mu_{d}$ in the following. 
Let us denote by $\mu_1,\dots,\mu_{d-1}$  the $d-1$  remaining eigenvalues of  
$\Hess f(y)$, the associated eigenspace being $T_y\partial \mathsf M$.
These are also the eigenvalues of $\hess f|_{\partial \mathsf M}(y)$.
Let us recall that, up to an orthogonal transformation on
$x'=(x_{1},\dots,x_{d-1})$, it holds, in a neighborhood of $0$ and  in the $x$-coordinates, 
\begin{equation}\label{eq.cv-pa-omega-nablaf=04}
f(x)\ =\ f(0)+  \sum_{j=1}^d \frac{\mu_j}{2} \, x_j^2+ O(\vert x\vert^3),
\end{equation} 
which is precisely~\eqref{eq.decf}. 
\begin{remark}
Let us mention that~\eqref{eq.cv-pa-omega-nablaf=04}  only requires that   $\mathsf n_{\mathsf M}(y)$ is an eigenvector  $\Hess f(y)$.  The stronger  assumption that $\partial_{\mathsf n_{\mathsf M}}f=0$ on $\pa \mathsf M\cap \mathsf V_y$ will be necessary to use the results of   Proposition~\ref{pr.prop-loca1-case1}.
\end{remark}
\noindent
In addition, it holds:
 \begin{equation}\label{eq.stable2}
 \forall x'\in \mathbb R^{d-1} \cap x(\mathsf V_y), \  \partial_{\mathsf n_{\mathbb R^d_-}}f(x',0)=0.  
  \end{equation}
In order to use Proposition~\ref{pr.prop-loca1-case1}, we extend the function $f$ and the metric~$\mbf g_{\mathsf M}$ from $x(\mathsf V_y)$ to $\overline{\mathbb R^{d}_-}$ so that they  satisfy respectively~\autoref{PotentialRd-} and~\autoref{MetricRd-}. We denote by $f_y$ and $\mbf g_y$ these extensions, defined on $\overline{\mathbb R^d_-}$. 
Notice that it holds since $\chi_y$ is supported in $\mathsf V_y$,
$$Q_{f_y,h}^{\mathsf{Di},(q)}( \mathbb R^{d-1},\mbf g_y )(\chi_y  w)=Q_{f,h}^{\mathsf{Di},(q)}( \mathsf M ,\mbf g_{\mathsf M} )(\chi_y  w) \text{ and } \Vert \chi_y w  \Vert_{L^2(\mathbb R^{d-1},\mbf g_y)}= \Vert \chi_y w  \Vert_{L^2(\mathsf M,\mbf g_{\mathsf M})},$$
where with a slight abuse of notation $\chi_y  w$ both denotes the $q$-form defined on $\mathsf M$ and in the $x$-coordinates. These equalities will be used many times in the rest of the proof.

\medskip

\noindent
\textbf{Step 3: Contributions of the  points in $ \mathsf B^{\pa \mathsf M,2}$.}  
 \medskip

\noindent
According to Step 2, one can use Proposition~\ref{pr.prop-loca1-case1}
to study $Q_{f,h}^{\mathsf{Di},(q)}( \mathsf M ,\mbf g_{\mathsf M}
)(\chi_y  w)$ when~$y\in \mathsf B^{\pa \mathsf M,2}$, where, we
recall, $w\in  \Lambda^0H^1_{\mbf T}(\mathsf M,\mbf g_{\mathsf M} )$. There are thus three possible cases:
\begin{enumerate}
\item By~\eqref{eq.last1-case1}, if $q=0$, there exists $C>0$ such that  for all $h$ small enough:
\begin{equation}\label{eq.last1-case1-22}
Q_{f,h}^{\mathsf{Di},(0)}( \mathsf M ,\mbf g_{\mathsf M} )(\chi_y  w)\ge Ch\,    \Vert \chi_y w  \Vert_{L^2(\mathsf M,\mbf g_{\mathsf M})}^2.
\end{equation}
 \item By~\eqref{eq.last2-case1}, for $q\in \{1,\ldots,d\}$,  if the
   index of $y$ as a critical point of $f|_{\pa \mathsf M}$ is not $q-1$, or if $\mu_{d}>0$, then, there exists $C>0$ such that  for all $h$ small enough:
 \begin{equation}\label{eq.qq2}
 Q_{f,h}^{\mathsf{Di},(q)}( \mathsf M ,\mbf g_{\mathsf M} )(\chi_y  w)\ge Ch\,   \Vert\chi_y w \Vert_{L^2(\mathsf M,\mbf g_{\mathsf M})}^2.
\end{equation}

\item 
For $q\in \{1,\ldots,d\}$, if the index of $y$ as a critical point of
$f|_{\pa \mathsf M}$ is  $q-1$ and  $\mu_{d}<0$ (namely if $y \in \mathsf U^{\pa \mathsf M,2}_q$), 
from~\eqref{eq.last3-case1}  there exists $\mathsf c_y>0$   such that
for all $h$ small enough:
\begin{equation}\label{eq.span=11}
  \Ran \, \pi_{[0,\mathsf c_yh]}\big (\Delta_{f_y,h}^{\mathsf{Di},(q)}( \mathbb R_-^d,\mbf g_y )\big )=\Ker \Delta_{f_y,h}^{\mathsf{Di},(q)}( \mathbb R_-^d,\mbf g_y ) \text{ has dimension  } 1.
\end{equation}
Moreover, let  $\chi: \overline{\mathbb R^{d}_-}\to [0,1]$   be a $\mathcal C^\infty$ function supported in  $\mathsf V_y$  which equals~$1$ in a neighborhood of $y$ in $\mathsf V_y$.
 Let  $\Psi_h^y\in \Ker \Delta_{f_y,h}^{\mathsf{Di},(q)}( \mathbb R_-^d,\mbf g_y )$ such that $\Vert \Psi_h^y \Vert_{L^2(\mathbb R^d_-,\mbf g_y)}=1$, then, 
 in the limit $h\to 0$, it holds (by~\eqref{eq.norme-case1}):
\begin{equation}\label{eq.norme-case1-2}
 \Vert \chi \Psi_h^y  \Vert_{L^2(\mathsf M,\mbf g_{\mathsf M})} =1+O  (h^{2}) \text{ and } Q_{f,h}^{\mathsf{Di},(q)}( \mathsf M ,\mbf g_{\mathsf M})( \chi \Psi_h^y)=O(h^2).
  \end{equation}
\end{enumerate}
Let us insist again on the fact that in~\eqref{eq.last1-case1-22}--\eqref{eq.qq2}, the constants $C$
and the interval $(0,h_0) \ni h$ do not depend on $w$.

\medskip
\noindent
\textbf{Step 4: End of the proof of Theorem~\ref{thm.main1}.}  
\medskip

 \noindent
 Let us consider $\eta_1>0$.  
Using the Min-Max principle,
Equations~\eqref{eq.norme-case1-2},~\eqref{eq.omega-z2}  together with
the fact that the supports of $(\chi_y)_{y\in \mathsf Q_q}$ are
pairwise disjoint, one gets that $\Delta_{f,h}^{\mathsf{Di},(q)}(
\mathsf M,\mbf g_{\mathsf M} )$ admits at least~$\mathsf m_q$
eigenvalues of order $O(h^{2})$ when $h\to 0$. Thus, for $h$
sufficiently small,
 $$\dim \Ran \, \pi_{[0,\eta_1h]}\big (\Delta_{f,h}^{\mathsf{Di},(q)}( \mathsf M,\mbf g_{\mathsf M} )\big )\ge \mathsf  m_q.$$
 Let us now prove the reverse inequality holds if $\eta_1$ is  small enough.   To this end, let us consider   $w\in  \Lambda^qH^1_{\mbf T}(\mathsf M,\mbf g_{\mathsf M} )$   such that $\Vert w\Vert_{L^2(\mathsf M,\mbf g_{\mathsf M} )}=1$ and  
$$Q_{f,h}^{\mathsf{Di},(q)}(\mathsf M,\mbf g_{\mathsf M} )(   w )\le \eta_1 h,$$
and let us prove that  the distance between $w$ and $\text{Span}\, \big (\chi_y\,\Psi_h^y, \, y\in \mathsf Q_q \big ) $  (which, we recall,  is of dimension $\mathsf m_q$ because $(\chi_y)_{y\in \mathsf Q_q}$  have supports which are pairwise  disjoint) 
goes to $0$ when $h\to 0$, for a sufficiently small $\eta_1$.
Using~\eqref{eq.IMSfin}, it holds for some $C_0>0$ independent of $h$:
\begin{align}\label{eq.IMS33}
\eta_1\, h  &\ge Q_{f,h}^{\mathsf{Di},(q)}( \mathsf M ,\mbf g_{\mathsf M} )(\tilde \chi  w)   -C_0   h^2  +\sum_{y\in \mathsf P_q} Q_{f,h}^{\mathsf{Di},(q)}( \mathsf M ,\mbf g_{\mathsf M} )(\chi_y w) .
\end{align}
In the following $\tilde c>0$ is a constant  independent of $h$, $\eta_1$ and $w$, which can change from one occurrence to another. 
Then~\eqref{eq.IMS33} together with~\eqref{eq.qq2}  and
\eqref{eq.last1-case1-22} yields that   for   all $y\in \mathsf B^{\pa
  \mathsf M,2} \setminus \mathsf U_q^{\pa \mathsf M,2}= \mathsf
P_q\setminus \mathsf Q_q$, for all $h$ small enough: 
\begin{equation}\label{eq.est-norm1}
 \Vert  \chi_y w\Vert_{L^2(\mathsf M)}\le \sqrt {\eta_1h +C_0h^2 }/{\sqrt{  C h }}  \le  \tilde c \, {\sqrt{\eta_1}}.
   \end{equation}
In addition, from Equations~\eqref{eq.IMS33} and~\eqref{eq.tildeC}, one has for $h$ small enough:
\begin{equation}\label{eq.est-norm2}
  \Vert\tilde \chi w \Vert_{L^2(\mathsf M)}\le  \sqrt {\eta_1h +C_0h^2 }/{\sqrt{\tilde C h }}\le   \tilde c \, {\sqrt{\eta_1}}.
  \end{equation}
Furthermore, one deduces from~\eqref{eq.IMS33} that  for all $y\in \mathsf P_q$, for $h$  small enough:
\begin{equation}\label{eq.est-Q3}
Q_{f,h}^{\mathsf{Di},(q)}( \mathsf M ,\mbf g_{\mathsf M} )(\chi_y w) \le 2\, \eta_1\, h.
  \end{equation}
For $y\in \mathsf Q_q= \mathsf U_q^{\mathsf M}\cup  \mathsf U^{\pa
  \mathsf M,1}_q\cup  \mathsf U^{\pa \mathsf M,2}_q$, set (see the
quasi-modes introduced in~\eqref{eq.span=12}--\eqref{eq.omega-z2} and~\eqref{eq.span=11}--\eqref{eq.norme-case1-2})
\begin{equation}\label{eq.PHIY}
\Phi_h^y := \chi_y\Psi_h^y\, /\,  \Vert \chi_y \Psi_h^y  \Vert_{L^2(\mathsf M)} .
  \end{equation}
It holds:
\begin{align*}
&{\rm dist}_{L^2( \mathsf M )}\Big( w, \text{Span}\, \big ( \Phi_h^y, \, y\in \mathsf Q_q \big ) \Big)^2= \Big \Vert w- \sum \limits_{y \in \mathsf Q_q}   \lp w,  \Phi_h^y\rp _{L^2( \mathsf M )} \Phi_h^y   \Big \Vert_{L^2( \mathsf M )}^2\\
&= \sum_{z\in \mathsf P_q}\Big \Vert \chi_z\Big(  w- \sum \limits_{y \in \mathsf Q_q}   \lp w,  \Phi_h^y\rp _{L^2( \mathsf M )} \Phi_h^y\Big)  \Big \Vert_{L^2( \mathsf M )}^2 +\Big\Vert \tilde \chi\Big(  w- \sum \limits_{y \in \mathsf Q_q}   \lp w,  \Phi_h^y\rp _{L^2( \mathsf M )} \Phi_h^y\Big)   \Big\Vert_{L^2( \mathsf M )}^2.
%
%
 \end{align*}
 The first inequalities in~\eqref{eq.norme-case1-2} and \eqref{eq.omega-z2} imply that for all $y\in \mathsf Q_q$, as $h\to 0$:
 $$ \Vert (1-\tilde \chi)\chi_y \Psi_h^y  \Vert_{L^2(\mathsf M)} =1+O  (h^{2}).$$ 
 Therefore, using in addition~\eqref{eq.est-norm2}, one deduces that:
 $$  \Big\Vert \tilde \chi\Big(  w- \sum \limits_{y \in \mathsf Q_q}   \lp w,  \Phi_h^y\rp _{L^2( \mathsf M )} \Phi_h^y\Big)   \Big\Vert_{L^2( \mathsf M )}^2 \!\!\le \rho \big ( \Vert\tilde \chi w \Vert_{L^2(\mathsf M)}^2 +  \sum \limits_{y \in \mathsf Q_q}     \Vert\tilde \chi  \Phi_h^y \Vert_{L^2(\mathsf M)}^2 \big)\le \tilde c \big(  \eta_1 +h^2\big),$$
 where  $\rho>0$ is independent of $h$, $\eta_1$ and $w$,
 and  since the supports of $(\chi_y)_{y\in \mathsf P_q}$ are pairwise disjoint, 
 $$\sum_{z\in \mathsf P_q\setminus \mathsf Q_q}\Big \Vert \chi_z\Big(  w- \sum \limits_{y \in \mathsf Q_q}   \lp w,  \Phi_h^y\rp _{L^2( \mathsf M )} \Phi_h^y\Big)  \Big \Vert_{L^2( \mathsf M )}^2 \le \tilde c\, \eta_1.$$
 On the other hand, when $z\in  \mathsf Q_q$, one has since the supports of $(\chi_y)_{y\in \mathsf P_q}$ are pairwise disjoint and using  Lemma~\ref{quadra},~\eqref{eq.est-Q3},~\eqref{eq.span=11}, and~\eqref{eq.span=12},
\begin{align*}
\Big \Vert \chi_z\Big(  w- \sum \limits_{y \in \mathsf Q_q}   \lp w,  \Phi_h^y\rp _{L^2( \mathsf M )} \Phi_h^y\Big)  \Big \Vert_{L^2( \mathsf M )}^2\!\!  &=  \Big \Vert \chi_z   w   -  \lp w,  \chi_z \Psi_h^z\rp _{L^2( \mathsf M )}  \frac{\chi_z^2 \Psi_h^z }{ \Vert \chi_z\Psi_h^z  \Vert_{L^2(\mathsf M)} ^2}\Big  \Vert_{L^2( \mathsf M )}^2   \\
&\le \big \Vert (1-\pi_{[0,\mathsf c_zh]} (\mathsf T_z ) )(\chi_z  w)\big  \Vert_{L^2( \mathsf E ,\mbf g_z  )}^2 +  \tilde{c} h^2\\
&\le  \tilde{c}(    \eta_1/\mathsf c_z  +h^2 \big),
\end{align*}
 where   $\mathsf T_z=\Delta_{f,h}^{\mathsf{Di},(q)}( \mathbb R^d_-,\mbf g_z )$ if $z\in \pa \mathsf M$ and $\mathsf T_z=\Delta_{f,h}^{(q)}( \mathbb R^d,\mbf g_z )$ if $z\in \mathsf M$ (recall that $\mathsf E=\mathbb R^d_-$ if $y\in  \pa \mathsf M$ and $\mathsf E=\mathbb R^d$ if $ y\in  \mathsf M$). 
 In conclusion,   as $\eta_1\to 0$ and $h\to 0$,
 $${\rm dist}_{L^2( \mathsf M )}\Big( w, \text{Span}\, \big ( \Phi_h^y, \, y\in \mathsf Q_q \big ) \Big)\to 0.$$ 
This implies that there exist $\eta>0$ and $h_0>0$ such that for all $\eta_1\in (0,\eta)$ and $h\in (0,h_0)$,   $\dim \Ran \, \pi_{[0,\eta_1h]}\big (\Delta_{f,h}^{\mathsf{Di},(q)}( \mathsf M,\mbf g_{\mathsf M} )\big )\le \mathsf  m_q$.  This concludes the proof of Theorem~\ref{thm.main1}.

 \subsection{Application of Theorem~\ref{thm.main1} to the infinitesimal generator of the diffusion~(\ref{eq.langevin})}
 \label{sec.seclienwitten}
Let us go back to the setting introduced in Section~\ref{sec:intro}.
Recall that $\Omega$ is a smooth bounded   domain of~$\mathbb R^d$,
and let us apply the results stated above to $\mathsf M=\Omega$
endowed with the standard Euclidean metric tensor:  $\mbf
g_{\mathsf M} =(\delta_{i,j} \, dx_i \, dx_j)_{i,j=1,\ldots,d}$. For
the ease of notation, we henceforth omit the reference to the  metric tensor in the notation of the Witten Laplacian and the Sobolev spaces.

\subsubsection{Notation for weighted Sobolev spaces}

\noindent
For $q\in \{0, \ldots,d\}$ and $m\in \mathbb N$, one denotes by
$\Lambda^qH^m_w(\Omega)$\label{page.wsobolevq} the weighted Sobolev spaces of $q$-forms with
regularity index $m$, for the weight $e^{-\frac{2}{h} f(x)}dx$ on
$\Omega$ (hence the subscript $w$ in $\Lambda^qH^m_w(\Omega)$). 
We refer again  for example to~\cite{GSchw} for an introduction to Sobolev spaces on manifolds with boundaries. For $q\in\{0,\ldots,d\}$ and $m> \frac{1}{2}$, the set $\Lambda^qH^m_{w,\mbf T}(\Omega)$ is defined by 
$$\Lambda^qH^m_{w,\mbf T}(\Omega):= \left\{v\in \Lambda^qH^m_w(\Omega)
  \,|\,  \mathbf  tv=0 \ {\rm on} \ \partial \Omega\right\}.$$
The space
$\Lambda^qH^0_w(\Omega)$ is denoted by  $\Lambda^qL^2_w(\Omega)$. Let us mention that  the space $\Lambda^0H^1_{w,\mbf T}(\Omega)$ (resp. $\Lambda^0L^2_w(\Omega)$)  is
the space $H^{1}_{0}(\Omega,e^{-\frac 2h f}dx)$  (resp. $L^2(\Omega,e^{-\frac 2h f}dx)$)   that we introduced in Section~\ref{sec.diri}  to
define the domain of $\mathsf L^{\mathsf{Di},(0)}_{f,h}(\Omega)$. 
We will denote by $\Vert . \Vert_{H^q_w}$  the norm on the weighted space $\Lambda^qH^m_w
(\Omega)$ (without referring to the degree of the forms). Moreover $\langle
\cdot , \cdot\rangle_{L^2_w}$ denotes the scalar product in $\Lambda^qL^2_w (\Omega)$.  We  will  also simply denote $\Lambda^0H^q_w(\Omega)$ by  $H^q_w(\Omega)$ if there is no possibility for  confusion.

\subsubsection{Link between $\mathsf L^{\mathsf{Di}, (0)}_{f,h}(\Omega)$ and  $\Delta^{\mathsf{Di}, (0)}_{f,h}(\Omega)$, and proof of~\eqref{eq.dim-no}}

\noindent
The infinitesimal generator $-\mathsf L^{(0)}_{f,h}$ of the diffusion
\eqref{eq.langevin} (see Section~\ref{sec.diri}) is linked to the
Witten Laplacian $\Delta^{(0)}_{f,h} = \Delta^{(0)}_{\mbf H}  +\vert
\nabla f\vert^2+h\Delta^{(0)}_{\mbf H}f$ (where we recall that the
Hodge Laplacian writes here: $\Delta^{(0)}_{\mbf H}=-\div  \nabla=-\Delta$)
through the unitary transformation:
$$  \phi \in L^2_w(\Omega) \mapsto e^{-\frac{f}{h}} \phi \in L^2(\Omega).$$
Indeed, one can check that
\begin{equation}\label{eq:unitary_p0}
\Delta^{ (0)}_{f,h}  = 2\, h \, e^{-\frac{f}{h}} \, \mathsf  L^{(0)}_{f,h}  \, e^{\frac{f}{h}}.
\end{equation}
Let us now generalize this to $q$-forms, using extensions of
$\mathsf L^{(0)}_{f,h}$ to $q$-forms.
\begin{proposition} \label{Lp}
Let $q\in \{0,\ldots,d\}$. 
The Friedrichs extension of the quadratic form 
$$Q_{f,h}^{\mathsf{Di},(q)}(\Omega): v \in \Lambda^q  H^1_{w,\mbf T}(\Omega)\mapsto \frac{h}{2}
 \big\Vert \mathsf  d^{(q)}v\big \Vert_{  L^2_w(\Omega)}^2
  +\frac{h}{2}\big \Vert
   e^{\frac{2f}{h}}  \big (\mathsf  d^{(q)}\big)^* e^{-\frac{2f}{h}} v\big \Vert_{ L^2_w(\Omega)}^2 $$ 
on $\Lambda^qL^2_w(\Omega)$, is denoted
$\big(\mathsf L^{\mathsf{Di},(q)}_{f,h}(\Omega), \ \mathcal 
  D\big(L^{\mathsf{Di},(q)}_{f,h}(\Omega)\big)\big)$. The operator $\mathsf L^{\mathsf{Di},(q)}_{f,h}(\Omega)$ is a  positive unbounded self-adjoint operator on $\Lambda^qL^2_w(\Omega)$. Besides, one has 
$$\mathcal D\big(\mathsf L^{\mathsf{Di},(q)}_{f,h}(\Omega)\big)=\big\{  v \in    \Lambda^qH^2_w(\Omega) \, | \, \mbf tv=0, \ \mbf   t \mathsf d^*\big(e^{-\frac{2f}{h}}v\big)=0   \big\}.$$
\end{proposition}
\noindent
For $p=0$,   the operator $\mathsf L^{\mathsf{Di},(0)}_{f,h}(\Omega)$
is the one introduced in Section~\ref{sec.diri}.  In particular,
for $v\in \mathcal D\big(\mathsf L^{\mathsf{Di},(0)}_{f,h}(\Omega)\big)$, 
$\mathsf L_{f,h}^{\mathsf{Di}, (0)}(\Omega)v =  \mathsf L_{f,h}^{  (0)}v$.
 For $p=1$ the operator $\mathsf L^{\mathsf{Di},(1)}_{f,h}(\Omega)$ is the one introduced in Section~\ref{sec.strategy1}. In particular,  
 for  $v\in \mathcal D\big(\mathsf L^{\mathsf{Di},(1)}_{f,h}(\Omega)\big)$,
$\mathsf L_{f,h}^{\mathsf{Di}, (1)} (\Omega)v = \mathsf L_{f,h}^{(1)}  v$
where we recall that $\mathsf L_{f,h}^{(1)} =  \frac{h}{2}
\Delta^{(1)}_{\mbf H}  +  \nabla f  \cdot \nabla   +   \hess f$,
see~\eqref{eq:L1-intro}. 

As a generalisation of~\eqref{eq:unitary_p0}, one gets:
\begin{equation}\label{eq:unitary} 
\Delta^{\mathsf{Di},(q)}_{f,h}(\Omega) = 2\, h \, e^{-\frac{f}{h}} \big( \mathsf L^{\mathsf{Di},(q)}_{f,h}(\Omega) \big)\, e^{\frac{f}{h}}.
\end{equation}
The intertwining relations~\eqref{eq.complexeMM1} and~\eqref{eq.complexeMM2}
write on $\mathsf L_{f,h}^{\mathsf{Di},(q)}(\Omega)$: for all $v\in  \Lambda^q  H^1_{w,\mbf T}(\Omega)$,
\begin{equation}\label{commutationL}
\pi_{E}\big (\mathsf L_{f,h}^{\mathsf{Di},(q+1)}(\Omega) )\mathsf dv =\mathsf  d \pi_{E}\big (\mathsf L_{f,h}^{\mathsf{Di},(q)}(\Omega))v
\end{equation}
and 
\begin{equation}\label{commutationL2}
 \pi_{E}\big (\mathsf L_{f,h}^{\mathsf{Di},(q-1)}(\Omega) \big )\mathsf d_{2f,h}^{*} v= \mathsf d^{*}_{2f,h} \pi_{E}\big (L_{f,h}^{\mathsf{Di},(q)}(\Omega)\big )v,
\end{equation}
Thanks to the relation~\eqref{eq:unitary}, the operators
$\mathsf L_{f,h}^{\mathsf{Di},(q)}(\Omega)$ and $\Delta^{\mathsf{Di},(q)}_{f,h}(\Omega)$ have the
same spectral properties. In particular the operators
$\mathsf L_{f,h}^{\mathsf{Di},(q)}(\Omega)$ and $\Delta^{\mathsf{Di},(q)}_{f,h}(\Omega)$ both have
compact resolvents, and thus a discrete spectrum (see Proposition~\ref{pr.defDfhD}). 
 \medskip

\noindent
Equation~\eqref{eq.dim-no} is a consequence of Theorem~\ref{thm.main1}  as stated in the following results, which also gives a first estimate of $\lambda_h$. 
 \begin{corollary}\label{thm-pc}
 Let us assume that \autoref{A} is satisfied. Then, there exists $\mathsf c>0$ and $h_0>0$ such that for all $h \in (0,h_0)$,
$$
\dim  \range   \pi_{[0,\mathsf c]} \big(\mathsf  L^{\mathsf{Di},(0)}_{f,h}(\Omega)\big)
=1 \text{ and } \dim  \range   \pi_{[0,\mathsf c]} \big( \mathsf  L^{\mathsf{Di},(1)}_{f,h}(\Omega)\big)  =n.
$$ 
Moreover, $\lambda_h$, the principal eigenvalue of $\mathsf  L^{\mathsf{Di},(0)}_{f,h}(\Omega)$,  is exponentially small as $h\to 0$.   
 \end{corollary}
 
For ease of notation,   we set 
 \begin{equation}
 \label{eq.proj}
\pi_h^{(q)}=  \pi_{[0,\mathsf c]} \big(\mathsf  L^{\mathsf{Di},(q)}_{f,h}(\Omega)\big), \text{ for $q\in \{0,1\}$}, 
 \end{equation} 
   where $\mathsf c>0$ is the constant introduced   in Corollary \ref{thm-pc}.

 \begin{proof}
 First of all, by item 2 in \autoref{A}, for any~$x\in \partial \Omega$ such that $\vert \nabla f(x)\vert =0$, there exists a neighborhood $\mathsf V_x^{\partial \Omega}$ of~$x$ in $\partial \Omega$ such that $ \partial_{\mathsf n_{ \Omega}}f =0 \text{ on } \mathsf V_x^{\partial \Omega}$. 
Therefore, $\mathsf M=\Omega$ and $f$ satisfy~\autoref{B}.
By Theorem~\ref{thm.main1}  and~\eqref{eq:unitary}, 
for all $q\in \{0,\ldots,d\}$, there exists $\mathsf c>0$ and $h_0>0$ such that for all $h\in (0,h_0)$: 
$$\dim \Ran \, \pi_{[0,\mathsf c]}\big (\mathsf L_{f,h}^{\mathsf{Di},(q)}( \Omega )\big )=\mathsf  m_q, \,  \, 
\text{ where by~\eqref{numb-generalize}, }\, \,  
 \mathsf m_q=\text{ Card}\big (\mathsf U_q^{\Omega}\cup  \mathsf U^{\partial \Omega,1}_q\cup  \mathsf U^{\partial \Omega,2}_q\big ).$$
Let us first consider the case $q=0$. Recall that $ \mathsf
U^{\partial \Omega,1}_0\cup  \mathsf U^{\partial
  \Omega,2}_0=\emptyset$. Thus $\mathsf m_0=\text{ Card}\big (\mathsf
U_0^{\Omega})=1$, since by Lemma~\ref{le.start}, $f$ has only one
local minimum in $\Omega$ which is $x_0$. 

Let us now consider the cas $q=1$. Notice that
$\mathsf U_1^{\Omega}=\emptyset$ since the  minimum $x_0$ is the
only
critical point of $f$ in $\Omega$.  One then has
$\mathsf m_1=\text{ Card}\big ( \mathsf U^{\partial \Omega,1}_1\cup
\mathsf U^{\partial \Omega,2}_1\big )$. By item 3 in \autoref{A} and
by the definition~\eqref{eq.PSGENE} of $\mathsf U^{\partial \Omega,1}_1$, it holds $\mathsf U^{\partial \Omega,1}_1=\emptyset$. 
By~\eqref{eq.PS-nouveau},  $\mathsf U_1^{\pa  \Omega,2}$  is the set
of saddle points of $f$ on $\partial \Omega$. Thus, from
Definition~\ref{eq.de-introd}, $\mathsf U_1^{\pa
  \Omega,2}=\{z_1,\ldots,z_n\}$. In conclusion,  $\mathsf m_1= n$. 

It remains to prove that $\lambda_h$ is exponentially small when $h\to
0$. Let us recall the proof of this well-known result. Let  $\chi :\mathbb R^d\to [0,1]$ be a $\mathcal C^\infty$ function supported in $\Omega$ such that $\chi=1$ in a neighborhood of $x_0$ in $\Omega$. Then, since  $x_0$ is the only global minimum of $f$ in $\overline \Omega$ (see Lemma~\ref{le.start}), there exists $\delta>0$ such that $f\ge f(x_0)+\delta$ on supp$\nabla \chi$. In addition, because Hess$f(x_0)>0$,  
$\int_{\Omega}  \chi^2 e^{-\frac 2h f}  = (\pi h)^{\frac d2} (1+O(h) )e^{-\frac 2h f(x_0)}   /\sqrt{   \det \Hess f(x_0)}  $, in the limit $h\to 0$. Thus, for $h$ small enough, it holds:
$$\lambda_h\le \langle \mathsf L_{f,h}^{\mathsf{Di}, (0)}(\Omega)\chi, \chi\rangle_{L^2_w}= \frac h2 \frac {\int_{\Omega}  |\nabla \chi|^2 e^{-\frac 2h f}  }{\int_{\Omega}  \chi^2 e^{-\frac 2h f}}  \le Ce^{-\frac \delta h},$$ 
where $\delta>0$ is independent of $h$.
This ends the proof of   Corollary~\ref{thm-pc}.  
 \end{proof}

\section{Construction of quasi-modes associated with  $(z_k)_{k=1,\ldots,n}$}
\label{sec.zk-QM}
By  Corollary~\ref{thm-pc}, for $h$ small enough, the rank of the
spectral projector $\pi_h^{(1)}$ (defined by \eqref{eq.proj})
is the number $n$ of saddle points of $f$ and the rank of the spectral projector
$\pi_h^{(0)}$ is $1$ (the number of local minima of $f$). To prove Theorem~\ref{thm1},  we will
construct $n$ quasi-modes $\{\mathsf f_1^{(1)} , \ldots, \mathsf
f_n^{(1)}\} $ for $\mathsf L^{\mathsf{Di},(1)}_{f,h}(\Omega)$  and a
quasi-mode $\mathsf u^{(0)}$ for  $\mathsf
L^{\mathsf{Di},(0)}_{f,h}(\Omega)$ which form  respectively   a basis
of~Ran~$\pi_h^{(1)}$ and of~Ran~$\pi_h^{(0)}$. We will build
quasi-modes which satisfy appropriate  estimates, listed in Section~\ref{sec.strategy0}, in order  to
get the results of
Theorem~\ref{thm1}.

As already outlined in
Section~\ref{sec.strategy1}, the strategy to build the quasimode
$\mathsf f_k^{(1)}$ consists in   constructing   a
quasi-mode $ \mathsf v_k^{(1)}  \in \Lambda^1\mathcal C^\infty
(\overline \Omega)$ for $\Delta_{f,h}^{\mathsf{Di},(1)}(\Omega)$
associated with the saddle point $z_k\in \partial \Omega$ for each
$k\in \{1,\ldots,n\}$, from which a quasi-mode $\mathsf f_k^{(1)}= e^{\frac
  fh} \mathsf v_k^{(1)}$ for $\mathsf
L_{f,h}^{\mathsf{Di},(1)}(\Omega)$ is deduced. This quasi-mode~$ \mathsf v_k^{(1)} $ is
built  as follows. We first introduce in Section~\ref{sec:omegakM} a
subdomain~$\Omega_k^{\textsc{M}} $ of~$\Omega$ which satisfies some geometric conditions (in particular, $z_k$ is the only saddle point
of $f$ in $\overline{\Omega_k^{\textsc{M}}}$, and $\nabla f\cdot
\mathsf n_{\mathsf \Omega_k^{\textsc{M}}} \ge 0$ on $\partial \mathsf
\Omega_k^{\textsc{M}}$). Then, we introduce in Section~\ref{sec.Witten-TN}
an auxiliary Witten Laplacian on~$\Omega_k^{\textsc{M}}$ with mixed
Dirichlet-Neumann boundary conditions, and we prove that it has only one
eigenvalue $\lambda(\mathsf \Omega_k^{\textsc{M}})$ smaller than
$\mathsf ch$ when considered on functions and $1$-forms. 
 The quasi-mode $ \mathsf v_k^{(1)}$  is then defined using the
 principal $1$-eigenform of this Witten Laplacian (denoted by $
 \mathsf u_k^{(1)}$) multiplied by a suitable cut-off function, see
 Section~\ref{sec:vk1}.

Let us emphasize that since $\vert \nabla f(z_k)\vert =0$, the constructions of the quasi-mode $\mathsf
v_k^{(1)} $  are  very different from those done previously in the literature~\cite{DLLN, helffer-nier-06,DLLN-saddle1,HKN}. In particular, WKB approximations of 
$\mathsf v_k^{(1)} $ are not sufficient to prove the required
estimates (see Section~\ref{sec:WKB} for more details). Instead of using  a WKB-approximation, we will use  an asymptotic equivalent
of $\lambda(\mathsf \Omega_k^{\textsc{M}})$ in the limit $h\to 0$,
inspired by~\cite{DoNe2}.  For $\lambda(\mathsf \Omega_k^{\textsc{M}})$ to be   different from
$0$, we require in particular that   $\mathsf \Omega_k^{\textsc{M}}$
contains~$x_0$, which was not the case in \cite{DLLN}. 




 \subsection{Sufficient estimates on the quasi-modes for $\mathsf L^{\mathsf{Di},(0)}_{f,h}(\Omega)$ and $\mathsf L^{\mathsf{Di},(1)}_{f,h}(\Omega)$} 
 \label{sec.strategy0}


Let us exhibit sufficient conditions on the quasi-modes to get the
results of Theorem~\ref{thm1} (recall that~$n_0$ is the cardinal of
$\arg\min f|_{\partial \Omega}$, see~\eqref{eq.n0}).
 
 \begin{proposition} \label{ESTIME}
 Let us assume that \autoref{A} is satisfied. 
Assume that there exists a family  $\{\mathsf f_1^{(1)}, \ldots,
\mathsf f_n^{(1)}\} $ of smooth $1$-forms on $\overline \Omega$,  and
a smooth function $\mathsf u^{(0)}$ on $\overline \Omega$  such that:
 \begin{enumerate}
 \item The function $\mathsf u^{(0)}$ belongs to $H^1_w(\Omega)$ and is normalized in $L^2_w(\Omega)$. 
 For all $ k\in \left\{1,\ldots,n\right\}$, $\mathsf f_k^{(1)}$
 belongs to $\Lambda^1H^1_{w,\mbf T}(\Omega)$ and is normalized in $\Lambda^1L^2_w(\Omega)$. 
 \item
 \begin{itemize}
 \item[(a)]       There exists $\ve_1>0$  such that for all $ k\in \left\{1,\ldots,n\right\}$, in the limit $h\to 0$:
 \begin{equation}\label{eq:assump_2a_psi}
  \big \|    \big(1-\pi_h^{(1)}\big) \mathsf f_k^{(1)}   \big \|_{H^1_{w}(\Omega)}^2  \le e^{-\frac{\ve_1}{h}}.
     \end{equation} 
 \item[(b)] For any $r>0$, $\mathsf u^{(0)}$ can be chosen such that there exist $C_r>0$ such that for $h$ small enough:
 $$   \|   \nabla  \mathsf u^{(0)}\|_{L^2_w(\Omega)}^2   \le C_r e^{-\frac{2}{h}(f(z_1)-f(x_0) - r)}.$$ 
 \end{itemize}
\item   There exists $\ve_0>0$   such that  for $h$ small enough, $\forall (k,\ell)
\in \left\{1,\ldots,n\right\}^2$ with $k\neq \ell$:
 $$\big\vert \big\lp\mathsf f_k^{(1)}, \mathsf f_\ell^{(1)}\big\rp_{L^2_w(\Omega)}\big\vert  \le  e^{-\frac{\ve_0}{h}}  .$$
 \item
 \begin{itemize}
 \item[(a)] There exist constants
  $(\mathsf K_k)_{k=1,\ldots,n_0}$ and  $\mathsf p$ which do not depend on~$h$ such that for all $k\in \left\{1,\ldots,n_0\right\}$, in the limit $h\to 0$:
\begin{equation*}
  \big \lp       \nabla\mathsf u^{(0)}   ,   \mathsf f_k^{(1)}\big\rp_{L^2_w(\Omega)} =        \mathsf K_k\ h^{\mathsf p}  e^{-\frac{1}{h}(f(z_1)- f(x_0))}   \    (  1  +     O(\sqrt h )   ) ,
 \end{equation*}
  where we recall $f(z_k)=f(z_1)$ for $k=1,\dots,n_0$. 
If  $k>n_0$, it holds for $h$ small enough:
 \begin{equation*}
 \big \vert  \big \lp       \nabla\mathsf u^{(0)}   ,   \mathsf f_k^{(1)}\big\rp_{L^2_w(\Omega)} \big \vert \le  e^{-\frac{1}{h}(f(z_1)- f(x_0) +\ve)}.
 \end{equation*}
 \item[(b)] There exist  constants $(\mathsf b_k)_{k=1,\dots,n_0}$ and $\mathsf m$ which do not depend on~$h$ such that for all $ (k,\ell)
\in \left\{1,\ldots,n\right\}^2$, in the limit $h\to 0$:
 \begin{equation*}
   \int_{\Sigma_{z_\ell}}    \mathsf f_k^{(1)} \cdot \mathsf n_\Omega  \   e^{- \frac{2}{h} f}  d\sigma =\begin{cases}    0   &   \text{ if } k\neq \ell,\\
 -  \mathsf  b_k \, h^{\mathsf m }  \     e^{-\frac{1}{h} f(z_1)}  \    (  1  +     O(\sqrt h )    )   &  \text{ if } k=\ell \in \{1,\dots,n_0\}   \\
  O(e^{-\frac 1h (f(z_1) +c)}) &  \text{ if } k=\ell \in \{n_0+1,\ldots,n\}
   \end{cases} 
  \end{equation*}
 where all the $\Sigma_{z_\ell}$'s are such that \eqref{Sigma_k} holds.
  \end{itemize}
   \end{enumerate}
   Then,  in the limit $h\to 0$:
   $$\lambda_h=\frac{h^{2\mathsf p+1}}{2}   \, e^{-\frac{2}{h}(f(z_1)- f(x_0))}    \sum_{k=1}^{n_0} \mathsf K_k^2  \    (  1  +     O(\sqrt h )   ),
   $$
   where $\lambda_h$ is the principal eigenvalue of $\mathsf L^{\mathsf{Di},(0)}_{f,h}(\Omega)$.
In addition,    for all $k\in \left\{1,\ldots,n_0\right\}$, in the limit $h\to 0$:
 \begin{equation*} 
 \int_{\Sigma_{z_k}}  (\partial_{\mathsf n_\Omega} u_h)\,   e^{-\frac{2}{h}f} d\sigma = -\mathsf    K_k \mathsf b_k \ h^{\mathsf p+\mathsf m}  \,  e^{-\frac{1}{h}(2f(z_1)- f(x_0))}    \   (1+  O(\sqrt h) ),  
   \end{equation*} 
   where $u_h$ is the principal eigenfunction of $\mathsf L^{\mathsf{Di},(0)}_{f,h}(\Omega)$  which satisfies \eqref{eq.u_norma0}. Finally, there exits $c>0$ such that, when $h\to 0$
 $$
\int_{\partial \Omega \setminus \cup_{k=1}^{n_0}\Sigma_{z_k}}   \partial_{\mathsf n_\Omega} u_h \,   e^{- \frac{2}{h}  f} d\sigma= O\big(e^{-\frac{1}{h}(2f(z_1)-f(x_0)+c)} \big).
$$

\end{proposition}
\noindent 
Let us emphasize that, even if this is not explicitely indicated, the
family $\{\mathsf f_1^{(1)} , \ldots,
\mathsf f_n^{(1)}\} $ depends on $h>0$, and the function $\mathsf u^{(0)}$ depends on $h>0$ and
$r>0$. The proof of Proposition~\ref{ESTIME} is based on finite dimensional
linear algebra computations, and is similar to the proof
of~\cite[Theorem~5]{DLLN-saddle1}. It is
therefore note reproduced here. Notice that
Equations~\eqref{eq.dnuh_asymptot} and~\eqref{eq.dnuh_asymptot2}  in
Theorem~\ref{thm1} and Equation~\eqref{eq.lambda_h-Equiv} in
Proposition~\ref{pr.LP-N} will follow from the construction of
quasi-modes  $\{\mathsf f_1^{(1)} , \ldots, \mathsf f_n^{(1)}\} $ and
$\mathsf u^{(0)}$ satisfying all the assumptions of
Proposition~\ref{ESTIME}. This construction is made in the rest of Section~\ref{sec.zk-QM}  (see
the formulas~\eqref{eq.kappa_0} and~\eqref{eq.constantK} for  the
constants $\mathsf b_k$,  $\mathsf m$,  $\mathsf K_k$, and~$\mathsf
p$, and Section~\ref{sec.proof0} for more details).

To prove  Equation~\eqref{eq.dnuh_asymptot3} in Theorem~\ref{thm1} (i.e. to get an asymptotic equivalent of $\int_{\Sigma_{z_k}}  (\partial_{\mathsf n_\Omega} u_h)\,   e^{-\frac{2}{h}f} d\sigma$ for $k>n_0$, as $h\to 0$), one needs stronger assumptions   on these quasi-modes. 

 \begin{proposition} \label{ESTIME2}
 Let us assume that \autoref{A} is satisfied.  Assume that there exists a family  $\{\mathsf f_1^{(1)} , \ldots,
\mathsf f_n^{(1)}\} $ of smooth $1$-forms on $\overline \Omega$,  and
a smooth function $\mathsf u^{(0)}$ on $\overline \Omega$  satisfying
all the assumptions of Proposition~\ref{ESTIME} with the following
additional requirements:
 \begin{enumerate}
 \item Concerning item 2(a) in Proposition~\ref{ESTIME}, there exists $\ve_2>0$  such that for all $ k\in \left\{1,\ldots,n\right\}$, in the limit $h\to 0$:
 \begin{equation}\label{eq:assump_2a_psi-estime2}
   \big \|    \big(1-\pi_h^{(1)}\big) \mathsf f_k^{(1)}   \big \|_{H^1_w(\Omega)}^2  \le  e^{-\frac{2}{h}( \max[f(z_n)-f(z_k), f(z_k)-f(z_1)] +\ve_2)}.
      \end{equation} 
 \item Concerning item 3 in  Proposition~\ref{ESTIME}, 
 there exists $\ve_3>0$   such that  $\forall (k,\ell)
 \in \left\{1,\ldots,n\right\}^2$ with $k>\ell$, in the limit $h\to 0$: 
  $$\big\vert \big\lp\mathsf f_k^{(1)}, \mathsf f_\ell^{(1)}\big\rp_{L^2_w(\Omega)}\big\vert \le  e^{-\frac{1}{h}(f(z_k)- f(z_\ell)+\varepsilon_3)} .$$

 \item Concerning  item 4(a) in Proposition~\ref{ESTIME},  there exist  $(\mathsf K_k)_{k=n_0,\ldots,n}$ and  $\mathsf p$ which do not depend on~$h$ such that for all  $k>n_0$, in the limit $h\to 0$:
\begin{equation*}
  \big \lp       \nabla\mathsf u^{(0)}   ,    \mathsf f_k^{(1)}\big\rp_{L^2_w(\Omega)} =        \mathsf K_k\ h^{\mathsf p}  e^{-\frac{1}{h}(f(z_k)- f(x_0))}   \    (  1  +     O(\sqrt h )   ) .
  \end{equation*}
 
 \item Concerning  item 4(b) in Proposition~\ref{ESTIME}, there exist  constants $(\mathsf b_k)_{k=n_0,\ldots,n}$ and $\mathsf m$ which do not depend on~$h$ such that for all $k
\in \left\{n_0+1,\ldots,n\right\}$, in the limit $h\to 0$:
 \begin{equation*}
   \int_{\Sigma_{z_k}}    \mathsf f_k^{(1)} \cdot \mathsf n_\Omega  \   e^{- \frac{2}{h} f}  d\sigma =- \mathsf  b_k \ h^{\mathsf m }  \     e^{-\frac{1}{h} f(z_k)}  \    (  1  +     O(\sqrt h )    ),
  \end{equation*}
 where all the $\Sigma_{z_k}$'s are such that \eqref{Sigma_k} holds. 
 \end{enumerate}
 \noindent
 Then,  for all $k\in \left\{ n_0+1,\ldots,n\right\}$, in the limit $h\to 0$:
 \begin{equation*} 
 \int_{\Sigma_{z_k}}  (\partial_{\mathsf n_\Omega} u_h)\,   e^{-\frac{2}{h}f} d\sigma = -\mathsf    K_k \mathsf b_k \ h^{\mathsf p+\mathsf m}  \,  e^{-\frac{1}{h}(2f(z_k)- f(x_0))}    \   (1+  O(\sqrt h) ).  
   \end{equation*} 

\end{proposition}
\noindent
Notice that the assumptions of  Proposition~\ref{ESTIME2} on the
quasi-modes are stronger than those of Proposition~\ref{ESTIME} (see indeed~\eqref{eq.n0}).
Again, the proof of Proposition~\ref{ESTIME2} is similar to the proof
of~\cite[Proposition 25]{DLLN}, and is therefore not reproduced
here. Notice that Equation~\eqref{eq.dnuh_asymptot3}  in
Theorem~\ref{thm1} will follow from the construction of quasi-modes
satisfying the assumptions of Proposition~\ref{ESTIME2}. To construct such quasi-modes, the assumptions \eqref{hypo1} on the Agmon distance and \eqref{hypo2} on $f(x_0)$ will be used. 

Let us finally mention that once Theorem~\ref{thm1} and Proposition~\ref{pr.LP-N} are proved, Theorem~\ref{thm2} and Corollary~\ref{co.ek} are direct consequences of Theorem~\ref{thm1}    together with~\eqref{eq.dens},~\eqref{eq.expk}, and Proposition~\ref{pr.LP-N}.

 \subsection{Construction of the subdomains $(\mathsf
   \Omega_k^{\textsc{M}})_{k=1,\ldots,n}$ of $\Omega$}\label{sec:omegakM}
Let us recall that $\Omega$ is a smooth bounded domain of $\mathbb R^d$. 
In this section, we construct a Lipschitz subdomain $\mathsf \Omega_k^{\textsc{M}}$ of $\Omega$ associated with each saddle point $z_k$ of $f$ in $\partial \Omega$, $k=1,\ldots,n$. This subdomain will  then be used to define in the next section  a Witten Laplacian   
with mixed Dirichlet-Neumann boundary
conditions on $\pa \mathsf \Omega_k^{\textsc{M}}$. We   construct
$\mathsf \Omega_k^{\textsc{M}}$  such that: (i) its boundary is
composed of two parts $\Gamma_{k,\mbf D}^{\textsc{M}}$ and
$\Gamma_{k,\mbf N}^{\textsc{M}}$, (ii) $\partial_{\mathsf n_{  \mathsf
    \Omega_k^{\textsc{M}} } }f=0$ on $\Gamma_{k,\mbf D}^{\textsc{M}}$
and $\partial_{\mathsf n_{  \mathsf \Omega_k^{\textsc{M}} } }f>0$ on
$\Gamma_{k,\mbf N}^{\textsc{M}}$,  (iii) $x_0\in \mathsf
\Omega_k^{\textsc{M}}$, and finally (iv) $\Gamma_{k,\mbf D}^{\textsc{M}}$ and $\Gamma_{k,\mbf N}^{\textsc{M}}$ meet at an angle strictly smaller than $\pi$ (see  Definition~\ref{de.angle} below). Conditions (ii) and (iii) will then be used to deduce in  Section~\ref{sec.Witten-TN} the number of small eigenvalues of this  Witten Laplacian on $\mathsf \Omega_k^{\textsc{M}}$, and  the condition (iv) will be necessary to have existence of traces and  subelliptic estimates for forms in the domain of this Witten Laplacian. 
\medskip

\subsubsection{Preliminary results}
Before going through the construction of   $\mathsf \Omega_k^{\textsc{M}}$ (see  Proposition~\ref{pr.omegakpoint}), we need preliminary results stated in  Propositions~\ref{pr.gammak} and~\ref{pr.omegak}. 

\begin{proposition}\label{pr.gammak}
 Let us assume that the assumption \autoref{A} is satisfied. Consider $k\in \{1,\ldots,n\}$ and $\mathsf F$ a compact subset of the open set $\Gamma_{z_k}$. Then, there exists a $\mathcal C^\infty$ simply connected subdomain~$\Gamma_{\mathsf F}$ of  $\partial \Omega$ containing~$z_k$ such that $ \overline {\Gamma_{\mathsf F}} \subset \Gamma_{z_k}$, ${\mathsf F}\subset \Gamma_{\mathsf F}$, and 
\begin{equation}\label{eq.ultrastable}
\nabla f\cdot \mathsf n_{ \Gamma_{\mathsf F} }>0  \text{ on } \pa \Gamma_{\mathsf F},
\end{equation}
 where $\mathsf n_{ \Gamma_{\mathsf F}}\in T\partial \Omega$ is the
 unit outward normal to $ \Gamma_{\mathsf F}$. 
\end{proposition}
Since $\Omega$ is a  stable  domain  for the dynamics~\eqref{eq.hbb},
one can prove a similar result on $x_0$ and $\Omega$, as the one
obtained in Proposition~\ref{pr.gammak} on $z_k$ and $\Gamma_{z_k}$. 

\begin{proposition}\label{pr.omegak}
 Let us assume that the assumption \autoref{A} is satisfied. Then, for any compact subset~$\mathsf K$ of $\Omega$ there exists a $\mathcal C^\infty $ simply connected   subdomain $\Omega_{\mathsf K}$ of $\Omega$ containing $x_0$ such that $\mathsf K\subset \Omega_{\mathsf K}$, $\overline{\Omega_{\mathsf K}}\subset \Omega$, and 
$$\nabla f\cdot \mathsf n_{ \Omega_{\mathsf K}}> 0  \text{ on } \partial \Omega_{\mathsf K}.$$
\end{proposition}

The proofs of Propositions~\ref{pr.gammak} and~\ref{pr.omegak} are 
tedious, and we therefore postpone them to Section~\ref{sec:gammak},
in the appendix.

\subsubsection{Construction of $\mathsf \Omega_k^{\textsc{M}}$}
We are now in position to construct, for each $k\in \{1,\ldots,n\}$,
the subdomain $\mathsf \Omega_k^{\textsc{M}}$ of $\Omega$ associated
with the saddle point $z_k$ and its neighborhood $\Sigma_{z_k}$ (see~\eqref{Sigma_k}).

\begin{proposition}\label{pr.omegakpoint}
 Let us assume that the assumption \autoref{A} is satisfied and
 consider $k\in \{1,\ldots,n\}$. Then, there exists a Lipschitz
 subdomain $\mathsf \Omega_k^{\textsc{M}}$ of $\Omega$ containing
 $x_0$ and such that:
\begin{enumerate}
\item  It holds $\pa \mathsf \Omega_k^{\textsc{M}} \cap \partial \Omega = \overline{\Gamma_{k,\mbf D}^{\textsc{M}}}$ where $\Gamma_{k,\mbf D}^{\textsc{M}}$ is a $\mathcal C^\infty$ subdomain of $\Gamma_{z_k}$ containing $\overline{\Sigma_{z_k}}$ which satisfies:
\begin{enumerate}
\item  $\nabla f\cdot \mathsf n_{ \Gamma_{k,\mbf D}^{\textsc{M}} }>0  \text{ on } \pa \Gamma_{k,\mbf D}^{\textsc{M}}$ (recall that $ \mathsf n_{ \Gamma_{k,\mbf D}^{\textsc{M}} }\in T\partial \Omega\cap (T\pa \Gamma_{k,\mbf D}^{\textsc{M}})^\bot$ is the unit outward normal to $  \Gamma_{k,\mbf D}^{\textsc{M}}$) and
 \item  \text{ a.e. on }    $\overline{ \Gamma_{k,\mbf D}^{\textsc{M}}}$,
$$\nabla f\cdot  \mathsf n_{ \mathsf \Omega_k^{\textsc{M}}}=0,$$
where, here and in the following, a.e. is with respect  to the surface measure on $\pa  \mathsf \Omega_k^{\textsc{M}}$. 
\end{enumerate}

 \item On  $\Gamma_{k,\mbf N}^{\textsc{M}}:=\pa \mathsf \Omega_k^{\textsc{M}} \cap   \Omega$ it holds a.e.:
 $$ \nabla f \cdot \mathsf n_{\mathsf \Omega_k^{\textsc{M}}} >0.$$

 \item The sets $\Gamma_{k,\mbf D}^{\textsc{M}}$ and $\Gamma_{k,\mbf N}^{\textsc{M}}$ meet  at an angle smaller than
  $\pi$ (see Definition~\ref{de.angle} below). This angle will be actually $\pi/2$ from the construction below.  
  
 \item For all $\delta>0$, $\mathsf \Omega_k^{\textsc{M}}$ can be chosen such that 
\begin{equation}\label{eq.gamma2_close_to_Bzc}
\sup_{x\in \Gamma_{k,\mbf N}^{\textsc{M}}}\mathsf d_{\, \overline \Omega}(x, \partial \Omega \setminus \Gamma_{z_k} )\le\delta,
\end{equation}
where  $\mathsf d_{\, \overline \Omega}$ denotes   the geodesic distance in $\overline \Omega$. 
 \end{enumerate}
 \end{proposition}
 
 \noindent
  Schematic representations of $\mathsf \Omega_k^{\textsc{M}}$, $\Gamma_{k,\mbf D}^{\textsc{M}}$, and $\Gamma_{k,\mbf N}^{\textsc{M}}$ are given in Figure~\ref{fig:domain_dotomega} below. 
 The subscript $\mbf D$ (resp. $\mbf N$) in $\Gamma_{k,\mbf D}^{\textsc{M}}$ (resp. in $\Gamma_{k,\mbf N}^{\textsc{M}}$)  refers to  the fact that   Dirichlet (resp. Neumann) boundary conditions will be applied  on $\Gamma_{k,\mbf D}^{\textsc{M}}$ (resp. on $\Gamma_{k,\mbf N}^{\textsc{M}}$) when defining the Witten Laplacian with mixed  Dirichlet-Neumann boundary conditions   on $\mathsf \Omega_k^{\textsc{M}}$, see Section~\ref{sec.Witten-general-result} below. 
 Let us recall the definition of an angle between two hypersurfaces  used in item 3 of Proposition~\ref{pr.omegakpoint}  (see~\cite{brown-94,jakab-mitrea-mitrea-09}).

 \begin{definition}\label{de.angle}
Let  $\mathsf D$ be a bounded  Lipschitz  domain of $\mathbb R^d$. Let $\Gamma_{\mbf D}$
and $\Gamma_{\mbf N}$ be two open disjoint subsets of $\partial \mathsf D$ such that $\overline{\Gamma_{\mbf D}} \cup \overline{\Gamma_{\mbf N}} = \pa
\mathsf D$. 
The sets $\Gamma_{\mbf D}$ and $\Gamma_{\mbf N}$ meet  at an angle smaller than
  $\pi$ (in~$\mathsf D$)  if  locally around any point
 $\mathsf y \in \overline{\Gamma_{\mbf D}} \cap \overline{\Gamma_{\mbf N}}$,  there exists a local system of coordinates $(y_1,y'',y_d) \in \mathbb R
 \times \mathbb R^{d-2} \times \mathbb R$ on a
 neighborhood $\mathsf V_{\mathsf y}$ of $\mathsf y$, and two Lipschitz functions
 $\varphi_{\mathsf y}:\mathbb R^{d-1} \to \mathbb R$ and $ \psi_{\mathsf y}:\mathbb R^{d-2} \to \mathbb R$ such
  that $\mathsf D  \cap \mathsf  V_{\mathsf y}=\{y_d >   \varphi_{\mathsf y}(y_1,y'')\}$, $\Gamma_{\mbf D} \cap
  \mathsf V_{\mathsf y}=\{y_d=  \varphi_{\mathsf y}(y_1,y'') \text{ and } y_1 >   \psi_{\mathsf y}(y'')\}$ and $\Gamma_{\mbf N} \cap
 \mathsf V_{\mathsf y}=\{y_d=  \varphi_{\mathsf y}(y_1,y'') \text{ and } y_1 <   \psi_{\mathsf y}(y'')\}$ and
$$
 \begin{aligned}
  \partial_{y_1}  \varphi_{\mathsf y} (y_1,y'') \ge \kappa &\text{ on }
  y_1>\psi_{\mathsf y}(y'')\\
 \partial_{y_1} \varphi_{\mathsf y} (y_1,y'')  \le -\kappa &\text{ on }
  y_1<\psi_{\mathsf y}(y''),
 \end{aligned}
$$
for some   $\kappa>0$. 
\end{definition}
From a geometric viewpoint, the fact that $\Gamma_{\mbf D}$ and $\Gamma_{\mbf N}$ meet  at an angle smaller than
  $\pi$ is equivalent to the existence of a
smooth vector field $\theta$ on $\partial \mathsf D$ such that
$\langle \theta , \mathsf n_{\mathsf D} \rangle < 0$ on $\Gamma_{\mbf D}$ and $\langle \theta,
\mathsf n_{\mathsf D}\rangle > 0$ on
$\Gamma_{\mbf N}$.
 Let us now prove Proposition~\ref{pr.omegakpoint}. 
\begin{proof}[Proof of Proposition~\ref{pr.omegakpoint}] Let $k\in \{1,\ldots,n\}$. The domain $\mathsf \Omega_k^{\textsc{M}}$ will be defined as the union of two intersecting subdomains of $\Omega$.    The proof of Proposition~\ref{pr.omegakpoint} is divided into several steps. 
\medskip

\noindent
\textbf{Step 1: Definition of $\mathsf \Omega_k^{\textsc{M}}$.}   
\medskip

\noindent
\textbf{Step 1a: Adapted system of coordinates and preliminary constructions.}
 \medskip
        
        \noindent
\textbf{The set $\Gamma_{k,\mbf D}^{\text{M}}$}. Recall
(see~\eqref{Sigma_k}) that 
$
z\in  \Sigma_{z_k}  \ \text{ and } \ \overline{\Sigma_{z_k} }\subset  \Gamma_{z_k}.$
Using  Proposition~\ref{pr.gammak}, there exists   a $\mathcal
C^\infty$ subdomain $\Gamma_{k,\mbf D}^{\text{M}}$ of $\Gamma_{z_k}$
such that  $\overline{\Sigma_{z_k}}\subset \Gamma_{k,\mbf
  D}^{\text{M}}$, $\overline{\Gamma_{k,\mbf D}^{\text{M}}}\subset
\Gamma_{z_k}$,  which can be as large as needed in $\Gamma_{z_k}$, and
such that
\begin{equation}\label{eq.derivee_normale_positive}
\nabla f\cdot \mathsf n_{ \Gamma_{k,\mbf D}^{\text{M}} }>0  \text{ on } \pa \Gamma_{k,\mbf D}^{\text{M}}.
\end{equation}
  In step 1b below (see indeed~\eqref{eq.onm}), we will check that
  from the definition~\eqref{eq.DEFDOTOMEGA} of $\mathsf
  \Omega_k^{\textsc{M}}$, $\overline{\Gamma_{k,\mbf D}^{\text{M}}}=
  \pa \mathsf \Omega_k^{\textsc{M}}\cap \partial \Omega$, and this
  will therefore prove item 1(a) of Proposition~\ref{pr.omegakpoint}. 
\medskip

        \noindent
\textbf{Systems of coordinates near $\partial \Omega$ and $\pa \Gamma_{k,\mbf D}^{\text{M}}$}.  
In the following we introduce two  systems of  coordinates: one  around $z\in \partial \Omega$ in $\overline \Omega$ (see $(x',x_d)$ in~\eqref{eq.coord1}),
 and one around $z\in \pa \Gamma_{k,\mbf D}^{\text{M}}$ in $\partial \Omega$ (see~$x'$ in~\eqref{eq.pauz} and~\eqref{eq.pauz2}). They will be used  to define $\mathsf \Omega_k^{\textsc{M}}$. 
\medskip


 Recall that, for $\ve>0$ small enough,  for all $x\in \overline \Omega$ such that $\mathsf d_{\, \overline \Omega}(x,\partial \Omega)<\ve$, there exists a unique 
point $\mathsf z(x)\in \partial \Omega$ such that  
 \begin{equation}\label{eq.coordxdd} 
x_d(x):=\mathsf d_{\, \overline \Omega}(x,\partial \Omega)=\mathsf d_{\, \overline \Omega}(x,\mathsf z(x)),
 \end{equation} 
 where we recall  $\mathsf d_{\, \overline \Omega}$ denotes   the
 geodesic distance in $\overline \Omega$. Moreover the function
 $x\mapsto \mathsf d_{\, \overline \Omega}(x,\partial \Omega)$ is
 smooth on the set $\{x\in \overline \Omega, \mathsf d_{\, \overline
   \Omega}(x,\partial \Omega)<\ve\}$. 
 Let $z\in \partial \Omega$ and   $x'$ be a system of coordinates in $\partial \Omega$ centred at $z$. Then, there exists a neighborhood $\mathsf V_z$ of $z$ in $\overline \Omega$ such that the  function  
 \begin{equation}\label{eq.coord1} 
 v\in \mathsf V_z\mapsto  (x'(\mathsf z(v)),x_d(v))\in \mathbb R^{d-1}\times \mathbb R_+
 \end{equation}
   is    a system of coordinates in $\mathsf V_z$ (this is the
   tangential-normal system of coordinates already introduced above in~\eqref{eq.norm-teng}).   For ease of notation, we omitted to write the dependency on $z$ when writing  $(x',x_d)$, and we write with a slight abuse of notation, $x'(v)$ instead of $x'(\mathsf z(v))$.  
Let us assume, up to choosing $\mathsf V_z$ smaller that for   $\ve_z>0$ small enough, $\mathsf V_z$ is a  cylinder in the $(x',x_d)$-coordinates:  
\begin{equation}\label{eq.uz} 
\mathsf  V_z=\big \{v\in \mathsf V_z, \ \vert x'(v)\vert < \ve_z \, \text{ and } \, x_d(v)\in [0,\ve_z)\big \}.  
\end{equation}

Let us now be more precise on $x'$ when $z\in \overline{\Gamma_{k,\mbf D}^{\text{M}}}$. If $z\in \Gamma_{k,\mbf D}^{\text{M}}$, we choose $\ve_z>0$ small enough such that 
\begin{equation}\label{eq.pauz}
\partial \Omega \cap \mathsf V_z=\big \{v\in \mathsf V_z, \ \vert x'(v)\vert \le \ve_z \, \text{ and } \, x_d(v)=0\big \}\subset \Gamma_{k,\mbf D}^{\text{M}}.  
\end{equation}
If $z\in \pa{\Gamma_{k,\mbf D}^{\text{M}}}$,   the system  $x'=(x_1,\ldots, x_{d-1})$ in $\partial \Omega$ is chosen such that:
\begin{equation}\label{eq.pauz2}
\Gamma_{k,\mbf D}^{\text{M}}\cap (\partial \Omega \cap  \mathsf V_z)=\{v\in \partial \Omega \cap \mathsf    V_z,\  x_1(v)> 0\} ,
\end{equation}  
and 
\begin{equation}\label{eq.pauz3}
\pa  \Gamma_{k,\mbf D}^{\text{M}}\cap (\partial \Omega \cap\mathsf V_z)=\{v\in \partial \Omega \cap \mathsf V_z, \ x_1(v)= 0\}.
\end{equation}
This implies that for all $z\in \pa \Gamma_{k,\mbf D}^{\text{M}}$, 
\begin{equation}\label{eq.ngammak'}
\mathsf n_{\Gamma_{k,\mbf D}^{\text{M}}}(z)= -\frac{\nabla x_1(z)}{\vert \nabla x_1\vert(z)} \in T_z\partial \Omega.
\end{equation}

\begin{figure}
\begin{center}
  \includegraphics[width=0.5\textwidth]{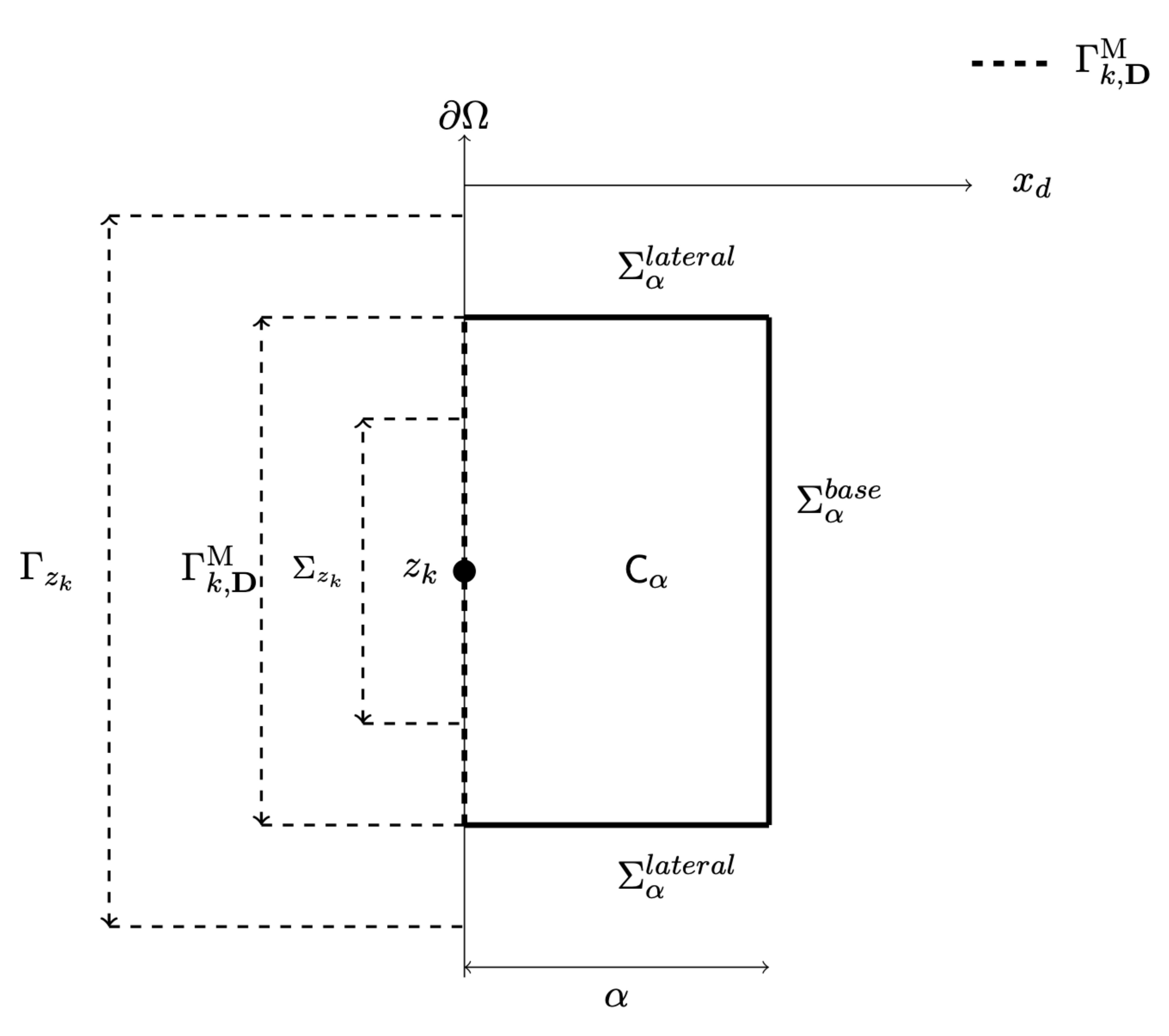}
\caption{The cylinder $\mathsf C_\alpha$.}
\label{fig:representation_Calpha}
\end{center}
\end{figure}

\textbf{Constructions of two subdomains of $\Omega$: $\mathsf C_\alpha$ and $\Omega_{\mathsf K_{\alpha/2}}$.} 
 Define,  for $\alpha>0$ small enough,   the open cylinder 
\begin{equation}\label{eq.Calpha}
\mathsf C_\alpha=\big  \{x\in \overline \Omega, \, \mathsf z(x)\in \Gamma_{k,\mbf D}^{\text{M}} , \, x_d(x)\in (0,\alpha)\big \},
\end{equation}
(see Figure~\ref{fig:representation_Calpha} for a schematic representation of $\mathsf C_\alpha$),
and   the compact set
$$\mathsf K_{\alpha/2}=\big \{v\in \overline \Omega, \ \mathsf d_{\, \overline \Omega}(v,\partial \Omega)\ge  {\alpha} / {2}\big \} \subset \Omega.$$
From Proposition~\ref{pr.omegak}, there exists a $\mathcal C^\infty $ subdomain $\Omega_{\mathsf K_{\alpha/2}}$ of $\Omega$ containing $x_0$ such that $\mathsf K_{\alpha/2}\subset \Omega_{\mathsf K_{\alpha/2}}$, $\overline{\Omega_{\mathsf K_{\alpha/2}}}\subset \Omega$, and 
\begin{equation}\label{eq.n-omeka-alpha}
\nabla f\cdot \mathsf n_{ \Omega_{\mathsf K_{\alpha/2}}}> 0  \text{ on } \partial \Omega_{\mathsf K_{\alpha/2}}.
\end{equation}
A schematic representation of $\Omega_{\mathsf K_{\alpha/2}}$ and
$\mathsf C_\alpha$ is given in Figure~\ref{fig:domain_dotomega}.

\begin{figure}[htbp]
\begin{center}
  \includegraphics[width=0.45\textwidth]{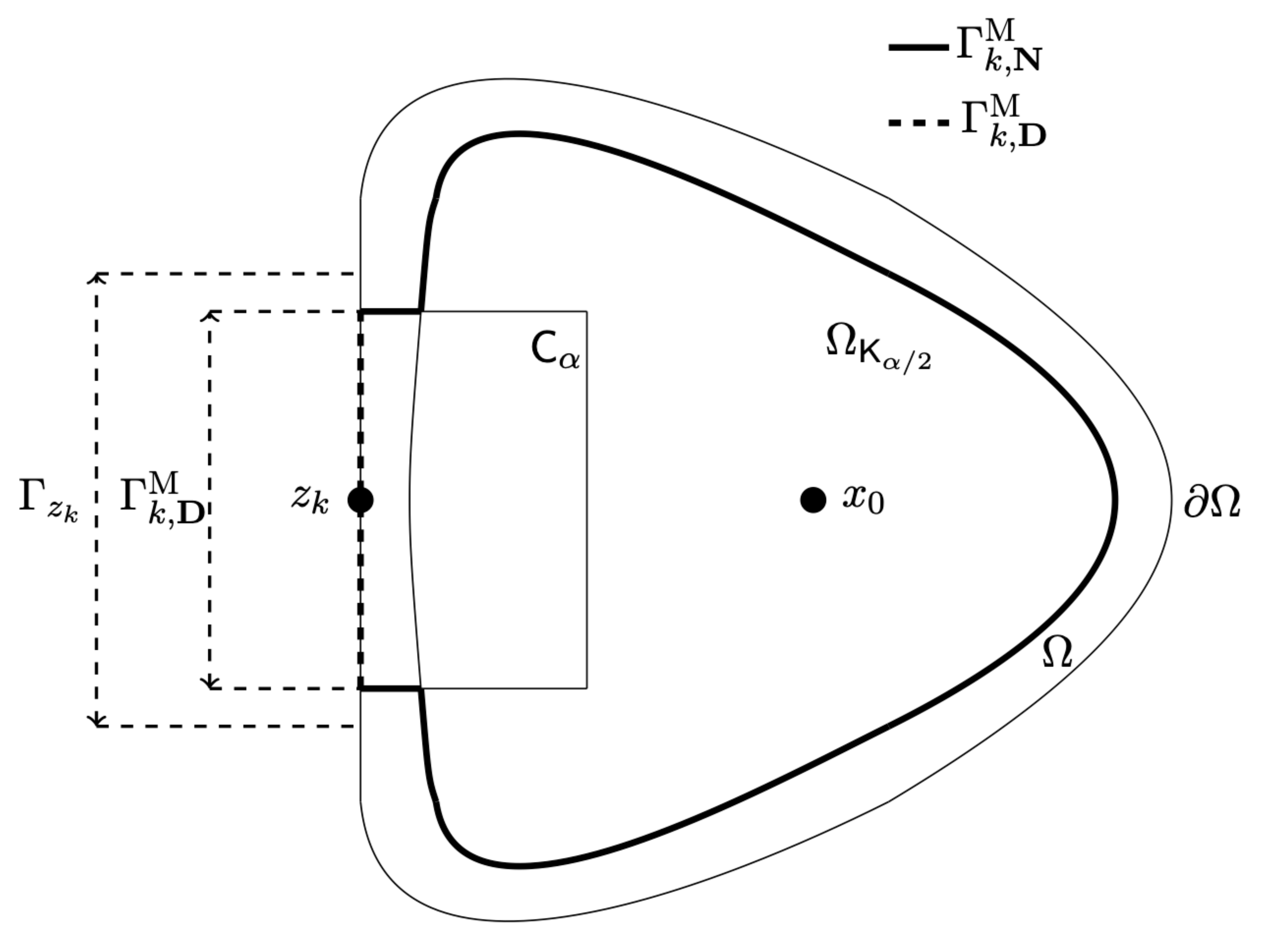}
  \includegraphics[width=0.45\textwidth]{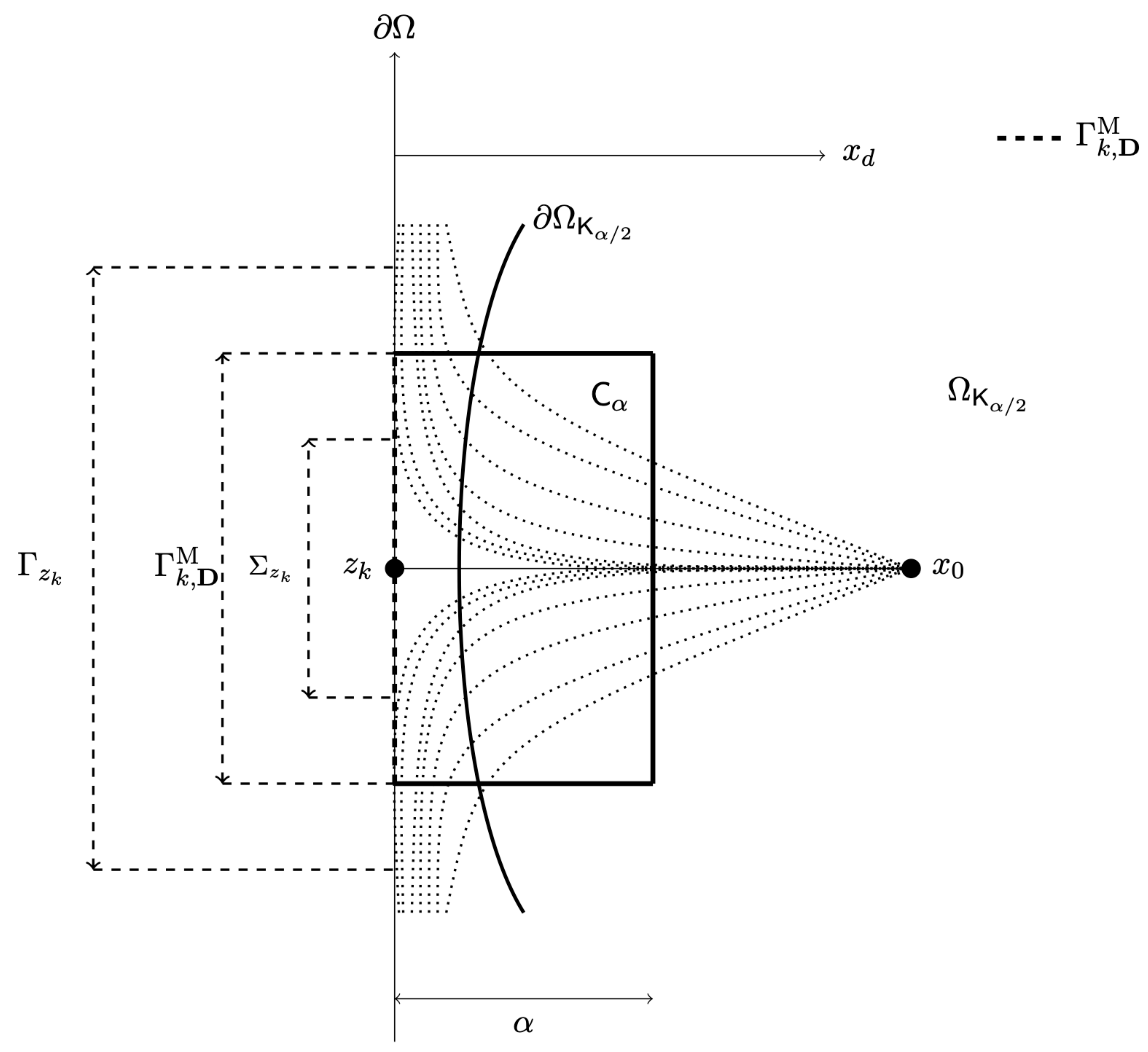}
\caption{Schematic representations of $\mathsf
  \Omega_k^{\textsc{M}}=\mathsf C_\alpha\cup\Omega_{\mathsf
    K_{\alpha/2}} $, $\Gamma_{k,\mbf D}^{\textsc{M}}$, and
  $\Gamma_{k,\mbf N}^{\textsc{M}}$. On the right, a zoom in the
  neighborhood of $\Gamma_{z_k}$, where the dotted lines represent the
  flows of $\varphi_x$ near the saddle point $z_k$ of $f$
  (see~\eqref{eq.hbb} and  item 1 in \autoref{A}).} 

\label{fig:domain_dotomega}

\end{center}

\end{figure}


Moreover it holds (see Figure~\ref{fig:representation_Calpha}):
 \begin{equation}\label{eq.3P}
 \pa  \mathsf C_\alpha =  \overline{\Gamma_{k,\mbf D}^{\text{M}}}\, \, \cup { \Sigma_\alpha^{lateral} }\, \cup  \, \, \overline{ \Sigma_\alpha^{base}},
 \end{equation}
 where 
 \begin{equation}\label{eq.Blat}
 \Sigma_\alpha^{lateral}=\big \{x\in \overline \Omega, \,\mathsf  z(x)\in \pa \Gamma_{k,\mbf D}^{\text{M}} , \, x_d(x)\in (0,\alpha)\big \} \subset  \Omega,
  \end{equation}
 and 
 $$ \Sigma_\alpha^{base}=\big \{x\in \overline \Omega, \, \mathsf z(x)\in   \Gamma_{k,\mbf D}^{\text{M}} , \, x_d(x) =\alpha \big \}  \subset \Omega.$$

 Let us now prove that  there exists  $\alpha_0>0$, such that for all $\alpha\in (0,\alpha_0)$, one has:
 \begin{equation}\label{eq.signBO}
 \nabla f \cdot \mathsf n_{ \mathsf C_\alpha} >0\text{ on }\Sigma_\alpha^{lateral}.
 \end{equation}
 It holds $\Sigma_\alpha^{lateral}=\big \{v\in \Omega, \, x_1(v)=0, \,
 x_d(v)\in (0,\alpha)\big \}$ (from
 ~\eqref{eq.pauz}--\eqref{eq.pauz3},~\eqref{eq.Calpha},and ~\eqref{eq.Blat}), and hence, one has for all $v\in \Sigma_\alpha^{lateral}$: 
 \begin{equation}\label{eq.normalLATERAL}
 \mathsf n_{ \mathsf C_\alpha}(v)=- \frac{\nabla x_1(v)}{\vert \nabla x_1\vert(v)}.
 \end{equation}
Therefore, by a continuity argument, using~\eqref{eq.derivee_normale_positive} and~\eqref{eq.ngammak'}, there exists $ \alpha_0>0$ such that for all $\alpha\in (0,\alpha_0)$ and  for all $v\in \Sigma_\alpha^{lateral}$,
$$ \nabla f(v)\cdot \mathsf n_{ \mathsf C_\alpha}(v)>0.$$
This concludes the proof of~\eqref{eq.signBO}.
\medskip

\noindent
\textbf{Step 1b: definition of  $\mathsf \Omega_k^{\textsc{M}}$ such that  $  \pa \mathsf \Omega_k^{\textsc{M}} \cap \partial \Omega =\overline{\Gamma_{k,\mbf D}^{\text{M}}}$.}
Let us introduce (see Figure~\ref{fig:domain_dotomega})
  \begin{equation}\label{eq.DEFDOTOMEGA}
  \mathsf \Omega_k^{\textsc{M}}:=   \mathsf C_\alpha \cup\Omega_{\mathsf K_{\alpha/2}}  ,
  \end{equation}
  which is included in $\Omega$.  Let us mention  that $ \mathsf \Omega_k^{\textsc{M}}$  depends on  two  parameters:  $\alpha>0$ (through  both~$\mathsf C_\alpha$ and~$\Omega_{\mathsf K_{\alpha/2}}$) and   $\Gamma_{k,\mbf D}^{\text{M}}$ (through $\mathsf C_\alpha$).  
  One obviously has $\overline{\Gamma_{k,\mbf D}^{\text{M}}}\subset  \pa \mathsf \Omega_k^{\textsc{M}}$. 
Let us define  
  \begin{equation}\label{eq.DEFDOTOMEGA2}
  \Gamma_{k,\mbf N}^{\text{M}}=  \pa \mathsf \Omega_k^{\textsc{M}}\setminus \overline{\Gamma_{k,\mbf D}^{\text{M}}},
    \end{equation}
  so that $ \pa \mathsf \Omega_k^{\textsc{M}}$ is the disjoint union of $\overline{\Gamma_{k,\mbf D}^{\text{M}}}$ and $ \Gamma_{k,\mbf N}^{\text{M}}$.  
  By definition,  $\mathsf \Omega_k^{\textsc{M}} $ is the union of two intersecting open connected subsets $\mathsf C_\alpha$ and $\Omega_{\mathsf K_{\alpha/2}}$ of $\Omega$, it is thus open and connected. 
  Notice that   one has:
  \begin{equation}\label{eq.unionboundary1}
  \pa \mathsf \Omega_k^{\textsc{M}}\subset \pa  \mathsf C_\alpha \cup \partial \Omega_{\mathsf K_{\alpha/2}}
  \end{equation}
   and,   since  $\overline{ \Sigma_\alpha^{base}}\subset \mathsf K_{\alpha/2} \subset    \Omega_{\mathsf K_{\alpha/2}} \subset  \mathsf  \Omega_{k}^{\textsc{M}}$ (see~\eqref{eq.3P}), one has: 
   \begin{equation}\label{eq.unionboundary2}
  \partial \Omega_{\mathsf K_{\alpha/2}}\cap \overline{ \Sigma_\alpha^{base}}=\emptyset \text{ and }  \pa \mathsf \Omega_k^{\textsc{M}}\cap \overline{ \Sigma_\alpha^{base}}=\emptyset.
   \end{equation}
In addition, from the fact that  $\partial \Omega_{\mathsf K_{\alpha/2}} \subset \Omega$  and $\overline{\Gamma_{k,\mbf D}^{\text{M}}}\subset \partial \Omega$, it 
holds:
   \begin{equation}\label{eq.unionboundary3}
\partial \Omega_{\mathsf K_{\alpha/2}}\cap \overline{\Gamma_{k,\mbf D}^{\text{M}}}=\emptyset.
 \end{equation}
 Thus,  from~\eqref{eq.unionboundary1},~\eqref{eq.unionboundary2}, and~\eqref{eq.unionboundary3} together with the definition of $\Gamma_{k,\mbf N}^{\text{M}}$, it holds: $\Gamma_{k,\mbf N}^{\text{M}} \subset (\pa \mathsf C_\alpha  \cup  \partial \Omega_{\mathsf K_{\alpha/2}} )\setminus (  \overline{\Gamma_{k,\mbf D}^{\text{M}}}\cup  \overline{ \Sigma_\alpha^{base}}) = \pa \mathsf C_\alpha   \setminus  (  \overline{\Gamma_{k,\mbf D}^{\text{M}}}\cup  \overline{ \Sigma_\alpha^{base}})  \cup  \partial \Omega_{\mathsf K_{\alpha/2}}$ and thus, from~\eqref{eq.3P}, 
   \begin{equation}\label{eq.GK2}
    \Gamma_{k,\mbf N}^{\text{M}}  \subset \Sigma_\alpha^{lateral} \cup \partial \Omega_{\mathsf K_{\alpha/2}} \subset \{v\in  \Omega, \ \mathsf d_{\, \overline \Omega}(v,\partial \Omega)<  {\alpha} \big \},
    \end{equation}
    where the last inclusion follows from the fact that $\partial \Omega_{\mathsf K_{\alpha/2}}\subset  \{v\in  \Omega, \ \mathsf d_{\, \overline \Omega}(v,\partial \Omega)\le   \alpha / {2}\big \}$ and~\eqref{eq.Blat}. 
   In particular, this implies that, since $\overline{\Gamma_{k,\mbf D}^{\text{M}}}\subset \partial \Omega$, 
      \begin{equation}\label{eq.onm}
       \pa \mathsf \Omega_k^{\textsc{M}} \cap   \Omega =\Gamma_{k,\mbf N}^{\text{M}} \text{ and } \pa \mathsf \Omega_k^{\textsc{M}} \cap \partial \Omega = \overline{\Gamma_{k,\mbf D}^{\text{M}}}.
        \end{equation}

        \medskip
        
        \noindent
        \textbf{Step 2: Proofs of items 3 and 4 in  Proposition~\ref{pr.omegakpoint}.}
       \medskip
        
        \noindent
        \textbf{Step 2a.}  Let us   check that $\Gamma_{k,\mbf
          D}^{\textsc{M}} $ and $\Gamma_{k,\mbf N}^{\textsc{M}}$ meet
        at an angle strictly smaller than $\pi$ in $\mathsf
        \Omega_k^{\textsc{M}}$   (in the sense of  Definition~\ref{de.angle}). To this end, let us prove that  
        \begin{equation}\label{eq.anglepi}
        \overline{\Gamma_{k,\mbf D}^{\textsc{M}} }\cap \overline{\Gamma_{k,\mbf N}^{\textsc{M}}}=\overline{\Gamma_{k,\mbf D}^{\textsc{M}} }\cap \overline{\Sigma_\alpha^{lateral}}.
        \end{equation} 
        Notice that~\eqref{eq.anglepi} implies  that $\pa \mathsf
        \Omega_k^{\textsc{M}}$ is Lipschitz  near
        $\overline{\Gamma_{k,\mbf D}^{\textsc{M}} }\cap
        \overline{\Gamma_{k,\mbf N}^{\textsc{M}}}$   as the union of
        the closures of two disjoint open transverse $\mathcal
        C^\infty$ hypersurfaces $\Gamma_{k,\mbf D}^{\textsc{M}}$ and
        $\Sigma_\alpha^{lateral}$ (this will be used in Step 3b below). Furthermore,~\eqref{eq.anglepi}
        implies that $\Gamma_{k,\mbf D}^{\textsc{M}} $ and
        $\Gamma_{k,\mbf N}^{\textsc{M}}$ meet at an angle  $\pi/2$
        (see Figure \ref{fig:angle}), which thus yields item 3 in Proposition~\ref{pr.omegakpoint}.

\begin{figure}[htbp]
\begin{center}
  \includegraphics[width=0.5\textwidth]{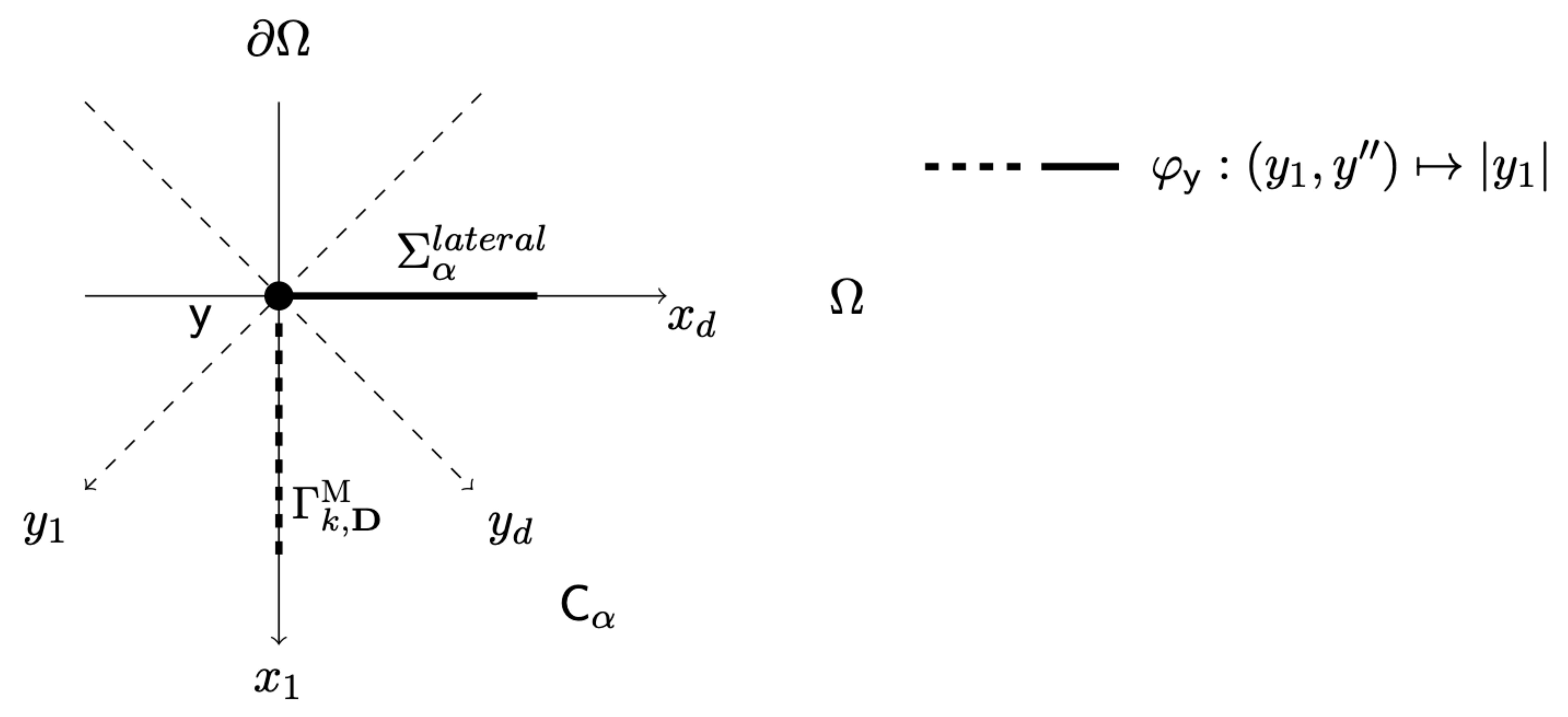}
\caption{The sets $\Gamma_{k,\mbf D}^{\textsc{M}}$ and $\Sigma_\alpha^{lateral}$ meet at an angle
   $\pi/2$ in $\mathsf C_\alpha$ (see Definition~\ref{de.angle},~\eqref{eq.pauz2}, \eqref{eq.pauz3}, and~\eqref{eq.3P}). On the figure,  $\mathsf y\in \overline{\Gamma_{k,\mbf D}^{\textsc{M}}}\cap\overline{ \Sigma_\alpha^{lateral}} $ and $\{x_d>0\}=\Omega$, 
  $y_1=\frac{-x_d+x_{1}}{2}$, $y_d=\frac{x_d+x_{1}}{2}$,
  $y''=(x_{2},\dots, x_{d-1})$ (which is, schematically,   the coordinates  perpendicular to  the plane $(x_1,x_d)$ centred at $\mathsf y$) , $ \psi_{\mathsf y}(y'')=0$,  and $  \varphi_{\mathsf y}(y_1,y'')=|y_1|$ in  Definition~\ref{de.angle}.}
 \label{fig:angle}
  \end{center}
\end{figure}

   
     Let us thus prove~\eqref{eq.anglepi}. 
        From~\eqref{eq.GK2} together with~\eqref{eq.unionboundary3}, it holds:
        $\overline{\Gamma_{k,\mbf D}^{\textsc{M}} }\cap \overline{\Gamma_{k,\mbf N}^{\textsc{M}}}\subset \overline{\Gamma_{k,\mbf D}^{\textsc{M}} }\cap \overline{\Sigma_\alpha^{lateral}}$. Now, let us consider $x\in \overline{\Gamma_{k,\mbf D}^{\textsc{M}} }\cap \overline{\Sigma_\alpha^{lateral}}$. Then, there exists a sequence $(x_n)_{n\ge 0}\in \Sigma_\alpha^{lateral}$ such that $x_n\to x$ as $n\to +\infty$.  Let us prove that for all $n$ large enough, $x_n\in \pa \mathsf \Omega_k^{\textsc{M}}$. 
         For $n$ large enough, $x_n$ does not belong to $\overline {\Omega_{\mathsf K_{\alpha/2}}}$ because $\overline{\Omega_{\mathsf K_{\alpha/2}}}\subset \Omega$ and $x_n\to x\in \partial \Omega$.   In addition $x_n \notin \mathsf C_\alpha$ (indeed $x_n\in \pa \mathsf C_\alpha$  since $x_n\in \Sigma_\alpha^{lateral}$).  
          Therefore, for $n$ large enough, 
         $x_n\notin \mathsf \Omega_k^{\textsc{M}}.$
         On the other hand, since $x_n\in \pa \mathsf C_\alpha$, 
         $x_n \in \overline { \mathsf C_\alpha} \subset \overline{\mathsf \Omega_k^{\textsc{M}}}.$
In conclusion, $x_n\in \Omega\cap  \pa{\mathsf \Omega_k^{\textsc{M}}}=\Gamma_{k,\mbf N}^{\textsc{M}}$ (see~\eqref{eq.onm}) and thus, $x\in \overline{\Gamma_{k,\mbf N}^{\textsc{M}}}$. 
This concludes the proof of~\eqref{eq.anglepi}.

       \medskip
        
        \noindent
        \textbf{Step 2b.}  
     Let us now prove item 4 in  Proposition~\ref{pr.omegakpoint}. To this end, let $\delta>0$.  Let us choose $\Gamma_{k,\mbf D}^{\textsc{M}} $  such that the distance between  $  \overline{\Gamma_{k,\mbf D}^{\textsc{M}} }$ and $\partial \Omega \setminus   \Gamma_{z_k}$ is smaller than $\delta/2$ (recall that $\overline{ \Gamma_{k,\mbf D}^{\textsc{M}} }\subset  \Gamma_{z_k}$ can be chosen as large as needed in $\Gamma_{z_k}$, see  Step 1a  above), i.e. 
   \begin{equation}\label{eq.ddgamma1}
 \mathsf d_{\, \overline \Omega}(\overline{\Gamma_{k,\mbf D}^{\textsc{M}} }, \partial \Omega \setminus \Gamma_{z_k} )\le \delta/2.
  \end{equation}

  Let us consider $x\in \Gamma_{k,\mbf N}^{\textsc{M}}$. According to~\eqref{eq.GK2},  $x\in \Sigma_\alpha^{lateral}$ or $x\in \partial \Omega_{\mathsf K_{\alpha/2}}$. If   $x\in \Sigma_\alpha^{lateral}$, then by the triangular inequality, it holds:
     $$\mathsf d_{\, \overline \Omega}(x, \partial \Omega \setminus \Gamma_{z_k} ) \le \mathsf d_{\, \overline \Omega}(x,  \overline{\Gamma_{k,\mbf D}^{\textsc{M}} }) +\mathsf d_{\, \overline \Omega}(\overline{\Gamma_{k,\mbf D}^{\textsc{M}} }, \partial \Omega \setminus \Gamma_{z_k} )\le \alpha+\delta/2,$$
     where we have used that according to~\eqref{eq.Blat},   $\mathsf d_{\, \overline \Omega}(x,  \overline{\Gamma_{k,\mbf D}^{\textsc{M}} })\le \mathsf d_{\, \overline \Omega}(x,  \pa{\Gamma_{k,\mbf D}^{\textsc{M}} })\le \alpha$, for all $x\in \Sigma_\alpha^{lateral}$. If $x\in \partial \Omega_{\mathsf K_{\alpha/2}} $, then $ \mathsf d_{\, \overline \Omega}(x,\partial \Omega)\le \alpha/2<\alpha$. 
Because $x\notin \mathsf C_\alpha$ (since $x\in \pa \mathsf \Omega_k^{\textsc{M}}$ and $\mathsf C_\alpha$ is an open subset of  $\mathsf \Omega_k^{\textsc{M}}$), one has $\mathsf z(x)\in \partial \Omega\setminus \Gamma_{k,\mbf D}^{\textsc{M}} $. Therefore, 
 $$\mathsf d_{\, \overline \Omega}(x, \partial \Omega \setminus \Gamma_{z_k} )\le \mathsf d_{\, \overline \Omega}(x,\mathsf z(x))+\mathsf d_{\, \overline \Omega}(\mathsf z(x),\partial \Omega \setminus \Gamma_{z_k})\le \alpha/2 + \delta/2,$$
 where we have used that either $\mathsf z(x)\in  \partial \Omega\setminus \Gamma_{z_k}$ (in which case $\mathsf d_{\, \overline \Omega}(\mathsf z(x),\partial \Omega \setminus \Gamma_{z_k})=0$) or $\mathsf z(x)\in  \Gamma_{z_k}\setminus \Gamma_{k,\mbf D}^{\textsc{M}} $ (in which case   $\mathsf d_{\, \overline \Omega}(\mathsf z(x),\partial \Omega \setminus \Gamma_{z_k})\le \delta/2$, see~\eqref{eq.ddgamma1} together with the fact that $\mathsf z(x)\notin \Gamma_{k,\mbf D}^{\textsc{M}} $). In conclusion 
      $$\sup_{x\in \Gamma_{k,\mbf N}^{\textsc{M}}}\mathsf d_{\, \overline \Omega}(x, \partial \Omega \setminus \Gamma_{z_k} )\le \alpha+\delta/2.$$
      Choosing $\alpha\le \delta/2$ concludes the proof of item 4 in  Proposition~\ref{pr.omegakpoint}.
%
%

    \medskip

  \noindent
\textbf{Step 3: Proof that $\mathsf \Omega_k^{\textsc{M}}$ is
  Lipschitz, and study of the sign of $\nabla f\cdot \mathsf n _{\mathsf \Omega_k^{\textsc{M}}}$.} 
Let us  first check that $\mathsf \Omega_k^{\textsc{M}}$ is
Lipschitz.  Notice that the union of two Lipchitz (even smooth)
subdomains of $\Omega$ is not necessarily a Lipchitz domain (the
boundary is even not  necessarily a manifold). In our setting, one has:  
$$\partial \Omega_{\mathsf K_{\alpha/2}}\cap \pa \mathsf
C_\alpha=\partial \Omega_{\mathsf K_{\alpha/2}} \cap
\Sigma_\alpha^{lateral} \text{
  (see~\eqref{eq.3P},~\eqref{eq.unionboundary2},
  and~\eqref{eq.unionboundary3})},$$
where   (i) $\Sigma_\alpha^{lateral}$ and $\partial \Omega_{\mathsf K_{\alpha/2}}$ are  smooth, and (ii)  the  normal derivatives $\nabla f(v)\cdot \mathsf n_{\mathsf C_\alpha}(v)$ and $\nabla f(v)\cdot \mathsf n_{ \Omega_{\mathsf K_{\alpha/2}}}(v)$  of $f$  at  $v\in \Sigma_\alpha^{lateral}\cap \partial \Omega_{\mathsf K_{\alpha/2}}$ are positive (so that the two normal vectors cannot be opposite, a situation which could create cusps). These two points will be used to prove that     the boundary of $\mathsf \Omega_k^{\textsc{M}}$ is  Lipschitz.   
One has:
\begin{equation}\label{eq.=pa-domk-union}
\pa \mathsf \Omega_k^{\textsc{M}} \subset \partial \Omega_{\mathsf K_{\alpha/2}} \cup \pa  \mathsf C_\alpha.
\end{equation}
Define the two open subsets of $\pa \mathsf \Omega_k^{\textsc{M}}$
\begin{equation}\label{eq.A1A2}
\mathsf A_1:=\pa \mathsf \Omega_k^{\textsc{M}} \cap (\partial \Omega_{\mathsf K_{\alpha/2}}\setminus \pa \mathsf C_\alpha) , \ \ \, \mathsf A_2:=   \pa \mathsf \Omega_k^{\textsc{M}}\cap  ( \pa \mathsf C_\alpha\setminus\partial \Omega_{\mathsf K_{\alpha/2}}),
\end{equation}
and the closed subset of $\pa \mathsf \Omega_k^{\textsc{M}}$
$\mathsf A_3:= \pa \mathsf \Omega_k^{\textsc{M}} \cap (\partial \Omega_{\mathsf K_{\alpha/2}}\cap \pa \mathsf C_\alpha),$
so that $\pa \mathsf \Omega_k^{\textsc{M}}$ is the disjoint union of
$\mathsf A_1$, $\mathsf A_2$, and $\mathsf A_3$. Let us now prove
that, for $j\in \{1,2,3\}$,  $\pa \mathsf \Omega_k^{\textsc{M}}$ is
Lipschitz in a neighborhood of any point of $\mathsf A_j$,  and let us
also study the sign of $\nabla f\cdot  \mathsf n_{ \mathsf
  \Omega_k^{\textsc{M}}}$ on  $\mathsf A_j$. 
\medskip

\noindent
\textbf{Step 3a: Study of $\mathsf A_1$. }  
First notice that (because $\partial \Omega_{\mathsf K_{\alpha/2}}\subset \Omega$),
\begin{equation}\label{eq.A1}
\mathsf A_1\subset  \partial \Omega_{\mathsf K_{\alpha/2}}\setminus \pa \mathsf C_\alpha \subset \Omega,
\end{equation}
Let $z\in\mathsf A_1$. Then, there exists a neighborhood $\mathsf O_z$ of $z$ in $\mathbb R^d$ such that $\mathsf O_z\cap \overline{\mathsf C_\alpha}=\emptyset$. Indeed, if not,  $z$ would belong to $\overline{\mathsf C_\alpha}=\pa \mathsf C_{\alpha}\cup \mathsf C_{\alpha}$, and $z$ cannot belong to $\pa \mathsf C_{\alpha}$ (by definition of $\mathsf A_1$) and $z$ cannot belong to $  \mathsf C_{\alpha}$ (because $z\in \pa \mathsf \Omega_k^{\textsc{M}}$). Using~\eqref{eq.DEFDOTOMEGA}, it then holds  $\mathsf O_z\cap \overline{\mathsf \Omega_k^{\textsc{M}}}=  \mathsf O_z \cap \overline{\Omega_{\mathsf K_{\alpha/2}}}$ (because $\overline{\mathsf \Omega_k^{\textsc{M}}}=\overline{\mathsf C_\alpha}\cup \overline{\Omega_{\mathsf K_{\alpha/2}}}$). 
Therefore, since in addition $\Omega_{\mathsf K_{\alpha/2}}$ is a smooth domain,   $\mathsf A_1$ is a smooth part of the boundary of $ \mathsf \Omega_k^{\textsc{M}}$  and  $\mathsf n_{ \mathsf \Omega_k^{\textsc{M}}}=\mathsf n_{ \Omega_{\mathsf K_{\alpha/2}}}$ on $\mathsf A_1$. 
Finally, using~\eqref{eq.n-omeka-alpha}, it holds:
 \begin{equation}\label{eq.panf1}
 \partial_{ \mathsf n_{ \mathsf \Omega_k^{\textsc{M}}}} f>0 \text{ on }\mathsf A_1.
 \end{equation}
\textbf{Step 3b: Study of $\mathsf A_2$.}
        It holds $\mathsf A_2  \subset \pa \mathsf C_\alpha\setminus\partial \Omega_{\mathsf K_{\alpha/2}}$. 
 With the same arguments as in Step 3a (see the lines
 after~\eqref{eq.A1}), $\mathsf O_z\cap \overline{\mathsf
   \Omega_k^{\textsc{M}}}=  \mathsf O_z \cap \overline{ \mathsf
   C_\alpha}$ for some neighborhood $\mathsf O_z$ af any point $z \in \mathsf A_2$.
Moreover, from~\eqref{eq.3P}, $\pa  \mathsf C_\alpha$ is
 $\mathcal C^\infty$ except on $\pa{\Gamma_{k,\mbf D}^{\textsc{M}}
 }\cup \pa \Sigma_\alpha^{base}$, where it is Lipschitz since $\Gamma_{k,\mbf D}^{\textsc{M}}
 $ and $\Sigma_\alpha^{lateral}$, and $\Sigma_\alpha^{base}$ and $\Sigma_\alpha^{lateral}$ are transverse (see Step 2a above). Thus,   $\mathsf A_2$ is a Lipschitz part of the boundary of $ \mathsf \Omega_k^{\textsc{M}}$ and  
 \begin{equation}\label{eq.nn=A2}
 \mathsf n_{ \mathsf \Omega_k^{\textsc{M}}}=\mathsf n_{  \mathsf C_\alpha} \ \text{ on $\mathsf A_2\setminus (\pa{\Gamma_{k,\mbf D}^{\textsc{M}} }\cup \pa \Sigma_\alpha^{base})$, i.e.  a.e. on $\mathsf A_2$}.
 \end{equation}   

 Let us now study the sign of  $\nabla f\cdot  \mathsf n_{ \mathsf \Omega_k^{\textsc{M}}}$ on $\mathsf A_2$. 
Recall that $\pa  \mathsf C_\alpha =  \overline{\Gamma_{k,\mbf D}^{\textsc{M}} }\cup { \Sigma_\alpha^{lateral} }\cup  \overline{ \Sigma_\alpha^{base}}$ (see~\eqref{eq.3P}),   $\overline{\Gamma_{k,\mbf D}^{\textsc{M}} } \cap \partial \Omega_{\mathsf K_{\alpha/2}} =\emptyset$ (see indeed~\eqref{eq.unionboundary3}), and $\pa \mathsf \Omega_k^{\textsc{M}} \cap   \overline{ \Sigma_\alpha^{base}}=\emptyset$  (see~\eqref{eq.unionboundary2}). Hence, it  holds:
\begin{align}\label{eq.=A2}
\mathsf A_2=\pa \mathsf \Omega_k^{\textsc{M}}\cap  ( \pa \mathsf C_\alpha \setminus  \partial \Omega_{\mathsf K_{\alpha/2}}) &=\underbrace{(\pa \mathsf \Omega_k^{\textsc{M}} \cap  \overline{\Gamma_{k,\mbf D}^{\textsc{M}} })}_{=\overline{\Gamma_{k,\mbf D}^{\textsc{M}} }} \cup   (\pa \mathsf \Omega_k^{\textsc{M}} \cap   {\Sigma_\alpha^{lateral} }\setminus \partial \Omega_{\mathsf K_{\alpha/2}}).
\end{align}
Let $z\in \mathsf A_2$. If $z\in \Gamma_{k,\mbf D}^{\textsc{M}} $,  then   $  \mathsf n_{  \mathsf C_\alpha}(z) =\mathsf n_\Omega(z)$ and thus, using~\eqref{eq.nn=A2}, it holds: 
 \begin{equation}\label{eq.panf2}
 \nabla f\cdot  \mathsf n_{ \mathsf \Omega_k^{\textsc{M}}}=0\text{ on }    {\Gamma_{k,\mbf D}^{\textsc{M}} },
 \end{equation}
where we also used the fact that  $\Gamma_{k,\mbf D}^{\textsc{M}} \subset \Gamma_{z_k}$ together with  item 2 in \autoref{A}. 
 If $z\in \Sigma_\alpha^{lateral}$, then  from~\eqref{eq.signBO} and~\eqref{eq.nn=A2}, it holds, 
 \begin{equation}\label{eq.panf3}
 \nabla f\cdot  \mathsf n_{ \mathsf \Omega_k^{\textsc{M}}}>0\text{ on }\pa \mathsf \Omega_k^{\textsc{M}} \cap   {\Sigma_\alpha^{lateral} }\setminus \partial \Omega_{\mathsf K_{\alpha/2}}.
  \end{equation}

\noindent
\textbf{Step 3c: Study of $\mathsf A_3$. }
Notice that $\mathsf A_3=\partial \Omega_{\mathsf K_{\alpha/2}}\cap \pa \mathsf C_\alpha$ (because $\partial \Omega_{\mathsf K_{\alpha/2}}\cap \pa \mathsf C_\alpha\subset \pa \mathsf \Omega_k^{\textsc{M}}$). Notice also that since $\partial \Omega_{\mathsf K_{\alpha/2}}\subset \Omega$, 
\begin{equation}\label{eq.A3}
\mathsf A_3 \subset  \Omega
\end{equation}
 Using~\eqref{eq.3P},~\eqref{eq.unionboundary2}, and~\eqref{eq.unionboundary3}, it   holds:
\begin{equation}\label{eq.A3-lateral}
\mathsf A_3= \partial \Omega_{\mathsf K_{\alpha/2}}\cap \pa \mathsf C_\alpha =\partial \Omega_{\mathsf K_{\alpha/2}}\cap \Sigma_\alpha^{lateral}.
\end{equation}  
Thus,  $ \partial \Omega_{\mathsf K_{\alpha/2}}$ intersects $ \pa \mathsf C_\alpha$ where $\pa \mathsf C_\alpha$ is smooth (i.e. on  $\Sigma_\alpha^{lateral}$). 
Let us consider  $v\in \mathsf A_3$. Let us conclude the proof by
considering successively the case when $\mathsf n_{\Omega_{\mathsf
    K_{\alpha/2}}}(v)$ is not collinear to $\mathsf n_{\mathsf
  C_\alpha}(v)$, and the case when $\mathsf n_{\Omega_{\mathsf K_{\alpha/2}}}(v)=\pm \mathsf n_{\mathsf C_\alpha}(v)$.

Let us first  consider \textbf{the case when $\mathsf n_{\Omega_{\mathsf K_{\alpha/2}}}(v)$ is not collinear to $\mathsf n_{\mathsf C_\alpha}(v)$}. By a continuity argument, there exists a neighborhood $\mathsf O_v$ of $v$ in $\Omega$ such that $\mathsf O_v\cap \pa \mathsf C_\alpha=\mathsf O_v\cap \Sigma_\alpha^{lateral}$ (so  that $\mathsf n_{\mathsf C_\alpha}$ is defined everywhere  and continuous on $\mathsf O_v\cap \pa \mathsf C_\alpha$)  and such that $\mathsf n_{\Omega_{\mathsf K_{\alpha/2}}}$ is not collinear to  $\mathsf n_{\mathsf C_\alpha}$ on $\mathsf O_v$. Consequently,   $ \partial \Omega_{\mathsf K_{\alpha/2}}$ and  $ \pa \mathsf C_\alpha$ are transverse on   $\mathsf O_v$ (or equivalently, the natural immersion map $\mathsf i:  \partial \Omega_{\mathsf K_{\alpha/2}}\to \mathbb R^d$ is transverse to  $ \pa \mathsf C_\alpha$   on   $\mathsf O_v$). Thus,   $\mathsf O_v\cap  \pa \mathsf \Omega_k^{\textsc{M}}$ is Lipschitz.
In addition, as a consequence  of the inverse image of a regular value Theorem~\cite[Theorem (5.12)]{brocker1982introduction} and its proof (see also~\cite{thom1954quelques,milnor1997topology,guillemin2010differential}) 
applied here to the  smooth function $\mathsf i$, one has, up to choosing  $\mathsf O_v$ smaller, $\mathsf O_v \cap \mathsf i^{-1}( \pa \mathsf C_\alpha)= \mathsf O_v\cap  (\partial \Omega_{\mathsf K_{\alpha/2}}\cap \pa \mathsf C_\alpha)$ (because $\mathsf i^{-1}( \pa \mathsf C_\alpha)=  \partial \Omega_{\mathsf K_{\alpha/2}}\cap \pa \mathsf C_\alpha$)  is a   $1$-codimensional  smooth submanifold of $\partial \Omega_{\mathsf K_{\alpha/2}}$ (i.e. a   $2$-codimensional  smooth submanifold of $\mathbb R^d$ included in $\pa \mathsf \Omega_k^{\textsc{M}}$).   
Therefore, for all $v\in \mathsf A_3$ such that  $\mathsf n_{\Omega_{\mathsf K_{\alpha/2}}}(v)$ is not collinear to $\mathsf n_{\mathsf C_\alpha}(v)$, there exists a neighborhood $\mathsf O_v$ of $v$ in $\Omega$ such that  
     \begin{equation}\label{eq.point1A3}
     \text{$\mathsf O_v\cap  \mathsf A_3$ is of measure $0$  for the surface measure on $\mathsf O_v\cap \pa \mathsf \Omega_k^{\textsc{M}}$}.
        \end{equation}

Let us finally consider \textbf{the case when $\mathsf n_{\Omega_{\mathsf K_{\alpha/2}}}(v)=\pm \mathsf n_{\mathsf C_\alpha}(v)$}. Using~\eqref{eq.n-omeka-alpha} and~\eqref{eq.signBO}, 
$\mathsf n_{\Omega_{\mathsf K_{\alpha/2}}}(v)=  +\mathsf n_{\mathsf C_\alpha}(v)$. Moreover, from~\eqref{eq.Calpha},~\eqref{eq.Blat}, and~\eqref{eq.A3-lateral} there exists a neighborhood $\mathsf O_v$ of $v$ in $\Omega$ such that $ \mathsf O_v\cap \pa \mathsf C_\alpha= \mathsf O_v\cap \Sigma_\alpha^{lateral}$ and thus (see~\eqref{eq.pauz2} and~\eqref{eq.pauz3}), 
$$
  \mathsf O_v\cap \mathsf C_\alpha=\mathsf O_v\cap\big \{w\in \overline \Omega, \, x_1(w)>0, \, x_d(w)\in (0,\alpha)\big \} 
$$
and (see~\eqref{eq.Blat})
  \begin{equation}\label{eq.normalCalpha}
  \mathsf O_v\cap \pa \mathsf C_\alpha=\mathsf O_v\cap\big \{w\in \overline \Omega, \, x_1(w)=0, \, x_d(w)\in (0,\alpha)\big \}\  (=\mathsf O_v\cap  \Sigma_\alpha^{lateral}).
          \end{equation}
          In the following, with a slight abuse of notation, we will denote by $x=(x_1,\tilde x)$ both a point   in $\mathsf O_v$ and its coordinates in the local basis. 
          
  In addition, since $\mathsf n_{\Omega_{\mathsf K_{\alpha/2}}}(v)=  +\mathsf n_{\mathsf C_\alpha}(v)$  and   $\partial \Omega_{\mathsf K_\alpha}$ is smooth, up to choosing $\mathsf O_v$ smaller, there exists a smooth function $\Psi: \mathbb R^{d-1} \to \mathbb R$ such that  $\Psi(\tilde x(v))=x_1(v)=0$ and 
\begin{equation}\label{eq.==p}
\mathsf O_v\cap \partial \Omega_{\mathsf K_{\alpha/2}}=\{( \Psi(\tilde x), \tilde x), \, x=(x_1,\tilde x)\in   \mathsf O_v\}
    \end{equation}
     is the graph\footnote{The
     fact that $\mathsf O_v\cap \partial \Omega_{\mathsf
       K_{\alpha/2}}$ is the graph of a function of $\tilde x$ is a
     consequence of the implicit function theorem, since $T_v\partial \Omega_{\mathsf K_\alpha}= \{\nabla {x_1}(v)\}^\bot$ ($\mathsf n_{\Omega_{\mathsf K_{\alpha/2}}}(v)=  +\mathsf n_{\mathsf C_\alpha}(v)$ and~\eqref{eq.normalCalpha}).  Indeed, in a neighborhood of $y_0:=(x_1(v), \tilde x(v))=(0,\tilde x(v))$ in $\mathbb R^d$,  $\partial \Omega_{\mathsf K_\alpha}$ is the set of points $(x_1,\tilde x)$ such that $\phi(x_1,\tilde x)=0$ where $\phi:\mathbb R^d\to \mathbb R$ is smooth.   In particular,  $\nabla \phi(y_0)\neq 0$ is collinear to $\mathsf n_{\Omega_{\mathsf K_{\alpha/2}}}(v)$ and  $\nabla_{\mbf T}\phi(y_0)=0$, where $\nabla_{\mbf T}$ is the tangential gradient of $\phi$  along $\partial \Omega_{\mathsf K_{\alpha/2}}$). Since  $T_v\partial \Omega_{\mathsf K_\alpha}= \{\nabla {x_1}(v)\}^\bot$ and $\nabla {x_1}(v)\perp \nabla {\tilde x_q}(v)$ for $q=2,\ldots,d$ (we choose normal coordinates systems), one has $\nabla_{\tilde x}\phi(y_0)=\nabla_{\mbf T}\phi(y_0)=0$ and thus,  $\partial_{x_1}\phi(y_0)\neq 0$. Equation~\eqref{eq.==p} then follows from the implicit function theorem.}  of $\Psi$ in the $(x_1,\tilde x)$ coordinates, where
     we set $\tilde x:=(x_2,\ldots, x_{d-1},x_d)$.
Moreover, one has  
$$\mathsf O_v\cap   \Omega_{\mathsf K_{\alpha/2}}=\{x=(x_1,\tilde x)\in   \mathsf O_v \text{ such that }  \, x_1>\Psi(\tilde x)\}.$$
Therefore, from~\eqref{eq.DEFDOTOMEGA}, 
$ \mathsf O_v\cap \mathsf \Omega_k^{\textsc{M}} =\{x=(x_1,\tilde x)\in   \mathsf O_v \text{ such that }  \, x_1>\min(\Psi(\tilde x),0)\}
$ and thus, 
 $$ \mathsf O_v\cap \pa \mathsf \Omega_k^{\textsc{M}} =\{x=(x_1,\tilde x)\in   \mathsf O_v \text{ such that }  \, x_1=\min(\Psi(\tilde x),0)\}
$$
is Lipschitz (indeed $\Upsilon: \tilde x\mapsto \min(\Psi(\tilde x),0)$ is a Lipschitz function). In addition,  for a.e. $x=(x_1,\tilde x)\in \mathsf O_v\cap \pa \mathsf \Omega_k^{\textsc{M}}$: $\mathsf n_{\mathsf \Omega_k^{\textsc{M}}}(x)\in \{ \mathsf n_{  \mathsf C_\alpha }(x), \mathsf n_{  \Omega_{\mathsf K_\alpha} }(x)\}$. Indeed, in the $(x_1,\tilde x)$-coordinates, one has for a.e. $x=(x_1,\tilde x)\in \mathsf O_v\cap \pa \mathsf \Omega_k^{\textsc{M}}$, 
$$T_x \pa \mathsf \Omega_k^{\textsc{M}}=\Big\{\big (\tilde p\cdot \nabla \Upsilon (\tilde x),  \tilde p\big ), \ \tilde p \in \mathbb R^{d-1}\Big\} \text{ and } \mathsf n_{\mathsf \Omega_k^{\textsc{M}}}(x)=\frac{(-1, \nabla \Upsilon (\tilde x))}{\sqrt{1+ \vert \nabla\Upsilon (\tilde x)\vert ^2  }}.$$
\noindent
Because for a.e. $\tilde x$, $\nabla \Upsilon (\tilde x)\in \{0,\nabla \Psi(\tilde x)\}$, it holds  for a.e. $x=(x_1,\tilde x)\in \mathsf O_v\cap \pa \mathsf \Omega_k^{\textsc{M}}$, $\mathsf n_{\mathsf \Omega_k^{\textsc{M}}}(x)\in \{ \mathsf n_{  \mathsf C_\alpha }(x), \mathsf n_{  \Omega_{\mathsf K_\alpha} }(x)\}$. 
Moreover, using~\eqref{eq.n-omeka-alpha} and~\eqref{eq.signBO}, it holds  for a.e. $x=(x_1,\tilde x)\in \mathsf O_v\cap \pa \mathsf \Omega_k^{\textsc{M}}$:
 \begin{equation}\label{eq.point1A3bis}
\nabla f(x)\cdot \mathsf n_{\mathsf \Omega_k^{\textsc{M}}}(x)>0.
\end{equation}
%
%
%
%
%
%

From~\eqref{eq.point1A3} and~\eqref{eq.point1A3bis}, we thus conclude
that for any point $v \in \mathsf A_3$, there exists a neighborhood~$\mathsf O_v$ of $v$ in $\Omega$ such that, for the surface measure on $\mathsf O_v\cap   \pa \mathsf \Omega_k^{\textsc{M}} $,   
either $\mathsf O_v \cap \mathsf A_3$ is of measure $0$  or 
$\nabla f\cdot  \mathsf n_{ \mathsf \Omega_k^{\textsc{M}}}>0$ 
  a.e. on  $\mathsf O_v \cap \mathsf A_3$.
This implies that, for the surface measure on $\pa \mathsf \Omega_k^{\textsc{M}} $,  
 \begin{equation}\label{eq.panf4}
\nabla f\cdot  \mathsf n_{ \mathsf \Omega_k^{\textsc{M}}}>0 \text{
  a.e. on } \mathsf A_3.
  \end{equation}

 In conclusion, $\mathsf \Omega_k^{\textsc{M}}$ is a Lipschitz subdomain  of $\Omega$. Furthermore, we have proved that: 
 $$\nabla f\cdot  \mathsf n_{ \mathsf \Omega_k^{\textsc{M}}}=0\text{ a.e. on } \overline{\Gamma_{k,\mbf D}^{\textsc{M}} }=\pa \mathsf \Omega_k^{\textsc{M}}\cap \partial \Omega \ \text{ (see~\eqref{eq.panf2} and~\eqref{eq.onm})}.$$
In addition, since (see~\eqref{eq.onm},~\eqref{eq.A1},~\eqref {eq.=A2}, and~\eqref{eq.A3})
$$\Gamma_{k,\mbf N}^{\textsc{M}}=\pa \mathsf \Omega_k^{\textsc{M}}\cap \Omega= \mathsf A_1\,  \cup\, ( \mathsf A_2\ \cap \  \Omega) \, \cup  \, \mathsf A_3,$$
and $\mathsf A_2\ \cap \  \Omega=\pa \mathsf \Omega_k^{\textsc{M}}
\cap   {\Sigma_\alpha^{lateral} }\setminus \partial \Omega_{\mathsf
  K_{\alpha/2}} $, one deduces from~\eqref{eq.panf1},~\eqref{eq.panf3}, and~\eqref{eq.panf4}, that 
 $$\nabla f\cdot  \mathsf n_{ \mathsf \Omega_k^{\textsc{M}}}>0\text{ a.e. on } \Gamma_{k,\mbf N}^{\textsc{M}}=\pa \mathsf \Omega_k^{\textsc{M}}\cap \Omega. $$
 This concludes the proof of Proposition~\ref{pr.omegakpoint}. 
 \end{proof}
 \noindent
 

\subsection{Witten Laplacians with mixed  Dirichlet-Neumann boundary conditions associated with $(z_k)_{k=1,\ldots,n}$}
\label{sec.Witten-TN}

In this section, we define a Witten Laplacian with mixed Dirichlet-Neumann boundary conditions associated with each saddle point $z_k$ of $f$ using the domain $\mathsf \Omega_k^{\textsc{M}}$ constructed in the previous section. The idea is to define a Witten Laplacian in $\Lambda^qL^2(\mathsf \Omega_k^{\textsc{M}})$  with  
   Dirichlet   boundary conditions  on $\Gamma_{k,\mbf D}^{\textsc{M}}$ (where $ \nabla f\cdot \mathsf n_{  \mathsf \Omega_k^{\textsc{M}} }  =0$) and Neumann  boundary conditions on $\Gamma_{k,\mbf N}^{\textsc{M}}$ (where $ \nabla f\cdot \mathsf n_{  \mathsf \Omega_k^{\textsc{M}} }  >0$), see Proposition \ref{pr.omegakpoint}. Since   $x_0\in  \mathsf \Omega_k^{\textsc{M}}$ is the only minimum of $f$ in $\mathsf \Omega_k^{\textsc{M}}$ and $z_k\in \pa  \mathsf \Omega_k^{\textsc{M}}$ is the only saddle point of $f$ in $\overline{ \mathsf \Omega_k^{\textsc{M}}}$, we expect,  in view of  Theorem~\ref{thm.main1} and the results of~\cite{le-peutrec-10}, that   such Witten Laplacians have only one   eigenvalue smaller than $\mathsf ch$ when $q=0$ and $q=1$.   
Thanks to Witten's complex   structure,  
this eigenvalue, already introduced as   $\lambda(\mathsf
\Omega_k^{\textsc{M}})$ at the beginning of Section~\ref{sec.zk-QM},
will be the same for $q=0$ and $q=1$. The quasi-mode $\mathsf
v_k^{(1)} $ of   $\Delta_{f,h}^{\mathsf{Di},(1)}(\Omega)$ associated
with $z_k$ will then be defined by multiplying   by a cut-off function the principal $1$-eigenform $\mathsf u_k^{(1)}$ of this Witten Laplacian with mixed Dirichlet-Neumann boundary conditions.
\medskip

We first give the  definition of Witten Laplacians with mixed
   Dirichlet-Neumann boundary conditions on
   Lipschitz domains in  Section~\ref{sec.Witten-general-result}.  We
   then study the spectral properties of these Witten Laplacians and
   derive some estimates on the principal eigenvalues and eigenforms in
   Sections~\ref{sec.specM},~\ref{sec.equivM} and~\ref{sec.Agmon1}

\subsubsection{Witten Laplacians with mixed
 Dirichlet-Neumann boundary conditions on Lipschitz domains}
 \label{sec.Witten-general-result}
 
In this section, in order to ease the notation, we drop the subscript $k$ in
$(\mathsf \Omega^{\textsc{M}}_{k}, \Gamma^{\textsc{M}}_{k,\mbf D},
\Gamma^{\textsc{M}}_{k,\mbf N})$, since the results will then be applied to each of
this triplet, for $k \in \{1, \ldots,n\}$.
 Let thus $\mathsf \Omega^{\textsc{M}}$ be a Lipschitz subdomain of
 $\Omega$.  Let $\Gamma^{\textsc{M}}_{\mbf D}$ and
 $\Gamma^{\textsc{M}}_{\mbf N}$ be two disjoint open subsets of
 $\partial \mathsf \Omega^{\textsc{M}}$ such that
 $\overline{\Gamma^{\textsc{M}}_{\mbf D}} \cup
 \overline{\Gamma^{\textsc{M}}_{\mbf N}} = \pa \mathsf
 \Omega^{\textsc{M}}$. 


This section is organized as follows. We  first recall the definition
of weak traces for forms $w\in  \Lambda^q H_{\mathsf d} ( \mathsf
\Omega^{\textsc{M}} )\cap \Lambda^q H_{\mathsf d^*} ( \mathsf
\Omega^{\textsc{M}} )$
where for $q\in\{0,\ldots,d\}$, 
\begin{equation}
\label{eq.H-d}
\Lambda^q H_{\mathsf d} ( \mathsf \Omega^{\textsc{M}} ):= \left\{w\in \Lambda^q L^2 ( \mathsf \Omega^{\textsc{M}} ),\ \mathsf dw\in\Lambda^{q+1} L^2 ( \mathsf \Omega^{\textsc{M}} )\right\}
\end{equation}
and  
\begin{equation}
\label{eq.H-d*}
\Lambda^q H_{\mathsf d^*} ( \mathsf \Omega^{\textsc{M}} )\ :=\  \left\{w\in \Lambda^q L^2 ( \mathsf \Omega^{\textsc{M}} ),\ \mathsf d^*w\in\Lambda^{q-1} L^2 ( \mathsf \Omega^{\textsc{M}} )\right\}
\end{equation}
are equipped  with their natural graph norms. Let us recall the
convention $ \Lambda^{-1} L^2=\Lambda^{d+1} L^2=\{0\}$.
  Secondly, we  state trace estimates and subelliptic estimates for forms $w\in \Lambda^q H_{\mathsf d} ( \mathsf \Omega^{\textsc{M}} )\cap \Lambda^q H_{\mathsf d^*} ( \mathsf \Omega^{\textsc{M}} )$ such that $\mbf t w=0$ on $\Gamma^{\textsc{M}}_{\mbf D}$ and $\mbf n w=0$ on $\Gamma^{\textsc{M}}_{\mbf N}$. 
Indeed (see \cite{brown-94,jakab-mitrea-mitrea-09}),  a trace in $\Lambda^q L^2(\partial \mathsf
\Omega^{\textsc{M}})$ does not exist in general for such forms except if  $\Gamma^{\textsc{M}}_{\mbf D}$ and $\Gamma^{\textsc{M}}_{\mbf N}$ meet at an angle strictly
smaller than~$\pi$ (measured in $\mathsf \Omega^{\textsc{M}}$), in the
sense of
Definition~\ref{de.angle}. This explains the role of item~$3$ in
Proposition~\ref{pr.omegakpoint}. Finally, we introduce the Witten
Laplacians of interest, together with an associated Green
formula. This formula will be crucial   to study the spectral
properties of these operators in the next section. 
\medskip

\noindent
\textbf{Weak definitions of traces for elements in $\Lambda^q H_{\mathsf d} ( \mathsf \Omega^{\textsc{M}} )$ or in $\Lambda^q H_{\mathsf d^*} ( \mathsf \Omega^{\textsc{M}} )$.}
Let us recall that  for a differential form $u$ in $\Lambda^qL^2 (\pa\mathsf \Omega^{\textsc{M}} )$, the tangential and normal components are defined as follows:
\begin{equation}
\label{eq.decomp-bdy}
u= \mathbf t u+\mathbf n u\quad\text{with}\quad
\mathbf t u= \mathbf{i}_{\mathsf n_{\mathsf \Omega^{\textsc{M}}}}(\mathsf n_{\mathsf \Omega^{\textsc{M}}}^{\flat} \wedge u)\quad\text{and}\quad 
 \mathbf n u= \mathsf n_{\mathsf \Omega^{\textsc{M}}}^{\flat}\wedge(\mathbf{i}_{\mathsf n_{\mathsf \Omega^{\textsc{M}}}} u),
\end{equation}
where the superscript $\flat$ stands for the usual musical
isomorphism ($\mathsf n_{\mathsf \Omega^{\textsc{M}}}^\flat$ is the 1-form associated with $\mathsf n_{\mathsf \Omega^{\textsc{M}}}$,  $\mathsf n_{\mathsf \Omega^{\textsc{M}}}$ being is the unit outward normal to $\mathsf \Omega^{\textsc{M}}$). Notice that   $\mathbf t u$  is orthogonal to $\mbf n u$ in $\Lambda^qL^2 (\pa\mathsf \Omega^{\textsc{M}} )$. Let us
recall that the mapping 
\begin{equation}
\label{eq.trace-usual}
w\in \Lambda^qH^1(\mathsf \Omega^{\textsc{M}} )\mapsto w|_{\pa \mathsf \Omega^{\textsc{M}}}\in \Lambda^qH^{1/2} (\pa\mathsf \Omega^{\textsc{M}} )
\end{equation}
 is well defined, continuous, and surjective. We would like here to
 recall the procedure to extend the notion of traces to elements in
 the subspaces of $\Lambda^qH^1(\mathsf \Omega^{\textsc{M}} )$: $\Lambda^q H_{\mathsf d} ( \mathsf \Omega^{\textsc{M}} )$ and $\Lambda^q H_{\mathsf d^*} ( \mathsf \Omega^{\textsc{M}} )$. This is achieved   using a duality argument and  the standard Green formula  which reads for  differential forms $(u,v)\in \Lambda^{q}H^1 ( \mathsf \Omega^{\textsc{M}} ) \times  \Lambda^{q
+1}H^1 (\mathsf \Omega^{\textsc{M}} )$:
\begin{align}
\nonumber
\langle \mathsf du, v\rangle_{ L^2 ( \mathsf \Omega^{\textsc{M}} )}- \langle u, \mathsf d^*v\rangle_{ L^2 ( \mathsf \Omega^{\textsc{M}} )}&= \int_{\partial \mathsf \Omega^{\textsc{M}}}  \langle  \mathsf n_{\mathsf \Omega^{\textsc{M}}}^\flat \wedge u,v\rangle_{ T^*_{\sigma}\mathsf \Omega^{\textsc{M}}} d\sigma= \int_{\partial \mathsf \Omega^{\textsc{M}}}  \langle  \mathsf n_{\mathsf \Omega^{\textsc{M}}}^\flat \wedge u, \mathbf n v\rangle_{ T^*_{\sigma}\mathsf \Omega^{\textsc{M}}} d\sigma\\
\label{eq:usual_Green}
&=\int_{\partial \mathsf \Omega^{\textsc{M}}} \langle  u,\mathbf{i}_{\mathsf n_{\mathsf \Omega^{\textsc{M}}}} v\rangle_{ T^*_{\sigma}\mathsf \Omega^{\textsc{M}}} d\sigma=\! \int_{\partial \mathsf \Omega^{\textsc{M}}} \!\!  \langle \mathbf t  u,\mathbf{i}_{\mathsf n_{\mathsf \Omega^{\textsc{M}}}} v\rangle_{ T^*_{\sigma}\mathsf \Omega^{\textsc{M}}} d\sigma,
\end{align}
where we used the fact that the adjoint of  $\mathsf n_{\mathsf \Omega^{\textsc{M}}}^\flat\wedge$ in $\Lambda^qL^2 ( \pa \mathsf \Omega^{\textsc{M}} )$ is $\mathbf i_{\mathsf n_{\mathsf \Omega^{\textsc{M}}}}$. Let us now consider   $w \in \Lambda^q
H_{\mathsf d} ( \mathsf \Omega^{\textsc{M}} )$. Then,  $\mathsf
n_{\mathsf \Omega^{\textsc{M}}}^\flat \wedge w$ is defined as an
element in $ \Lambda^{q+1} H^{-\frac12} (\pa\mathsf
\Omega^{\textsc{M}} )$ by: $\forall \phi \in\Lambda^{q+1}H^{\frac12}
(\pa\mathsf \Omega^{\textsc{M}} )$
\begin{equation}
\label{eq.n-wedge-u}
\langle \mathsf n_{\mathsf \Omega^{\textsc{M}}}^\flat \wedge w ,  \phi \rangle_{H^{-\frac 12} (\pa\mathsf \Omega^{\textsc{M}} ),H^{\frac 12} (\pa\mathsf \Omega^{\textsc{M}} )} = 
\langle \mathsf d w,  \Phi \rangle_{ L^2 ( \mathsf \Omega^{\textsc{M}} )}-\langle w, \mathsf d^*\Phi\rangle_{ L^2 ( \mathsf \Omega^{\textsc{M}} )},
\end{equation}
where $\Phi$ is any form in $\Lambda^{q+1}H^{1} ( \mathsf \Omega^{\textsc{M}} )$ whose trace
in $\Lambda^{q+1}H^{\frac12} (\pa\mathsf \Omega^{\textsc{M}} )$ is $\phi$. Recall that this definition is independent of the  chosen extension $\Phi$ of $\phi$ (this follows from~\eqref{eq:usual_Green} and the density of $\Lambda^q\mathcal C^\infty\big (\, \overline{\mathsf \Omega^{\textsc{M}}}\, \big )$
 in   $\Lambda^q H_{\mathsf d} ( \mathsf \Omega^{\textsc{M}} )$, see
 for example
 \cite[Proposition~3.1]{jakab-mitrea-mitrea-09}). Similarly, for any
 $w \in \Lambda^q H_{\mathsf d^*} ( \mathsf \Omega^{\textsc{M}} )$,
 $\mathbf{i}_{\mathsf n_{\mathsf \Omega^{\textsc{M}}}} w\in
 \Lambda^{q-1}H^{-\frac12} (\pa\mathsf \Omega^{\textsc{M}} )$
is defined by: $\forall \phi \in \Lambda^{q-1} H^{\frac12} (\pa\mathsf \Omega^{\textsc{M}} )$, 
\begin{equation}
\label{eq.i-n-u} 
\langle \mathbf{i}_{\mathsf n_{\mathsf \Omega^{\textsc{M}}}} w , \phi \rangle_{H^{-\frac12} (\pa\mathsf \Omega^{\textsc{M}} ),H^{\frac12} (\pa\mathsf \Omega^{\textsc{M}} )} =
\langle w, \mathsf d\Phi\rangle_{ L^2 ( \mathsf \Omega^{\textsc{M}} )}-\langle \mathsf d^*w, \Phi\rangle_{ L^2 ( \mathsf \Omega^{\textsc{M}} )},
\end{equation}
where $\Phi$ is any extension of $\phi$ in $\Lambda^{q-1}H^{1} ( \mathsf \Omega^{\textsc{M}})$. 

Let us now recover   the decomposition \eqref{eq.decomp-bdy} for forms
$w\in \Lambda^q H_{\mathsf d} ( \mathsf \Omega^{\textsc{M}} )\cap
\Lambda^q H_{\mathsf d^*} ( \mathsf \Omega^{\textsc{M}} )$ such that,
on a subset $\Gamma$ of $\pa \mathsf \Omega^{\textsc{M}} $, the
tangential trace or the normal trace are defined in a weak sense.
Let $w\in   \Lambda^q H_{\mathsf d} ( \mathsf \Omega^{\textsc{M}} )$. 
 If $ \mathsf n_{\mathsf \Omega^{\textsc{M}}}^\flat \wedge w \in
 \Lambda^{q+1} L^2 ( \Gamma)$, we define   $\mbf tw|_\Gamma$, the tangential trace of $w$  on $\Gamma$, by
\begin{equation}
\label{eq.tan-gen}\mathbf t w|_{\Gamma}  :=  \mathbf{i}_{\mathsf n_{\mathsf \Omega^{\textsc{M}} }}
(\mathsf n_{\mathsf \Omega^{\textsc{M}} }^{\flat} \wedge w) \in \Lambda^{q} L^2(\Gamma),\text{ so that } \|\mathbf t w\|_{ L^2(\Gamma)} = \|\mathsf n_{\mathsf \Omega^{\textsc{M}} }^{\flat} \wedge w\|_{ L^2(\Gamma)}.
\end{equation}  
In particular $\mbf tw|_\Gamma=0$ if $ \mathsf n_{\mathsf \Omega^{\textsc{M}}}^\flat \wedge w|_\Gamma=0$. 
Let us now consider $w  \in \Lambda^q H_{\mathsf d^*} ( \mathsf \Omega^{\textsc{M}} )$.   
 When  $ \mathbf{i}_{\mathsf n_{\mathsf \Omega^{\textsc{M}}}} w\in
 \Lambda^{q-1} L^2 ( \Gamma )$, we define   $\mbf nw|_\Gamma$, the
 normal trace of $w$  on $\Gamma$, by 
\begin{equation}
\label{eq.norm-gen}\mathbf n u|_{\Gamma}  :=  \mathsf n_{\mathsf \Omega^{\textsc{M}} }^{\flat} \wedge
(\mathbf{i}_{\mathsf n_{\mathsf \Omega^{\textsc{M}} }}u) \in \Lambda^{q} L^2(\Gamma)\ , \text{
 so that } \|\mathbf n u\|_{ L^2(\Gamma)}
=\|\mathbf{i}_{\mathsf n_{\mathsf \Omega^{\textsc{M}} }}u\|_{ L^2(\Gamma)}.
\end{equation}
In particular,  $\mathbf n u|_{\Gamma} =0$  if $ \mathbf{i}_{\mathsf n_{\mathsf \Omega^{\textsc{M}}}} w|_\Gamma=0$. 
Lastly, if $w\in \Lambda^qH_{\mathsf d}(\mathsf \Omega^{\textsc{M}})\cap \Lambda^qH_{\mathsf d^*}(\mathsf \Omega^{\textsc{M}})$ is such that $\mathsf n_{\mathsf \Omega^{\textsc{M}} }^{\flat}\wedge w|_{\Gamma}\in \Lambda^{q+1}L^2(\Gamma)$ and $\mathbf{i}_{\mathsf n_{\mathsf \Omega^{\textsc{M}} }}w\in \Lambda^{q-1}L^2(\Gamma)$ then $w$ admits a trace $w|_{\Gamma}$  in $\Lambda^qL^2(\Gamma)$  defined by  (see~\eqref{eq.tan-gen} and~\eqref{eq.norm-gen}),
\begin{equation}
\label{eq.trace-gen}
w|_{\Gamma} := \mathbf tw|_{\Gamma}+
\mathbf n w|_{\Gamma}. 
\end{equation}
 In addition, one has  for such $w$:
$$
\|w|_{\Gamma}\|^2_{ L^2(\Gamma)} = 
\| \mathbf t w|_{\Gamma}\|^2_{ L^2(\Gamma)}
+
\|\mathbf n w|_{\Gamma}\|^2_{ L^2(\Gamma)}
=
\|\mathsf n_{\mathsf \Omega^{\textsc{M}} }^{\flat} \wedge w\|_{ L^2(\Gamma)}^2
+
\|\mathbf{i}_{\mathsf n_{\mathsf \Omega^{\textsc{M}} }}w\|_{ L^2(\Gamma)}^2
.
$$ 
Let us mention that all the above definitions coincide 
with the usual ones when $w$ belongs to $\Lambda^q H^1 ( \mathsf \Omega^{\textsc{M}} )$. In particular, \eqref{eq.trace-gen}   can be seen as an extension of  \eqref{eq.decomp-bdy}.
\medskip
\medskip

\noindent
\textbf{Trace estimates  for forms $ \Lambda^q H_{\mathsf d} ( \mathsf \Omega^{\textsc{M}} )\cap \Lambda^q H_{\mathsf d^*} ( \mathsf \Omega^{\textsc{M}} )$ satisfying  mixed Dirichlet-Neumann  boundary conditions  and when $\mathsf \Omega^{\textsc{M}}$ is not smooth.}
Let $\Gamma$ be any open Lipschitz subset of $\partial \mathsf \Omega^{\textsc{M}}$. According to \cite[Proposition~3.1]{jakab-mitrea-mitrea-09},
the space
$$
 \big \{w\in\Lambda^q\mathcal C^\infty\big (\overline{\mathsf \Omega^{\textsc{M}}}\big ),\ 
w= 0 \text{ in a neighborhood of } \pa\mathsf \Omega^{\textsc{M}}\setminus \Gamma \big \}
$$
is dense  in
$$
\Lambda^q H_{\mathsf d,\Gamma} ( \mathsf \Omega^{\textsc{M}} ) := \big \{w\in\Lambda^q H_{\mathsf d} ( \mathsf \Omega^{\textsc{M}} ),\ 
\supp(\mathsf n_{\mathsf \Omega^{\textsc{M}}}^{\flat} \wedge w)\subset \overline\Gamma \big \}
$$
and in 
$$
\Lambda^q H_{\mathsf d^*,\Gamma} ( \mathsf \Omega^{\textsc{M}} ) := \big \{w\in\Lambda^q H_{\mathsf d^*} ( \mathsf \Omega^{\textsc{M}} ),\ 
\supp(\mathbf{i}_{\mathsf n_{\mathsf \Omega^{\textsc{M}}}}  w)\subset \overline\Gamma \big \}.
$$ 
We are now in position to state the following result which is a consequence  of \cite[Theorems~1.1 and 1.2]{jakab-mitrea-mitrea-09} (see also \cite[Theorems~4.1 and~4.2]{goldshtein-mitrea-mitrea-11}).
\begin{proposition}
\label{pr.QTN}
Let us assume that $\mathsf \Omega ^{\textsc{M}}\subset \mathbb R^d$ is a Lipschitz domain. Let $\Gamma^{\textsc{M}}_{\mbf D}$
and $\Gamma^{\textsc{M}}_{\mbf N}$ be two disjoint Lipschitz   open subsets of $\partial \mathsf \Omega ^{\textsc{M}}$ such that $\overline{\Gamma^{\textsc{M}}_{\mbf D}} \cup \overline{\Gamma^{\textsc{M}}_{\mbf N}} = \pa
\mathsf \Omega ^{\textsc{M}}$ and such that $\Gamma^{\textsc{M}}_{\mbf D}$ and $\Gamma^{\textsc{M}}_{\mbf N}$ meet at an angle
strictly smaller than $\pi$ (in the sense of Definition~\ref{de.angle}). Then, the following results hold:

\noindent
{(i)} Let $w$ be a differential form such that (see \eqref{eq.H-d},~\eqref{eq.H-d*},~\eqref{eq.tan-gen}, and \eqref{eq.norm-gen})
 $$
w\in \Lambda^qH_{\mathsf d}(\mathsf \Omega^{\textsc{M}})\cap \Lambda^qH_{\mathsf d^*}(\mathsf \Omega^{\textsc{M}}), \  
\mathbf t w|_{\Gamma^{\textsc{M}}_{\mbf D}}=0 \text{ and } 
\mathbf n w|_{\Gamma^{\textsc{M}}_{\mbf N}}=0.$$
Then $w$ satisfies 
$$w\in\Lambda^{q} H^{\frac12} ( \mathsf \Omega ^{\textsc{M}} )\quad\text{and}\quad
\mathbf{i}_{\mathsf n_{\mathsf \Omega ^{\textsc{M}}}}u,\ \mathsf n_{\mathsf \Omega ^{\textsc{M}}}^{\flat} \wedge u \in
\Lambda L^2 (\pa\mathsf \Omega ^{\textsc{M}} )
$$
as well as the subelliptic estimate:
\begin{equation}\label{eq:subelliptic}
\|w\|_{ H^{\frac12} ( \mathsf \Omega ^{\textsc{M}} )}+
\|w|_{\pa\mathsf \Omega ^{\textsc{M}}}\|_{ L^{2} (\pa\mathsf \Omega ^{\textsc{M}} )}\
\leq\ C\left(\|w\|_{ L^{2} ( \mathsf \Omega ^{\textsc{M}} )}
+\|\mathsf d w\|_{ L^{2} ( \mathsf \Omega ^{\textsc{M}} )}+\|\mathsf d^*w\|_{ L^{2} ( \mathsf \Omega ^{\textsc{M}} )} \right),
\end{equation}
where $w|_{\pa\mathsf \Omega ^{\textsc{M}}}$ is defined by \eqref{eq.trace-gen} and $C>0$ is independent of $w$.

\medskip
\noindent
(ii) Assume that $f:\overline{\mathsf \Omega ^{\textsc{M}}}\to \mathbb R$ is a $\mathcal C^\infty$ function. The unbounded operators $\mathsf d_{f,h,\mbf T}^{(q)} ( \mathsf \Omega ^{\textsc{M}} )$ and $\delta_{f,h,\mbf N}^{(q)} ( \mathsf \Omega ^{\textsc{M}} )$  on $\Lambda^{q} L^{2} ( \mathsf \Omega ^{\textsc{M}} )$
defined by
$$
\mathsf d_{f,h,\mbf T}^{(q)}  ( \mathsf \Omega ^{\textsc{M}} ) = \mathsf d_{f,h}^{(q)}$$
with domain 
$$\mathcal D\big (\mathsf d_{f,h,\mbf T}^{(q)}  ( \mathsf \Omega ^{\textsc{M}} )\big )=\left\{w\in \Lambda^{q}L^2 ( \mathsf \Omega ^{\textsc{M}} ),\ \mathsf  d_{f,h}w \in \Lambda^{q+1}L^2(\mathsf \Omega ^{\textsc{M}}),\ \mathbf{t}w|_{\Gamma^{\textsc{M}}_{\mbf D}}=0 \right\},
$$
and
$$
\delta_{f,h,\mbf N}^{(q)} ( \mathsf \Omega ^{\textsc{M}} ) =   \mathsf d_{f,h}^{(q)} \, ^*$$
with domain 
$$ \mathcal D\big (\delta_{f,h,\mbf N}^{(p)} ( \mathsf \Omega ^{\textsc{M}} )\big )=\left\{w\in \Lambda^{q} L^2 ( \mathsf \Omega ^{\textsc{M}} ),\ \mathsf d^*_{f,h}w\in \Lambda^{q-1}L^2( \mathsf \Omega ^{\textsc{M}} ),\ \mathbf{n}w|_{\Gamma^{\textsc{M}}_{\mbf N}}=0 \right\} ,
$$
are closed, densely defined, and adjoint one of each other in $\Lambda^{q}L^2 ( \mathsf \Omega ^{\textsc{M}} )$.
\end{proposition}
On can check that (see~\cite[Equation (130]{DLLN})
\begin{equation}\label{eq.dom-dT} 
\left\{
\begin{aligned}
&\overline{ \Im  \mathsf d_{f,h,\mbf T}}\subset \Ker  \mathsf d_{f,h,\mbf T} \quad\text{and}\quad  \mathsf d_{f,h,\mbf T}^{2}=0,\\
&\overline{ \Im  \delta_{f,h,\mbf N}}\subset \Ker  \delta_{f,h,\mbf N} \quad\text{and}\quad
\delta_{f,h,\mbf N}^{2}=0.
\end{aligned}
\right.
\end{equation}

\noindent
\textbf{Witten Laplacian with mixed Dirichlet-Neumann boundary
  conditions on $\pa \mathsf \Omega ^{\textsc{M}}$.} We are now in
position to define the   Witten Laplacians with mixed Dirichlet-Neumann boundary conditions on~$\pa \mathsf \Omega ^{\textsc{M}}$ (see also~\cite[p. 89]{DLLN}). 
\begin{proposition}
\label{pr.DeltaM}
Let us assume that $\mathsf \Omega ^{\textsc{M}}$, $\Gamma^{\textsc{M}}_{\mbf D}$, and   $\Gamma^{\textsc{M}}_{\mbf N}$ satisfy the assumptions of Proposition~\ref{pr.QTN}. Let    $q=0,\ldots,d$. Let us define on $ \Lambda L^{2} (
\mathsf \Omega ^{\textsc{M}} )$ the operator 
\begin{equation}
\label{eq.DeltaTN}
\Delta_{f,h}^{\textsc{M},(q)} ( \mathsf \Omega ^{\textsc{M}} ) := \mathsf d_{f,h,\mbf T}^{(q-1)} ( \mathsf \Omega ^{\textsc{M}} )\circ \delta_{f,h,\mbf N}^{(q)} ( \mathsf \Omega ^{\textsc{M}} )+\delta_{f,h,\mbf N}^{(q+1)} ( \mathsf \Omega ^{\textsc{M}} )\circ \mathsf d_{f,h,\mbf T}^{(q)} ( \mathsf \Omega ^{\textsc{M}} ),
\end{equation}
in the sense of composition of unbounded operators, see  Proposition~\ref{pr.QTN} for the definitions of $\mathsf d_{f,h,\mbf T} ( \mathsf \Omega ^{\textsc{M}} )$ and $\delta_{f,h,\mbf N} ( \mathsf \Omega ^{\textsc{M}} )$. This operator  is a densely defined nonnegative self-adjoint  operator  
and its domain is given by  
\begin{equation}
\label{eq.domDeltaTN}
\begin{aligned}
\mathcal D\big  (\Delta_{f,h}^{\textsc{M},(q)} ( \mathsf \Omega ^{\textsc{M}} )\big ) =\Big\{&w \in \Lambda^qL^{2}
( \mathsf \Omega ^{\textsc{M}} ), \, 
\mathsf d_{f,h} w,\  \mathsf  d^*_{f,h}w,\  \mathsf d^*_{f,h}\mathsf d_{f,h} w,\  \mathsf  d_{f,h}\mathsf d^*_{f,h} w
\in  \Lambda  L^{2} ( \mathsf \Omega ^{\textsc{M}} ), \\
& \mathbf{t}w|_{\Gamma^{\textsc{M}}_{\mbf D}}=0,\ 
\mathbf{t}\mathsf d^*_{f,h}w|_{\Gamma^{\textsc{M}}_{\mbf D}}=0,\ 
\mathbf{n}w|_{\Gamma^{\textsc{M}}_{\mbf N}}=0,\ 
\mathbf{n}\mathsf d_{f,h}w|_{\Gamma^{\textsc{M}}_{\mbf N}}=0
\Big\}.
\end{aligned}
\end{equation}
  In addition, the domain $\mathcal D\big(  Q_{f,h}^{\textsc{M},(q)} ( \mathsf \Omega ^{\textsc{M}} ) \big)$ of the closed quadratic form $Q_{f,h}^{\textsc{M},(q)} ( \mathsf \Omega ^{\textsc{M}} )$
associated with  $\Delta_{f,h}^{\textsc{M},(q)} ( \mathsf \Omega ^{\textsc{M}} )$
is given by
\begin{align*}
 \mathcal D\big( Q^{\textsc{M},(q)}_{f,h} ( \mathsf \Omega
  ^{\textsc{M}} )\big )&=  \mathcal  D \big(\mathsf d_{f,h,\mbf T}^{(q)} ( \mathsf \Omega ^{\textsc{M}} )\big)\cap
\mathcal D\big(\delta_{f,h,\mbf N}^{(q)} ( \mathsf \Omega ^{\textsc{M}} )\big)\\
&=\left\{w\in \Lambda^qH_{\mathsf d}(\mathsf \Omega^{\textsc{M}})\cap \Lambda^qH_{\mathsf d^*}(\mathsf \Omega^{\textsc{M}}),\, 
\mathbf t w|_{\Gamma^{\textsc{M}}_{\mbf D}}=0 \text{ and } 
\mathbf n w|_{\Gamma^{\textsc{M}}_{\mbf N}}=0\right\}
\end{align*}
\label{page.qm}
and for any $u,w\in\mathcal   D\big (  Q^{\textsc{M},(q)}_{f,h} ( \mathsf \Omega ^{\textsc{M}} )\big )$,  $$ Q_{f,h}^{\textsc{M},(q)} ( \mathsf \Omega ^{\textsc{M}} )(u,w)=\langle \mathsf d_{f,h,\mbf T} u,\mathsf d_{f,h,\mbf T}w \rangle_{L^2(\mathsf \Omega ^{\textsc{M}})} +\langle \delta_{f,h,\mbf N}u,\delta_{f,h,\mbf N}w \rangle_{L^2(\mathsf \Omega ^{\textsc{M}})}.$$ 
\end{proposition}

Let us mention an important consequence of Proposition~\ref{pr.QTN}. For $w\in \mathcal D\big  (\Delta_{f,h}^{M,(q)} ( \mathsf \Omega ^{\textsc{M}} )\big )$, the traces $\mathbf{t}\mathsf d^*_{f,h}w$ and $\mathbf{n} \mathsf  d_{f,h} w$ are a
priori defined in  $ \Lambda H^{-\frac12} ( \pa \mathsf \Omega ^{\textsc{M}} )$ but actually belong
to~$ \Lambda L^{2} (\pa  \mathsf \Omega ^{\textsc{M}} )$. Indeed,
 $\mathbf{n}\mathsf d_{f,h}u|_{\Gamma^{\textsc{M}}_{\mbf N}}=0$ by
 definition of $\mathcal D\big(\Delta_{f,h}^{\textsc{M},(q)} ( \mathsf
 \Omega ^{\textsc{M}} )\big )$ and $\mathbf t \mathsf
 d_{f,h}w|_{\Gamma^{\textsc{M}}_{\mbf D}}=0$ using~\eqref{eq.dom-dT}. Therefore,
 $\mathsf d_{f,h}w$ is in $\mathcal D\big(  Q^{\textsc{M},(q+1)}_{f,h} ( \mathsf \Omega ^{\textsc{M}}  )  \big )$ and therefore has
 a trace in $ \Lambda L^{2} (\partial  \mathsf \Omega ^{\textsc{M}} )$ according to Proposition~\ref{pr.QTN}. This argument also holds for $\mathsf d^{*}_{f,h}w\in \mathcal D\big ( Q^{\textsc{M},(q-1)}_{f,h} ( \mathsf \Omega ^{\textsc{M}} )\big  )$.

We end up this section with a Green formula which will be frequently
used in the sequel (see \cite[Lemma 2.10]{peutrec2020bar}).
\begin{lemma}
\label{le.GreenWeak}
 Let us assume that $\mathsf \Omega ^{\textsc{M}}$, $\Gamma^{\textsc{M}}_{\mbf D}$, and   $\Gamma^{\textsc{M}}_{\mbf N}$ satisfy the assumptions of Proposition~\ref{pr.QTN}. Let    $q=0,\ldots,d$.    
Let $\varphi$ be a real-valued Lipschitz
function on $\overline{\mathsf \Omega^{\textsc{M}}}$. Then, for any
$w\in \mathcal D\big (   Q^{\textsc{M},(q)}_{f,h} ( \mathsf
\Omega^{\textsc{M}} )\big)$, one has: 
\begin{equation}
\label{eq.compfor}
\begin{aligned}
 {Q}^{\textsc{M},(q)}_{f,h} ( \mathsf \Omega^{\textsc{M}} )( w,e^{\frac 2h \varphi} w)&=
 h^{2}\big  \| \mathsf d (e^{\frac{\varphi}{h}} w)\big \|^{2}_{ L^{2} ( \mathsf \Omega^{\textsc{M}} )}+
h^{2} \big \|  \mathsf d^{*} (e^{\frac{\varphi}{h}} w) \big \|^{2}_{ L^{2} ( \mathsf \Omega^{\textsc{M}} )}\\
&\quad + \big \langle
(|\nabla f|^{2}-|\nabla \varphi|^{2}+h\mathcal{L}_{\nabla
  f}+h\mathcal{L}_{\nabla
  f}^{*})e^{\frac{\varphi}{h}} w, e^{\frac{\varphi}{h}}
 w\big  \rangle_{ L^{2} ( \mathsf \Omega^{\textsc{M}} )}\\
&\quad +h \left( \int_{\Gamma^{\textsc{M}}_{\mbf N}} - \int_{\Gamma^{\textsc{M}}_{\mbf D}}\right) 
\langle w,w
\rangle_{ T_{\sigma}^{*}\mathsf \Omega^{\textsc{M}}}\;e^{\frac{2 }{h}\varphi }\partial_{\mathsf n_{  \mathsf \Omega^{\textsc{M}} } }
f\, d\sigma,
\end{aligned}
\end{equation}
where we recall that $\mathcal{L}$ stands for the Lie derivative.
 Moreover, when $  w\in
\mathcal D\big (   \Delta_{f,h}^{\textsc{M} ,(q)} ( \mathsf \Omega^{\textsc{M}} )\big)$, the left-hand
 side of \eqref{eq.compfor} equals $\langle
  e^{\frac 2h \varphi}\Delta_{f,h}^{\textsc{M} ,(q)}  ( \mathsf \Omega^{\textsc{M}} ) w,w\rangle_{L^{2} ( \mathsf \Omega^{\textsc{M}} )}$.   
\end{lemma}
In the following, we will use this Lemma several times with $(\mathsf
\Omega^{\textsc{M}}, \Gamma^{\textsc{M}}_{\mbf D},
\Gamma^{\textsc{M}}_{\mbf N}) = (\mathsf \Omega_k^{\textsc{M}},
\Gamma_{k,\mbf D}^{\textsc{M}}, \Gamma_{k,\mbf N}^{\textsc{M}})$ (for
$k \in \{1, \ldots,n\}$), in which case $\partial_{\mathsf n_{ \mathsf \Omega^{\textsc{M}}}}f=0$ on
$\Gamma^{\textsc{M}}_{\mbf D}$ and $\partial_{\mathsf n_{ \mathsf \Omega^{\textsc{M}}}}f>0$ on
$\Gamma^{\textsc{M}}_{\mbf N}$ (see items 1(b) and 2 in Proposition~\ref{pr.omegakpoint}).

\subsubsection{Spectral properties of $\Delta_{f,h}^{\textsc{M}, (q)}(  \mathsf \Omega_k^{\textsc{M}} )$}
\label{sec.specM}

 In view of Proposition~\ref{pr.omegakpoint}, the results of Section~\ref{sec.Witten-general-result} can be applied, for any
 $k\in\{1,\ldots,n\}$, to $(\mathsf \Omega^{\textsc{M}}, \Gamma^{\textsc{M}}_{\mbf D}, \Gamma^{\textsc{M}}_{\mbf N}) = (\mathsf \Omega_k^{\textsc{M}}, \Gamma_{k,\mbf D}^{\textsc{M}}, \Gamma_{k,\mbf N}^{\textsc{M}})$.
The main result of this section concerns the spectrum of the operator
$\Delta_{f,h}^{\textsc{M}, (q)} (  \mathsf \Omega_k^{\textsc{M}} )$,
defined in Proposition~\ref{pr.DeltaM}.
\begin{proposition}
\label{pr.DeltaTN} 
Let us assume that assumption \autoref{A} is satisfied.  Let
$k\in\{1,\ldots,n\}$ and $ \mathsf \Omega_k^{\textsc{M}} $ be the domain introduced in Proposition~\ref{pr.omegakpoint}. For $q\in\{0,\ldots,d\}$, let $\Delta_{f,h}^{\textsc{M}, (q)}(  \mathsf \Omega_k^{\textsc{M}} )$ be the unbounded
 nonnegative self-adjoint  operator   on 
$ \Lambda^qL^{2} (  \mathsf \Omega_k^{\textsc{M}} )$ defined by \eqref{eq.DeltaTN}--\eqref{eq.domDeltaTN} with $(\mathsf \Omega^{\textsc{M}}, \Gamma^{\textsc{M}}_{\mbf D}, \Gamma^{\textsc{M}}_{\mbf N}) = (\mathsf \Omega_k^{\textsc{M}}, \Gamma_{k,\mbf D}^{\textsc{M}}, \Gamma_{k,\mbf N}^{\textsc{M}})$.
Then, the following holds true:
\begin{enumerate}
\item[(i)] The operator $\Delta_{f,h}^{\textsc{M},(q)} (   \mathsf \Omega_k^{\textsc{M}}  )$ has compact resolvent.
\item[(ii)] For any eigenvalue $\lambda$ of $\Delta_{f,h}^{\textsc{M}, (q)} (  \mathsf \Omega_k^{\textsc{M}} )$
and associated eigenform $w_{}^{(q)}\in \mathcal D\big (\Delta_{f,h}^{\textsc{M}, (q)} (  \mathsf \Omega_k^{\textsc{M}} )\big)$, one has
$
\mathsf d_{f,h}w_{}^{(q)}\in\mathcal  D\big (\Delta_{f,h}^{\textsc{M},(q+1)}(   \mathsf \Omega_k^{\textsc{M}}  ) \big)
$ and
$
\mathsf d^{*}_{f,h}w_{}^{(q)}\in \mathcal  D\big  (\Delta_{f,h}^{\textsc{M},(q-1)}(   \mathsf \Omega_k^{\textsc{M}}  ) \big)
$,
with
$$
\mathsf d_{f,h}\Delta_{f,h}^{\textsc{M}, (q)}(   \mathsf
\Omega_k^{\textsc{M}}  )  w_{}^{(q)}= 
\Delta_{f,h}^{\textsc{M}, (q+1)}(   \mathsf \Omega_k^{\textsc{M}}  )\mathsf d_{f,h}w_{}^{(q)}
= \lambda \mathsf d_{f,h}w_{}^{(q)}$$
and
$$
\mathsf d^*_{f,h}\Delta_{f,h}^{\textsc{M},(q)} (  \mathsf \Omega_k^{\textsc{M}} )w_{}^{(q)}= 
\Delta_{f,h}^{\textsc{M},(p-1)} (  \mathsf \Omega_k^{\textsc{M}} )\mathsf d^*_{f,h}w_{}^{(q)}
= \lambda\mathsf d^*_{f,h}w_{}^{(q)}.
$$
If in addition $\lambda\neq 0$, either $\mathsf d_{f,h}w_{}^{(q)}$ or $\mathsf d^*_{f,h}w_{}^{(q)}$ is non-zero.
\item[(iii)] There exist $\mathsf c>0$ and $h_0>0$ such that for any $q\in\{0,\dots,d\}$ and $h\in (0,h_0)$. 
$$
\dim\Ran \ \pi_{[0,\mathsf ch ]}\big (\Delta_{f,h}^{\textsc{M},(q)} (  \mathsf \Omega_k^{\textsc{M}} )\big )
=\left\{
  \begin{aligned}[c]
 1 \ &\text{if}\ q\in\{0,1\},\\
0 \ &\text{if} \ q\in\{2,\ldots,d\},
  \end{aligned}
\right.$$ 
In addition, for all $h\in (0,h_0)$,     there exists $\lambda(\mathsf \Omega_k^{\textsc{M}} )\ge 0$ such that for $q\in \{0,1\}$,
$$
\Sp \big(\Delta_{f,h}^{\textsc{M} ,(q)} (  \mathsf \Omega_k^{\textsc{M}} )\big)\cap [0,\mathsf ch ]
=
\{ \lambda(\mathsf \Omega_k^{\textsc{M}} )\}.
$$
\end{enumerate}
Finally, $\lambda(\mathsf \Omega_k^{\textsc{M}} )$ is non-zero  and  is exponentially small when $h\to 0$. 
\end{proposition}
\begin{proof}
Item  $(i)$  is a consequence of  the compactness
of the embedding $\Lambda^q H^{\frac12} ( \mathsf \Omega_k^{\textsc{M}}  )\hookrightarrow
 \Lambda^qL^{2} (  \mathsf \Omega_k^{\textsc{M}}  )$ and of the
 continuous inclusion $\mathcal D\big(\Delta_{f,h}^{\textsc{M} ,(q)} (
 \mathsf \Omega_k^{\textsc{M}} )\big) \hookrightarrow
 \Lambda^qH^{\frac12} ( \mathsf \Omega_k^{\textsc{M}}  )$ (see Proposition~\ref{pr.QTN}). 
Item  $(ii)$ is  a straightforward  consequence
of the characterization of the domain of $\Delta_{f,h}^{\textsc{M} ,(q)} (  \mathsf \Omega_k^{\textsc{M}} )$
together with~\eqref{eq.dom-dT}. Moreover, if
$\lambda\neq 0$, then
\begin{align*}
0 \neq  \lambda\Vert w_{}^{(q)}\Vert^2_{L^{2} ( \mathsf \Omega_k^{\textsc{M}})}&= \langle \Delta_{f,h}^{\textsc{M} ,(q)}  ( \mathsf \Omega_k^{\textsc{M}})w_{}^{(q)},w_{}^{(q)}\rangle_{L^{2} (  \mathsf \Omega_k^{\textsc{M}})}\\
&= \langle \mathsf d_{f,h} w^{(q)} ,\mathsf d_{f,h}w_{}^{(q)} \rangle_{L^{2} (  \mathsf \Omega_k^{\textsc{M}} )} +\langle \mathsf d^{*}_{f,h}w_{}^{(q)},\mathsf d^{*}_{f,h}w_{}^{(q)} \rangle_{L^{2} (  \mathsf \Omega_k^{\textsc{M}} )} 
\end{align*}
which implies that  either $\mathsf d_{f,h}w_{}^{(q)}$ or $\mathsf d^*_{f,h}w_{}^{(q)}$ is non-zero.
Let us now prove item $(iii)$ in  Proposition~\ref{pr.DeltaTN}. It is
a consequence of Lemma~\ref{le.GreenWeak} (with $\varphi=0$) and the
fact that the normal derivative of $f$ on $\Gamma_{k,\mbf
  N}^{\textsc{M}}$ is non negative (see item 2 in
Proposition~\ref{pr.omegakpoint})  together with arguments already used in Section~\ref{sec.numb-small}. Let us be more precise on this. The function $f$ is $\mathcal C^\infty$ on $\overline {\mathsf \Omega_k^{\textsc{M}}}$ and its critical points in~$\overline {\mathsf \Omega_k^{\textsc{M}}}$ are exactly $x_0$ and $z_k$. 
Moreover, for $\ve >0$ small enough, 
\begin{equation}
\label{eq.=vec}
\overline {\mathsf \Omega_k^{\textsc{M}}}\cap \{x\in \overline \Omega, x_d(x) \in [0,\ve]\}=  \overline {\mathsf C_\alpha }\cap \{x\in \overline \Omega, x_d(x)\in [0,\ve]\} =\overline {\mathsf C_\ve }.
\end{equation}
In particular, $ \mathsf \Omega_k^{\textsc{M}}$ is smooth near $z_k$ and 
\begin{equation}\label{eq.=normaCalpha}
\text{$\mathsf n_{{\mathsf \Omega_k^{\textsc{M}}}}=\mathsf n_{ \mathsf C_\alpha}=\mathsf n_{ \Omega}$ on $\Gamma_{k,\mbf D}^{\textsc{M}}$.}
\end{equation} 
From assumption \autoref{A},  it thus holds    $\partial_{\mathsf n_{ \mathsf \Omega_k^{\textsc{M}}}}f=0$ on $\Gamma_{k,\mbf D}^{\textsc{M}}$. 
Therefore, we can consider  $\mathsf V_{x_0}$ and  $\mathsf V_{z_k}$  two neighborhoods of respectively $x_0$ and $z_k$ in $\overline {\mathsf \Omega_k^{\textsc{M}}}$  such that  
 \begin{itemize}
\item   $\overline{\mathsf V_{x_0}}\subset {\mathsf \Omega_k^{\textsc{M}}}$ and $x_0$ is the only critical point of $f$ in $\overline{\mathsf V_{x_0}}$,
\item $\overline {\mathsf V_{z_k}} \cap \overline{\Gamma_{k,\mbf N}^{\textsc{M}}}=\emptyset$, $ \partial_{\mathsf n_{ \mathsf \Omega_k^{\textsc{M}}}}f=0$ on $\pa\overline {\mathsf \Omega_k^{\textsc{M}}}  \cap\overline {\mathsf V_{z_k}}$ 
and, $z_k$ is the only critical point of  $f$ in $\overline {\mathsf V_{z_k}}$, 
\item $\overline{\mathsf V_{x_0}}\cap \overline {\mathsf V_{z_k}}=\emptyset$. 
\end{itemize}
For  $y\in \{x_0,z_k\}$, let $\chi_y: \overline {\mathsf \Omega_k^{\textsc{M}}}\to [0,1]$ be a  $\mathcal C^\infty$ supported in $\mathsf V_y$ and such that  $\chi_y=1$ in a neighborhood of $y$ in $\overline {\mathsf \Omega_k^{\textsc{M}}}$. 
Then, one defines: 
$$\tilde \chi :=\sqrt{ 1-  \chi_{x_0}^2-\chi_{z_k}^2},$$
so that on $\overline {\mathsf \Omega_k^{\textsc{M}}}$, $\tilde \chi ^2+\chi_{x_0}^2+\chi_{z_k}^2=1$.  
Let $w\in \mathcal D(Q_{f,h}^{\textsc{M},(q)}(\mathsf \Omega_k^{\textsc{M}}))$.   
The IMS formula~\cite{CFKS,helffer-nier-06}  yields:
  \begin{align}
\nonumber
Q_{f,h}^{\textsc{M},(q)}(\mathsf \Omega_k^{\textsc{M}})(   w ) &=  Q_{f,h}^{\textsc{M},(q)}(\tilde \chi  w)    - h^2\,  \Vert w\, \nabla  \tilde \chi     \Vert_{L^2( \mathsf \Omega_k^{\textsc{M}} )}^2 + \sum_{y\in  \{x_0,z_k\}} Q_{f,h}^{\textsc{M},(q)}(\chi_y w)-      h^2\,\big \Vert w\,  \nabla \chi_y   \big \Vert_{L^2( \mathsf \Omega_k^{\textsc{M}}  )}^2.
\end{align}
This formula easily follows from  Lemma~\ref{le.GreenWeak} (with $\varphi=0$) and the fact that $\tilde \chi ^2+\chi_{x_0}^2+\chi_{z_k}^2=1$ and $\chi w\in \mathcal D(Q_{f,h}^{\textsc{M},(q)}(\mathsf \Omega_k^{\textsc{M}}))$ for any smooth function $\chi :\overline {\mathsf \Omega_k^{\textsc{M}}}\to \mathbb R$. 

In the following $C>0$ and $c>0$ are constants independent of $h$ and $w$, and which can change from one occurrence to another. 
Since   $\vert \nabla f\vert ^2\ge  c$ on the support of $\tilde \chi$ in  $\overline {\mathsf \Omega_k^{\textsc{M}}}$, one deduces  from   Lemma~\ref{le.GreenWeak} (applied to $\tilde \chi  w$ with  $\varphi=0$), and  the fact that $\partial_{\mathsf n_{ \mathsf \Omega_k^{\textsc{M}}}}f>0$ on $\Gamma_{k,\mbf N}^{\textsc{M}}$ that  for $h$ small enough 
$$Q_{f,h}^{\textsc{M},(q)}(\mathsf \Omega_k^{\textsc{M}})(\tilde \chi  w) \ge c \Vert w\,    \tilde \chi     \Vert_{L^2( \mathsf \Omega_k^{\textsc{M}} )}^2.$$
Then, using  the previous IMS formula, it holds for $h$ small enough,
\begin{equation}\label{eq.CIMSM}
Q_{f,h}^{\textsc{M},(q)}(\mathsf \Omega_k^{\textsc{M}})(   w )\ge c \Vert w\,    \tilde \chi     \Vert_{L^2( \mathsf \Omega_k^{\textsc{M}} )}^2 +  \sum_{y\in  \{x_0,z_k\}} Q_{f,h}^{\textsc{M},(q)}(\mathsf \Omega_k^{\textsc{M}})(\chi_y w) -C h^2 \Vert w \Vert_{L^2( \mathsf \Omega_k^{\textsc{M}} )}^2.
\end{equation}

Let us assume that  $q\ge 2$. Then,   by  the same analysis as in item~2 in Step~1.b and item~2 in Step~3 in Section~\ref{sec.numb-small},  one has (up to choosing $\mathsf V_{x_0}$ and $\mathsf V_{z_k}$ smaller), for all $y \in \{x_0,z_k\}$ and $h$ small enough
$ Q_{f,h}^{\textsc{M},(q)}(\mathsf \Omega_k^{\textsc{M}})(\chi_y w)\ge Ch   \Vert \chi_yw \Vert_{L^2( \mathsf \Omega_k^{\textsc{M}} )}^2.$
Hence, using \eqref{eq.CIMSM}, it follows that, when $q\ge 2$, 
$$Q_{f,h}^{\textsc{M},(q)}(\mathsf \Omega_k^{\textsc{M}})(  w)\ge Ch   \Vert w \Vert_{L^2( \mathsf \Omega_k^{\textsc{M}} )}^2.$$
This proves the first statement in item $(iii)$ in  Proposition~\ref{pr.DeltaTN} when $q\ge 2$. 

Let us now consider $q\in\{0,1\}$.  By the same analysis as in
item~2 in Step~1.b and  item~1 in Step~3  in
Section~\ref{sec.numb-small}, one has that (up to choosing $\mathsf V_{x_0}$ and $\mathsf V_{z_k}$ smaller) for $h$ small enough
$ Q_{f,h}^{\textsc{M},(0)}(\mathsf \Omega_k^{\textsc{M}})(\chi_{z_k} w)\ge Ch   \Vert \chi_{z_k} w \Vert_{L^2( \mathsf \Omega_k^{\textsc{M}} )}^2$
  and
  $  Q_{f,h}^{\textsc{M},(1)}(\mathsf \Omega_k^{\textsc{M}})(\chi_{x_0} w)\ge Ch   \Vert \chi_{x_0} w \Vert_{L^2( \mathsf \Omega_k^{\textsc{M}} )}^2 .$
Let us now assume that 
$$Q_{f,h}^{\textsc{M},(q)}(\mathsf \Omega_k^{\textsc{M}})(   w ) \le \mathsf ch\Vert w \Vert_{L^2( \mathsf \Omega_k^{\textsc{M}} )}^2,$$
for some $\mathsf c>0$. Using the same arguments than those used in Step 4 in Section~\ref{sec.numb-small} (up to choosing $\mathsf V_{x_0}$ and $\mathsf V_{z_k}$ smaller), one obtains that, if $q=0$ (resp. $q=1$),  $w$ is at a distance $(\sqrt {\mathsf c} +o(1))\Vert w \Vert_{L^2( \mathsf \Omega_k^{\textsc{M}} )}$ of the one dimensional vector space   spanned by   $\Phi_h^{x_0}= \chi_{x_0}\Psi_h^{x_0}\, /\,  \Vert \chi_{x_0} \Psi_h^{x_0}  \Vert_{L^2(\mathsf \Omega_k^{\textsc{M}})}$, see~\eqref{eq.span=12},~\eqref{eq.omega-z2}, and~\eqref{eq.PHIY} (resp. spanned by $\Phi_h^{z_k}=\chi_{z_k}\Psi_h^{z_k}\, /\,  \Vert \chi_{z_k} \Psi_h^{z_k}  \Vert_{L^2(\mathsf \Omega_k^{\textsc{M}})}$, see~\eqref{eq.span=11},~\eqref{eq.norme-case1-2}, and~\eqref{eq.PHIY}). Hence,  for $\mathsf c>0$ small enough and $h$ small enough
$$\dim\Ran \ \pi_{[0,\mathsf ch ]}\big (\Delta_{f,h}^{\textsc{M},(q)}
(  \mathsf \Omega_k^{\textsc{M}} )\big )\le 1.$$ 

Besides, using
Proposition~\ref{pr.DeltaM} , $\Phi_h^{x_0}\in  \mathcal
D(Q_{f,h}^{\textsc{M},(0)}(\mathsf \Omega_k^{\textsc{M}}))$ because
the function $\Phi_h^{x_0}$ is smooth  and is supported  in $\overline
{\mathsf V_{x_0}}\subset \mathsf \Omega_k^{\textsc{M}}$. It also holds
$\Phi_h^{z_k}\in  \mathcal D(Q_{f,h}^{\textsc{M},(1)}(\mathsf
\Omega_k^{\textsc{M}}))$. Indeed, the $1$-form $\Phi_h^{z_k}$ is
smooth,  supported in $ \overline {\mathsf V_{z_k}} \subset
\overline{\mathsf \Omega_k^{\textsc{M}}}$  and $\overline {\mathsf
  V_{z_k}} \cap \overline{\Gamma_{k,\mbf N}^{\textsc{M}}}=\emptyset$,
and therefore: $\mbf t \Phi_h^{z_k}= 0$ on $\Gamma_{k,\mbf
  D}^{\textsc{M}}$ and $\Phi_h^{z_k}=0$ on $\Gamma_{k,\mbf
  N}^{\textsc{M}}$. Using the Min-Max principle,
Equations~\eqref{eq.omega-z2} (with $y=x_0$)
and~\eqref{eq.norme-case1-2} (with $y=z_k$), one deduces that
$\Delta_{f,h}^{\textsc{M},(q)}(  \mathsf \Omega_k^{\textsc{M}} )$
admits at least one eigenvalue $\lambda^{\textsc{M}, (q)}$ of order
$O(h^{2})$ when $h\to 0$. This shows that $\dim\Ran \ \pi_{[0,\mathsf ch ]}\big (\Delta_{f,h}^{\textsc{M},(q)}
(  \mathsf \Omega_k^{\textsc{M}} )\big )= 1$ if $q \in \{0,1\}$.

Using the complex property (see $(ii)$
in Proposition~\ref{pr.DeltaTN}), it holds $\lambda^{\textsc{M},
  (0)}=\lambda^{\textsc{M}, (1)}=:\lambda( \mathsf \Omega_k^{\textsc{M}})$ for $h$ small enough. In addition,
$\lambda^{\textsc{M}, (0)}>0$ because $e^{-\frac 1hf}$ does not belong
to the domain of $\Delta_{f,h}^{\textsc{M},(0)}(  \mathsf
\Omega_k^{\textsc{M}} )$. Finally, the fact that $\lambda^{\textsc{M}, (0)}$ is exponentially
small when $h\to 0$  follows by standard arguments, using the test
function $\chi_{x_0} e^{-\frac 1hf}$ in the Min-Max principle for
$\Delta_{f,h}^{\textsc{M},(0)}(  \mathsf \Omega_k^{\textsc{M}} )$ (see
the end of the proof of Corollary~\ref{thm-pc} for a similar reasoning). The proof of Proposition~\ref{pr.DeltaTN} is complete.  \end{proof}

\subsubsection{Asymptotic equivalents of $\lambda(\mathsf
  \Omega_k^{\textsc{M}} )$ and of $\int_{\Gamma_{k,\mbf
      D}^{\textsc{M}}}\mathsf u^{(1)}_k \cdot \mathsf n_{  \mathsf
    \Omega_k^{\textsc{M}}  }\  e^{-\frac 1hf}$}\label{sec.equivM}

Let us now provide asymptotic results on the principal eigenvalue and
eigenform of $ \Delta_{f,h}^{\textsc{M} ,(1)} ( \mathsf \Omega_k^{\textsc{M}} ) $.
\begin{proposition}
\label{pr.VP-DeltaTN} 
Let us assume that assumption \autoref{A} is satisfied.  Let
$k\in\{1,\ldots,n\}$ and $ \mathsf \Omega_k^{\textsc{M}} $ be the
domain introduced in Proposition~\ref{pr.omegakpoint}. For $q\in\{0,1\}$, let $\Delta_{f,h}^{\textsc{M}, (q)}(  \mathsf \Omega_k^{\textsc{M}} )$ be the unbounded
 nonnegative self-adjoint  operator   on 
$ \Lambda^qL^{2} (  \mathsf \Omega_k^{\textsc{M}} )$ defined by \eqref{eq.DeltaTN}--\eqref{eq.domDeltaTN} with $(\mathsf \Omega^{\textsc{M}}, \Gamma^{\textsc{M}}_{\mbf D}, \Gamma^{\textsc{M}}_{\mbf N}) = (\mathsf \Omega_k^{\textsc{M}}, \Gamma_{k,\mbf D}^{\textsc{M}}, \Gamma_{k,\mbf N}^{\textsc{M}})$. 

Let $\lambda(\mathsf \Omega_k^{\textsc{M}} )$ be the principal
eigenvalue  of  $\Delta_{f,h}^{\textsc{M}, (q)}(  \mathsf
\Omega_k^{\textsc{M}} )$ (as introduced in item $(iii)$ of Proposition~\ref{pr.DeltaTN}). Then, it holds in the limit $h\to 0$:
\begin{align}
\label{eq.DL-blambda1}
 \lambda(\mathsf \Omega_k^{\textsc{M}} )=    \mathsf A_{x_0,z_k} h \, e^{-\frac 2h (f(z_k)-f(x_0))} (1+O(\sqrt h))
 \end{align}
 with 
\begin{align}
\label{eq.ax0zk}
 \mathsf A_{x_0,z_k}:=\frac{ 2\vert \mu_{z_k}\vert \big( {\rm det \ Hess } f   (x_0)   \big)^{\frac12}              }{  \pi \big \vert   {\rm det  \,  Hess }f (z_k) \big \vert   ^{\frac 12}}
  \end{align}
  where  $\mu_{z_k}$ is the negative eigenvalue of  ${\rm Hess }f (z_k)$.

Let  $\mathsf u^{(1)}_k$ be a   $L^2(\mathsf
\Omega_k^{\textsc{M}})$-normalized  eigenform of $
\Delta_{f,h}^{\textsc{M} ,(1)} ( \mathsf \Omega_k^{\textsc{M}} ) $
associated with the eigenvalue $\lambda( \mathsf
\Omega_k^{\textsc{M}})$. The $1$-form $\mathsf u^{(1)}_k$ is  unique
up to a multiplication by~$\pm 1$. This multiplicative factor can be
chosen such that:  in the limit $h\to 0$,
\begin{align}\label{eq.DL-boun1}
 \int_{\Gamma_{k,\mbf D}^{\textsc{M}}}\mathsf u^{(1)}_k \cdot \mathsf n_{  \mathsf \Omega_k^{\textsc{M}}  }\  e^{-\frac 1hf}=-\mathsf b_k h^{\mathsf m} \, e^{-\frac 1h f(z_k)} (1+O(\sqrt h)),\
\end{align}
where  
\begin{align}\label{eq.kappa_0}
\mathsf b_k:=\sqrt{\mathsf A_{x_0,z_k}\kappa_{x_0}}, 
\ \ \text {and } \ \ \kappa_{x_0}:=\frac{\pi^{\frac d2}}{ \sqrt{   \det \Hess f(x_0)}},\  \mathsf m= {\frac d4}-\frac 12.
\end{align}
 \end{proposition}
\begin{proof}
The proof of Proposition~\ref{pr.VP-DeltaTN} is divided into three steps.

\medskip

\noindent
\textbf{Step 1: Construction of the   quasi-mode $\varphi^{\textsc{M},(0)}_k$ for
  $\Delta_{f,h}^{\textsc{M}, (0)}(  \mathsf \Omega_k^{\textsc{M}} )$.}
Let   $\ve >0$ be small enough such that   $\overline{\Omega_{\mathsf K_{\alpha/2}}} \subset \{x_d> 5\ve\}$. Then it holds (see~\eqref{eq.=vec} and~\eqref{eq.Calpha}) 
\begin{equation}\label{eq.veve}
\overline {\mathsf \Omega_k^{\textsc{M}}}\cap \{x\in \overline \Omega, x_d(x) \in [0,4\ve]\}=  \overline {\mathsf C_\alpha }\cap \{x\in \overline \Omega, x_d(x)\in [0,4\ve]\} =  \overline {\mathsf C_{4\ve} }.
\end{equation} 
Notice that since $x_0\in \Omega_{\mathsf K_{\alpha/2}}$, it then holds 
\begin{equation}\label{eq.x_0ve}
x_0\in   {\mathsf \Omega_k^{\textsc{M}}}\cap \{x\in \overline \Omega, x_d(x) >4\ve\}.
\end{equation} 
Since $z_k$ belongs to the open set $\Gamma_{k,\mbf D}^{\textsc{M}}$,
one can consider  $r>0$ small enough such that $\overline{\mathsf B_{\partial \Omega}(z_k,r)}\subset \Gamma_{k,\mbf D}^{\textsc{M}}$ (where $\mathsf B_{\partial \Omega}(z_k,r)$ is the open ball of radius $r>0$ centred at $z_k$ in $\partial \Omega$). 
Define 
\begin{equation}\label{eq.Ve}
\mathsf V_{4\ve}^r(z_k) :=\{x\in \overline \Omega, \, \mathsf z(x)\in
\overline{\mathsf B_{\partial \Omega}(z_k,r)} \text{ and } x_d\in [0,4\ve]\}\subset \overline{\mathsf C_{4\ve}}.
\end{equation} 
A schematic representation of $\mathsf V_{4\ve}^r(z_k)$ is given in
Figure~\ref{fig:v-ve}.

\begin{figure}
\begin{center}
  \includegraphics[width=0.5\textwidth]{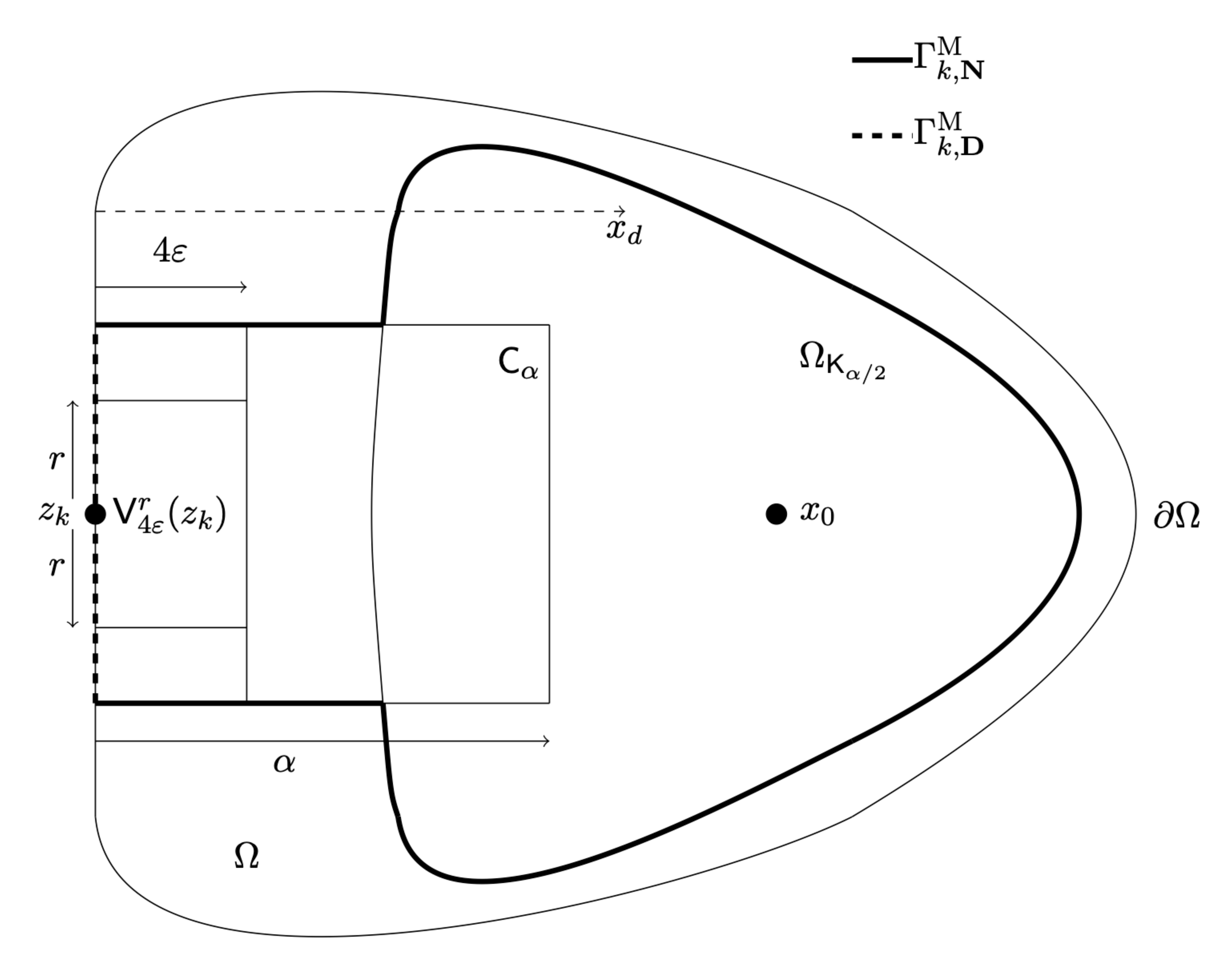}
\caption{A schematic representation of $\mathsf V_{4\ve}^r(z_k)$.}

\label{fig:v-ve}

\end{center}

\end{figure}

For each $z\in \overline{ \Gamma_{k,\mbf D} ^{\textsc{M}}}$, recall
that $x'=(x_1,\ldots,x_{d-1})$ is a system of coordinates    defined
in a neighborhood of   $z$  in $\partial \Omega$ such that
\eqref{eq.pauz}, \eqref{eq.pauz2}, and  \eqref{eq.pauz3} hold. Recall
that    $x\mapsto (\mathsf z(x),x_d(x))$  introduced
in~\eqref{eq.norm-teng} and~\eqref{eq.coord1} defines a $\mathcal C^\infty$ diffeomorphism on  $\{x\in\overline\Omega, x_d(x)\in [0,3\ve)\}$.  To ease the notation, from now on, we simply write $\{ x_d \in \mathsf O\}$ for the set $\{x\in \overline \Omega, x_d(x) \in \mathsf O\}$ for $\mathsf O\subset \mathbb R_+$.  
Recall that by~\eqref{g_y} and since $\mathsf n_{\mathsf M}(z_k)$ is an eigenvector of $\Hess f(z_k)$ for the eigenvalue $\mu_{z_k}<0$,
  up to choosing  $\ve>0$   and  $r>0$ smaller,
\begin{equation}\label{eq.DLzk}
\forall x=(x',x_d)\in \mathsf V_{4\ve}^r(z_k), \ f(x)\ =\ f(0)+ \frac 12 x' \cdot \text{Hess} f|_{\partial \Omega} (z_k)x'  - \frac{\vert \mu_{z_k}\vert }{2} x_d^2 + O(\vert x\vert^3).
\end{equation} 
Since $\text{Hess} f|_{\partial \Omega} (z_k)$ is  positive-definite, we may assume in the following that $r>0$ and $\ve>0$ are  small enough such that 
\begin{equation}\label{eq.min-zK}
\{z_k\}=\argmin_{ {  \mathsf V_{4\ve}^r(z_k)}} (f+  \vert \mu_{z_k} \vert x_d^2).
\end{equation} 
Moreover, because $\overline{ \Gamma_{k,\mbf D} ^{\textsc{M}}}\subset \Gamma_{z_k}$ and $\Gamma_{z_k}\subset \mathsf W_+^{z_k}$ (see \eqref{eq.incluWW}), one has 
$\{z_k\}=\argmin_{\overline{ \Gamma_{k,\mbf D} ^{\textsc{M}}}}f.$
With a slight abuse of notation, we still denote by $f$ the function $f$ in the $(\mathsf z,x_d)$ variable. 
Since $f(\mathsf z,x_d)= f(\mathsf z,0)+o_{\ve}(1)$ uniformly on $x=(\mathsf z,x_d)\in \overline{\mathsf C_{4\ve}}$ as $\ve \to 0$, it thus holds if  in addition $\mathsf z\in  \overline{ \Gamma_{k,\mbf D} ^{\textsc{M}}}\setminus \overline{\mathsf B_{\partial \Omega}(z_k,r)}$, $f(\mathsf z,x_d)\ge f(z_k)+c-o_{\ve}(1)$ for some $c>0$ independant of $x_d\in [0,4\ve]$. This implies that up to choosing $\ve >0$ smaller, it holds for some $c>0$, 
\begin{equation}\label{eq.min>>zK}
f> f(z_k)+c/2\ \text{on}\ \overline{\mathsf C_{4\ve}}\setminus \mathsf V_{4\ve}^r(z_k)=\{x=(\mathsf z,x_d), \mathsf z\in \overline{ \Gamma_{k,\mbf D} ^{\textsc{M}}}\setminus \overline{\mathsf B_{\partial \Omega}(z_k,r)}, x_d\in [0,4\ve]\}.
\end{equation} 

Let us consider   $\chi\in\mathcal  C^\infty(\mathbb R_+,[0,1])$    such that 
$
\text{supp } \chi\subset  [0,  \ve] \, \text{ and } \,   \chi=1 \text{ on }  [0,\ve/2].
$
 Inspired by~\cite{BEGK} (see also~\cite[Section 4.2]{DoNe2} and~\cite{di2017small}),   we build a    quasi-mode for $\Delta_{f,h}^{\textsc{M}, (0)}(  \mathsf \Omega_k^{\textsc{M}} )$  using the function $\phi_k^{\textsc{M},(0)}$  defined on $\overline{ \mathsf C_\alpha} \cap  \{x_d \in [0,2\ve]\}= \overline {\mathsf C_{2\ve}   }$ (see~\eqref{eq.Calpha})   by:
 \begin{equation}\label{eq.phi=def}
\forall x = (\mathsf z,x_d) \in   \overline {\mathsf C_{2\ve}   }, \  \  \phi^{\textsc{M},(0)}_k(\mathsf z,x_d) :=\frac{  \int_0^{x_d}\chi(t)e^{-\frac 1h\vert \mu_{z_k}\vert  \, t^2}  dt}{\int_0^{2\ve}\chi(t)\, e^{-\frac 1h\vert \mu_{z_k}\vert  \, t^2}  dt}.
\end{equation}
Notice that the function $\phi^{\textsc{M},(0)}_k$   only depends on the variable $x_d$.
Moreover, one has:
$$
    \phi^{\textsc{M},(0)}_{k}\in   \mathcal C^\infty\big ( \overline {\mathsf C_{2\ve}   }\big ) \text{ and } 
\forall x = (\mathsf z,x_d) \in  \overline {\mathsf C_{2\ve}   },\, \phi^{\textsc{M},(0)}_k(x)=1 \text{  if  } x_d \in [\ve,2\ve].
$$
Let us set for $x = (\mathsf z,x_d) \in \overline {\mathsf C_{2\ve}
}$: 
\begin{equation}\label{eq.defPsi}
\psi^{\textsc{M},(0)}_k(x)= \phi^{\textsc{M},(0)}_k(\mathsf z,x_d).
\end{equation}
We extend $\psi^{\textsc{M},(0)}_k$ from   $ \overline {\mathsf C_{2\ve} } =\overline {\mathsf \Omega_k^{\textsc{M}}}\cap \{  x_d \in [0,2\ve]\}$ (see~\eqref{eq.veve})  to $ \overline {\mathsf \Omega_k^{\textsc{M}}}$ by setting 
$\psi^{\textsc{M},(0)}_k=1 \text{ on } \overline {\mathsf \Omega_k^{\textsc{M}}}\cap \{  x_d> 2\ve\}=  \overline {\mathsf \Omega_k^{\textsc{M}}}\setminus \overline {\mathsf C_{2\ve} } .$
 One then has $\psi^{\textsc{M},(0)}_k\in  \mathcal C^\infty\big ( \overline {\mathsf \Omega_k^{\textsc{M}}} \big )$. 
 Notice that from~\eqref{eq.x_0ve}, 
 \begin{equation}\label{eq.psi=1x_0}
\psi^{\textsc{M},(0)}_k= 1 \text{ in a neighborhood of $x_0$ in $\mathsf \Omega_k^{\textsc{M}}$}.
\end{equation}
Then, define on $ \overline{\mathsf \Omega_k^{\textsc{M}}}$:
\begin{equation}\label{eq.QM-M0}
\varphi^{\textsc{M},(0)}_k=\frac{\psi^{\textsc{M},(0)}_k\, e^{-\frac 1hf} }{   \Vert \psi^{\textsc{M},(0)}_k\, e^{-\frac 1hf} \Vert_{L^2( \mathsf \Omega_k^{\textsc{M}})}   } \in  \mathcal C^\infty\big ( \overline {\mathsf \Omega_k^{\textsc{M}}} \big ).
\end{equation}

Let us check that $\varphi^{\textsc{M},(0)}_k$ belongs to the domain
of $\Delta_{f,h}^{\textsc{M}, (0)}(  \mathsf \Omega_k^{\textsc{M}} )$,
defined in~\eqref{eq.domDeltaTN}. Because it is smooth on the bounded set $\overline{\mathsf \Omega_k^{\textsc{M}}}$, one just has to check that it satisfies the boundary conditions on $\pa  \mathsf \Omega_k^{\textsc{M}} $. By definition of $\phi^{\textsc{M},(0)}_k$ above, $\varphi^{\textsc{M},(0)}_k(x)=0$ for all $x=(\mathsf z,x_d) \in  \overline{ \mathsf C_{2\ve}} \cap  \{x_d =0\} = \overline{\Gamma_{k,\mbf D}^{\textsc{M}}}$ (see~\eqref{eq.Calpha}). Let us now check that $\pa _{\mathsf n_{ \mathsf \Omega_k^{\textsc{M}} }}(e^{\frac fh} \varphi^{\textsc{M},(0)}_k)(x)=0$ for a.e.   $x\in \Gamma_{k,\mbf N}^{\textsc{M}}$, i.e. that $\pa _{\mathsf n_{ \mathsf \Omega_k^{\textsc{M}} }}\psi^{\textsc{M},(0)}_k(x)=0$ for a.e.   $x\in \Gamma_{k,\mbf N}^{\textsc{M}}$. Recall that $\Gamma_{k,\mbf N}^{\textsc{M}}=\pa \mathsf \Omega_k^{\textsc{M}} \cap \Omega$ (see~\eqref{eq.onm}). Let us first consider the case $x\in  \Gamma_{k,\mbf N}^{\textsc{M}}\cap \{x_d\in (0,3\ve)\}$. 
From~\eqref{eq.veve} and~\eqref{eq.Blat}, it holds:
\begin{equation}\label{eq.=Sve}
 \pa \mathsf \Omega_k^{\textsc{M}}  \cap \{x_d\in (0,3\ve)\} =\Sigma_{3\ve}^{lateral}.
\end{equation}
Because  $\overline{\Omega_{\mathsf K_{\alpha/2}}} \subset \{x_d> 5\ve\}$ and  $\Sigma_{3\ve}^{lateral} \subset \pa \mathsf C_\alpha$, one deduces that (see~\eqref{eq.=pa-domk-union} and~\eqref{eq.A1A2}),
$$\Sigma_{3\ve}^{lateral}= \Gamma_{k,\mbf N}^{\textsc{M}}\cap \{x_d\in (0,3\ve)\} \subset  \mathsf A_2.$$
Then using~\eqref{eq.nn=A2} and~\eqref{eq.normalLATERAL}, $\mathsf n_{ \mathsf \Omega_k^{\textsc{M}}}=-  {\nabla x_1}/ {\vert \nabla x_1\vert}$ on $\Sigma_{3\ve}^{lateral}$.  Since $\nabla \psi^{\textsc{M},(0)}_k$ is collinear to $\nabla x_d$ on $\overline{  \mathsf C_{3\ve} }$ which is, in view of~\eqref{g_y}, orthogonal to $\nabla x_1$, it holds:
$$\pa _{\mathsf n_{ \mathsf \Omega_k^{\textsc{M}} }}\psi^{\textsc{M},(0)}_k(x)=0 \text{ for  } x \in \Gamma_{k,\mbf N}^{\textsc{M}} \cap \Sigma_{3\ve}^{lateral}.$$
Let us now consider  the case $x\in  \Gamma_{k,\mbf N}^{\textsc{M}}\cap \{x_d\ge 3\ve\}$. 
Because $\psi^{\textsc{M},(0)}_k=1$ on $ \overline{\mathsf \Omega_k^{\textsc{M}} }\setminus \overline{\mathsf C_{\ve}}= \overline{\mathsf \Omega_k^{\textsc{M}} }\cap\{x_d>\ve\}$ (therefore $\vert \nabla \psi^{\textsc{M},(0)}_k\vert (x)=0$ on this set)   and $\Gamma_{k,\mbf N}^{\textsc{M}}\cap\{x_d\ge 3\ve\}= \pa \mathsf \Omega_k^{\textsc{M}} \cap\{x_d\ge 3\ve\}$, it holds, 
$$\pa _{\mathsf n_{ \mathsf \Omega_k^{\textsc{M}} }}\psi^{\textsc{M},(0)}_k(x)=0 \text{ for a.e.  } x\in \Gamma_{k,\mbf N}^{\textsc{M}}\cap\{x_d\ge 3\ve\}.$$
In conclusion, one has
\begin{equation}\label{eq.psi-DM}
\varphi^{\textsc{M},(0)}_k \in \mathcal D(  \Delta_{f,h}^{\textsc{M}, (0)}(  \mathsf \Omega_k^{\textsc{M}} ) ) .
\end{equation}

\medskip
\noindent
\textbf{Step 2: Asymptotic estimates of $\langle \varphi^{\textsc{M},(0)}_k,\Delta_{f,h}^{\textsc{M}, (0)}(  \mathsf \Omega_k^{\textsc{M}} ) \varphi^{\textsc{M},(0)}_k \rangle_{L^2( \mathsf \Omega_k^{\textsc{M}})}$ and $\Vert \Delta_{f,h}^{\textsc{M}, (0)}(  \mathsf \Omega_k^{\textsc{M}} )\varphi^{\textsc{M},(0)}_k \Vert_{L^2( \mathsf \Omega_k^{\textsc{M}})}^2$ as $h\to 0$. }
Let us first deal with $\mathsf Z_{z_k}:=\Vert \psi^{\textsc{M},(0)}_k\, e^{-\frac 1hf} \Vert_{L^2( \mathsf \Omega_k^{\textsc{M}})} $. Because $|\psi^{\textsc{M},(0)}_k|\le 1$, $\{x_0\}= \argmin_{\overline{\mathsf \Omega_k^{\textsc{M}}}} f$ (which follows from the fact that $x_0\in \mathsf \Omega_k^{\textsc{M}} \subset \Omega$ and Lemma~\ref{le.start})  and $\psi^{\textsc{M},(0)}_k=1$ near $x_0$ in  $\mathsf \Omega_k^{\textsc{M}}$ (see~\eqref{eq.psi=1x_0})
it holds, using Laplace's method, in the limit $h\to 0$:
\begin{equation}\label{lap.x0K}
\mathsf Z_{z_k}=  \sqrt{\kappa_{x_0}}\, h^{\frac d4}\, e^{-\frac 1h f(x_0)} \big(1+O(h)\big),
\end{equation}
where $\kappa_{x_0}$ is defined in~\eqref{eq.kappa_0}.
Let us now consider the term $\langle \varphi^{\textsc{M},(0)}_k,\Delta_{f,h}^{\textsc{M}, (0)}(  \mathsf \Omega_k^{\textsc{M}} ) \varphi^{\textsc{M},(0)}_k \rangle_{L^2( \mathsf \Omega_k^{\textsc{M}})}$. One has (see~\eqref{eq.psi-DM} and  Proposition~\ref{pr.DeltaM}):
\begin{align*}
\langle \varphi^{\textsc{M},(0)}_k,\Delta_{f,h}^{\textsc{M}, (0)}(  \mathsf \Omega_k^{\textsc{M}} ) \varphi^{\textsc{M},(0)}_k \rangle_{L^2( \mathsf \Omega_k^{\textsc{M}})}=    \int_{\mathsf \Omega_k^{\textsc{M}}} \vert \mathsf d_{f,h} \varphi^{\textsc{M},(0)}_k \vert^2    =  \frac{h^2 \int_{\mathsf C_{2\ve}} \vert \nabla  \psi^{\textsc{M},(0)}_k  \vert^2 e^{-\frac 2h f}   }{\mathsf Z_{z_k}^2}
\end{align*}
where we also used the fact that $\psi^{\textsc{M},(0)}_k=1$ on $
\overline{\mathsf \Omega_k^{\textsc{M}} }\setminus \overline{\mathsf
  C_{2\ve}}$. For $\eta\in [0,2\ve]$, set  $  \Gamma  (\eta)=\{ x
\in\overline{ \mathsf C_{2\ve}}, x_d(x)=\eta\}$. Note that  $ \Gamma
(0)=\Gamma^{\textsc{M}}_{k,\mbf D}$ and that for any $\eta\in
[0,2\ve]$, $\Gamma  (\eta)$ is naturally parametrized by
$\Gamma^{\textsc{M}}_{k,\mbf D}$ through the mapping  $\mathsf z\in
\Gamma^{\textsc{M}}_{k,\mbf D}\mapsto (\mathsf z,\eta)$ with Jacobian
determinant $\mathsf j(\mathsf z,\eta)$ with  $\mathsf j(\mathsf
z,0)=1$. One has using \eqref{eq.defPsi},~\eqref{eq.phi=def}, the co-area formula~\cite{evans-gariepy-18} ($dx=d\sigma_{\Gamma (\eta)}  |\nabla x_d|^{-1} d\eta$), and the fact that 
 $\vert \nabla x_d\vert=1$: 
\begin{align*}
 {h^2 \int_{\mathsf C_{2\ve}} \vert\nabla  \psi^{\textsc{M},(0)}_k  \vert^2 e^{-\frac 2h f}   } &= \frac{ \int_{\eta=0}^{2\ve} \chi^2(\eta)  \int_{\Gamma (\eta) }     |\nabla x_d|^2  \, e^{-\frac 2h (f  +\vert \mu_{z_k}\vert  \eta^2)}  d\sigma _{\Gamma (\eta)} |\nabla x_d|^{-1} d\eta  }{ \Big(\int_0^{2\ve}\chi(t)\, e^{-\frac 1h\vert \mu_{z_k}\vert  \, t^2}  dt\Big)^2  }\\
 &=\frac{ \int_{\eta=0}^{2\ve} \chi^2(\eta) \int_{\mathsf z \in  \Gamma_{k,\mbf D} ^{\textsc{M}}  }      e^{-\frac 2h (f(\mathsf z,\eta) +\vert \mu_{z_k}\vert  \eta^2)}  \mathsf j(\mathsf z,\eta) d\sigma _{\Gamma^{\textsc{M}}_{k,\mbf D}}  d\eta   }{ \Big(\int_0^{2\ve}\chi(t)\, e^{-\frac 1h\vert \mu_{z_k}\vert  \, t^2}  dt\Big)^2  }.
\end{align*}
A  straightforward computation  implies that there exists $c>0$ such that in
 the limit $h\to 0$,
\begin{align}\label{eq.chiUU2}
\mathsf N_{z_k}:=\int_0^{2\ve}\chi(t)\, e^{-\frac 1h\vert \mu_{z_k}\vert  \, t^2}  dt=\frac{\sqrt{\pi h} }{2\sqrt{\vert \mu_{z_k}\vert}} \big(1+O(e^{-\frac ch})\big).
\end{align}
Using~\eqref{eq.min>>zK} and~\eqref{eq.Ve}, one has for $h$ small enough:
\begin{align*}
\int_{\eta=0}^{2\ve} \chi^2(\eta) \int_{\mathsf z \in  \Gamma_{k,\mbf D} ^{\textsc{M}}  }  \!\!  \!\!  e^{-\frac 2h (f(\mathsf z,\eta) +\vert \mu_{z_k}\vert  \eta^2)}   \mathsf j(\mathsf z,\eta)d\sigma _{\Gamma^{\textsc{M}}_{k,\mbf D}}  d\eta&= \int_{\eta=0}^{2\ve} \chi^2(\eta) \int_{|\mathsf z |\le r }      \!\!  \!\!   e^{-\frac 2h (f(\mathsf z,\eta) +\vert \mu_{z_k}\vert  \eta^2)}  \mathsf j(\mathsf z,\eta) d\sigma _{\Gamma^{\textsc{M}}_{k,\mbf D}}  d\eta  \\
&\quad+O(e^{-\frac 2h (f(z_k)+c)}),
\end{align*}
for some $c>0$ independent of $h$. 
Using in addition~\eqref{eq.DLzk},   the same computations as in the proof of the   step 1.b of \cite[Proposition 24]{DoNe2}  imply that in the limit $h\to 0$:
$$\int_{\eta=0}^{2\ve} \chi^2(\eta) \int_{|\mathsf z |\le r }      \!\!  \!\!   e^{-\frac 2h (f(\mathsf z,\eta) +\vert \mu_{z_k}\vert  \eta^2)} \mathsf j(\mathsf z,\eta)  d\sigma _{\Gamma^{\textsc{M}}_{k,\mbf D}}  d\eta  =\frac{ (\pi h)^{\frac{d}{2} }  e^{-\frac 2h f(z_k)}}{2\sqrt{ \mu_1 \cdots \mu_{d-1} \vert \mu_{z_k}\vert}}
 \big(1+ O ( \sqrt h) \big),
$$
where the $O ( \sqrt h)$ is optimal in general. In conclusion, using also \eqref{lap.x0K}, one has as $h\to 0$:
\begin{equation}\label{eq.Ax0ps}
 \langle \varphi^{\textsc{M},(0)}_k,\Delta_{f,h}^{\textsc{M}, (0)}(  \mathsf \Omega_k^{\textsc{M}} ) \varphi^{\textsc{M},(0)}_k \rangle_{L^2( \mathsf \Omega_k^{\textsc{M}})}= \mathsf A_{x_0,z_k}h \ e^{-\frac 2h (f(z_k)-f(x_0))} (1+O(\sqrt h)).
 \end{equation}

Let us now consider the term $\Vert \Delta_{f,h}^{\textsc{M}, (0)}(  \mathsf \Omega_k^{\textsc{M}} )\varphi^{\textsc{M},(0)}_k \Vert_{L^2( \mathsf \Omega_k^{\textsc{M}})}^2$. Using~\eqref{eq:unitary_p0} and the definition of $\psi^{\textsc{M},(0)}_k$, it holds on $\mathsf \Omega_k^{\textsc{M}}$:
\begin{align} 
\label{eq.=LW}
\Delta_{f,h} ^{(0)}\varphi^{\textsc{M},(0)}_k&= \frac{2h e^{-\frac 1hf} }{\Vert \psi^{\textsc{M},(0)}_k\, e^{-\frac 1hf} \Vert_{L^2( \mathsf \Omega_k^{\textsc{M}})}}\Big (\,  \frac h2 \Delta_{\mbf H}^{(0)}+\nabla f\cdot \nabla \,  \Big)\psi^{\textsc{M},(0)}_k \ \text{ is supported in $\overline{\mathsf C_{2\ve}}$.} 
\end{align} 
By \eqref{eq.defPsi}, \eqref{eq.chiUU2}, and~\eqref{eq.phi=def},  for $h$ small enough,  $ \Vert \Delta_{\mbf H}^{(0)} \psi^{\textsc{M},(0)}_k\Vert_{L^\infty(\mathsf C_{2\ve})}$ and $\Vert \nabla\psi^{\textsc{M},(0)}_k \Vert_{L^\infty(\mathsf C_{2\ve})} $  are $O(h^\nu)$ for some $\nu\in \mathbb R$. 
Then, using~\eqref{eq.min>>zK} and \eqref{lap.x0K},  one has for $h$ small enough (see~\eqref{eq.Ve}):
\begin{align}
\nonumber 
\Vert \Delta_{f,h}^{\textsc{M}, (0)}(  \mathsf \Omega_k^{\textsc{M}} )\varphi^{\textsc{M},(0)}_k \Vert_{L^2( \mathsf \Omega_k^{\textsc{M}})}^2& =\Vert \Delta_{f,h}^{\textsc{M}, (0)}( \mathsf \Omega_k^{\textsc{M}} )\varphi^{\textsc{M},(0)}_k \Vert_{L^2( \mathsf C_{2\ve})}^2\\
\label{eq.NeG}
&= \Vert \Delta_{f,h}^{\textsc{M}, (0)}( \mathsf \Omega_k^{\textsc{M}} )\varphi^{\textsc{M},(0)}_k\Vert_{L^2(\mathsf V_{2\ve}^r(z_k))}^2    +O(e^{-\frac 2h (f(z_k)-f(x_0)+c)}),
\end{align}
for some $c>0$ independent of $h$. Let us recall that $\mbf g$ denotes
the  metric tensor in the $(x',x_d)$ coordinates (see \eqref{g_y}).
In the following, with a slight abuse of  notation, we also denote by
$\mbf g$ the matrix $(\mbf G, 0;0, 1)$.    
 In the $(x',x_d)$-coordinates, $\Delta_{\mbf H}^{(0)}$ writes 
 $$\Delta_{\mbf H}^{(0)} \phi^{\textsc{M},(0)}_k=-\frac{1}{\sqrt{\vert \mbf g\vert}}\sum_{i,j=1}^d \partial_{x_i}\big( \sqrt{\vert  \mbf g\vert} \,  \mbf g^{i,j} \partial_{x_j}\phi^{\textsc{M},(0)}_k  \big),$$
  where~$\vert  \mbf g\vert$ denotes the determinant of  $ \mbf g$  and $\mbf g^{i,j} $ the $(i,j)$ entry of $ \mbf g^{-1}$. 
%
%
Then, from  \eqref{eq.=LW} and \eqref{eq.phi=def}, one has on $\mathsf V_{2\ve}^r(z_k)$,
\begin{align*}
\Delta_{f,h} \varphi^{\textsc{M},(0)}_k&=\frac{2h\, e^{-\frac f h}}{\mathsf Z_{z_k} }\Big[ -\frac{h}{2\sqrt{\vert \mbf g\vert}}\sum_{i,j=1}^d \partial_{x_i}\big( \sqrt{\vert  \mbf g\vert} \,  \mbf g^{i,j} \partial_{x_j}\phi^{\textsc{M},(0)}_k  \big)   +  \, \sum_{i,j=1}^d  \mbf g^{i,j} \partial_{x_i}f\partial_{x_j}\phi^{\textsc{M},(0)}_k \,\Big]\\
&=  \frac{2h \, e^{-\frac 1h(f + \vert \mu_{z_k}\vert  x_d^2)}}{\mathsf Z_{z_k}\mathsf N_{z_k}}
\Big[ -\frac{h}{2\sqrt{\vert \mbf g\vert}}\sum_{i=1}^d \partial_{x_i}\big( \sqrt{\vert  \mbf g\vert} \,  \mbf g^{i,d}    \big)\chi(x_d)  +   \mbf g^{d,d}\chi (x_d)\vert \mu_{z_k}\vert  x_d  
 \\
 &\qquad \qquad \qquad \qquad \qquad +  \chi(x_d) \sum_{i=1}^d  \mbf g^{i,d} \partial_{x_i}f - \frac h2\,\chi'(x_d) \,  \mbf g^{d,d} \Big]  \\
 &=\frac{2h\,   e^{-\frac 1h(f + \vert \mu_{z_k}\vert  x_d^2)}}{\mathsf Z_{z_k}\mathsf N_{z_k}}
\big[ O(h) +  O(\vert x\vert ^2)  \big],
\end{align*}
\noindent
where   $\mathsf N_{z_k}$ is defined   in   \eqref{eq.chiUU2}, and where in the last inequality we have used that $\mbf g^{i,d}=0$ for $i=1,\ldots,d-1$ (see \eqref{g_y}), and $\partial_{x_d}f(x',x_d)= -\vert \mu_{z_k}\vert  x_d+O(\vert x\vert^2)$ (see \eqref{eq.DLzk}). 
Notice the cancelation of the $O(x)$ terms in the previous computations due to the precise form of the quasi-mode $\varphi^{\textsc{M},(0)}_k$. 
Thus, by \eqref{eq.min-zK}, \eqref{lap.x0K}, \eqref{eq.DLzk}, \eqref{eq.chiUU2}, and~\eqref{eq.Ax0ps}, one deduces using Laplace's method that as $h\to 0$,
\begin{align*}
\Vert \Delta_{f,h}^{\textsc{M}, (0)}( \mathsf \Omega_k^{\textsc{M}} )\varphi^{\textsc{M},(0)}_k\Vert_{L^2(\mathsf V_{\ve}^r(z_k))}^2&= \frac{ h^2  h^{\frac{d}{2} } }{  h^{\frac{d}{2} } h}\, O(h^2)       \,    e^{-\frac 2h (f(z_k)-f(x_0))} \\
&=O(h^2) \vert \langle \varphi^{\textsc{M},(0)}_k,\Delta_{f,h}^{\textsc{M}, (0)}(  \mathsf \Omega_k^{\textsc{M}} ) \varphi^{\textsc{M},(0)}_k \rangle_{L^2( \mathsf \Omega_k^{\textsc{M}})}\vert.
\end{align*}
Consequently, one deduces using \eqref{eq.NeG}, that 
\begin{equation}\label{eq.Ax0ps2}
\Vert \Delta_{f,h}^{\textsc{M}, (0)}( \mathsf \Omega_k^{\textsc{M}} )\varphi^{\textsc{M},(0)}_k\Vert_{L^2(\mathsf \Omega_k^{\textsc{M}})}=O(h) \sqrt{\vert \langle \varphi^{\textsc{M},(0)}_k,\Delta_{f,h}^{\textsc{M}, (0)}(  \mathsf \Omega_k^{\textsc{M}} ) \varphi^{\textsc{M},(0)}_k \rangle_{L^2( \mathsf \Omega_k^{\textsc{M}})}\vert}.
\end{equation}

\noindent
\textbf{Step 3:  End of the proof of
  Proposition~\ref{pr.VP-DeltaTN}.} 
Let us introduce the constant $\mathsf c>0$ from item~$(iii)$ in
Proposition~\ref{pr.DeltaTN}.  Because $\varphi^{\textsc{M},(0)}_k \in
\mathcal D(  \Delta_{f,h}^{\textsc{M}, (0)}(  \mathsf
\Omega_k^{\textsc{M}} ) )$, see indeed~\eqref{eq.psi-DM},  and since
$\lambda(\mathsf \Omega_k^{\textsc{M}} )$ is exponentially small when
$h\to 0$ (actually $o(h)$ as $h\to 0$ would be enough),
using   the fact that (see the proof of \cite[Proposition~27]{DoNe2})
$$(1-\pi_{[0,\mathsf ch ]}\big (\Delta_{f,h}^{\textsc{M},(0)} (  \mathsf \Omega_k^{\textsc{M}} )\big ))\,  \varphi^{\textsc{M},(0)}_k =-\frac{1}{2\pi i} \int_{C(\mathsf ch/2)} z^{-1}(z-   \Delta_{f,h}^{\textsc{M},(0)}  )^{-1}\Delta_{f,h}^{\textsc{M},(0)}  \varphi^{\textsc{M},(0)}_k  dz,$$
where $C(\mathsf ch/2)\subset \mathbb C$ is the circle of radius $\mathsf ch/2$ centered at $0$, 
  it holds for $h$ small enough
$$\Vert(1-\pi_{[0,\mathsf ch ]}\big (\Delta_{f,h}^{\textsc{M},(0)} (
\mathsf \Omega_k^{\textsc{M}} )\big ))\,  \varphi^{\textsc{M},(0)}_k
\Vert_{L^2(\mathsf \Omega_k^{\textsc{M}})} \le Ch^{-1}\Vert
\Delta_{f,h}^{\textsc{M}, (0)}( \mathsf \Omega_k^{\textsc{M}}
)\varphi^{\textsc{M},(0)}_k\Vert_{L^2(\mathsf
  \Omega_k^{\textsc{M}})}.$$
Therefore, using~\eqref{eq.Ax0ps2}, it holds:
\begin{equation}\label{eq.le-k}
\Vert(1-\pi_{[0,\mathsf ch ]}\big (\Delta_{f,h}^{\textsc{M},(0)} (
\mathsf \Omega_k^{\textsc{M}} )\big ))\,  \varphi^{\textsc{M},(0)}_k
\Vert_{L^2(\mathsf \Omega_k^{\textsc{M}})} \le C \sqrt{\vert \langle \varphi^{\textsc{M},(0)}_k,\Delta_{f,h}^{\textsc{M}, (0)}(  \mathsf \Omega_k^{\textsc{M}} ) \varphi^{\textsc{M},(0)}_k \rangle_{L^2( \mathsf \Omega_k^{\textsc{M}})}\vert}.
\end{equation}
In particular, using~\eqref{eq.Ax0ps} and the fact that $\Vert
\varphi^{\textsc{M},(0)}_k \Vert_{L^2(\mathsf \Omega_k^{\textsc{M}})}
=1$, there exists $c>0$ (because $f(z_k)>f(x_0)$ see
Lemma~\ref{le.start}) such that, for $h>0$ small enough, 
 \begin{align}
 \label{eq.norme-piO}
\Vert\pi_{[0,\mathsf ch ]}\big (\Delta_{f,h}^{\textsc{M},(0)} (  \mathsf \Omega_k^{\textsc{M}} )\big ) \varphi^{\textsc{M},(0)}_k \Vert_{L^2(\mathsf \Omega_k^{\textsc{M}})} =1+O(e^{-\frac ch}),
 \end{align}
and the following function is therefore well defined :
 \begin{align}
 \label{eq.QMK0=}
 \mathsf u^{(0)}_k= \frac{  \pi_{[0,\mathsf ch ]}\big (\Delta_{f,h}^{\textsc{M},(0)} (  \mathsf \Omega_k^{\textsc{M}} )\big ) \varphi^{\textsc{M},(0)}_k}{ \Vert\pi_{[0,\mathsf ch ]}\big (\Delta_{f,h}^{\textsc{M},(0)} (  \mathsf \Omega_k^{\textsc{M}} )\big ) \varphi^{\textsc{M},(0)}_k \Vert_{L^2(\mathsf \Omega_k^{\textsc{M}})}  }.
 \end{align}
One has
 \begin{align}
 \label{eq.lambdak-O=}
 \lambda(\mathsf \Omega_k^{\textsc{M}})&= \langle    \mathsf u^{(0)}_k   , \Delta_{f,h}^{\textsc{M},(0)} (  \mathsf \Omega_k^{\textsc{M}} )    \mathsf u^{(0)}_k   \big \rangle_{ L^2( \mathsf \Omega_k^{\textsc{M}})}\\
 \nonumber
 &= \langle   \pi_{[0,\mathsf ch ]}\big (\Delta_{f,h}^{\textsc{M},(0)} (  \mathsf \Omega_k^{\textsc{M}} )\big )  \varphi^{\textsc{M},(0)}_k   , \Delta_{f,h}^{\textsc{M},(0)} (  \mathsf \Omega_k^{\textsc{M}} )    \varphi^{\textsc{M},(0)}_k   \big \rangle_{ L^2( \mathsf \Omega_k^{\textsc{M}})} (1+O(e^{-\frac ch})),
 \end{align}
since the orthogonal projector $\pi_{[0,\mathsf ch ]}\big (\Delta_{f,h}^{\textsc{M},(0)} (  \mathsf \Omega_k^{\textsc{M}} )\big ) $ and $\Delta_{f,h}^{\textsc{M},(0)} (  \mathsf \Omega_k^{\textsc{M}} )$ commute  on $\mathcal D(\Delta_{f,h}^{\textsc{M},(0)} (  \mathsf \Omega_k^{\textsc{M}} ))$ and $\varphi^{\textsc{M},(0)}_k \in \mathcal D(  \Delta_{f,h}^{\textsc{M}, (0)}(  \mathsf \Omega_k^{\textsc{M}} ) )$. 
In addition, one has using \eqref{eq.le-k} and \eqref{eq.Ax0ps2},
 \begin{align*}
&\langle   \pi_{[0,\mathsf ch ]}\big (\Delta_{f,h}^{\textsc{M},(0)} (  \mathsf \Omega_k^{\textsc{M}} )\big )  \varphi^{\textsc{M},(0)}_k   , \Delta_{f,h}^{\textsc{M},(0)} (  \mathsf \Omega_k^{\textsc{M}} )    \varphi^{\textsc{M},(0)}_k   \big \rangle_{ L^2( \mathsf \Omega_k^{\textsc{M}})} \\
&=  \langle    \varphi^{\textsc{M},(0)}_k   , \Delta_{f,h}^{\textsc{M},(0)} (  \mathsf \Omega_k^{\textsc{M}} )     \varphi^{\textsc{M},(0)}_k   \big \rangle_{ L^2( \mathsf \Omega_k^{\textsc{M}})} - \langle    (1-\pi_{[0,\mathsf ch ]}\big (\Delta_{f,h}^{\textsc{M},(0)} (  \mathsf \Omega_k^{\textsc{M}} )\big ) )   \varphi^{\textsc{M},(0)}_k   ,  \Delta_{f,h}^{\textsc{M},(0)} (  \mathsf \Omega_k^{\textsc{M}} )     \varphi^{\textsc{M},(0)}_k  \big \rangle_{ L^2( \mathsf \Omega_k^{\textsc{M}})}\\
&=\langle    \varphi^{\textsc{M},(0)}_k   , \Delta_{f,h}^{\textsc{M},(0)} (  \mathsf \Omega_k^{\textsc{M}} )     \varphi^{\textsc{M},(0)}_k   \big \rangle_{ L^2( \mathsf \Omega_k^{\textsc{M}})}  (1+O(h)),\\
&=  \mathsf A_{x_0,z_k} h \ e^{-\frac 2h (f(z_k)-f(x_0))} (1+O(\sqrt h)).
\end{align*}
where we used~\eqref{eq.Ax0ps}. This proves 
~\eqref{eq.DL-blambda1}.
It remains to prove Equation~\eqref{eq.DL-boun1}. 

Let  $\mathsf u^{(1)}_k$ be a $L^2(\mathsf
\Omega_k^{\textsc{M}})$-normalized  eigenform of $
\Delta_{f,h}^{\textsc{M} ,(1)} ( \mathsf \Omega_k^{\textsc{M}} ) $
associated with the eigenvalue $\lambda( \mathsf
\Omega_k^{\textsc{M}})$. In view of item $(ii)$ and $(iii)$ in
Proposition~\ref{pr.DeltaTN} it holds ($\mathsf u^{(0)}_k$ is indeed a
$L^2(\mathsf
\Omega_k^{\textsc{M}})$-normalized principal eigenform of $\Delta_{f,h}^{\textsc{M},(0)}
(  \mathsf \Omega_k^{\textsc{M}} )$, see~\eqref{eq.QMK0=}), 
$\mathsf u^{(1)}_k=\pm \mathsf d_{f,h}   \mathsf u^{(0)}_k   /  \Vert \mathsf d_{f,h}    \mathsf u^{(0)}_k \Vert_{L^2(  \mathsf \Omega_k^{\textsc{M}} )}$.
Let us choose 
\begin{equation}\label{eq.u1k}
\mathsf u^{(1)}_k=\frac{ \mathsf d_{f,h}    \mathsf u^{(0)}_k  } { \mathsf N_k^{(1)}} \text{ with } \mathsf N_k^{(1)}=\big \Vert \mathsf d_{f,h}    \mathsf u^{(0)}_k\big \Vert_{L^2(   \mathsf \Omega_k^{\textsc{M}})} .
\end{equation}
From~\eqref{eq.lambdak-O=}, one has, $\lambda(\mathsf
\Omega_k^{\textsc{M}})= \langle    \mathsf d_{f,h} \mathsf u^{(0)}_k
, \mathsf d_{f,h}     \mathsf u^{(0)}_k   \big \rangle_{ L^2( \mathsf
  \Omega_k^{\textsc{M}})}=(\mathsf N_k^{(1)})^2$, and thus, using \eqref{eq.QMK0=} and the fact that $\mathsf d_{f,h} \pi_{[0,\mathsf ch ]}\big (\Delta_{f,h}^{\textsc{M},(0)} (  \mathsf \Omega_k^{\textsc{M}} )\big )= \pi_{[0,\mathsf ch ]}\big (\Delta_{f,h}^{\textsc{M},(1)} (  \mathsf \Omega_k^{\textsc{M}} )\big ) \mathsf d_{f,h}$ (see item $(ii)$ in Proposition~\ref{pr.DeltaTN}),
 \begin{align*}
 \lambda(\mathsf \Omega_k^{\textsc{M}})&=\frac{\mathsf N_k^{(1)} \,\langle    \mathsf d_{f,h} \varphi^{\textsc{M},(0)}_k   ,   \mathsf u^{(1)}_k  \big \rangle_{ L^2( \mathsf \Omega_k^{\textsc{M}})}}{   \Vert\pi_{[0,\mathsf ch ]}\big (\Delta_{f,h}^{\textsc{M},(0)} (  \mathsf \Omega_k^{\textsc{M}} )\big ) \varphi^{\textsc{M},(0)}_k \Vert_{L^2(\mathsf \Omega_k^{\textsc{M}})} } \\
 &=\frac{\mathsf N_k^{(1)} \,\langle    \mathsf d  \psi^{\textsc{M},(0)}_k   ,  h e^{-\frac 1hf}   \mathsf u^{(1)}_k  \big \rangle_{ L^2( \mathsf \Omega_k^{\textsc{M}})}}{ \Vert\pi_{[0,\mathsf ch ]}\big (\Delta_{f,h}^{\textsc{M},(0)} (  \mathsf \Omega_k^{\textsc{M}} )\big ) \varphi^{\textsc{M},(0)}_k \Vert_{L^2(\mathsf \Omega_k^{\textsc{M}})}  \mathsf Z_{z_k} },
 \end{align*}
 where we also used~\eqref{eq.QM-M0} at the last line. Therefore, because $\mathsf N_k^{(1)}=\sqrt { \lambda(\mathsf \Omega_k^{\textsc{M}})}$, it follows from~\eqref{eq.DL-blambda1},~\eqref{lap.x0K} and~\eqref{eq.norme-piO}, that, as  $h\to 0$, it holds:
\begin{equation}\label{eq.P11}
\langle    \mathsf d  \psi^{\textsc{M},(0)}_k   ,  h e^{-\frac 1hf}   \mathsf u^{(1)}_k  \big \rangle_{ L^2( \mathsf \Omega_k^{\textsc{M}})}=\sqrt{\mathsf A_{x_0,z_k}\, h\, \kappa_{x_0}\, h^{\frac d2}}\, e^{-\frac 1h f(z_k)} (1+O(\sqrt h)).
\end{equation}
Besides,  using the fact that $\mathsf
u^{(1)}_k \in \mathcal D( \Delta_{f,h}^{\textsc{M},(1)} (  \mathsf
\Omega_k^{\textsc{M}} ))$ (see item $(ii)$ in
Proposition~\ref{pr.DeltaTN}) and the Green formula \eqref{eq.i-n-u}, one deduces that
\begin{align}
\nonumber
&\langle    \mathsf d  \psi^{\textsc{M},(0)}_k   ,  h e^{-\frac 1hf}   \mathsf u^{(1)}_k  \big \rangle_{ L^2( \mathsf \Omega_k^{\textsc{M}})}= -\langle    \mathsf d (1- \psi^{\textsc{M},(0)}_k)   ,  h e^{-\frac 1hf}   \mathsf u^{(1)}_k  \big \rangle_{ L^2( \mathsf \Omega_k^{\textsc{M}})}\\
\label{eq.P12}
&= \langle     ( 1- \psi^{\textsc{M},(0)}_k ) e^{-\frac 1hf} ,  \mathsf d_{f,h}^*    \mathsf u^{(1)}_k  \big \rangle_{ L^2( \mathsf \Omega_k^{\textsc{M}})}  -h\int_{\pa \mathsf \Omega_k^{\textsc{M}}} (1- \psi^{\textsc{M},(0)}_k)  \mathsf u^{(1)}_k \cdot \mathsf n_{  \mathsf \Omega_k^{\textsc{M}}  }  e^{-\frac 1hf},
\end{align}
(where here and in the following, we use the notation $ \mathsf
u^{(1)}_k \cdot \mathsf n_{  \mathsf \Omega_k^{\textsc{M}}
}=\mathbf{i}_{\mathsf n_{  \mathsf \Omega_k^{\textsc{M}}
}}  \mathsf
u^{(1)}_k$) and
$$\mathsf u^{(1)}_k
\cdot \mathsf n_{  \mathsf \Omega_k^{\textsc{M}}  }= 0 \text{ on } \Gamma_{k,\mbf N}^{\textsc{M}}.$$
Moreover, $\psi^{\textsc{M},(0)}_k=0$ on $\Gamma_{k,\mbf
  D}^{\textsc{M}}$. 
Thus,
\begin{equation}\label{eq.P13}
\int_{\pa \mathsf \Omega_k^{\textsc{M}}} (1- \psi^{\textsc{M},(0)}_k)  \mathsf u^{(1)}_k \cdot \mathsf n_{  \mathsf \Omega_k^{\textsc{M}}  } \, e^{-\frac 1hf}=\int_{\Gamma_{k,\mbf D}^{\textsc{M}}}  \mathsf u^{(1)}_k \cdot \mathsf n_{  \mathsf \Omega_k^{\textsc{M}}  } \, e^{-\frac 1hf}.
\end{equation}
Let us now deal with the term $\langle     ( 1- \psi^{\textsc{M},(0)}_k ) e^{-\frac 1hf} ,  \mathsf d_{f,h}^*    \mathsf u^{(1)}_k  \big \rangle_{ L^2( \mathsf \Omega_k^{\textsc{M}})}$. It holds,
\begin{align*}
\big \vert  \langle     ( 1- \psi^{\textsc{M},(0)}_k ) e^{-\frac 1hf} ,  \mathsf d_{f,h}^*    \mathsf u^{(1)}_k  \big \rangle_{ L^2( \mathsf \Omega_k^{\textsc{M}})}\big \vert &\le\big  \Vert ( 1- \psi^{\textsc{M},(0)}_k ) e^{-\frac 1hf} \big \Vert_{ L^2( \mathsf \Omega_k^{\textsc{M}})} \sqrt{ \lambda(\mathsf \Omega_k^{\textsc{M}}) }\\
&\le Ce^{-\frac 1h \min_{\text{supp}(1- \psi^{\textsc{M},(0)}_k )} f} \sqrt{ \lambda(\mathsf \Omega_k^{\textsc{M}}) }\\
&\le Ce^{-\frac 1h ( f(x_0)+\delta)} \sqrt{ \lambda(\mathsf \Omega_k^{\textsc{M}}) },\\
&\le Ce^{-\frac 1h (f(z_k)+\delta)},
\end{align*}
  where we used   the fact that, from~\eqref{eq.psi=1x_0} and since $x_0$ is the global minimum of $f$ in $\overline \Omega$ (see Lemma~\ref{le.start}), $\min_{\text{supp}(1- \psi^{\textsc{M},(0)}_k )} f\ge  f(x_0)+\delta$, for some $\delta>0$. Notice that we also used~\eqref{eq.DL-blambda1}  at the last line of the previous computation. 
Equation~\eqref{eq.DL-boun1} then follows from the previous inequality together with~\eqref{eq.P11}, \eqref{eq.P12}, and \eqref{eq.P13}.
 This concludes the proof of  Proposition~\ref{pr.VP-DeltaTN}. 
\end{proof}

\subsubsection{Agmon estimates on $\mathsf u^{(1)}_k$}\label{sec.Agmon1}

The aim of this section is to prove  that $\mathsf u^{(1)}_k$ (the
principal eigenform of $ \Delta_{f,h}^{\textsc{M} ,(1)} ( \mathsf \Omega_k^{\textsc{M}} ) $) decays
exponentially fast away from $z_k$ (see Proposition~\ref{pr.Agmon} below): these are so-called  Agmon
estimates.

Recall the  Definition~\ref{de.da} of the Agmon distance. These are basic properties of the Agmon distance    which follows from~\cite[Appendice 2]{helffer-sjostrand-85}, see also~\cite[Lemma 3.2]{LeNi}:
\begin{proposition}
\label{pr.p-agmon}
Let us assume that $f:\overline \Omega\to \mathbb R$ is a $\mathcal C^\infty$ function. Then, the Agmon pseudo-distance $(x,y)\in \overline \Omega \times\overline \Omega \mapsto \mathsf d_a(x,y)$ (see Definition~\ref{de.da}) is symmetric and satisfies the triangular inequality. In addition, it is a distance if $f$ has a finite number of critical points in $\overline \Omega$. 
Moreover, for any fixed $y \in \overline \Omega$,  $x \in \overline \Omega\mapsto
\mathsf d_a\left(x,y\right)$ is Lipschitz (therefore, its gradient is well
defined almost everywhere). For all subset $U$ of $\overline \Omega$ and for almost every $x \in \Omega$,
\begin{equation} \label{eq.eqq}
\vert \nabla_x \mathsf d_a\left(x,U\right) \vert  \leq \vert \nabla f (x) \vert.  
\end{equation} 
Moreover,  for all $x,y \in \overline \Omega$, we have
\begin{equation} \label{eq.ineq}
\vert f(x)-f(y)\vert \leq \mathsf d_a\left(x,y\right).  
\end{equation} 
\end{proposition}

The main result of this section if the following:
\begin{proposition}
\label{pr.Agmon}
Let us assume that assumption \autoref{A} is satisfied.  Let
$k\in\{1,\ldots,n\}$ and $ \mathsf \Omega_k^{\textsc{M}} $ be the
subdomain of $\Omega$ introduced in Proposition~\ref{pr.omegakpoint}. Let $\Delta_{f,h}^{\textsc{M}, (1)}(  \mathsf \Omega_k^{\textsc{M}} )$ be the unbounded
 nonnegative self-adjoint  operator   on 
$ \Lambda^1L^{2} (  \mathsf \Omega_k^{\textsc{M}} )$ defined by \eqref{eq.DeltaTN}--\eqref{eq.domDeltaTN} with $(\mathsf \Omega^{\textsc{M}}, \Gamma^{\textsc{M}}_{\mbf D}, \Gamma^{\textsc{M}}_{\mbf N}) = (\mathsf \Omega_k^{\textsc{M}}, \Gamma_{k,\mbf D}^{\textsc{M}}, \Gamma_{k,\mbf N}^{\textsc{M}})$.  
Let  $\mathsf u^{(1)}_k$ be a $L^2(\mathsf
\Omega_k^{\textsc{M}})$-normalized  eigenform of $
\Delta_{f,h}^{\textsc{M} ,(1)} ( \mathsf \Omega_k^{\textsc{M}} ) $
associated with the eigenvalue $\lambda( \mathsf
\Omega_k^{\textsc{M}})$, as introduced in Proposition~\ref{pr.DeltaTN}. 
Then, for any $\delta>0$, there exists $h_\delta>0$  such that it holds for $h\in (0,h_\delta)$:
$$
\big\|e^{\frac{\Psi_k}{ h}} \mathsf u^{(1)}_k\big\|_{  L^{2} ( \mathsf \Omega_k^{\textsc{M}}  )}
+\big\|\mathsf d \big (e^{\frac{ \Psi_k}{ h}} \mathsf u^{(1)}_k\big) \big\|_{  L^{2} ( \mathsf \Omega_k^{\textsc{M}} )}+\big\|\mathsf d^*\big (e^{\frac{ \Psi_k}{ h}} \mathsf u^{(1)}_k\big) \big\|_{  L^{2} ( \mathsf \Omega_k^{\textsc{M}} )} 
\le  e^{\frac \delta h}, 
$$
where  $\Psi_k(x):=\mathsf d_a(x,z_k)$.
\end{proposition}

\begin{proof} 
 Using Lemma~\ref{le.GreenWeak} on $\mathsf \Omega_k^{\textsc{M}}$ with $w=\mathsf u^{(1)}_k$  and since $\nabla f \cdot \mathsf n_{\mathsf \Omega_k^{\textsc{M}} }> 0$ a.e.  $\Gamma_{k, \mbf N}^{\text{M}}$ (see item 2 in Proposition~\ref{pr.omegakpoint}),
 it holds, 
 \begin{align}
 \nonumber
  \lambda( \mathsf \Omega_k^{\textsc{M}}) \Vert  e^{\frac{\varphi}{h}} \mathsf u^{(1)}_k \Vert_{ L^{2} ( \mathsf \Omega_k^{\textsc{M}} )}^2&\ge 
  h^{2}\big  \| \mathsf d (e^{\frac{\varphi}{h}} \mathsf u^{(1)}_k)\big \|^{2}_{ L^{2} ( \mathsf \Omega_k^{\textsc{M}} )}+
 h^{2} \big \| \mathsf d^{*} (e^{\frac{\varphi}{h}} \mathsf u^{(1)}_k) \big \|^{2}_{ L^{2} ( \mathsf \Omega_k^{\textsc{M}} )}\\
 \label{eq.Agmon-est}
&\quad + \big \langle
(|\nabla f|^{2}-|\nabla \varphi|^{2}+h\mathcal{L}_{\nabla
  f}+h\mathcal{L}_{\nabla
  f}^{*})e^{\frac{\varphi}{h}} \mathsf u^{(1)}_k, e^{\frac{\varphi}{h}}
 \mathsf u^{(1)}_k\big  \rangle_{ L^{2} ( \mathsf \Omega_k^{\textsc{M}} )}.
\end{align}
Using \eqref{eq.Agmon-est} and   \eqref{eq.eqq}, it is then standard to get the estimate of Proposition \ref{pr.Agmon} with $\mathsf d_a(\cdot ,\{z_k\}\cup \{x_0\})$ instead of $\mathsf d_a(\cdot,z_k)$ using the same arguments as those used in the boundaryless case \cite[Proposition 3.3.1]{helffer-88}.  Proving  Proposition \ref{pr.Agmon} requires a finer analysis. To this end,  we follow the analysis  of \cite[Section 2.2]{HelSjII} and \cite[Section 6.c]{DiSj}.
The proof is divided into two steps. 
 \medskip
  
  \noindent
  \textbf{Step 1:  A Witten  Laplacian on $1$-forms with a spectrum bounded from below by~$ch$.} 
Roughly speaking, recall that in view of the proof of item $(iii)$ in
Proposition \ref{pr.DeltaTN},     $z_k$ is the only point which
``creates'' a small eigenvalue for $ \Delta_{f,h}^{\textsc{M} ,(1)} (
\mathsf \Omega_k^{\textsc{M}} ) $, namely $\lambda(\mathsf
\Omega_k^{\textsc{M}} )$. Thus, if
  we ``remove''  $z_k$ from~$\overline{\mathsf \Omega_k^{\textsc{M}}}
  $, the spectrum of the Witten Laplacian $ \Delta_{f,h}^{\textsc{M} ,(1)}$ will be bounded from below by $ch$. To do so, we proceed as follows. 
 Let us take $\eta>0$ small enough such that  $ \overline{ \mathsf B_{a}(z_k,3\eta) }\cap \Omega \subset \mathsf \Omega_k^{\textsc{M}} $  and 
 $z_k$ is the only critical point of $f$ in $  \overline{ \mathsf B_{a}(z_k,3\eta) }$,
 where $\mathsf B_{a}(x,r)$ denotes the open ball of center $x$ and radius $r$ for the Agmon distance $d_{a}$ (which is indeed a distance since $f$ is a Morse function). Define  
$$\mathsf D_{k,\eta} := \mathsf \Omega_k^{\textsc{M}} \setminus   \overline{ \mathsf B_{a}(z_k,\eta)}.$$
  We have  $\partial \mathsf D_{k,\eta} = \overline{\Sigma_{k, \mbf N}  }\, \cup\,  \Sigma_{k,\mbf{D}}  \, \cup  \, \overline{\Sigma_{k,\mbf{FD}}  }$,
     where 
   $$\Sigma_{k, \mbf N}  := \Gamma_{k,\mbf N}^{\textsc{M}}, \ \, \,  \Sigma_{k,\mbf{D}}   :=  \Gamma_{k,\mbf D}^{\textsc{M}}\setminus \overline{ \mathsf B_{a}(z_k,\eta)\cap \partial \mathsf \Omega_k^{\textsc{M}}},\text{ and }\,  \ \Sigma_{k,\mbf{FD}}   :=  {\partial \mathsf B_{a}(z_k,\eta)\cap   \mathsf \Omega_k^{\textsc{M}}}.$$
We refer to Figure \ref{fig:zjr} for a schematic representation of
$\mathsf D_{k,\eta} $ and its boundary.    We use the
subscript~$\mbf{FD}$ because we will consider a Witten Laplacian  with
full Dirichlet boundary conditions on $\Sigma_{k,\mbf{FD}}$.

\begin{figure}
\begin{center}
  \includegraphics[width=0.6\textwidth]{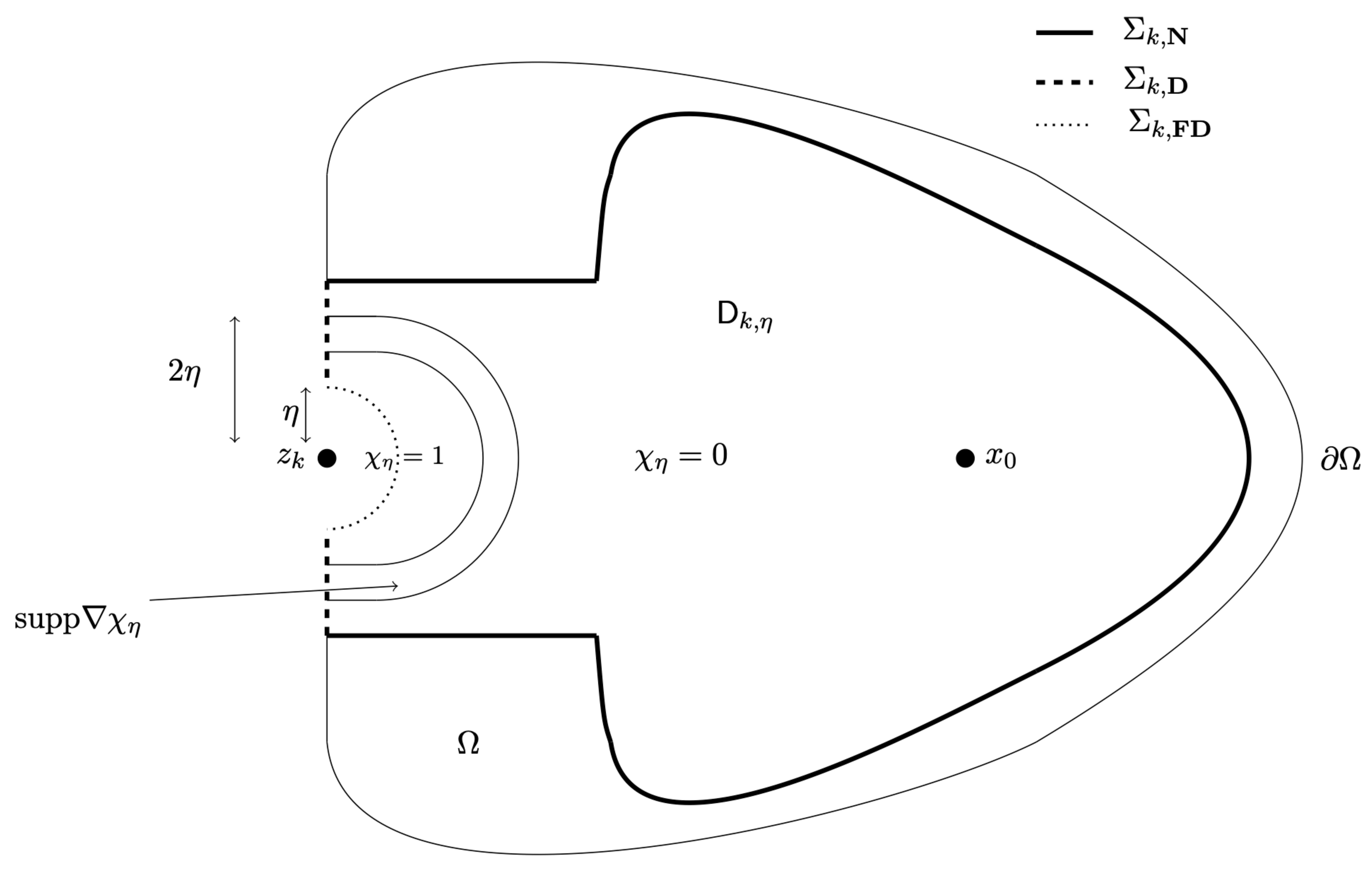}

\caption{A schematic representation of $\mathsf D_{k,\eta} $.}

\label{fig:zjr}

\end{center}

\end{figure}
   
Following the procedure of Section \ref{sec.Witten-general-result}, we
can consider     the Friedrichs extension $\Delta_{f,h}^{
  \mathcal M,(1)} (    \mathsf D_{k,\eta}  )$ (which has different
boundary conditions from the mixed Laplacian $\Delta_{f,h}^{
  \textsc{M},(1)}$ introduced in Proposition~\ref{pr.DeltaM}, hence
the different notation) of the closed quadratic form
 $$ Q_{f,h}^{  \mathcal M,(1)} (    \mathsf D_{k,\eta}  )(u,w)=\langle \mathsf d_{f,h} u,\mathsf d_{f,h}w \rangle_{L^2(    \mathsf D_{k,\eta}  )} +\langle \mathsf d^*_{f,h}u,\mathsf d^*_{f,h}w \rangle_{L^2(    \mathsf D_{k,\eta}  )},$$
for all $u,w\in\mathcal   D\big (  Q_{f,h}^{  \mathcal M,(1)} (    \mathsf D_{k,\eta}  )\big )$, where 
\begin{align*}
\mathcal D\big (  Q_{f,h}^{  \mathcal M,(1)} (    \mathsf D_{k,\eta}  )\big )&:=\big \{w\in \Lambda^1L^2(    \mathsf D_{k,\eta}  ),   \mathsf d_{f,h}w \text{ and } \mathsf d^*_{f,h} w\in  \Lambda L^2(    \mathsf D_{k,\eta}  ) \text{ with }\\
&\quad \quad   \mathbf n w|_{\Sigma_{k,\mbf N}  }=0 , \mathbf t w|_{\Sigma_{k,\mbf D}  }=0, \text{ and } 
  w|_{\Sigma_{k,\mbf{FD}}  }=0    \big\}.
  \end{align*} 
Let $w\in \mathcal D\big (  Q_{f,h}^{  \mathcal M,(1)} (    \mathsf D_{k,\eta}  )\big )$ and $\varphi$ be  a real-valued Lipschitz
function on $\overline{\mathsf  D_{k,\eta}}$. 
 Since $w|_{\Sigma_{k,\mbf{FD}}  }=0 $ and $\nabla f \cdot \mathsf n_{\mathsf \Omega_k^{\textsc{M}} }=0$ on   $\Gamma_{k, \mbf D}^{\text{M}}\supset \Sigma_{k, \mbf D}$, one  has using the same arguments as those used to prove Lemma~\ref{le.GreenWeak},  
\begin{align}
\nonumber
 Q_{f,h}^{  \mathcal M,(1)} (    \mathsf D_{k,\eta}  )( w,e^{\frac 2h \varphi} w)&=
 h^{2}\big  \| \mathsf d (e^{\frac{\varphi}{h}} w)\big \|^{2}_{ L^{2} (  \mathsf D_{k,\eta})}+
h^{2} \big \|  \mathsf d^{*} (e^{\frac{\varphi}{h}} w) \big \|^{2}_{ L^{2} (  \mathsf D_{k,\eta})}\\
\nonumber
&\quad + \big \langle
(|\nabla f|^{2}-|\nabla \varphi|^{2}+h\mathcal{L}_{\nabla
  f}+h\mathcal{L}_{\nabla
  f}^{*})e^{\frac{\varphi}{h}} w, e^{\frac{\varphi}{h}}
 w\big  \rangle_{ L^{2} (  \mathsf D_{k,\eta} )}\\
\nonumber
&\quad +h \int_{\Sigma_{k,\mbf N}}
\langle w,w
\rangle_{ T_{\sigma}^{*}  \mathsf D_{k,\eta} }\;e^{\frac{2 }{h}\varphi }\partial_{\mathsf n_{  \mathsf \Omega_k^{\textsc{M}} } }f\, d\sigma\\
\nonumber
&\ge 
h^{2}\big  \| \mathsf d (e^{\frac{\varphi}{h}} w)\big \|^{2}_{ L^{2} (  \mathsf D_{k,\eta})}+
h^{2} \big \|  \mathsf d^{*} (e^{\frac{\varphi}{h}} w) \big \|^{2}_{ L^{2} (  \mathsf D_{k,\eta})}\\
 \label{eq.Spec2}
&\quad + \big \langle
(|\nabla f|^{2}-|\nabla \varphi|^{2}+h\mathcal{L}_{\nabla
  f}+h\mathcal{L}_{\nabla
  f}^{*})e^{\frac{\varphi}{h}} w, e^{\frac{\varphi}{h}}
 w\big  \rangle_{ L^{2} (  \mathsf D_{k,\eta} )},
\end{align} 
where we have used that $\partial_{\mathsf n_{  \mathsf \Omega_k^{\textsc{M}} } }f\ge 0$ a.e. on $\Gamma_{k,\mbf N}^{\textsc{M}}=\Sigma_{k,\mbf N}$. 
Thus, using \eqref{eq.Spec2} and the same analysis as the one made to prove item $(iii)$ in Proposition \ref{pr.DeltaTN},  there exists $c>0$ such that for $h$ small enough:
\begin{equation}\label{eq.Spec1}
\sigma \big (   \Delta_{f,h}^{  \mathcal M,(1)} (    \mathsf D_{k,\eta}  )\big)\ge ch.
\end{equation} 

  \noindent
  \textbf{Step 2:  Resolvent estimates.} 
   When $\mathsf D$ is a subdomain of $\Omega$,  and $w\in \Lambda^1L^2( \mathsf D  )$ is  such that  $ \mathsf d w$ and  $\mathsf d^*  w$ belong to   $\Lambda L^2(   \mathsf D)$, we define  
 \begin{equation}\label{eq.norm1}
 \Vert w\Vert _{\mathrm W^1(\mathsf D)}^2:= \Vert w\Vert _{L^2(\mathsf D)}^2+\Vert \mathsf dw\Vert _{L^2(\mathsf D)}^2+\Vert \mathsf d^*w\Vert _{L^2(\mathsf D)}^2.
 \end{equation}
 Notice that by \eqref{eq.Agmon-est} (with $\varphi=0$), it holds 
  \begin{equation}\label{eq.estimeW1}
  \Vert \mathsf u^{(1)}_k \Vert _{\mathrm W^1( \mathsf \Omega_k^{\textsc{M}})}\le Ch^{-1/2}.
  \end{equation}
  By  \eqref{eq.Spec1} and since $\lambda(\mathsf \Omega_k^{\textsc{M}} )$ is exponentially small as $h$ goes to $0$ (see \eqref{eq.DL-blambda1}),  the distance of  $\lambda(\mathsf \Omega_k^{\textsc{M}} )$ to $\sigma \big (   \Delta_{f,h}^{  \mathcal M,(1)} (    \mathsf D_{k,\eta}  )\big)$ is bounded from below by $ch/2$, as $h\to 0$.   
Then, adopting  the notation of
\cite[p.~56]{DiSj},  and using  \eqref{eq.eqq}  and  \eqref{eq.Spec2},
we obtain using resolvent estimates as  in the proof of
\cite[Proposition 6.5]{DiSj} (with, in our context,  $K(h)=
\{\lambda(\mathsf \Omega_k^{\textsc{M}} ) \}$),
  \begin{equation}\label{eq.estimeRESO} 
   \big (\Delta_{f,h}^{  \mathcal M,(1)} (    \mathsf D_{k,\eta}  )- \lambda(\mathsf \Omega_k^{\textsc{M}} ) \big )^{-1}(x,y)= \widehat O( e^{-\frac 1h \mathsf d_a(x,y)}\big ) \text{ for all } x,y \in   \mathsf D_{k,\eta}.
     \end{equation}
The  $\widehat O$ in \eqref{eq.estimeRESO} means that for any $x,y\in    \mathsf D_{k,\eta}$ and $\epsilon>0$, there exist neighborhoods $V_x$ and $V_y$ in  $\mathsf D_{k,\eta}$  of $x$ and $y$ respectively such that for $h$ small enough,
$$\big \Vert \big (\Delta_{f,h}^{  \mathcal M,(1)} (    \mathsf D_{k,\eta}  )- \lambda(\mathsf \Omega_k^{\textsc{M}} ) \big )^{-1}w\Vert_{\mathrm W^1(V_x)}\le  e^{-\frac 1h ( \mathsf d_a(x,y)-\epsilon)}\Vert w\Vert_{L^2(V_y)},$$
for all $w\in \Lambda^1L^2(  \mathsf D_{k,\eta})$ supported in  $V_y$. 
We are now in position to prove Proposition \ref{pr.Agmon}. 

\medskip  \noindent
  \textbf{Step 3:  Proof of the Agmon estimate.} 
 Let $\chi_\eta$  be a smooth cut-off function supported  in $\mathsf
 B_{a}(z_k,2\eta)$  which equals $1$ on $\mathsf B_{a}(z_k,3\eta/2)$ and such that  $
\nabla \chi_\eta \cdot \ft n_{\mathsf \Omega_k^{\textsc{M}}}  =0$. 
We   claim that
\begin{equation}\label{eq.domMM}
\text{$(1-\chi_\eta)\mathsf u^{(1)}_k\in \mathcal D( \Delta_{f,h}^{  \mathcal M,(1)} (    \mathsf D_{k,\eta}  ))$ and $ \Delta_{f,h}^{  \mathcal M,(1)} (    \mathsf D_{k,\eta}  )((1-\chi_\eta)\mathsf u^{(1)}_k) = \Delta_{f,h}^{(1)} ((1-\chi_\eta)\mathsf u^{(1)}_k)$}.
\end{equation}

To prove~\eqref{eq.domMM}, we use the integration
by parts formula~\cite[Equation (120)]{DLLN} on $\mathsf D_{k,\eta}$ with,
using the notation there,
$u=(1-\chi_\eta) \mathsf u^{(1)}_k$ and an arbitrary $v\in \mathcal D\big (
Q_{f,h}^{  \mathcal M,(1)} (    \mathsf D_{k,\eta}  )\big )$ and we
observe that all the boundary terms vanish. To do
so, we check that $u=(1-\chi_\eta) \mathsf u^{(1)}_k$ satisfies the
required regularity, and that the boundary terms are zero. This
shows that $Q_{f,h}^{  \mathcal M,(1)} (    \mathsf D_{k,\eta}
)(u,v) = \langle \Delta_{f,h}^{(1)}(u),v \rangle_{L^2(    \mathsf D_{k,\eta}  )}$ is bounded by $ C(u) \|v\|_{L^2(\mathsf D_{k,\eta})} $. Thus $u \in
\mathcal D( \Delta_{f,h}^{  \mathcal M,(1)} (    \mathsf D_{k,\eta}
))$, and $\Delta_{f,h}^{  \mathcal M,(1)} (    \mathsf D_{k,\eta}  ) u
= \Delta_{f,h}^{(1)}u $.

Let us give some more details on the regularity and trace of
$u=(1-\chi_\eta) \mathsf u^{(1)}_k$.  It is easy to check that $u \in \mathcal D\big (
Q_{f,h}^{  \mathcal M,(1)} (    \mathsf D_{k,\eta}  )\big
)$. Moreover, $\mathbf
n\mathsf d_{f,h}u=0$ on $\Sigma_{k, \mbf N}$ and   $\mathsf
d^*_{f,h}u=0$ on $\Sigma_{k, \mbf {FD}}\cup \Sigma_{k, \mbf {D}}$ are
consequences of the fact that $\mathsf u^{(1)}_k\in \mathcal
D(\Delta_{f,h}^{\text{M}, (1)}(\mathsf \Omega_k^{\textsc{M}}))$ and $u=0$ in  a neighborhood of $\overline{\Sigma_{k, \mbf
    {FD}}}$ in $\overline{\mathsf D_{k,\eta} }$. In particular,  $\mathsf d^*_{f,h}u=0 \text{ on }  \Sigma_{k, \mbf {D}}$,
since,   $\mathsf d_{f,h}^*u=-\nabla \chi_\eta\cdot \mathsf
u^{(1)}_k=0$ on $\Sigma_{k, \mbf D}$ (because $\mathsf d_{f,h}^*\mathsf u^{(1)}_k=0$ and $\mathbf  t\mathsf
u^{(1)}_k=0$ on $\Gamma_{k,
  \mbf D}^{\text{M}}\supset \Sigma_{k, \mbf D}$).  This yields
$Q_{f,h}^{  \mathcal M,(1)} (    \mathsf D_{k,\eta}
)(u,v) = \langle \Delta_{f,h}^{(1)}u,v \rangle_{L^2(    \mathsf
  D_{k,\eta}  )}$, using~\cite[Equation (120)]{DLLN},  since $v\in \mathcal D\big (
Q_{f,h}^{  \mathcal M,(1)} (    \mathsf D_{k,\eta}  )\big )$, $\mathbf
n\mathsf d_{f,h}u=0$ on $\Sigma_{k, \mbf N}$ and   $\mathsf
d^*_{f,h}u=0$ on $\Sigma_{k, \mbf {FD}}\cup \Sigma_{k, \mbf {D}}$, and
thus concludes the proof of~\ref{eq.domMM}.

\medskip

  We have, using  \eqref{eq.domMM}, and    since $\mathsf D_{k,\eta}\subset  \mathsf \Omega_k^{\textsc{M}}$, 
  $$\big (\Delta_{f,h}^{\text{M}, (1)} (    \mathsf D_{k,\eta}  )  - \lambda(\mathsf \Omega_k^{\textsc{M}} ) \big ) ( (1-\chi_\eta)\mathsf u^{(1)}_k)= \big [ \Delta_{f,h}^{   (1)} ,(1-\chi_\eta)\big] \mathsf u^{(1)}_k$$
  is supported in  $\overline{\mathsf B_{a}(z_k,2\eta)}\setminus \mathsf
  B_{a}(z_k,3\eta/2)$ (we used here the commutator brackets notation). 
   Using~\eqref{eq.estimeW1} and~\eqref{eq.estimeRESO}, and the fact
   that $\big [ \Delta_{f,h}^{ (1)}  ,(1-\chi_\eta)\big] $ is a
   bounded linear operator from  $\Lambda^1\mathrm{W}^1(\mathsf
   D_{k,\eta})$  to $\Lambda^1L^2(\mathsf D_{k,\eta})$, for all $x\in
   \mathsf D_{k,\eta}$ and $\epsilon>0$, there exists a
     neighborhood $V_x$ of $x$ in $\mathsf D_{k,\eta}$ such that for
     $h$ small enough: 
 $$\Vert (1-\chi_\eta)\mathsf u^{(1)}_k \Vert_{\mathrm W^1( V_x)} \le e^{\frac \epsilon h}  e^{-\frac 1h( \mathsf d_a(x,z_k)-2\eta)} \, \Vert \mathsf u^{(1)}_k \Vert _{\mathrm W^1(\mathsf D_{k,\eta})  }\le  e^{\frac \epsilon h} e^{-\frac 1h( \mathsf d_a(x,z_k)-3\eta)}.$$
Proposition \ref{pr.Agmon} is a consequence of the previous estimate, a compactness argument,  and the fact that  $ \mathsf u^{(1)}_k=\chi_\eta\mathsf u^{(1)}_k+(1-\chi_\eta)\mathsf u^{(1)}_k$ and  $\Vert e^{\frac 1h \mathsf d_a(x,z_k) } \chi_\eta \mathsf u^{(1)}_k  \Vert _{\mathrm W^1( \mathsf \Omega_k^{\textsc{M}})}\le  e^{\frac {3\eta}{h}}$ (by   \eqref{eq.estimeW1} and the continuity of the Agmon distance $\mathsf d_a(\cdot ,z_k)$).     
 \end{proof}

\subsection{Quasi-modes associated with $(z_k)_{k=1,\ldots,n}$}\label{sec:vk1}
 The principal eigenform $\mathsf u^{(1)}_k$  of
 $\Delta_{f,h}^{\textsc{M} ,(1)} ( \mathsf \Omega_k^{\textsc{M}} )$
 introduced in Proposition~\ref{pr.VP-DeltaTN} (see~\eqref{eq.u1k}) will be used as a quasi-mode for $\Delta_{f,h}^{\mathsf{Di} ,(1)} ( \Omega  ) $.  To do so, we   multiply it by a smooth cut-off function $\chi_k^{\textsc{M}} $ whose gradient is supported as close as needed to $\Gamma_{k,\mbf N}^{\textsc{M}}$ and so that $\chi_k^{\textsc{M}} \mathsf u^{(1)}_k$  belongs to the form domain of  $\Delta_{f,h}^{\mathsf{Di} ,(1)} ( \Omega  ) $, namely $\Lambda^1H^1_{\mbf T}(\Omega) $ (as required by item 1 in Proposition \ref{ESTIME}, see also \eqref{eq:unitary}). More precisely, we have the following result.

\begin{proposition}\label{pr.AgmonQM}
Let us assume that the assumptions of Proposition~\ref{pr.Agmon} hold. 
Let  $\mathsf u^{(1)}_k$ be defined by~\eqref{eq.u1k}. 
 Let $\beta>0$ and  $\chi_k^{\textsc{M}} (\beta): \overline{  \mathsf \Omega_k^{\textsc{M}}}  \to [0,1]$ be a $\mathcal C^\infty$ function such that  
 \begin{equation}\label{eq.supp-close1}
 \chi_k^{\textsc{M}} (\beta)=1 \text{ on  } \big \{x\in  \overline{ \mathsf \Omega_k^{\textsc{M}}}, \  \mathsf d_{\overline \Omega}(x, \overline{\Gamma_{k,\mbf N}^{\textsc{M}}})>2\beta\big \},
\end{equation}
 and,
 \begin{equation}\label{eq.supp-close2}
 \chi_k^{\textsc{M}} (\beta)=0 \text{ on  }\big  \{x\in   \overline{ \mathsf \Omega_k^{\textsc{M}}},  \ \mathsf d_{\overline \Omega}(x, \overline{\Gamma_{k,\mbf N}^{\textsc{M}}})\le \beta\big \},
\end{equation}
where  we recall that $\mathsf d_{\, \overline \Omega}$ denotes   the geodesic distance in $\overline \Omega$. We extend $ \chi_k^{\textsc{M}} (\beta)$ by $0$ on $\overline{\Omega}\setminus  \overline{ \mathsf \Omega_k^{\textsc{M}}}$,  and thus $\chi_k^{\textsc{M}} (\beta)\in \mathcal C^\infty(\overline \Omega)$    (see Figure~\ref{fig:domain_chik} for a schematic representation of the support of $ \chi_k^{\textsc{M}}$). 
Then, one defines   
 \begin{equation}\label{eq.QM1}
 \mathsf v_k^{(1)}:=\frac{ \chi_k^{\textsc{M}} (\beta)    \mathsf u_k^{(1)}  }{\Vert \chi_k^{\textsc{M}} (\beta)\mathsf u_k^{(1)}\Vert  _{ L^2 (\Omega)  }},
 \end{equation}
for any $\beta \in (0,\beta_0)$ with $\beta_0>0$ small enough so that
$\chi_k^{\textsc{M}} (\beta) \neq 0$.
For ease of notation, we do not   refer to $\beta$ when writing $\mathsf v_k^{(1)}$.
Then
$$\mathsf v_k^{(1)}\in \Lambda^1H^1_{\mbf T}(\Omega) \cap \Lambda^1\mathcal C^\infty_c (\overline \Omega).$$
Finally, for any $\delta>0$, there exists $h_\delta>0$  such that  for
all $h\in (0,h_\delta)$ and $\beta \in (0,\beta_0)$: 
$$
\big  \| \mathsf d_{f,h} \mathsf v_k^{(1)} \big \|^{2}_{ L^{2} ( \Omega)}+
  \big \|  \mathsf d^{*}_{f,h} \mathsf v_k^{(1)} \big \|^{2}_{ L^{2} (\Omega)}\le C\lambda(\mathsf \Omega_k^{\textsc{M}})+  e^{\frac \delta h} e^{-\frac 2h \inf_{{\rm supp} \nabla \chi_k^{\textsc{M}} (\beta)} \mathsf d_a(\cdot,z_k)},
$$
where $C>0$ is independent of $h$, $\beta$, and $\delta$.
\end{proposition}

\begin{figure}
\begin{center}
  \includegraphics[width=0.7\textwidth]{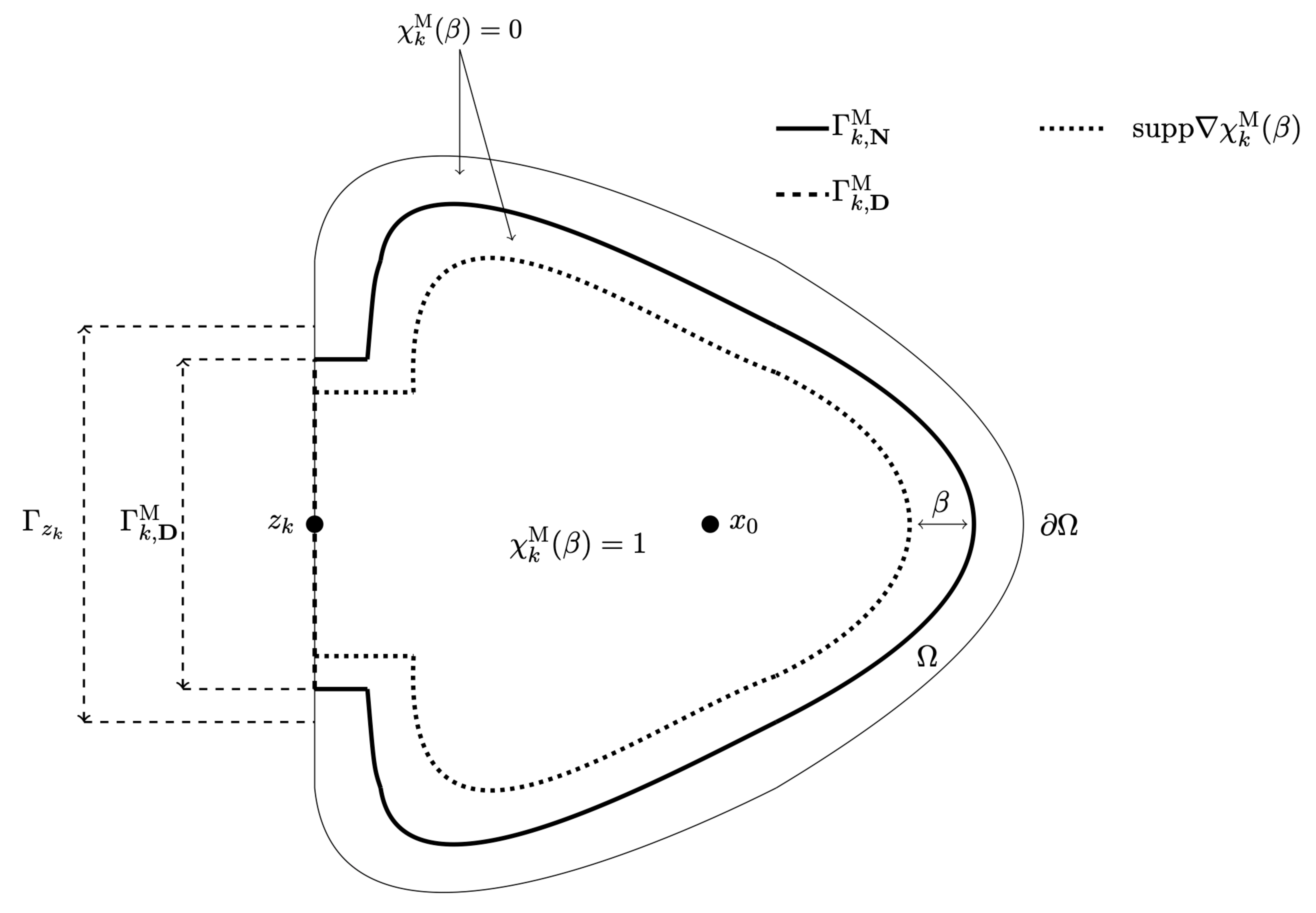}

\caption{Schematic representation of  the  cut-off function $\chi_k^{\textsc{M}}(\beta)$, see Proposition~\ref{pr.AgmonQM}. The support of $\nabla \chi_k^{\textsc{M}}(\beta)$  is as close as needed to $\overline{\Gamma_{k,\mbf N}^{\textsc{M}}}$, and $\overline{\Gamma_{k,\mbf N}^{\textsc{M}}}$ can be as closed as needed to $\partial \Omega\setminus \Gamma_{z_k}$.
 } 

\label{fig:domain_chik}

\end{center}

\end{figure}

Notice that since $\overline{ \Sigma_{z_k}}$ is included in the open subset $\Gamma_{k,\mbf D}^{\textsc{M}}$ of $\pa \mathsf \Omega_k^{\textsc{M}}= \Gamma_{k,\mbf D}^{\textsc{M}}\cup \overline{\Gamma_{k,\mbf N}^{\textsc{M}}}$ (see item 1 in Proposition~\ref{pr.omegakpoint}), from~\eqref{eq.supp-close1},  for $\beta>0$ small enough, 
\begin{equation}\label{eq.chi=1Gzk}
\chi_k^{\textsc{M}} (\beta)=1 \text{ in a neighborhood of } \overline{\Sigma_{z_k} }\text{ in } \overline{ \mathsf \Omega_k^{\textsc{M}}},
\end{equation}
or equivalently, in  a neighborhood of $ \overline{\Sigma_{z_k}} $ in   $\overline{ \Omega}$, by~\eqref{eq.=vec}-\eqref{eq.Calpha}). In addition, one also has that,  since $x_0\in  { \mathsf \Omega_k^{\textsc{M}}}$, for $\beta>0$ small enough,  
\begin{equation}\label{eq.chi=1Gx0}
\chi_k^{\textsc{M}} (\beta)=1 \text{ in a neighborhood of } x_0\text{ in }  { \mathsf \Omega_k^{\textsc{M}}} \text{ (or equivalently, in $\Omega$).}
\end{equation}   In the following, we assume that   $\beta>0$ is small enough such that \eqref{eq.chi=1Gzk} and \eqref{eq.chi=1Gx0}  hold.

\begin{proof}
From   \eqref{eq.domDeltaTN}, $\mathsf u_k^{(1)}, \mathsf d \mathsf u_k^{(1)},  \mathsf d^*  \mathsf u_k^{(1)}\in \Lambda^1L^{2}( \mathsf \Omega ^{\textsc{M}}_k )$  and  $\mathbf{t}\ft u_k^{(1)}|_{\Gamma^{\textsc{M}}_{k,\mbf D}}=0$. Since $ \chi_k^{\textsc{M}} (\beta)=0$ on $\overline{\Omega}\setminus  \overline{ \mathsf \Omega_k^{\textsc{M}}}$,
$$\text{$\mathsf v_k^{(1)}, \mathsf d \mathsf v_k^{(1)},  \mathsf d^*  \mathsf v_k^{(1)}\in \Lambda^1L^{2}(   \Omega  )$}.$$
Since $ \chi_k^{\textsc{M}} (\beta)=0$ on $\partial {\Omega}\setminus  \pa{ \mathsf \Omega_k^{\textsc{M}}}$ and  $\mathbf{t}\ft u_k^{(1)}|_{\text{int } (\partial {\Omega}\cap  \pa{ \mathsf \Omega_k^{\textsc{M}}} )}=0$ (because $\text{int }(\pa \mathsf \Omega_k^{\textsc{M}} \cap \partial \Omega )=  {\Gamma_{k,\mbf D}^{\textsc{M}}}$, see item~1 in Proposition \ref{pr.omegakpoint}), it holds:   $\mathbf{t}\ft v_k^{(1)}=0$ on $\partial \Omega$. 
Then, by  \cite[Lemma 73]{DLLN},
$$\mathsf v_k^{(1)}\in \Lambda^1H^1_{\mbf T}(\Omega).$$
In addition,  since $\Delta^{(1)}_{f,h}\mathsf u_k^{(1)}= \lambda(\mathsf \Omega_k^{\textsc{M}}) \mathsf u_k^{(1)} \in \Lambda^1L^2 ( \mathsf \Omega_{k}^{\textsc{M}} )$ with,  on the smooth open subset $\Gamma_{k,\mbf D}^{\textsc{M}}$ of $\mathsf \Omega_{k}^{\textsc{M}}$,  
$\mathbf{t}\ft u_k^{(1)}=0$  and   $\mathbf{t} \mathsf d_{f,h}^*\ft u_k^{(1)}=0$,
it holds, by elliptic regularity, $\mathsf u_k^{(1)}\in\Lambda^1\mathcal C^\infty (\mathsf \Omega_{k}^{\textsc{M}} \cup \Gamma_{k,\mbf D}^{\textsc{M}})$. Therefore,
 $$\mathsf v_k^{(1)}\in\Lambda^1\mathcal C^\infty_c (\overline \Omega).$$

 Let us now compute the energy of $\mathsf v_k^{(1)}$ in $\Omega$. 
Let us first deal with $\Vert \chi_k^{\textsc{M}} (\beta)\mathsf u_k^{(1)}\Vert  _{ L^2 (\Omega)  }$. First of all, $\Vert \chi_k^{\textsc{M}} (\beta)\mathsf u_k^{(1)}\Vert  _{ L^2 (\Omega)  }=\Vert \chi_k^{\textsc{M}} (\beta)\mathsf u_k^{(1)}\Vert  _{ L^2 (\mathsf \Omega ^{\textsc{M}}_k)  }\le 1$ (we have used that $\Vert  \mathsf u_k^{(1)}\Vert  _{ L^2 (\Omega)  }= 1$, $\chi_k^{\textsc{M}} (\beta)=0$ on $\overline{\Omega}\setminus  \overline{ \mathsf \Omega_k^{\textsc{M}}}$, and $\chi_k^{\textsc{M}}\in [0,1]$). 
On the other hand, 
It holds,
 $$
 \Vert \chi_k^{\textsc{M}} (\beta)\mathsf u_k^{(1)}\Vert  _{ L^2 (\mathsf \Omega_k^{\textsc{M}} ) } \ge 1- \big \Vert [1-\chi_k^{\textsc{M}} (\beta)]\mathsf u_k^{(1)}\big \Vert  _{ L^2 (\mathsf \Omega_k^{\textsc{M}} ) } 
  $$
  and
  $$ \big \Vert [1-\chi_k^{\textsc{M}} (\beta)]\mathsf u_k^{(1)}\big
  \Vert  _{ L^2 (\mathsf \Omega_k^{\textsc{M}} ) }  = \big \Vert
  [1-\chi_k^{\textsc{M}} (\beta)]e^{-\frac{\Psi_k}{h}}\, \mathsf
  u_k^{(1)}e^{\frac{\Psi_k}{h}}\big \Vert  _{ L^2 (\mathsf
    \Omega_k^{\textsc{M}} ) }, $$
where we introduced the function $\Psi_k(x)=\mathsf d_a(x,z_k)$.
Furthermore,  $\chi_k^{\textsc{M}} (\beta)=1$ in a  neighborhood of $z_k$ in $\overline \Omega$ (by~\eqref{eq.chi=1Gzk} together with the fact that  $z_k\in \Sigma_{z_k}$). Thus,  there exits $c>0$ such that 
$$\inf_{\text{supp}(1-\chi_k^{\textsc{M}} (\beta))} \Psi_k>c.$$
Then, using Proposition~\ref{pr.Agmon}, one deduces that $\big \Vert [1-\chi_k^{\textsc{M}} (\beta)]\mathsf u_k^{(1)}\big \Vert  _{ L^2 (\mathsf \Omega_k^{\textsc{M}} ) }=O(e^{-\frac ch})$,  for some $c>0$ and as $h\to 0$. Consequently, 
\begin{align}\label{eq.normevk1}
\Vert \chi_k^{\textsc{M}} (\beta)\mathsf u_k^{(1)}\Vert  _{ L^2 (\Omega)   }=1+O(e^{-\frac ch}).
\end{align}
In addition one has, using again Proposition~\ref{pr.Agmon}, 
\begin{align*}
\big  \| \mathsf d_{f,h} (\chi_k^{\textsc{M}} (\beta)\mathsf u_k^{(1)}) \big \|_{ L^{2} ( \Omega)}&\le   \| \chi_k^{\textsc{M}} (\beta) \mathsf d_{f,h}  \mathsf u_k^{(1)} \big \|_{ L^{2} ( \mathsf \Omega_k^{\textsc{M}})} + h \big \| \nabla \chi_k^{\textsc{M}} (\beta)\wedge \mathsf u_k^{(1)} \big \|_{ L^{2} ( \mathsf \Omega_k^{\textsc{M}})}\\
&\le  \sqrt {\lambda(\mathsf \Omega_k^{\textsc{M}})}+ e^{\frac \delta h} e^{-\frac 1h \inf_{{\rm supp} \nabla \chi_k^{\textsc{M}} (\beta)} \Psi_k}.
\end{align*}
The same inequality holds for $\big  \| \mathsf d^*_{f,h} (\chi_k^{\textsc{M}} (\beta)\mathsf u_k^{(1)}) \big \|_{ L^{2} ( \Omega)}$ because 
$$\big  \| \mathsf d^*_{f,h} (\chi_k^{\textsc{M}} (\beta)\mathsf u_k^{(1)}) \big \|_{ L^{2} ( \Omega)}\le   \| \chi_k^{\textsc{M}} (\beta) \mathsf d^*_{f,h}  \mathsf u_k^{(1)} \big \|_{ L^{2} ( \mathsf \Omega_k^{\textsc{M}})} + h \big \| \nabla \chi_k^{\textsc{M}} (\beta)\cdot\mathsf u_k^{(1)} \big \|_{ L^{2} ( \mathsf \Omega_k^{\textsc{M}})}.$$   The proof of Proposition~\ref{pr.AgmonQM} is complete using  \eqref{eq.normevk1} and \eqref{eq.QM1}.\end{proof}


 \noindent
According to~\eqref{eq.gamma2_close_to_Bzc},~\eqref{eq.supp-close1},
and~\eqref{eq.supp-close2},   for any $\gamma>0$,  one can choose 
$\mathsf \Omega_k^{\textsc{M}}$ in Proposition \ref{pr.omegakpoint} and
$\beta >0$ small enough in Proposition~\ref{pr.AgmonQM} (see  Figure~\ref{fig:domain_chik}) such that:
$$
\sup_{x\in {\rm supp} \nabla \chi_k^{\textsc{M}} (\beta)}\mathsf d_{\, \overline \Omega}(x, \partial \Omega \setminus \Gamma_{z_k} )\le \gamma.
$$ 
Hence, for any $\delta>0$,  one can choose $\beta >0$  and $\mathsf \Omega_k^{\textsc{M}}$   such that:
\begin{equation}\label{eq.distance3}
\inf_{ {\rm supp} \nabla \chi_k^{\textsc{M}} (\beta)} \mathsf d_a(\cdot,z_k)\ge \inf_{ \partial \Omega \setminus \Gamma_{z_k}} \mathsf d_a(\cdot,z_k)-\delta/4
\end{equation}
Then, once \eqref{eq.distance3} is satisfied, one can
use~\eqref{eq.DL-blambda1} and  Proposition~\ref{pr.AgmonQM} with such
$\beta>0$ and $\mathsf \Omega_k^{\textsc{M}}$ fixed as a function $\delta$,  to obtain the following result.

\begin{corollary}\label{co.AgmonQM}
\label{co.qfh-QfhM}
Let us assume that the assumptions of Proposition~\ref{pr.Agmon} hold.
 For any $\delta>0$,
there exists a domain $\mathsf \Omega_k^{\textsc{M}}$, $\beta>0$, and   $h_\delta>0$  such that  for $h\in (0,h_\delta)$:
$$
\big  \| \mathsf d_{f,h} \mathsf v_k^{(1)} \big \|^{2}_{ L^{2} ( \Omega)}+
  \big \|  \mathsf d^{*}_{f,h} \mathsf v_k^{(1)} \big \|^{2}_{ L^{2} (\Omega)}\le   Ch   e^{-\frac 2h(f(z_k)-f(x_0))}+ e^{\frac{ \delta}{h}}e^{-\frac 2h \inf_{ \partial \Omega \setminus \Gamma_{z_k}} \mathsf d_a(\cdot,z_k)},
$$
    where $C>0$ is independent of $h>0$ and $\delta>0$.\end{corollary}
 \noindent
    By Corollary \ref{co.qfh-QfhM}, and   because $\inf_{z\in \partial \Omega \setminus \Gamma_{z_k}} \mathsf d_a(z,z_k)>0$ and $f(z_k)>f(x_0)$, there exists $c>0$ such that for $h$ small enough:
$$
\big  \| \mathsf d_{f,h} \mathsf v_k^{(1)} \big \|^{2}_{ L^{2} ( \Omega)}+
  \big \|  \mathsf d^{*}_{f,h} \mathsf v_k^{(1)} \big \|^{2}_{ L^{2} (\Omega)}\le C e^{-\frac ch}.
$$
This implies that $ \mathsf v_k^{(1)} $ is a  quasi-mode 
associated with the spectrum in $[0,ch]$ of
$\Delta_{f,h}^{\textsc{M} ,(1)} ( \mathsf \Omega_k^{\textsc{M}} )$
because, using in addition Lemma~\ref{quadra} and the fact that, by Proposition~\ref{pr.AgmonQM}, $\mathsf v_k^{(1)}\in \Lambda^1H^1_{\mbf T}(\Omega)=\mathcal D(Q_{f,h}^{\mathsf{Di} ,(1)} ( \Omega  ) )$, it holds:
\begin{equation}\label{eq.QM-est-expo-petit}
 \big \Vert \big [1-\pi_{[0,\mathsf ch]} ( \Delta_{f,h}^{\mathsf{Di} ,(1)} ( \Omega  ))\big]   \mathsf v_k^{(1)} \big \|^{2}_{ L^{2} ( \Omega)}   \le C e^{-\frac ch}.
  \end{equation}

%
%
%
%
%
%
%

\section{Proofs of the main results}\label{sec.mainresu}
In this section, we  give the proofs of the main results stated in Section~\ref{ssec.mainresu}.

\subsection{Proofs of Theorem~\ref{thm1}, Proposition~\ref{pr.LP-N}, Theorem~\ref{thm2}, and Corollary~\ref{co.ek}}
\label{sec.proof0} 

The quasi-modes for $\mathsf L^{\mathsf{Di},(0)}_{f,h}(\Omega)$ and $\mathsf L^{\mathsf{Di},(1)}_{f,h}(\Omega)$  are defined as follows.

\begin{definition}\label{de.QM-pi}
Let us assume that \autoref{A} is satisfied.  Then, one defines for $k\in\{1,\ldots,n\}$ 
 (see~\eqref{eq.u1k} and \eqref{eq.QM1}):
$$
 \mathsf f_k^{(1)}:= e^{\frac 1h f}   \mathsf v_k^{(1)} \in \Lambda^1H^1_{w,\mbf T}(\Omega).$$
For $r\in (0,\min_{\partial \Omega}f- f(x_0))$, consider $\chi_r\in \mathcal C^\infty_c(\Omega)$ such that $\chi_r=1$ on $\{f<\min_{\partial \Omega}f-r\}$. Then, one defines:
$$
 \mathsf u^{(0)}:= \frac{\chi_r}{\Vert \chi_r\Vert_{L^2_w(  \Omega)}}
   \in \mathcal C^\infty_c (  \Omega).
$$
\end{definition}
\noindent
For ease of notation, we do not refer to $r>0$ in the notation of $
\mathsf u^{(0)}$. Recall that the family $\{\mathsf f_1^{(1)} ,
\ldots, \mathsf f_n^{(1)}\} $ depends on the parameter $\delta >0$
introduced in Corollary~\ref{co.AgmonQM}.
Let us now check that there exist $r>0$ and $\delta >0$ such that  the family of quasi-modes $\{\mathsf f_1^{(1)} , \ldots, \mathsf f_n^{(1)}\}  \cup \{ \mathsf u^{(0)} \}$ introduced in Definition~\ref{de.QM-pi} satisfies the assumptions of Propositions~\ref{ESTIME} and~\ref{ESTIME2}.  As explained at the end of this section, Theorem~\ref{thm1}, Proposition~\ref{pr.LP-N}, Theorem~\ref{thm2}, and Corollary~\ref{co.ek} are then   consequences of    the results of  Propositions~\ref{ESTIME} and~\ref{ESTIME2}.

Let us start with the following lemma. 
\begin{lemma}\label{de1}  
Let us assume that~\autoref{A} is satisfied. Let  $\{\mathsf f_1^{(1)} , \ldots, \mathsf f_n^{(1)}\} $  and  $\mathsf u^{(0)}$ be  as  introduced in Definition~\ref{de.QM-pi}. 
Then, item 1 in Proposition~\ref{ESTIME} is satisfied as well as item 2(b). Furthermore,   there exists $C>0$ such that for all $h$ small enough:
\begin{equation}\label{eq.QM0-}
 \big \Vert u_h - \mathsf u^{(0)}\big \Vert_{ L_w^2(\Omega)} \le C  \, h^{-\frac d4 +\frac 12} e^{-\frac 1h (f(z_1)-f(x_0)-r)}.
\end{equation}
 Finally,~\eqref{eq.moyenn-uh} is satisfied.
  \end{lemma}

\begin{proof}   Item 1   in Proposition~\ref{ESTIME}  is satisfied by Definition~\ref{de.QM-pi}. 
First of all, because  $\{x_0\}= \argmin_{\overline \Omega} f$ (see Lemma~\ref{le.start})  and  since $\chi_r=1$ near $x_0$ in  $\Omega$ (see Definition~\ref{de.QM-pi}), 
it holds, using Laplace's method, in the limit $h\to 0$:
\begin{equation}\label{lap.x0K--2}
\Vert \chi_r\Vert_{L^2_w(\Omega)} ^2=  \kappa_{x_0} h^{\frac d2} e^{-\frac 2h f(x_0)} \big(1+O(h)\big),
\end{equation}
where $\kappa_{x_0}$ is defined in~\eqref{eq.kappa_0}.
Recall  that from  Corollary~\ref{thm-pc} (see also \eqref{eq.proj} for the definition of~$ \pi_h^{(0)}$),  the $L^2_w(\Omega)$-orthogonal projector  $ \pi_h^{(0)}$ associated with $\mathsf L_{f,h}^{\mathsf{Di},(0)}(\Omega)$  has rank $1$. 
Because $\mathsf u^{(0)}\in \mathcal D(Q_{f,h}^{\mathsf{Di},(0)}(\Omega))$ (see Proposition~\ref{Lp} and Section~\ref{sec.diri}), it holds  thanks to Lemma~\ref{quadra}:
$$
\big \Vert(1- \pi_h^{(0)}) \mathsf u^{(0)} \big \Vert_{L^2_w(\Omega)}^2 \le\frac{1}{\mathsf c}Q_{f,h}^{\mathsf{Di},(0)}(\Omega)(\mathsf u^{(0)})= \frac {h}{2\mathsf c} \Vert \nabla \mathsf u^{(0)} \Vert_{L^2_w(\Omega)}^2  .
$$
Using~\eqref{lap.x0K--2} and because $\chi_r\in \mathcal C^\infty_c(\Omega)$ such that $\chi_r=1$ on $\{f<\min_{\partial \Omega}f-r\}$, one has for $h$ small enough: 
$$
\frac h2 \Vert \nabla \mathsf u^{(0)} \Vert_{L^2_w(\Omega)}^2  = \frac h2\frac{\Vert \nabla \chi_r\Vert_{L^2_w(\Omega)}^2}{\Vert \chi_r\Vert_{L^2_w(\Omega)}^2}
\le C \Vert \nabla \chi_r \Vert_{L^\infty(\Omega)}^2 \, h^{-\frac d2+  1}\, e^{-\frac{2}{h} (\min_{\partial \Omega}f -f(x_0)-r)}.
$$
Hence, because $f(z_1)=\min_{\partial \Omega}f$ (see Lemma~\ref{le.start} and~\eqref{eq.n0}), $\mathsf u^{(0)}$ satisfies item 2(b) in Proposition~\ref{ESTIME}. 
In addition, one has:
$$\big \Vert(1- \pi_h^{(0)}) \mathsf u^{(0)} \big \Vert_{L^2_w(\Omega)}^2  \le  C \Vert \nabla \chi_r \Vert_{L^\infty(\Omega)}^2 \, h^{-\frac d2+  1}\, e^{-\frac{2}{h} (f(z_1) -f(x_0)-r)}.
$$
Choosing $r>0$ small enough, it hence holds for $h$ small enough:
$$  \Vert    \pi_h^{(0)}\mathsf u^{(0)}  \Vert_{L^2_w(\Omega)}=1+O(e^{-\frac ch}) \neq 0,$$
and then (using in addition the fact that $u_{h}>0$, $\mathsf u^{(0)}\ge 0$ on $\Omega$, and $\mathsf u^{(0)}\neq 0$), 
\begin{align*}
u_{h}=\frac{\pi_h^{(0)}\mathsf u^{(0)} }{  \Vert    \pi_h^{(0)}\mathsf u^{(0)}  \Vert_{L^2_w(\Omega)} }= \frac{\mathsf u^{(0)}}{  \Vert    \pi_h^{(0)}\mathsf u^{(0)}  \Vert_{L^2_w(\Omega)}}  +  \frac{( \pi_h^{(0)} -1)\mathsf u^{(0)}  }{  \Vert    \pi_h^{(0)}\mathsf u^{(0)}  \Vert_{L^2_w(\Omega)}}  \text{ in } L^2_w(\Omega).
\end{align*}
Equation  \eqref{eq.QM0-} is a direct consequence of the three last equations.
Moreover,   the latter equation implies
$$ \int_\Omega u_h\, e^{-\frac 2h f}= (1+O(e^{-\frac ch}) ) \Big[\int_\Omega  \mathsf u^{(0)} \, e^{-\frac 2h f}+  e_h\Big],$$
where 
$$\vert e_h\vert \le \big \Vert(1- \pi_h^{(0)}) \mathsf u^{(0)} \big \Vert_{L^2_w(\Omega)}\big \Vert1\big \Vert_{L^2_w(\Omega)}\le C e^{-\frac{1}{h} (f(x_0) +c_r)}, $$
where $c_r=f(z_1)-f(x_0)-r>0$ (since  $r\in (0,f(z_1)-f(x_0))$).  On the other hand, from~\eqref{lap.x0K--2} together with the fact that $\int_\Omega  \chi \, e^{-\frac 2h f}$ has the same asymptotic equivalent as $\int_\Omega  \chi^2 \, e^{-\frac 2h f}$ when $h\to 0$, 
$$\int_\Omega  \mathsf u^{(0)} \, e^{-\frac 2h f}=h^{d/4}\sqrt{\kappa_{x_0}}e^{-\frac 1h f(x_0)} \big(1+O(h)\big).
$$
This proves~\eqref{eq.moyenn-uh} 
 and concludes the proof of Lemma~\ref{de1}. 
\end{proof}
  
Let us now check that $\{\mathsf f_1^{(1)} , \ldots, \mathsf f_n^{(1)}\} $ satisfies item 2(a) and item 3 in Proposition~\ref{ESTIME}. 
\begin{lemma}\label{de2}  
Assume that~\autoref{A} is satisfied. Let    $\{\mathsf f_1^{(1)} , \ldots, \mathsf f_n^{(1)}\} $ be the family of $1$-forms    introduced in Definition~\ref{de.QM-pi}. Let $k\in \{1,\ldots,n\}$. Then, 
for any $\delta>0$,
there exists  $h_\delta>0$  such that for $h\in (0,h_\delta)$:
\begin{equation}\label{eq.estimePrecise-co1}
\big \|    \big(1-\pi_h^{(1)}\big) \mathsf f_k^{(1)}   \big \|_{H^1_{w}(\Omega)}^2 \le   C  h^{-2}e^{-\frac 2h(f(z_k)-f(x_0))}+  e^{\frac{ \delta }{h}} e^{-\frac 2h( \inf_{ \partial \Omega \setminus \Gamma_{z_k}} \mathsf d_a(\cdot,z_k))} ,
\end{equation}
and for all  $\ell \in \{1,\ldots,n\}$, $\ell \neq k$,  
\begin{equation}\label{eq.estimePrecise-co2}
\big\vert \big\lp\mathsf f_k^{(1)}, \mathsf f_\ell^{(1)}\big\rp_{L^2_w(\Omega)}\big\vert  \le   e^{\frac{\delta}{ h}} e^{-\frac 1h  \mathsf d_a(z_k,z_\ell)}.
\end{equation}
In particular,  choosing $\delta>0$ small enough,      $\{\mathsf f_1^{(1)} , \ldots, \mathsf f_n^{(1)}\} $  satisfies  items 2(a) and 3 in Proposition~\ref{ESTIME}, and if   \eqref{hypo1}  and \eqref{hypo2}  hold, then  $\{\mathsf f_1^{(1)} , \ldots, \mathsf f_n^{(1)}\} $  satisfies  items 1 and 2  in Proposition~\ref{ESTIME2}. 
  \end{lemma}
\begin{proof}
 Using Lemma~\ref{quadra} and Proposition~\ref{pr.defDfhD},
 \begin{align*}
  \big \Vert \big[1-\pi_{[0,\mathsf ch]} ( \Delta_{f,h}^{\mathsf{Di} ,(1)} ( \Omega  )) \big]   \mathsf v_k^{(1)} \big \|^{2}_{ L^{2} ( \Omega)} \le \frac{ \big  \| \mathsf d_{f,h} \mathsf v_k^{(1)} \big \|^{2}_{ L^{2} ( \Omega)}+
   \big \|  \mathsf d^{*}_{f,h} \mathsf v_k^{(1)} \big \|^{2}_{ L^{2} (\Omega)}}{\mathsf ch}.
 \end{align*}
Therefore, using~Corollary~\ref{co.qfh-QfhM},  for any $\delta>0$,
there exists  $h_\delta>0$  such that for $h\in (0,h_\delta)$:
$$
 \big \Vert \big[1-\pi_{[0,\mathsf ch]} ( \Delta_{f,h}^{\mathsf{Di} ,(1)} ( \Omega  )) \big]   \mathsf v_k^{(1)} \big \|^{2}_{ L^{2} ( \Omega)}  \le    C   e^{-\frac 2h(f(z_k)-f(x_0))}+e^{\frac{ \delta }{h}} e^{-\frac 2h( \inf_{ \partial \Omega \setminus \Gamma_{z_k}} \mathsf d_a(\cdot,z_k))}.
$$ 
  Let us prove that this inequality also holds in $\Lambda^1H^1(\Omega)$. 
Set $\mathsf v_{k,\pi}^{(1)}=\big[1-\pi_{[0,\mathsf ch]} ( \Delta_{f,h}^{\mathsf{Di} ,(1)} ( \Omega  )) \big]   \mathsf v_k^{(1)}$. It holds using Proposition~\ref{pr.defDfhD}, 
$$
\big[1-\pi_{[0,\mathsf ch]} ( \Delta_{f,h}^{\mathsf{Di} ,(2)} ( \Omega  )) \big]\mathsf d_{f,h}  \mathsf v_k^{(1)} = \mathsf d_{f,h}\mathsf v_{k,\pi}^{(1)}=h\mathsf d \mathsf v_{k,\pi}^{(1)} +\nabla f\wedge  \mathsf v_{k,\pi}^{(1)}.
$$
Therefore,
 $$h \Vert \mathsf d \mathsf v_{k,\pi}^{(1)}  \|_{ L^{2} ( \Omega)} \le  \Vert \mathsf d_{f,h} \mathsf v_k^{(1)}  \|_{ L^{2} ( \Omega)} +  C\Vert   \mathsf v_{k,\pi}^{(1)}  \|_{ L^{2} ( \Omega)}. $$
Similarly, one has $h \Vert \mathsf d^* \mathsf v_{k,\pi}^{(1)}  \|_{ L^{2} ( \Omega)} \le  \Vert \mathsf d^*_{f,h} \mathsf v_k^{(1)}  \|_{ L^{2} ( \Omega)} +  C\Vert   \mathsf v_{k,\pi}^{(1)}  \|_{ L^{2} ( \Omega)}$. 
 Hence, using also the standard Gaffney inequality in $\Omega$  (see \cite{GSchw}), since  $\mathsf v_{k,\pi}^{(1)}\in \Lambda^1H^1_{\mbf T}(\Omega)$ (by Proposition \ref{pr.AgmonQM}), it holds:
\begin{align} 
\nonumber
C_{\text{Gaffney}}\big  \|   \mathsf v_{k,\pi}^{(1)} \big \|^{2}_{ H^1 ( \Omega)}&\le \big  \|   \mathsf v_{k,\pi}^{(1)} \big \|^{2}_{ L^{2} ( \Omega)} + \big  \| \mathsf d  \mathsf v_{k,\pi}^{(1)} \big \|^{2}_{ L^{2} ( \Omega)}+
  \big \|  \mathsf d^{*}  \mathsf v_{k,\pi}^{(1)} \big \|^{2}_{ L^{2} (\Omega)}\\
  \label{eq.est-QQ}
  & \le Ch^{-2}\big (\Vert \mathsf d_{f,h} \mathsf v_k^{(1)}  \|_{ L^{2} ( \Omega)}^2 +   \Vert \mathsf d^*_{f,h} \mathsf v_k^{(1)}  \|_{ L^{2} ( \Omega)}^2 +  \Vert   \mathsf v_{k,\pi}^{(1)}  \|_{ L^{2} ( \Omega)}^2\big ).
   \end{align}
   This implies that for $h\in (0,h_\delta)$: 
   \begin{equation}\label{eq.est-vk1}
   \big  \|   \mathsf v_{k,\pi}^{(1)} \big \|^{2}_{ H^1 ( \Omega)}\le C  h^{-2}  \big[  e^{-\frac 2h(f(z_k)-f(x_0))}+ e^{\frac{ \delta }{h}}e^{-\frac 2h( \inf_{ \partial \Omega \setminus \Gamma_{z_k}} \mathsf d_a(\cdot,z_k))}\big].
   \end{equation}
   From~\eqref{eq:unitary}, it holds:
$$\pi_{[0,\mathsf ch]} ( \Delta_{f,h}^{\mathsf{Di} ,(1)} ( \Omega  ))=
e^{-\frac 1hf}\pi_h^{(1)}e^{\frac 1 h f} .$$
Therefore,   by definition of $ \mathsf f_k^{(1)}$ (see Definition~\ref{de.QM-pi}) and from~\eqref{eq.est-vk1},   
$$
 \big \Vert (1-\pi_h^{(1)} )   \mathsf f_k^{(1)} \big \|^{2}_{ H^{1}_w ( \Omega)} \le     Ch^{-2}e^{-\frac 2h(f(z_k)-f(x_0))}+ e^{\frac{  \delta }{h}}  e^{-\frac 2h( \inf_{\partial \Omega \setminus \Gamma_{z_k}} \mathsf d_a(\cdot,z_k))} .
$$
 This proves Equation~\eqref{eq.estimePrecise-co1}. 

Let us now prove \eqref{eq.estimePrecise-co2}. One has (see Definition~\ref{de.QM-pi}) using the triangular inequality for $\mathsf d_a$, 
 $$\vert \big\lp\mathsf f_k^{(1)}, \mathsf f_\ell^{(1)}\big\rp_{L^2_w(\Omega)}|=|  \big\lp\mathsf v_k^{(1)}, \mathsf v_\ell^{(1)}\big\rp_{L^2(\Omega)}|\le e^{-\frac 1h \mathsf d_a(z_\ell,z_k)} \Vert    \mathsf v_k^{(1)}e^{\frac 1h \mathsf d_a(\cdot ,z_k)} \Vert_{L^2(\Omega)}\,  \Vert  \mathsf v_\ell^{(1)}e^{\frac 1h \mathsf d_a(\cdot ,z_\ell)} \Vert_{L^2(\Omega)}.$$
 Equation  \eqref{eq.estimePrecise-co2} is thus a consequence of  the previous inequality together with Proposition~\ref{pr.Agmon} and~\eqref{eq.normevk1}  (see also~\eqref{eq.QM1}).

Since $f(z_k)-f(x_0)>0$, $\mathsf d_a(z_k,z_\ell)>0$,  and  $ \inf_{z\in \partial \Omega \setminus \Gamma_{z_k}} \mathsf d_a(z_k,z)>0$ (because $\mathsf d_a$ is a distance),  $\{\mathsf f_1^{(1)} , \ldots, \mathsf f_n^{(1)}\} $  satisfies  items 2(a) and 3 in Proposition~\ref{ESTIME} hold   choosing $\delta>0$ small enough in~\eqref{eq.estimePrecise-co1} and~\eqref{eq.estimePrecise-co2}. Finally,   since $z_\ell \in \partial \Omega \setminus \Gamma_{z_k}$ (because $z_\ell\notin \mathsf W_+^{z_k}$ and $ \Gamma_{z_k}\subset   \mathsf W_+^{z_k} $, see~\eqref{eq.incluWW})
 $$\mathsf d_a(z_k,z_\ell)\ge \inf_{z\in \partial \Omega \setminus \Gamma_{z_k}} \mathsf d_a(z_k,z).
 $$
In addition,  $f(z_k)-f(x_0)\ge f(z_1)-f(x_0)$ and if $k>\ell$, $f(z_k)-f(z_1)\ge f(z_k)-f(z_\ell)$. 
Thus, if~\eqref{hypo1}  and~\eqref{hypo2}  hold, then  $\{\mathsf f_1^{(1)} , \ldots, \mathsf f_n^{(1)}\} $  satisfies  items 1 and 2  in Proposition~\ref{ESTIME2}. 
\end{proof}


\begin{lemma}\label{de3-boundary+interaction}  
Let us assume that~\autoref{A} is satisfied. Let     $\{\mathsf f_1^{(1)} , \ldots, \mathsf f_n^{(1)}\} $ be the family of $1$-forms    introduced in Definition~\ref{de.QM-pi}. Let $k,\ell\in \{1,\ldots,n\}^2$. 
Then, it holds,
$$\int_{\Sigma_{z_\ell}}    \mathsf f_k^{(1)} \cdot \mathsf n_\Omega  \   e^{- \frac{2}{h} f}  d\sigma =\begin{cases}    0   &   \text{ if } k\neq \ell,\\
   -\mathsf  b_k \ h^{\mathsf m }  \     e^{-\frac{1}{h} f(z_k)}  \    (  1  +     O(\sqrt h )    )   &  \text{ if } k=\ell,
   \end{cases} 
$$ 
  where $\mathsf b_k$ and $\mathsf m$  are defined in~\eqref{eq.kappa_0}.
  Let  $\mathsf u^{(0)}$ be as     introduced in Definition~\ref{de.QM-pi}. Then, for all $k \in \{1,\ldots,n\}$, there exists $c>0$ such that as $h\to 0$: 
  \begin{align*}
  \big \lp       \nabla\mathsf u^{(0)}   ,   \mathsf f_k^{(1)}\big\rp_{L^2_w(\Omega)} =
  \begin{cases}    \mathsf K_kh^{\mathsf p}\, e^{-\frac 1h( f(z_1)-f(x_0))} (1+O(\sqrt h))   &   \text{ if } k\in \{1,\ldots,n_0\},\\
   O( e^{-\frac 1h( f(z_1)-f(x_0) +c)}   )    &  \text{ if } k>n_0,
   \end{cases} 
\end{align*}
where $\mathsf K_k$ and $\mathsf p$ are defined in~\eqref{eq.constantK} below. 
In particular,    $\{\mathsf f_1^{(1)} , \ldots, \mathsf f_n^{(1)}\} $ and $\mathsf u^{(0)}$  satisfy  item  4 in Proposition~\ref{ESTIME}.
 If moreover \eqref{hypo1}  and \eqref{hypo2}  hold, one has as $h\to 0$,
 $$  \big \lp       \nabla\mathsf u^{(0)}   ,   \mathsf f_k^{(1)}\big\rp_{L^2_w(\Omega)} =\mathsf K_kh^{\mathsf p}\, e^{-\frac 1h( f(z_k)-f(x_0))} (1+O(\sqrt h)) .$$ 
Hence, if \eqref{hypo1}  and \eqref{hypo2}  hold,   $\{\mathsf f_1^{(1)} , \ldots, \mathsf f_n^{(1)}\} $ and $\mathsf u^{(0)}$  satisfy  items 3 and 4  in Proposition~\ref{ESTIME2}. 
  \end{lemma}
  
  \begin{proof} Recall the definitions of $\mathsf u^{(0)}$
    and  $\{\mathsf f_1^{(1)} , \ldots, \mathsf f_n^{(1)}\} $  in  Definition~\ref{de.QM-pi}. 
  The proof is divided into several steps. 
  \medskip
  
  \noindent
  \textbf{Step 1.} Let us first compute $ \int_{\Sigma_{z_k}}      \mathsf u_k^{(1)}   \cdot \mathsf n_\Omega  \   e^{- \frac{1}{h} f}   $.   One has since $\overline{\Sigma_{z_k}} \subset \Gamma_{k,\mbf D}^{     \textsc{M} }$  (see item 1 in Proposition \ref{pr.omegakpoint}), 
\begin{align*}
 \int_{\Sigma_{z_k}}      \mathsf u_k^{(1)}   \cdot \mathsf n_\Omega  \   e^{- \frac{1}{h} f}  =   \int_{\Gamma_{k,\mbf D}^{     \textsc{M} } }      \mathsf u_k^{(1)}   \cdot \mathsf n_\Omega  \   e^{- \frac{1}{h} f}  d\sigma - \int_{\Gamma_{k,\mbf D}^{     \textsc{M} }  \setminus \Sigma_{z_k} }     \mathsf u_k^{(1)}   \cdot \mathsf n_\Omega  \   e^{- \frac{1}{h} f}   .  
    \end{align*}
    It holds, using the trace estimate~\eqref{eq:subelliptic}  and Proposition~\ref{pr.Agmon}, for any $\delta>0$, there exists  $h_\delta>0$  such that  for $h\in (0,h_\delta)$:
    $$\Big \vert  \int_{\Gamma_{k,\mbf D}^{     \textsc{M} }  \setminus \Sigma_{z_k} }     \mathsf u_k^{(1)}   \cdot \mathsf n_\Omega  \   e^{- \frac{1}{h} f} \Big \vert  \le    e^{-\frac 1h \inf_{\Gamma_{k,\mbf D}^{     \textsc{M} }  \setminus \Sigma_{z_k}}(\mathsf d_a(\cdot, z_k)+f(z_k))}   e^{\frac{ \delta }{h}}.
    $$
   Notice that we have used that $f\ge f(z_k)$ on $\Gamma_{k,\mbf D}^{     \textsc{M} }$ (because $\Gamma_{k,\mbf D}^{     \textsc{M} }\subset \Gamma_{z_k} \subset \mathsf W_+^{z_k}$, see  Lemma~\ref{eq.incluWW}). 
    Thus, since $\inf_{\Gamma_{k,\mbf D}^{     \textsc{M} }  \setminus \Sigma_{z_k}}\mathsf d_a(\cdot, z_k)>0$,   for $\delta>0$ small enough, one has for $h$ small enough
        \begin{align}\label{eq.ici2}
    \Big \vert  \int_{\Gamma_{k,\mbf D}^{     \textsc{M} }  \setminus \Sigma_{z_k} }     \mathsf u_k^{(1)}   \cdot \mathsf n_\Omega  \   e^{- \frac{1}{h} f}   \Big \vert \le    e^{-\frac 1h (f(z_k)+c)}.
\end{align}
 Using~\eqref{eq.DL-boun1}, it then holds as $h\to 0$:
    \begin{align*}
 \int_{\Sigma_{z_k}}\mathsf u^{(1)}_k \cdot \mathsf n_{  \mathsf \Omega_k^{\textsc{M}}  }\  e^{-\frac 1hf}=-\mathsf b_k h^{\mathsf m} \, e^{-\frac 1h f(z_k)} (1+O(\sqrt h)),\
\end{align*}
 where $\mathsf b_k$ and $\mathsf m$  are defined in~\eqref{eq.kappa_0}.
  \medskip
  
  \noindent
  \textbf{Step 2.} Let us deal with the terms  $ \int_{\Sigma_{z_\ell}}     \mathsf f_k^{(1)} \cdot \mathsf n_\Omega  \,  e^{- \frac{1}{h} f}   $. 
 One has that 
  $\chi_k^{\textsc{M}} (\beta)=0 \text{ on }\partial \Omega\setminus \overline{\Gamma_{k,\mbf D}^{     \textsc{M} }}.$
Indeed, $\chi_k^{\textsc{M}} (\beta)$ is supported in $\overline{\mathsf \Omega_k^{\textsc{M}}}$ (see~\eqref{eq.supp-close1}, and~\eqref{eq.supp-close2}) and $\overline{ \mathsf \Omega_k^{\textsc{M}}}\cap \partial \Omega = \pa \mathsf \Omega_k^{\textsc{M}}\cap \partial \Omega= \overline{\Gamma_{k,\mbf D}^{\textsc{M}} }$ (see   item 1 Proposition~\ref{pr.omegakpoint}). In particular, because $\overline{\Gamma_{k,\mbf D}^{\textsc{M}} }\subset \Gamma_{z_k}$, and $\Gamma_{z_k}\cap \Sigma_\ell \subset  \Gamma_{z_k}\cap \Gamma_{z_\ell}=\emptyset$ when $k\neq \ell$ (see the line after  the proof of Lemma~\ref{le.start} and~\eqref{Sigma_k}), $\chi_k^{\textsc{M}} (\beta)=0$ on $\Sigma_{z_\ell}$ when $k\neq \ell$.
Then, one has  using~\eqref{eq.QM1},~\eqref{eq.chi=1Gzk} and~\eqref{eq.normevk1},
\begin{align*}
\int_{\Sigma_{z_\ell}}    \mathsf f_k^{(1)} \cdot \mathsf n_\Omega  \   e^{- \frac{2}{h} f}  &= \int_{\Sigma_{z_\ell}}    \mathsf v_k^{(1)} \cdot \mathsf n_\Omega  \   e^{- \frac{1}{h} f}  d\sigma\\
&= \frac{ 1}{\Vert \chi_k^{\textsc{M}} (\beta)\mathsf u_k^{(1)}\Vert  _{ L^2 (\Omega) }} \int_{\Sigma_{z_\ell}}     \chi_k^{\textsc{M}} (\beta)    \mathsf u_k^{(1)}   \cdot \mathsf n_\Omega  \   e^{- \frac{1}{h} f}   \\
&=(1+O(e^{-\frac ch}))\times 
 \begin{cases}    0   &   \text{ if } k\neq \ell,\\
  \int_{\Sigma_{z_k}}      \mathsf u_k^{(1)}   \cdot \mathsf n_\Omega  \   e^{- \frac{1}{h} f}      &  \text{ if } k=\ell,
   \end{cases} \\
&= \begin{cases}    0   &   \text{ if } k\neq \ell,\\
   -\mathsf  b_k \ h^{\mathsf m }  \     e^{-\frac{1}{h} f(z_k)}  \    (  1  +     O(\sqrt h )    )   &  \text{ if } k=\ell.
   \end{cases} 
  \end{align*}
  This proves the first statement in Lemma~\ref{de3-boundary+interaction}. In particular,   $\{\mathsf f_1^{(1)} , \ldots, \mathsf f_n^{(1)}\} $ satisfies  item  4(b) in Proposition~\ref{ESTIME} and  item  4  in Proposition~\ref{ESTIME2}. 
     \medskip
  
  \noindent
  \textbf{Step 3.} Let us finally deal with the terms  $
\big \lp       \nabla\mathsf u^{(0)}   ,   \mathsf f_k^{(1)}\big\rp_{L^2_w(\Omega)}$. 
 One has, since $\chi_r=0$ on $\partial \Omega$, from~\eqref{eq.QM1}, 
\begin{align*}
  \big \lp       \nabla \mathsf u^{(0)}   ,   \mathsf f_k^{(1)}\big\rp_{L^2_w(\Omega)}&= \frac{1}{\Vert \chi_r\Vert_{L^2_w(\Omega)}} \big \lp       \nabla  \chi_r  ,   e^{-\frac 1hf} \mathsf v_k^{(1)}\big\rp_{L^2(\Omega)}\\
  &=-\frac{1}{\Vert \chi_r\Vert_{L^2_w(\Omega)}} \big \lp       \nabla(1-  \chi_r)  ,   e^{-\frac 1hf} \mathsf v_k^{(1)}\big\rp_{L^2(\Omega)}\\
  &= \frac{1}{\Vert \chi_r\Vert_{L^2_w(\Omega)}} \Big[  h^{-1} \big \lp       (1-  \chi_r)  ,   e^{-\frac 1hf}\mathsf d^*_{f,h} \mathsf v_k^{(1)}\big\rp_{L^2(\Omega)}   -\int_{\partial \Omega} e^{-\frac 1hf} \mathsf v_k^{(1)} \cdot \mathsf n_\Omega   \Big]\\
  &= \frac{1}{\Vert \chi_r\Vert_{L^2_w(\Omega)}} \Big[ h^{-1}  \big \lp   (1-  \chi_r)   e^{-\frac 1hf} ,  \mathsf d^*_{f,h} \mathsf v_k^{(1)}\big\rp_{L^2(\Omega)}   -\int_{\partial \Omega \cap{\rm supp} \chi_k^{\textsc{M}} (\beta) } e^{-\frac 1hf} \mathsf v_k^{(1)} \cdot \mathsf n_\Omega   \Big].
  \end{align*}

   Let us first deal with the boundary term in the previous equality.  
 Because $\chi_k^{\textsc{M}} (\beta)=0$ on $\partial \Omega\setminus \overline{\Gamma_{k,\mbf D}^{     \textsc{M} }}$, from~\eqref{eq.QM1} and~\eqref{eq.normevk1},   it holds:
 \begin{align*}
 \int_{\partial \Omega \cap{\rm supp} \chi_k^{\textsc{M}} (\beta) } e^{-\frac 1hf} \mathsf v_k^{(1)} \cdot \mathsf n_\Omega   &=    (1+O(e^{-\frac ch}))\int_{ \Gamma_{k,\mbf D}^{     \textsc{M} }}  \chi_k^{\textsc{M}} (\beta) e^{-\frac 1hf} \mathsf u_k^{(1)} \cdot \mathsf n_\Omega \\
  &= (1+O(e^{-\frac ch}))\Big[ \int_{ \Gamma_{k,\mbf D}^{     \textsc{M} }} (\chi_k^{\textsc{M}} (\beta)-1)  e^{-\frac 1hf} \mathsf u_k^{(1)} \cdot \mathsf n_\Omega + \int_{ \Gamma_{k,\mbf D}^{     \textsc{M} }} e^{-\frac 1hf} \mathsf u_k^{(1)} \cdot \mathsf n_\Omega\Big].
  \end{align*}
  It holds, from~\eqref{eq.chi=1Gzk},
  $$\Big \vert\int_{ \Gamma_{k,\mbf D}^{     \textsc{M} }}
  (\chi_k^{\textsc{M}} (\beta)-1)  e^{-\frac 1hf} \mathsf u_k^{(1)}
  \cdot \mathsf n_\Omega\Big \vert = \Big\vert\int_{ \Gamma_{k,\mbf D}^{     \textsc{M} }\setminus \Sigma_{z_k}} (\chi_k^{\textsc{M}} (\beta)-1)  e^{-\frac 1hf} \mathsf u_k^{(1)} \cdot \mathsf n_\Omega\Big \vert\le C \Big \vert \int_{ \Gamma_{k,\mbf D}^{     \textsc{M} }\setminus \Sigma_{z_k}}  e^{-\frac 1hf} \mathsf u_k^{(1)} \cdot \mathsf n_\Omega\Big \vert .
  $$
Thus from~\eqref{eq.ici2} and~\eqref{eq.DL-boun1}, it then holds as $h\to 0$:
   $$
   \int_{\partial \Omega \cap{\rm supp} \chi_k^{\textsc{M}} (\beta) } e^{-\frac 1hf} \mathsf v_k^{(1)} \cdot \mathsf n_\Omega  =-\mathsf b_k h^{\mathsf m} \, e^{-\frac 1h f(z_k)} (1+O(\sqrt h)).
$$
 Hence, as $h\to 0$, one has using~\eqref{lap.x0K--2} (see also \eqref{eq.kappa_0}):
    \begin{align*}
  -\frac{1}{\Vert \chi_r\Vert_{L^2_w(\Omega)}} \int_{\partial \Omega \cap{\rm supp} \chi_k^{\textsc{M}} (\beta) } e^{-\frac 1hf} \mathsf v_k^{(1)} \cdot \mathsf n_\Omega  = \mathsf K_kh^{\mathsf p}\, e^{-\frac 1h( f(z_k)-f(x_0))} (1+O(\sqrt h)),\
\end{align*}
with 
 \begin{align}\label{eq.constantK}
\mathsf K_k= \frac{\mathsf b_k}{\sqrt{\kappa_{x_0}}}=\sqrt{\mathsf A_{x_0,z_k}}\ \text{and}\ \mathsf p=\mathsf m-\frac d4=  -\frac 12, \text{ where $\mathsf A_{x_0,z_k}$ is defined in \eqref{eq.ax0zk}.}
\end{align} 

   Let us now  deal with the error term $ \big \lp   (1-  \chi_r)   e^{-\frac 1hf} ,  \mathsf d^*_{f,h} \mathsf v_k^{(1)}\big\rp_{L^2(\Omega)} $. 
 Using~Proposition~\ref{pr.AgmonQM} and~Corollary~\ref{co.qfh-QfhM}, for any $\delta>0$, there exists $h_\delta>0$  such that   for $h\in (0,h_\delta)$:
 \begin{align*}
\big \vert \big \lp(1-  \chi_r),e^{-\frac 1hf}\mathsf d^*_{f,h} \mathsf v_k^{(1)}\big\rp_{L^2(\Omega)}\big \vert &\le C e^{-\frac 1h \min_{ {\rm supp}   (1-\chi_r) }f }\Vert\mathsf d^*_{f,h} \mathsf v_k^{(1)} \Vert_{L^2(\Omega)} \\
&\le Ce^{-\frac 1h ( f(z_1)-r)}  \big[   e^{-\frac 1h(f(z_k)-f(x_0))}+e^{\frac{ \delta }{h}}   e^{-\frac 1h( \inf_{z\in \partial \Omega \setminus \Gamma_{z_k}} \mathsf d_a(z_k,z))}\big].
  \end{align*}
      Therefore, using~\eqref{lap.x0K--2}, 
\begin{align*}
   \frac{  \big \vert \big \lp   (1-  \chi_r)   e^{-\frac 1hf} ,  \mathsf d^*_{f,h} \mathsf v_k^{(1)}\big\rp_{L^2(\Omega)} \big \vert }{\Vert \chi_r\Vert_{L^2_w(\Omega)}} & \le     C\big[ h^{-d/4}  e^{-\frac 1h(f(z_k)-f(x_0))} e^{-\frac 1h(f(z_1)-f(x_0)-r)}\\
   &\qquad +  e^{\frac{  \delta }{h}} e^{-\frac 1h( \inf_{z\in \partial \Omega \setminus \Gamma_{z_k}} \mathsf d_a(z_k,z))}e^{-\frac 1h ( f(z_1)-f(x_0)-r)}\big]\\
  & \le  e^{-\frac1h \mathsf  E_k(r,\delta) }, 
  \end{align*}
  where, for $r>0$ and $\delta>0$ small enough, one can choose
  $\mathsf E_k(r,\delta)>f(z_1)-f(x_0)$. Moreover, if   \eqref{hypo1}
  and \eqref{hypo2}  hold, then $\inf_{z\in \partial \Omega \setminus
    \Gamma_{z_k}} \mathsf d_a(z_k,z)>f(z_k)-f(z_1)$ and
  $f(z_1)-f(x_0)>f(z_k)-f(z_1)$. Thus, for $r>0$ and $\delta>0$ small
  enough, one can choose $\mathsf E_k(r,\delta)>f(z_k)-f(x_0)$. 
The proof of Lemma~\ref{de3-boundary+interaction}  is complete.   
  \end{proof}
  \noindent
  
In this section, we proved (see Lemmata~\ref{de1},~\ref{de2}, and~\ref{de3-boundary+interaction}) that the  quasi-modes $\{\mathsf f_1^{(1)} , \ldots, \mathsf f_n^{(1)}\}  \cup \{ \mathsf u^{(0)} \}$ satisfy   all   the assumptions of 
Propositions~\ref{ESTIME} and~\ref{ESTIME2}. We can now conclude the proofs of   Theorem~\ref{thm1}, Proposition~\ref{pr.LP-N}, Theorem~\ref{thm2}, and Corollary~\ref{co.ek}, using the results of Propositions~\ref{ESTIME} and~\ref{ESTIME2}.  
Theorem~\ref{thm1} is a consequence of   Propositions~\ref{ESTIME}
and~\ref{ESTIME2} together with  the formulas~\eqref{eq.kappa_0}
and~\eqref{eq.constantK} for  the constants $\mathsf b_k$,  $\mathsf
m$,  $\mathsf K_k$, and~$\mathsf p$.    Proposition~\ref{pr.LP-N} is a
consequence of Lemma~\ref{de1} and
Proposition~\ref{ESTIME} (notice
that using Lemma~\ref{le.start}, Proposition~\ref{pr.LP-N} is also  a
consequence of the results of~\cite{DoNe2}: we thus here provide a new
proof using $1$-forms). Theorem~\ref{thm2} is a consequence of
Theorem~\ref{thm1}  and Proposition~\ref{pr.LP-N} together
with~\eqref{eq.dens}. Corollary~\ref{co.ek} is a consequence of
Theorem~\ref{thm1}, Proposition~\ref{pr.LP-N},
and~\eqref{eq.expk}. It remains to prove Theorem~\ref{thm3}.  




\subsection{Generalization to deterministic initial conditions: proof of Theorem~\ref{thm3}}
\label{sec.proof}

 The proof of Theorem~\ref{thm3} relies on so-called leveling results  (see Corollary \ref{co.Px} below) which only requires that   $f:\overline \Omega\to \mathbb R$ is a $\mathcal C^\infty$    function which satisfies item 1 in \autoref{A}.  For $F\in\mathcal  C^{\infty}(\partial \Omega,\mathbb R)$, let us define 
\begin{equation}\label{eq.wh}
\forall x\in \overline
\Omega, \ 
w_h(x)=\mathbb
E_{x}   [ F (X_{\tau} )],
\end{equation}
where $\tau$ is defined by~\eqref{eq:tau}.

\subsubsection{{Leveling result} on $w_h$}
For any  {closed} subset $\mathsf F \subset \mathbb R^d$, one denotes by  
$$\tau_{\mathsf F}=\inf \{ t\geq 0\,  |\,  X_t \in  \mathsf F
\}$$ the first time the process~\eqref{eq.langevin} hits $\mathsf F$
(in particular, $\tau=\tau_{\Omega^c}$).  
Let  $x_0$ be a local minimum of $f$ in $\Omega$. Let  us recall that  ${\mathsf B(x_0,h)}$ is the open ball  centred at $x_0$ of radius $h$.      
Let us assume that $h$ is small enough so that $ \overline{\mathsf B}(x_0,h)\subset \Omega$, where $ \overline{\mathsf B}(x_0,h)$ is the closure of ${\mathsf B}(x_0,h)$. 
The function 
$$\mathsf p_{x_0}: x\mapsto \mathbb P_{x}[ \tau_{\Omega^c}<\tau_{ \overline{\mathsf B}(x_0,h)}]$$ is called the  \textit{committor  function} (or the \textit{equilibrium potential})  between  $\Omega^c$ and $ \overline{\mathsf B}(x_0,h)$.  
We have the following precise leveling result on $\mathsf p_{x_0}:\overline \Omega\to \mathbb R$. 
\begin{proposition}\label{pr.upper_boundeCF}
Let us assume that $f:\overline \Omega\to \mathbb R$ is a $\mathcal C^\infty$ Morse  function which satisfies item 1 in \autoref{A}. 
  Let $\mathsf K$ be a compact subset of $\{f<\min_{\partial \Omega}f\}$. Then, there exist $C_{\mathsf K}>0$ and $h_0>0$   such that for all $h\in (0,h_0)$ and $x\in \mathsf K$: 
\begin{equation}\label{eq.res1}
\mathsf p_{x_0} (x)\le  C_{\mathsf K}\, h^{-d} e^{-\frac 2h(\min_{\partial \Omega}f-f(x))}.
\end{equation} 
\end{proposition}
\noindent
We refer to Figure~\ref{fig:3}  for a schematic representation of $\{f<\min_{\partial \Omega}f\}$    and $\mathsf B(x_0,h)$ (recall that since item 1 in \autoref{A} holds, $f$ satisfies item 1 in Lemma \ref{le.start}). 
  \begin{figure}[h!]
\begin{center}
\begin{tikzpicture}[scale=1]
\tikzstyle{vertex}=[draw,circle,fill=black,minimum size=5pt,inner sep=0pt]
\tikzstyle{ball}=[circle, dashed, minimum size=1.3cm, draw]
\tikzstyle{point}=[circle, fill, minimum size=.01cm, draw]
\draw [rounded corners=10pt] (1,0.5) -- (-0.25,2.5) -- (1,5) -- (4.3,4.2) -- (4.8,2.75) -- (4.4,1) -- (4,0) -- (1.8,0) --cycle;
\draw [thick, densely dashed,rounded corners=10pt] (1.5,0.7) -- (.2,1.5) -- (0.5,2.5) -- (0.09,3.5) -- (3,3.75) -- (3.5,3) -- (3.5,2) -- (3,1) --cycle; 
\draw [thick, densely dashed]   (8.3 ,3.2) -- (9.3 ,3.2);
\draw  (11.3 ,3.2) node[]{$\{f= \min_{\partial \Omega}f\}$};
 \draw (1.7,1.46) node[]{\footnotesize{$ \{f<\min_{\partial \Omega}f\}$}};
     \draw  (3.8,1) node[]{$\Omega$};
    \draw  (5.2,3) node[]{$\partial \Omega$};
 \draw [thick]  (1.7 ,2.5) circle (0.78);
 \draw [<->] (1.7 ,2.4) -- (1.7 ,1.82);
 \draw (1.99 ,2.2)node[]{ $\small{h}$};
  \draw (1.7 ,2.8)node[]{$x_0$};
\draw (1.7 ,2.5) node[vertex,label=north west: { }](v){};
\draw (0.38,1.45) node[vertex,label=south west: {$z_1$}](v){};
\draw (0.17,3.4) node[vertex,label=north west: {$z_2$}](v){};
\end{tikzpicture}
 
 \caption{Schematic representation  of   $\{f<\min_{\partial \Omega}f\}$ and $
 \mathsf B(x_0,h)$.  On  the figure $\pa \{f<\min_{\partial \Omega}f\}\cap \partial \Omega=\{z_1,z_2\}$.  }
 \label{fig:3}
 \end{center}
\end{figure}
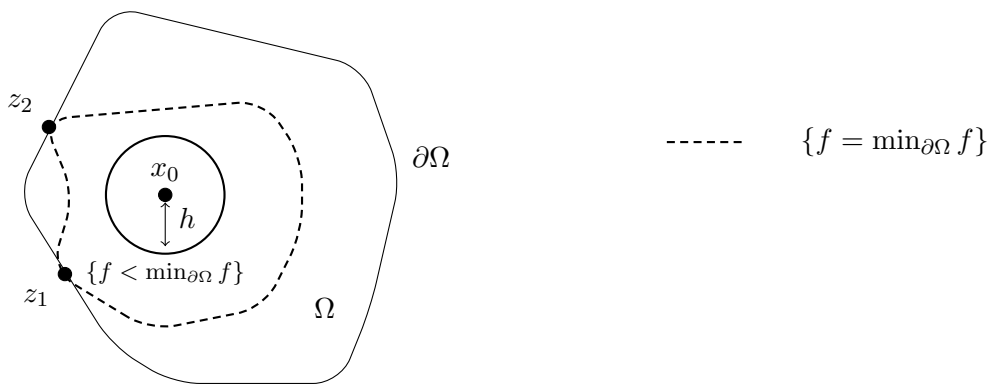

 \begin{proof}
 The proof of this result is inspired by the proofs of~\cite[Lemma 4.6]{BEGK} and~\cite[Proposition~7.9]{landim2017dirichlet}.  Let $x\in \Omega$. If $x\in  \overline{\mathsf B}(x_0,h)$, then $\mathsf p_{x_0}=0$. Let us thus deal with the case when $x\in \Omega\setminus   \overline{\mathsf B}(x_0,h)$. 
  \medskip
 
 \noindent
 \textbf{Step 1.}  First inequality for $  \mathsf p_{x_0}(x)$ using capacities. 
 \medskip
 
 \noindent
 In this step, we prove Equation~\eqref{eq.capa-ineq} below. Let us denote by  $\mathsf d_{\mathbb R^d}$ the standard Euclidean distance in $\mathbb R^d$. Let $G_h$ be the Green function   of  $ \mathsf L^{\mathsf{Di},(0)}_{f,h}(\Omega\setminus  \overline{\mathsf B}(x_0,h))$, see \cite[Equation~(2.3)]{BEGK}. 
 Set for  $x\in \Omega\setminus   \overline{\mathsf B}(x_0,h)$, 
 \begin{equation}\label{eq.cc}
 c=\mathsf d_{\mathbb R^d}(x, \Omega^c\cup     \overline{\mathsf B}(x_0,h)  )/2.
 \end{equation}
Define
$$\text{$\rho=ch>0$ and $R=\rho/3$}.$$
On the one hand, using this couple $(\rho,R)$ in the proof of~\cite[Lemma 4.6]{BEGK},     one deduces  that there exists $C_H>0$  such that for all $x\in\Omega\setminus  \overline{\mathsf B}(x_0,h) $,  $h\in (0,1]$,
\begin{equation}\label{eq.Green}
\sup_{z\in \pa \mathsf B(x,\rho)}G_h(x,z) \le C_H^{\pi \rho/R}\inf_{z\in \pa \mathsf B(x,\rho)}G_h (x,z).
\end{equation}
 Notice that $c$  depends on $h$ and $x$ (which was \textit{a priori} not the case in~\cite[Lemma 4.6]{BEGK}). Let us explain more precisely why~\eqref{eq.Green} remains  valid in our setting.   To get  Equation~\eqref{eq.Green}, one uses $k$~times  the  Harnack inequality \cite{MR1814364} (see also~\cite[Lemma 4.1]{BEGK})     on $k$ balls $\mathsf  B(x_i,\rho)$  where  $x_i\in  \pa\mathsf  B(x,\rho)$  with $\mathsf  B(x_i,R)\cap \mathsf  B(x_{i+1},R)\neq \emptyset$ ($i=1,\ldots,k$, $x_{k+1}=x_1$), and where  $k\le \pi \rho/R$. The constant $C_H$ in~\eqref{eq.Green}  is the one   from the Harnack inequality used on each  $\mathsf  B(x_i,R)$. In addition, this constant $C_H$ depends on $h^{-2} R^2$  and thus can be chosen  independently of $h$  since  for all $x\in\Omega\setminus  \overline{\mathsf B}(x_0,h)$ and $h>0$, 
$$h^{-2} R^2 \le \mathsf d_{\mathbb R^d}^2( x, \Omega^c\cup   \{x_0\})/6^2 \le \mathsf M_0^2/36,$$
where $\mathsf M_0:=\max_{y\in  \overline{\Omega} }\mathsf d_{\mathbb R^d}( y, \Omega^c\cup   \{x_0\})$. 
The condition $h\le 1$ in \eqref{eq.Green} ensures that we can use the Harnack inequality, since for all $x\in \Omega\setminus  \overline{\mathsf B}(x_0,h)$ and all $i=1,\ldots,k$,
$$     {\mathsf  B(x_i,2R)} \subset \Omega\setminus  \overline{\mathsf B}(x_0,h),$$
which follows from the fact that,   if $h\le 1$,  $\rho +2R=5\mathsf d_{\mathbb R^d}(x, \Omega^c\cup     \overline{\mathsf B}(x_0,h)  )h/6<\mathsf d_{\mathbb R^d}(x, \Omega^c\cup     \overline{\mathsf B}(x_0,h) )$. 
Finally notice that we have that 
$$C_H^{\pi \rho/R}=C_H^{\pi /3}.$$

On the other hand, using the arguments of the proof of~\cite[Proposition 7.9]{landim2017dirichlet} with $\mathcal C=   \mathsf  B(x,\rho)$ there, together with~\eqref{eq.Green}, one deduces that  there exists $C>0$ such that for all $x\in \Omega\setminus   \overline{\mathsf B}(x_0,h)$ and  $h$ small enough: 
 \begin{equation}\label{eq.capa-ineq}
  \mathsf p_{x_0}(x)\le  C \frac{\text{cap}(\mathsf  B(x,\rho),  \Omega^c)   }{\text{cap}(\mathsf  B(x,\rho),  \mathsf  B(x_0,h) \cup \Omega^c)},
\end{equation}
where we recall that (see \cite[Section 2]{BEGK}) for two   subsets $\mathsf C$ and $\mathsf  D$ of $ \mathbb R^d$   such that  $ \overline{\mathsf C}\cap\overline{ \mathsf D}=\emptyset$,
\begin{align}
\label{eq.mini}
{\rm cap}({\mathsf C},\mathsf D)&=\frac h2\inf_{\varphi\in H_{\mathsf C, \mathsf D}}  \int_{\mathbb R^d\setminus (\mathsf D\cup \mathsf C)}\big \vert \nabla \varphi (x)\big \vert^2e^{-\frac{2}{h}f(x)}dx,
\end{align}
where
$$
H_{\mathsf C, \mathsf D}=\big \{ \varphi \in H^1(\mathbb R^d),\,  \varphi (x)=1 \text{ for } x\in \mathsf C, \,    \varphi (x)=0 \text{ for } x\in\mathsf  D\}.$$

\medskip

\noindent
 \textbf{Step 2.} Upper bound on $\text{cap}(\mathsf  B(x,\rho),  \Omega^c)  $ and lower bound on $\text{cap}(\mathsf  B(x,\rho),  \mathsf  B(x_0,h) \cup \Omega^c)$. 
 \medskip

  \noindent
Let us first obtain a lower bound on $\text{cap}(\mathsf  B(x,\rho),  \mathsf  B(x_0,h) \cup \Omega^c)$. 
By the variational principle for capacities~\eqref{eq.mini}, it holds:
$$\text{cap}(\mathsf  B(x,\rho),  \mathsf  B(x_0,h) \cup \Omega^c)\ge   \text{cap}(\mathsf  B(x,\rho),  \mathsf  B(x_0,h)).$$
 Let $\mathsf K$ be a compact subset of $\Omega$. 
 Following the proof of~\cite[Lemma 7.10]{landim2017dirichlet} (see also~\cite[Proposition 4.7]{BEGK}) with here $\rho:[0,1]\to \Omega$ a smooth path connecting $x$ to $x_0$ and such that $t\in [0,1]\mapsto f(\rho(t))$ is decreasing, there exists $C>0$ such that for all $x\in \mathsf K$ and $h>0$ (recall $x_0$ is the global minimum of $f$ in $\overline{\Omega}$), 
  \begin{equation}\label{eq.cap1}
  \text{cap}(\mathsf  B(x,\rho),  \mathsf  B(x_0,h))\ge C h^d e^{-\frac 2h f(x)}.
      \end{equation}
Let us now deal with $\text{cap}(\mathsf  B(x,\rho),  \Omega^c)  $. 
Let 
 $\mathsf U_h$ be the  subdomain of $\{f<\min_{\partial \Omega}f\}$ such that $\overline{\mathsf U_h }\subset \{f<\min_{\partial \Omega}f\}  $ and for all $x\in \partial \mathsf U_h$,  
 $\mathsf d_{\mathbb R^d}(x,  {\{f<\min_{\partial \Omega}f\} }^c)=h$. 
   It then follows that for $h$ small enough, 
  \begin{equation}\label{eq.FFj}
   \min_{\{f<\min_{\partial \Omega}f\}   \setminus   {\mathsf U_h }}  f\ge \min_{\partial \Omega}f -\epsilon h,
   \end{equation}
  for some $\epsilon >0$ independent of $h$. Let $\phi_h$ be a smooth  function on $\mathbb R^d$ such that $\phi_h=0$ on $\{f<\min_{\partial \Omega}f\}^c$, $\phi_h=1$ on $ \mathsf U_h $, and for some $C>0$ independent of $h$,
    \begin{equation}\label{eq.nablaphih}
    \Vert \nabla \phi_h\Vert _{L^\infty(\Omega)}\le Ch^{-1}.
    \end{equation}
Assume now that  $\mathsf K$ is a compact subset of $\{f<\min_{\partial \Omega}f\}$. Then,  for $h$ small enough, it holds for all $x\in \mathsf K$:
  $$ {\mathsf  B(x,\rho)}\subset \mathsf U_h,$$
  where we recall   $\rho=ch>0$ where $c$ is defined by \eqref{eq.cc} and satisfies  $c\le \mathsf M_0/2$. 
 Hence, using the variational formula~\eqref{eq.mini}, it holds, for $h$ small enough and all $x\in \mathsf K$, 
 $$\text{cap}(\mathsf  B(x,\rho),  \Omega^c)  \le \frac h2 \int_{\mathbb R^d} \vert\nabla \phi_h \vert^2e^{-\frac 2h f}= \frac h2 \int_{\{f<\min_{\partial \Omega}f\}   \setminus   {\mathsf U_h }} \vert\nabla \phi_h \vert^2e^{-\frac 2h f}.$$
 Using in addition~\eqref{eq.FFj} and~\eqref{eq.nablaphih}, one deduces that  for $h$ small enough and all $x\in \mathsf K$,
 \begin{equation}\label{eq.cap2}
 \text{cap}(\mathsf  B(x,\rho),  \Omega^c)  \le C  e^{-\frac 2h  (\min_{\partial \Omega}f -ch)}    
     \end{equation}
 where $C>0$ is a constant independent of $x\in \mathsf K$ and $h$.  In conclusion, Equations~\eqref{eq.capa-ineq}, \eqref{eq.cap1}, and~\eqref{eq.cap2} imply~\eqref{eq.res1}. This concludes the proof of Proposition~\ref{pr.upper_boundeCF}.
 \end{proof}
 \noindent
The following result   (which is proved as~\cite[Lemma 1]{kamin1979elliptic} or \cite[Lemma 3]{NePC}),  gives a leveling results on $w_h$ in  $\mathsf B(x_0,h)$.

\begin{lemma}\label{le.upper_boundeF}
Let us assume that $f:\overline \Omega\to \mathbb R$ is a $\mathcal C^\infty$ Morse  function. Let $x_0$ be a local minimum of~$f$ in $\Omega$.  
 Then, it holds for $h$ small enough
 $$\sup_{x\in { \mathsf B(x_0,h)}} \vert w_h(x) -w_h(x_0)\vert\le C \sqrt h \, w_h(x_0)$$
\end{lemma} 
The two previous results have the following consequence  on $w_h$.

\begin{corollary}\label{co.Px}
Let us assume that $f:\overline \Omega\to \mathbb R$ is a $\mathcal C^\infty$ Morse  function which satisfies item 1 in \autoref{A}. 
  Let $\mathsf K$ be a compact subset of $\{f<\min_{\partial \Omega}f\}$.
Then, for $h$ small enough  and  uniformly in $x\in \mathsf K$, one has,
  $$w_h(x) =w_h(x_0)(1+O( \sqrt h ))+ O(h^{-d}e^{-\frac 2h (\min_{\partial \Omega }f-\max_{\mathsf K} f)}).$$
\end{corollary}
\noindent
Let us mention that a similar  result was proved in~\cite[Section 5.1.4]{DLLN} using~\cite[Theorem 1]{eizenberg-90} when $f$ has no critical point on the boundary of $\Omega$. When this is no longer the case, \cite[Theorem~1]{eizenberg-90} does not apply and we  prove this result using the strong Markov property together with  Proposition~\ref{pr.upper_boundeCF} and Lemma~\ref{le.upper_boundeF}. 
\begin{proof}
 Let $\mathsf K$ be a compact subset of $\{f<\min_{\partial \Omega}f\}$ and $x\in \mathsf K$. Then write:
 \begin{equation}\label{eq.EE-d}
\mathbb E_{x}   [F(X_{\tau_{\Omega^c}})]= \mathbb E_{x}  \big [F(X_{\tau_{\Omega^c}})\, \mathbf{1}_{ \tau_{\mathsf B(x_0,h)}<\tau_{\Omega^c} }\big]+\mathbb E_{x}  \big [F(X_{\tau_{\Omega^c}}) \mathbf{1}_{ \tau_{\mathsf B(x_0,h)}\ge \tau_{\Omega^c} }\big].
\end{equation}
By the strong Markov property,  
$$ \mathbb E_{x}  \big [F(X_{\tau_{\Omega^c}})\mathbf{1}_{ \tau_{\mathsf B(x_0,h)}<\tau_{\Omega^c} }\big]=  \mathbb E_{x}  \big [  \mathbf{1}_{ \tau_{\mathsf B(x_0,h)}<\tau_{\Omega^c} }\mathbb E_{X_{\tau_{\mathsf B(x_0,h)}}}[F(X_{\tau_{\Omega^c}})]\big].$$
Using Proposition~\ref{pr.upper_boundeCF}
 and Lemma~\ref{le.upper_boundeF},   for all $h$ small enough and $x\in \mathsf K$, it holds: 
  \begin{equation}\label{eq.-Exj} 
   \mathbb E_{x}  \big [F(X_{\tau_{\Omega^c}})\, \mathbf{1}_{ \tau_{\mathsf B_{r_h}(x_0)}<\tau_{\Omega^c} }\big]= \big (1+O(e^{-\frac ch})\big )\big (1+O(\sqrt h)\big )w_h(x_0), \text{ uniformly in $x\in \mathsf K$},
   \end{equation}
where $c>0$ is any  constant such that $c<2( \min_{\partial \Omega }f-\max_{\mathsf K} f)$. Let us now deal with the last term in~\eqref{eq.EE-d}.
For $x\in \mathsf K$, it holds:
$$\mathbb E_{x}  \big [F(X_{\tau_{\Omega^c}}) \mathbf{1}_{ \tau_{\mathsf B(x_0,h)}\ge \tau_{\Omega^c} }\big]\le \Vert F\Vert _{L^\infty(\partial \Omega)} \mathbb P_{x}    [\tau_{\mathsf B(x_0,h)}\ge \tau_{\Omega^c}].$$
Using Proposition~\ref{pr.upper_boundeCF}, it thus holds for $h$ small enough:
$$\max_{x\in \mathsf K}\mathbb E_{x}  \big [F(X_{\tau_{\Omega^c}}) \mathbf{1}_{ \tau_{\mathsf B(x_0,h)}\ge \tau_{\Omega^c} }\big]\le Ch^{-d}e^{-\frac 2h (\min_{\partial \Omega }f-\max_{\mathsf K} f)}.$$
Together with~\eqref{eq.-Exj}  and~\eqref{eq.EE-d}, this concludes the proof of Corollary~\ref{co.Px}. \end{proof}

\subsubsection{Proof of Theorem~\ref{thm3}}

We are now in position to prove Theorem~\ref{thm3}.
 \begin{proof}[Proof of Theorem~\ref{thm3}]
  Let us assume that Assumption~\autoref{A} is satisfied. 
 Let us define for $\alpha \in \mathbb R$, 
 $$\mathsf K_{\alpha}:=\{f\le \alpha\}.$$ 
In the following we consider  $f(x_0)<\alpha < \min_{\partial \Omega}f$ so that 
 $$\mathsf K_{\alpha} \text{ is a non empty compact subset of } \{f<\min_{\partial \Omega}f\}.$$
 Write 
 \begin{equation}\label{eq.Ex==}
 \mathbb E_{\nu_h}  [F(X_{\tau_{\Omega}})]={\mathsf Z_h^{-1}} \int_{\Omega \setminus \mathsf K_{\alpha}}w_h\,
 {u}_h\, e^{-\frac{2}{h}f}  + {\mathsf Z_h^{-1}} \int_{\mathsf K_{\alpha}} w_h\,  {u}_h\, e^{-\frac{2}{h}f},
  \end{equation}
  where we have defined $ \mathsf Z_h  :=\int_{\Omega }
 {u}_h\, e^{-\frac{2}{h}f}$. 
 Let us deal with the first term in the right hand side of~\eqref{eq.Ex==}. It holds:
 $${\mathsf Z_h^{-1}} \int_{\Omega \setminus \mathsf K_{\alpha}}w_h\,
 {u}_h\, e^{-\frac{2}{h}f}  \le \Vert F\Vert _{L^\infty(\partial \Omega)}\  {\mathsf Z_h^{-1}} \int_{\Omega \setminus \mathsf K_{\alpha}} 
 {u}_h\, e^{-\frac{2}{h}f}  . $$
 Moreover, using Lemma~\ref{de1}, it holds (see Definition~\ref{de.QM-pi})
 $$\int_{\Omega \setminus \mathsf K_{\alpha}} 
 {u}_h\, e^{-\frac{2}{h}f}=\frac{\int_{\Omega \setminus \mathsf K_{\alpha}} \chi_r e^{-\frac{2}{h}f} }{\Vert \chi_r\Vert_{L^2_w}} +O\big(  h^{-\frac d4 +\frac 12} e^{-\frac 1h (f(z_1)-f(x_0)-r)} \big) \sqrt { \int_{\Omega \setminus \mathsf K_{\alpha}} 
   e^{-\frac{2}{h}f}}.$$
Recall that when  $h\to 0$  (see Proposition~\ref{pr.LP-N}, \eqref{lap.x0K--2}, and~\eqref{eq.kappa_0}), 
$$\mathsf Z_h=\sqrt{\kappa_0}\, h^{d/4}\, e^{-\frac 1h f(x_0)}(1+O(h)) \ \text{and} \ \Vert \chi_r\Vert_{L^2_w}=\sqrt{\kappa_0}\, h^{d/4}\, e^{-\frac 1h f(x_0)}(1+O(h)).$$
Then, since $f\ge \alpha$ on $\Omega \setminus \mathsf K_{\alpha}$, there exists $\beta >0$ such that:
 $${\mathsf Z_h^{-1}}\int_{\Omega \setminus \mathsf K_{\alpha}} 
 {u}_h\, e^{-\frac{2}{h}f}\le Ch^{-\beta} \big [ e^{-\frac 2h (\alpha -f(x_0))}+ e^{-\frac 1h (f(z_1)+\alpha -2f(x_0)-r)}  \big]. $$
 Therefore, because $\min_{\partial \Omega }f=f(z_1)\ge \alpha$, for any $r>0$, it holds for $h$ small enough:
\begin{equation}\label{eq.oubl}
 {\mathsf Z_h^{-1}}\int_{\Omega \setminus \mathsf K_{\alpha}} 
 {u}_h\, e^{-\frac{2}{h}f}\le e^{-\frac 2h (\alpha -f(x_0)-r)}.
 \end{equation}
In conclusion, the first term in the right hand side of~\eqref{eq.Ex==} satisfies the following upper bound: for any $r>0$, it holds for $h$ small enough:
\begin{equation}\label{eq.Est1-EE2}
{\mathsf Z_h^{-1}} \int_{\Omega \setminus \mathsf K_{\alpha}}w_h\,
 {u}_h\, e^{-\frac{2}{h}f}  \le   e^{-\frac 2h (\alpha -f(x_0)-r)}.
 \end{equation}

  Let us now deal with the second term in the right hand side of~\eqref{eq.Ex==}. Using Corollary~\ref{co.Px} with $\mathsf K=\mathsf K_{\alpha}$ there, it holds:
 $${\mathsf Z_h^{-1}} \int_{\mathsf K_{\alpha}} w_h\,  {u}_h\, e^{-\frac{2}{h}f} = \Big[w_h(x_0)(1+O( \sqrt h ))+ O(h^{-d}e^{-\frac 2h (\min_{\partial \Omega }f-\max_{\mathsf K_{\alpha}} f)}) \Big] \times \frac{\int_{\mathsf K_{\alpha}}    {u}_h\, e^{-\frac{2}{h}f} }{ \int_{\Omega}   {u}_h\, e^{-\frac{2}{h}f} }.$$
In addition, by \eqref{eq.oubl},
 $$\frac{\int_{\mathsf K_{\alpha}}    {u}_h\, e^{-\frac{2}{h}f} }{ \int_{\Omega}   {u}_h\, e^{-\frac{2}{h}f} }=1-  {\mathsf Z_h^{-1}}\int_{\Omega \setminus \mathsf K_{\alpha}} 
 {u}_h\, e^{-\frac{2}{h}f}=1+O(e^{-\frac ch}).$$
 Thus, the second term in the right hand side of~\eqref{eq.Ex==} satisfies the equivalent in the limit $h\to 0$:
\begin{equation}\label{eq.Est1-EE3}
{\mathsf Z_h^{-1}} \int_{\mathsf K_{\alpha}} w_h\,  {u}_h\, e^{-\frac{2}{h}f} = \Big[w_h(x_0)(1+O( \sqrt h ))+ O(h^{-d}e^{-\frac 2h (\min_{\partial \Omega }f-\alpha)}) \Big] (1+O(e^{-\frac ch})).
 \end{equation} 

Choosing $r>0$ small enough,   Equations~\eqref{eq.px1}
 and~\eqref{eq.px2}  when $x=x_0$ are then  consequences of~\eqref{eq.Est1-EE2},~\eqref{eq.Est1-EE3},~\eqref{eq.Ex==} together with~\eqref{eq.puh1} and~\eqref{eq.puh2}    in Theorem~\ref{thm2}, and the fact that $f(x_0)<\alpha < \min_{\partial \Omega}f$. To obtain~\eqref{eq.px1}
 and~\eqref{eq.px2} uniformly on all $x$ in 
 any compact subset $\mathsf K$ of $\{f<\min_{\partial \Omega}f\}$, one uses 
  in addition  Corollary~\ref{co.Px} with $\mathsf K_\alpha$ where
  $\mathsf K$  is such that $\mathsf K\subset  \mathsf K_\alpha$ (with
  $f(x_0)<\alpha<\min_{\partial \Omega} f$).   Using the procedure of
  step 2 of the proof of~\cite[Proposition 14]{DLLN-saddle2} and since the domain of attraction $\mathcal A(\{f<\min_{\partial \Omega}f\})$ of $\{f<\min_{\partial \Omega}f\}$ for the dynamics 
 \eqref{eq.hbb} (see \cite[Section 1.2.2]{DLLN-saddle2} for the definition of $\mathcal A(\{f<\min_{\partial \Omega}f\})$) is equal to $\Omega$ (by item 1 of \autoref{A}),~\eqref{eq.px1}
 and~\eqref{eq.px2} extends to all $x\in \mathsf K$, $\mathsf K$ a compact set of $\Omega$.

 It   remains to prove~\eqref{eq.px3}.
 To this end,   assume that~\eqref{hypo1} and~\eqref{hypo2} are satisfied.  Assume in addition there exists $\ell_0\in \{n_0+1,\ldots,n\}$ such
that  (see~\eqref{eq:hypo2_bis})
$$
2( f(z_{\ell_0})-f(z_1))<f(z_1)-f(x_0).
$$
Let $k_0\in \{ n_0+1,\ldots,\ell_0\}$ and $ \alpha_* \in \mathbb R$ be such
that  $f(x_0) <  \alpha_*< 2 f(z_1)-f(z_{k_0})$. Notice that we can
assume without loss of generality (up to increasing $\alpha_*$ if $\alpha_*$ is
smaller than $f(x_0)+f(z_{k_0})-f(z_1)$, see~\eqref{eq:hypo2_bis}) that
\begin{equation}\label{eq:hyp_alpha}
f(x_0)+f(z_{k_0})-f(z_1)<\alpha_*< 2 f(z_1)-f(z_{k_0}).
\end{equation}
Let us consider $x\in \mathsf K_{\alpha^*}$ and $k\in \{ n_0+1,\ldots,k_0\}$. Thanks
to~\eqref{eq:hyp_alpha} and the fact that $f(z_k)\le f(z_{k_0})$ 
$$\min_{\partial \Omega}f-\alpha_*=f(z_1)-\alpha_*>f(z_{k_0})-f(z_1)\ge f(z_k)-f(z_1),$$
and 
$$ \alpha_*-f(x_0)>f(z_{k_0})-f(z_1)\ge f(z_k)-f(z_1).$$
Choosing $r>0$ small enough in~\eqref{eq.Est1-EE2}, and using~\eqref{eq.Est1-EE3} and~\eqref{eq.Ex==} together with~\eqref{eq.puh3}    in Theorem~\ref{thm2}, one then deduces Equation~\eqref{eq.px3} when $x=x_0$. Using Corollary~\ref{co.Px}, one proves that~\eqref{eq.px3}  holds uniformly on 
 $x\in \mathsf K_{\alpha^*}$.  The proof of Theorem~\ref{thm3} is complete. \end{proof}
\bigskip

\noindent
\textbf{Acknowledgements:} This work was partially supported by the
ANR-19-CE40-0010 (QuAMProcs) and the  ANR-17-CE40-0030 (EFI). BN is  supported by  the grant
IA20Nectoux from the Projet I-SITE Clermont CAP 20-25. 
 TL has received funding from the European ResearchCouncil (ERC) under
the European Union’s Horizon 2020 research and innovation programme
(grant agreement No 810367), project EMC2. Part of this project was carried out as TL was a visiting professor at Imperial College of London, with a visiting professorship grant from the Leverhulme
Trust. The Department of Mathematics at ICL and the Leverhulme Trust are warmly
thanked for their support.

\appendix

\section{Proofs of some technical results and additional comments}

\subsection{Proof of Lemma~\ref{le.start}}\label{sec:le.start}

\begin{proof}[Proof of Lemma~\ref{le.start}] The proof is divided into two steps: 
\medskip

\noindent
\textbf{Step 1: Proof of item 1 in Lemma~\ref{le.start}.}
Let us prove that $\partial_{\mathsf n_\Omega}f\ge 0$ on $\partial \Omega$. Let us assume that there exists $z\in \partial \Omega$ such that $\partial_{\mathsf n_\Omega}f(z)<0$. Then, there exists $s_z>0$ such that  $\varphi_z(t) \notin \overline \Omega$  for all $t\in (0,s_z]$. Let $\ve >0$ be such that $\overline{\mathsf B(\varphi_z(s_z),\ve)} \subset \mathbb R^d\setminus \overline \Omega$. Since $(s,x)\mapsto \varphi_x(s)$ is continuous, there exists $\alpha>0$ such that if $|s-s_z|\le \alpha$ and $|x-z|\le \alpha$, then $ \varphi_x(s)\in \mathsf B(\varphi_z(s_z),\ve)$. Therefore, for all $x\in \Omega$ such that $|x-z|\le \alpha$, $ \varphi_x(s)\notin \Omega$, which contradicts item 1 in \autoref{A}. Therefore $\partial_{\mathsf n_\Omega}f\ge 0$ on $\partial \Omega$. 

The fact that $x_0$ is the only critical point of  the function $f$ in $\Omega$ is also a direct consequence of  item~1 in \autoref{A}. 
In addition,  there  is no local minimum $x\in \partial \Omega$ of $f$ in $\overline \Omega$. Indeed, assume the existence of such a point $x\in \partial \Omega$. Necessarily $\partial_{\mathsf n_\Omega}f(x)\le 0$, and by the previous discussion, $\partial_{\mathsf n_\Omega}f(x)= 0$. Then, $x$ is a critical point of $f$ and a  local minimum   of $f$  in $\mathbb R^d$. Since Hess$f(x)>0$, $x$ is (positively) asymptotically stable for the flow~\eqref{eq.hbb}. This contradicts   item~1 in \autoref{A}.

 Let us now prove that  $f(x_0)=\min_{\overline \Omega}f<\min_{\partial \Omega}f$. For $\beta >0$, set $\mathsf V_\beta=f|_{\Omega}^{-1}((-\infty,f(x_0)+ \beta))$. Since Hess$f(x_0)>0$, for $\beta>0$ small enough,
$\mathsf V_\beta$ is a nonempty open neighborhood of $x_0$ in $\Omega$ such that  
  $x_0$ is the unique global minimum  of $f$ on  $ \overline{\mathsf V_\beta}$.
Let $x\in \Omega\setminus  {\mathsf V_\beta}$. Let $t_x:=\inf\{t\ge 0, \varphi_x(t_x)\in  \overline{\mathsf V_\beta}\}$. 
By item 1 in \autoref{A}, $t_x<+\infty$. In addition, by continuity of $t\mapsto \varphi_x(t)$,   $\varphi_x(t_x)\in \pa \mathsf V_\beta\subset \{f=f(x_0)+\beta\}$. Thus,  
$$f(x)   =f(\varphi_x({t_x})) + \int_0^{t_x} |\nabla f(\varphi_x(s))|^2ds  \ge  f(x_0)+\beta.$$
Let us now consider  $x\in \partial \Omega$. Let $x_n\in \Omega$ be such that  $x_n\to x$ as $n\to +\infty$. Since for $n$ large enough $f(x_n)\ge f(x_0)+\beta$, it follows that $f(x)\ge f(x_0)+\beta$. In conclusion 
 $$f(x_0)=\min_{\overline \Omega}f<\min_{\partial \Omega}f.$$
 
 It remains to prove that 
 $\{f< \min_{\partial \Omega}f\}$ is connected and $\pa \{f< \min_{\partial \Omega}f\}\cap \partial \Omega =\argmin_{\partial \Omega}f$.   The fact that $\{f< \min_{\partial \Omega}f\}$ is connected follows from the facts that $\{f< \min_{\partial \Omega}f\}$ is actually an open subset of $\Omega$ and  that there is only one local minimum of $f$ in $\Omega$ (namely $x_0$). Let us now prove that $\pa \{f< \min_{\partial \Omega}f\}\cap \partial \Omega =\argmin_{\partial \Omega}f$. It is clear that $\pa \{f< \min_{\partial \Omega}f\}\cap \partial \Omega \subset \argmin_{\partial \Omega}f$. Let $z\in \argmin_{\partial \Omega}f$. If  $z\notin  \pa \{f< \min_{\partial \Omega}f\}$, then there exists $\ve >0$ such that for all $x\in \overline{\mathsf B(z,\ve)}\cap \overline{\Omega}$, $f(x)\ge \min_{\partial \Omega}f=f(z)$. Thus, $z$ is a local minimum of $f$ in $\overline \Omega$, which is not possible. Consequently $z\in \pa \{f< \min_{\partial \Omega}f\}$. Therefore,  $\pa \{f< \min_{\partial \Omega}f\}\cap \partial \Omega =\argmin_{\partial \Omega}f$. This ends the proof of item 1 in Lemma~\ref{le.start}. 
 \medskip

\noindent
\textbf{Step 2:  Proof of item 2 in Lemma~\ref{le.start}.}
Let $z\in \partial \Omega$ be such that $\vert \nabla f\vert (z)=0$. 
Let $(e_1, , \ldots,e_d)$ be an orthonormal basis such that (i) $z + {\rm
Span}(e_1, \ldots, e_{d-1}) = T_z \partial \Omega$ and (ii) $e_d = \mathsf n_\Omega(z)$.
  Let us introduce the affine change of variables: $\varphi : (y_1,\ldots,y_d) \mapsto z + \sum_{i=1}^d y_i e_i$. The Hessian of $f$ at point $z$
is unitarily equivalent to the matrix with $(i,j)$-component
 $\frac{\partial^2 f}{\partial y_i \partial y_j}(0)$ for $1 \le i,j
 \le d$ (where with a slight abuse of notation, $f(y)$ refers to $f
(\varphi(y)))$. Let us prove that 
 $\frac{\partial^2 f}{\partial y_i
 \partial y_d}(0)=0$ for all $1 \le i \le d-1$. For any $i \in \{1, \ldots,  d-1\}$,
 let $t \in (-1,1) \mapsto \gamma(t) \in \partial \Omega$ be a
 curve such that $\gamma(0)=z$ and $\gamma'(0)=e_i$. By item 2 in
 \autoref{A}, one has 
 $\frac{d}{dt} \nabla f(\gamma(t)) \cdot \mathsf n_{\Omega}(\gamma(t)) \big|_{t=0} = 0$ (since $\nabla f(\gamma(t)) \cdot
\mathsf n_{\Omega}(\gamma(t))=0$ on a neighborhood of $t=0$). This writes:
$\frac{\partial^2 f}{\partial y_i \partial y_d}(0)=0$. This implies that
 $(0,\ldots,0,1)^T$ is an eigenvector of $\left(\frac{\partial^2
 f}{\partial y_i \partial y_j}(0)\right)_{1 \le i,j \le d}$ associated
 with the eigenvalue $\frac{\partial^2 f}{\partial y_d^2}(0)$. Since $f$
 is a Morse function in $\overline{\Omega}$, one has $\frac{\partial^2
f}{\partial y_d^2}(0)\neq 0$. Finally, item 1 in  \autoref{A} then
 implies that necessarily $\frac{\partial^2 f}{\partial y_d^2}(0) < 0$.
 This proves that $\mathsf n_\Omega(z)$ is an eigenvector of the Hessian of $f$ at   $z$ associated with a negative eigenvalue. \end{proof}

\subsection{On WKB-approximation for $ \mathsf v_k^{(1)}$}\label{sec:WKB}

As explained in Section~\ref{sec.zk-QM}, the quasi-modes $\mathsf f_k^{(1)}$
for $\mathsf L_{f,h}^{\mathsf{Di},(1)}(\Omega)$ are built using the
principal $1$-eigenform~$ \mathsf v_k^{(1)} $ of a Witten Laplacian on $\mathsf \Omega_k^{\textsc{M}}$  with
mixed Dirichlet-Neumann boundary conditions. Since $\vert \nabla f(z_k)\vert =0$, the constructions of the domain $\mathsf \Omega_k^{\textsc{M}}$ and  of the quasi-mode $\mathsf v_k^{(1)} $  are  very different from the ones done in~\cite{DLLN} and require to overcome a major technical issue.

Indeed,
we do not
have    a 
satisfactory
WKB-approximation of~$ \mathsf v_k^{(1)} $ near
$\Sigma_{z_k}$ in $\overline{\mathsf \Omega_k^{\textsc{M}}}$.  
An accurate WKB-approximation, constructed in \cite{helffer-nier-06},  was used  in~\cite{DLLN} (see also
\cite{DLLN-saddle1}
for similar computations)  to estimate
(in the limit $h \to 0$), the quantities
\begin{equation}\label{eq.int-bordvk}
\int_{\Sigma_{z_k}} \mathsf v_k^{(1)} \cdot \mathsf n_{\Omega} e^{-\frac 1hf}, k=1,\ldots,n,
\end{equation}
which were in turn used to compute asymptotically  $\int_{\Sigma_{z_k}}\partial_{\mathsf n_\Omega}u_h e^{-\frac 1hf}$, since (see Corollary~\ref{thm-pc},~\eqref{eq.nu-in} and~\eqref{eq:unitary})
$$\int_{\Sigma_{z_k}}\partial_{\mathsf n_\Omega}u_h e^{-\frac 1hf}\sim
\sum_{k=1}^n  \int_\Omega \nabla u_h\cdot \mathsf v_k^{(1)}  \,
e^{-\frac 1hf}\  \int_{\Sigma_{z_k}}  \mathsf v_k^{(1)} \cdot \mathsf
n_{\Omega} \, e^{-\frac 1hf}.$$

In our context (see~\cite[Section 1.4]{DoNe2} for more details),  
the only possible candidate is the WKB ansatz constructed in 
\cite[Section 2]{helffer-sjostrand-85} 
on $\mathsf B(z_k,\rho)$ (for
some $\rho>0$).
However, 
only
its first term $\omega_0e^{-\frac 1h  \mathsf d_a(.,z_k)}|_\Omega$
belongs to  the form domain of
$\Delta_{f,h}^{\mathsf{Di},(1)}(\Omega)$ in general, i.e. satisfies $\mbf t
a_0=0$ on $\partial \Omega\cap  \mathsf B(z_k,\rho)$.  Thus, only this first term
can be used
to approximate $\mathsf v_k^{(1)} $ with the help of Lemma~\ref{quadra}, but this approximation
is not accurate enough.
Let us briefly explain why, by showing that
for a smooth function
$\xi $ 
supported in $\mathsf B(z_k,\rho/2)$   which equals~$1$ in    $\mathsf
B(z_k,\rho/3)$, the $1$-form
$$\mathsf u_{{wkb,0}}=  \frac{\xi \omega_0e^{-\frac 1h  \mathsf d_a(.,z_k)}|_\Omega}{\Vert \xi \omega_0e^{-\frac 1h  \mathsf d_a(.,z_k)}|_\Omega\Vert_{L^{2}(\Omega)}} $$
  is in general not close enough   to $\mathsf v_k^{(1)} $ in $\Lambda^1H^1(\Omega)$ (recall that we are looking for an equivalent of~\eqref{eq.int-bordvk}).
 Using \eqref{eq.incluWW}, by construction of $\omega_0$  in~\cite[Th\' eor\`eme 2.5]{helffer-sjostrand-85}   and using an integration by parts, it holds, for $h$ small enough:
\begin{align*}
&
\Vert \xi \omega_0e^{-\frac 1h  \mathsf d_a(.,z_k)}|_\Omega\Vert^{2}_{L^{2}(\Omega)}\Big(
\Vert \mathsf d_{f,h}\mathsf u_{{wkb,0}}\Vert^2_{L^2(\Omega)}+\Vert \mathsf d^*_{f,h}\mathsf u_{{wkb,0}}\Vert^2_{L^2(\Omega)}\Big)\\
&  = \langle \Delta^{(1)}_{f,h}\mathsf u_{{wkb,0}},\mathsf u_{{wkb,0}}\rangle_{L^2(\Omega)}- h\int_{\partial \Omega\cap  \mathsf B(z_k,\rho)} \mbf t \mathsf d_{f,h}^* \mathsf u_{{wkb,0}}\,  \mathsf u_{{wkb,0}}\cdot \mathsf n_\Omega\\
&=O(h^2)  \Vert e^{-\frac 2h  \mathsf d_a(.,z_k)} \xi \omega_0\Delta_{\mbf H}\omega_0 \Vert^2_{L^2(\Omega)}   +O(e^{-\frac ch})\\
&\quad -h\int_{\partial \Omega\cap  \mathsf B(z_k,\rho)}     \xi^2 e^{-\frac 2h  \mathsf d_a(.,z_k)}    \mbf t   \big[h\mathsf d^*\omega_0+\underbrace{\mbf i_{   \nabla(\mathsf d_a(\cdot,z_k)+f)}(\omega_0)}_{=0} \big] \,  \omega_0\cdot \mathsf n_\Omega.
\end{align*}
Therefore, using Laplace's method and~\cite[Th\' eor\`eme 2.5]{helffer-sjostrand-85}, one cannot expect   in general a better estimate
when $h\to0$ than
$$\Vert \mathsf d_{f,h}\mathsf u_{{wkb,0}}\Vert^2_{L^2(\Omega)}+\Vert \mathsf d^*_{f,h}\mathsf u_{{wkb,0}}\Vert^2_{L^2(\Omega)}\le Ch^{3/2},$$
which only implies, using Lemma~\ref{quadra}, that   $\mathsf u_{{wkb,0}}$ is at a distance of the order $O(h^{\frac14})$
from $ \mathsf v_k^{(1)}
$ in $\Lambda^1L^2(\Omega)$.
In view of the computations made in the proof of~\cite[Proposition 90]{DLLN} (which is very similar to the proof of 
Lemma~\ref{de2}, see in particular \eqref{eq.est-QQ}), 
this is not sufficient to prove that the distance between 
$\mathsf u_{{wkb,0}}$  and $ \mathsf v_k^{(1)}$ converges to $0$
 in $\Lambda^1H^1(\Omega)$ as $h\to 0$.
 One would indeed 
 at least  need  that $ \Vert
\mathsf d_{f,h}\mathsf u_{{wkb,0}}\Vert^2_{L^2(\Omega)}+\Vert
\mathsf d^*_{f,h}\mathsf u_{{wkb,0}}\Vert^2_{L^2(\Omega)}=o(h^{3})$ as
$h\to 0$.  

\subsection{Proofs of Propositions~\ref{pr.gammak} and~\ref{pr.omegak}}\label{sec:gammak}

\begin{proof}[Proof of Proposition~\ref{pr.gammak}]
Let $k\in \{1,\ldots,n\}$. 
The proof of Proposition~\ref{pr.gammak} is divided into several steps. 
\medskip

\noindent
\textbf{Step 1: Properties of $\Gamma_{z_k}$ and preliminary definitions.} 
Let us recall that since $\Gamma_{z_k}\subset \mathsf W_+^{z_k}$ (see~\eqref{eq.incluWW} in \autoref{A}), it holds 
$$\text{  for all }  x\in  \Gamma_{z_k}, \ \ \nabla f(x)=\nabla_{\mbf T}f(x) \in T_x\partial \Omega.$$
Moreover, for all $y\in \Gamma_{z_k}$,  since $\varphi_y(s)\in \Gamma_{z_k}$ for all $s\ge 0$  (see~\eqref{eq.hbb}), it holds: $\lim_{s\to +\infty}\varphi_y(s)=z_k$. 
Let $r>0$ and define:
$$\mathsf C_{z_k} =(f|_{\Gamma_{z_k}})^{-1}\big((-\infty, f(z_k)+r)\big).$$  For $r>0$ small, $\vert \nabla f\vert \neq 0$ on $\pa \mathsf C_{z_k}$ and thus, ${\mathsf C_{z_k}}$ is a smooth open neighborhood  of $z_k$ and   $\nabla f \cdot \mathsf n_{\mathsf C_{z_k}}>0$ on $\pa \mathsf C_{z_k}$. For $y\in \Gamma_{z_k}\setminus \overline{  \mathsf C_{z_k} }$, let:
\begin{equation}\label{eq.sinf}
t_{\mathsf C_{z_k}}(y)=\inf\{\, s\ge 0, \, \varphi_{y}(s)\in \overline{\mathsf C_{z_k}}\,  \},
  \end{equation} 
  which is finite   since $\lim_{s\to +\infty}\varphi_y(s)=z_k$. 
By continuity of $s\ge 0\mapsto \varphi_y(s)$, for $y\in \Gamma_{z_k}\setminus \overline{  \mathsf C_{z_k} }$,    $\varphi_y(t_{\mathsf C_{z_k}}(y))\in \pa \mathsf C_{z_k}$ and for all $s> t_{\mathsf C_{z_k}}(y)$, $\varphi_y(s)\in \mathsf C_{z_k}$. Moreover, $t_{\mathsf C_{z_k}}(y)$ is defined by
  $$\int_0^{t_{\mathsf C_{z_k}}(y)} \vert \nabla f(\varphi_y(s))\vert^2ds= f(y)-(f(z_k)+r),$$
  and thus since $\vert \nabla f\vert \neq 0$ on $\pa \mathsf C_{z_k}$, by the implicit functions theorem,  
\begin{equation}\label{eq.sinfregularite}
y\in \Gamma_{z_k} \setminus \overline{  \mathsf C_{z_k} } \mapsto t_{\mathsf C_{z_k}}(y) \text{ is } \mathcal C^\infty. 
  \end{equation}
  
For all $x\in \pa \mathsf C_{z_k}$ and $s\in \mathbb R$, let $\gamma_x(s):=\varphi_x(-s)$  (see~\eqref{eq.hbb}) which  satisfy for all $s\in \mathbb R$
\begin{equation}\label{eq.hbb2}
  \frac{d}{ds}\gamma_x(s)=\nabla f(\gamma_x(s)) \text{ with } \gamma_x(0)=x.
  \end{equation} 
   Let us define 
for all $x\in \pa \mathsf C_{z_k}$,
 \begin{equation}\label{eq.des+}
  s_{\Gamma_{z_k}}(x):=\inf\{s\ge 0, \, \gamma_x(s)  \notin  {\Gamma_{z_k}} \}.
  \end{equation}
Let us prove that
\begin{equation}\label{s+}
   s_{\Gamma_{z_k}}: \pa \mathsf C_{z_k}\to \mathbb R_+ \text{  is lower semicontinuous. } 
\end{equation}
Let us first prove that  for all $x\in \pa \mathsf C_{z_k}$, $   s_{\Gamma_{z_k}}(x)<+\infty$. It it is not the case, there exists $y\in  \pa \mathsf C_{z_k}$ such that $\gamma_y(s)\in \Gamma_{z_k}$ for all $s\ge 0$. Thus, the curve $\gamma_y$ converges to a critical point of $f$ in  $\overline{\Gamma_{z_k}}$, the only one being $z_k$ (by~\eqref{eq.incluWW}), which is impossible because $\gamma_y(s)\notin \mathsf C_{z_k}$ for all $s\ge 0$. Let us now prove that $ s_{\Gamma_{z_k}}$ is lower semicontinuous. To this end, let $(x_n)_{n\ge 0}$ be a sequence  in $ \pa { \mathsf C_{z_k}}$ converging to $x_{\infty}\in  \pa { \mathsf C_{z_k}}$  and  a limit  $s_*$ of a subsequence of $( s_{\Gamma_{z_k}}(x_n))_{n\ge 0}$.  If $s_*=+\infty$, then, $ s_{\Gamma_{z_k}}(x_{\infty})\le s_*$. Let us then consider the case when $s_*<+\infty$. Up to extracting  a subsequence, we assume  that  $ s_{\Gamma_{z_k}}(x_n)\to s_*$ when $n\to \infty$. Notice that for all $n\ge 0$, since $ s_{\Gamma_{z_k}}(x_n)<+\infty$ and $s\mapsto \gamma_{x_n}(s)$ is continuous, 
$\gamma_{x_n}( s_{\Gamma_{z_k}}(x_n))\in \pa \Gamma_{z_k}$. In addition,  since $s_*<+\infty$, 
$\pa \Gamma_{z_k}$ is a closed set, and by continuity of $(x,t)\mapsto \gamma_x(t)$, it holds $\gamma_{x_n}( s_{\Gamma_{z_k}}(x_n))\to \gamma_{x_{\infty}}(s_*)\in \pa \Gamma_{z_k}$   when $n\to \infty$.  This implies that $ s_{\Gamma_{z_k}}(x_{\infty})\le s_*$ by definition of $   s_{\Gamma_{z_k}}$. This implies that $ s_{\Gamma_{z_k}}$ is lower semicontinuous and concludes the proof of~\eqref{s+}.  

Finally, since $\Gamma_{z_k}$ is open, one can   consider an open subset $\mathsf O_{\mathsf F}$ of $\Gamma_{z_k}$ such that 
$$\overline{ \mathsf C_{z_k} }\, \cup \, \mathsf F\subset \mathsf O_{\mathsf F} \ \text{ and } \ \overline{\mathsf O_{\mathsf F}}\subset \Gamma_{z_k}.$$ 
\textbf{Step 2: Construction of  a set  $\mathsf V_{\mathsf F}\subset \Gamma_{z_k}$ containing $\mathsf O_{\mathsf F}$ which is stable for~\eqref{eq.hbb}. }
Define for all $x\in \pa \mathsf C_{z_k}$, $ s_{\mathsf O_{\mathsf F}}(x):=\sup \{s\ge 0, \, \gamma_x(s) \in\overline{\mathsf O_{\mathsf F}}\}$. Let us prove that
\begin{equation}\label{vfgammaf}
 s_{\mathsf O_{\mathsf F}} <+\infty \text{ and } s_{\mathsf O_{\mathsf F}}< s_{\Gamma_{z_k}}.
  \end{equation}
To prove the first statement in~\eqref{vfgammaf}, we argue by contradiction:
 assume that there exists $x\in \pa \mathsf C_{z_k}$ such that  $s_{\mathsf O_{\mathsf F}}(x)=+\infty$. Then, there exists a sequence $s_n\in (\mathbb R_+)^{\mathbb N}$ such that $s_n\to +\infty$ when $n \to +\infty$ and for all $n$, $\gamma_x(s_n)\in \overline{\mathsf O_{\mathsf F}}$.  Thus,    $\gamma_x(s_n)$ converges when $n\to +\infty$ to a critical point of $f$ in $\overline{\mathsf O_{\mathsf F}}$, the only one being $z_k$, which is not possible because  $\gamma_x(s_n)\notin  { \mathsf C_{z_k}}$ for all $n$.
This proves  the first statement in~\eqref{vfgammaf}. To prove the second statement in~\eqref{vfgammaf},  let us consider $x\in \pa \mathsf C_{z_k}$. Notice that since $s_{\mathsf O_{\mathsf F}}(x)$ is finite and the trajectories of~\eqref{eq.hbb2} are continuous in time, $\gamma_x(   s_{\mathsf O_{\mathsf F}}(x) ) \in 
 \overline{\mathsf O_{\mathsf F}} \subset \Gamma_{z_k}$.  Since $\Gamma_{z_k}$ is stable for the dynamics~\eqref{eq.hbb}, $\gamma_x(s)=\varphi_{\gamma_x(   s_{\mathsf O_{\mathsf F}}(x) )} ( s_{\mathsf O_{\mathsf F}}(x)-s) \in \Gamma_{z_k}$ for all   $s\in [0,s_{\mathsf O_{\mathsf F}}(x)]$ (see~\eqref{eq.hbb} and~\eqref{eq.hbb2}).
Moreover, since $\Gamma_{z_k}$ is open, and the trajectories of~\eqref{eq.hbb2} are continuous, there exists $\ve_x>0$ such that: 
 \begin{equation}\label{inclutube}
  \big \{ \gamma_x(s), \, x\in \pa \mathsf C_{z_k} \text{ and } s\in [0,s_{\mathsf O_{\mathsf F}}(x)+\ve_x]\big \} \subset \Gamma_{z_k}.
  \end{equation}
  Therefore, $ s_{\Gamma_{z_k}}(x)\ge s_{\mathsf O_{\mathsf
      F}}(x)+\ve_x$. This concludes the proof of~\eqref{vfgammaf}.  

 Let us now define: 
 \begin{equation}\label{VF}
  \mathsf V_{\mathsf F}:= \mathsf C_{z_k} \,  \cup  \,  \big \{ \gamma_x(s), \, x\in \pa \mathsf C_{z_k} \text{ and } s\in [0,s_{\mathsf O_{\mathsf F}}(x)]\big \}.
    \end{equation}
By construction, the  set  $\mathsf V_{\mathsf F}$ 
 is    stable for the dynamics~\eqref{eq.hbb}. From~\eqref{inclutube} and since $\overline{\mathsf C_{z_k}}\subset  \Gamma_{z_k}$, one has $\mathsf V_{\mathsf F}\subset \Gamma_{z_k}$.  We now claim that
\begin{equation}\label{ofvf0}
  s_{\mathsf O_{\mathsf F}}: \pa \mathsf C_{z_k}\to \mathbb R_+ \text{  is upper semicontinuous and } 
   { \mathsf V_{\mathsf F}}\text{ is a closed set}. 
\end{equation} 
Let us first prove the first statement in~\eqref{ofvf0}.  
To this end, let $(x_n)_{n\ge 0}$ be a sequence  in $ \pa { \mathsf C_{z_k}}$ converging to $x_{\infty}\in  \pa { \mathsf C_{z_k}}$  and $s_*$ a limit of a subsequence of $(s_{\mathsf O_{\mathsf F}}(x_n))_{n\ge 0}$.
  Up to extracting   a subsequence, we assume  that  $s_{\mathsf O_{\mathsf F}}(x_n)\to s_*$ when $n\to \infty$.
Notice that for all $n\ge 0$, 
$\gamma_{x_n}(s_{\mathsf O_{\mathsf F}}(x_n))\in \overline{\mathsf O_{\mathsf F}}$. 
   Let us prove that  $s_*$ is finite. Assume that it is not the case,  i.e. that $s_{\mathsf O_{\mathsf F}}(x_n)\to +\infty$. From~\eqref{vfgammaf}, for all $t\in [0,s_{\mathsf O_{\mathsf F}}(x_n)]$, $\gamma_{x_n}(t)\in \Gamma_{z_k}$. Let $T>0$ and consider $N\ge 1$ such that $s_{\mathsf O_{\mathsf F}}(x_n)\ge T$ for all $n\ge N$. Then, for all $t\in [0,T]$ and $n\ge N$, $\gamma_{x_n}(t)\in \overline{\Gamma_{z_k}}$. Passing to the limit, one obtains that
$\gamma_{x}(t)\in \overline{\Gamma_{z_k}}$ for all $t\in [0,T]$. Since $T>0$ is arbitrary, one deduces that $\gamma_{x}(t)\in \overline{\Gamma_{z_k}}$ for all $t>0$. This is not possible because, as already explained,  the limit points of the curve $\gamma_{x}$ are outside  $ \overline{\Gamma_{z_k}}$. Thus $s_*$ is finite. 
Since 
$\overline{\mathsf O_{\mathsf F}}$ is a closed set, by continuity of $(x,t)\mapsto \gamma_x(t)$, it holds $\gamma_{x_n}(s_{\mathsf O_{\mathsf F}}(x_n))\to \gamma_{x_{\infty}}(s_*)\in \overline{\mathsf O_{\mathsf F}}$  when $n\to \infty$. Therefore, $s_{\mathsf O_{\mathsf F}}(x_{\infty})\ge s_*$, and thus, $s_{\mathsf O_{\mathsf F}}(x_{\infty})\ge \limsup_{n\to +\infty }s_{\mathsf V_{\mathsf F}}(x_n)$. 
This proves that $s_{\mathsf O_{\mathsf F}}$ is upper
semicontinuous. This proves the first statement in~\eqref{ofvf0}. 

 Let us now prove the second  statement in~\eqref{ofvf0}.  To prove that $\mathsf V_{\mathsf F}$ is a closed set it is sufficient to show that $A= \big \{ \gamma_x(s), \, x\in \pa \mathsf C_{z_k} \text{ and } s\in [0,s_{\mathsf O_{\mathsf F}}(x)]\big \}$ is a closed set. To this end,  let $(y_n)_{n\ge 0}$ be a sequence  in $A$  converging to $y_*$. Let us show that $y_*\in A$. Write  $y_n=\gamma_{x_n}(s_n)$ where $x_n \in  \pa { \mathsf C_{z_k}}$ and $0\le s_n\le s_{\mathsf O_{\mathsf F}}(x_n)$. By compactness and up to extracting a subsequence, let  $x_{\infty}\in  \pa { \mathsf C_{z_k}}$ such that $x_n\to x_{\infty}$ when $n\to +\infty$. Since $s_{\mathsf O_{\mathsf F}}<+\infty$ is upper semicontinuous on the  compact set $\pa \mathsf C_{z_k}$, $s_{\mathsf O_{\mathsf F}}$ is bounded on $\pa \mathsf C_{z_k}$. Therefore, $(s_n)_{n\ge 0}$ and  $(s_{\mathsf O_{\mathsf F}}(x_n))_{n\ge 0}$ are bounded. Denote by $s_{*}$  a limit of a subsequence of $(s_n)_{n\ge 0}$. Then, it holds 
$$s_*\le  \limsup_{n\to +\infty }s_n \le \limsup_{n\to +\infty }s_{\mathsf O_{\mathsf F}}(x_n) \le s_{\mathsf O_{\mathsf F}}(x_{\infty}),$$
where the last inequality follows from the fact that $s_{\mathsf
  O_{\mathsf F}}$ is upper semicontinuous.  Since $y_n=
\gamma_{x_n}(s_n)\to \gamma_{x_{\infty}}(s_*)=y_*$ when $n\to
+\infty$, and $s_*\le s_{\mathsf O_{\mathsf F}}(x_{\infty})$, this
implies that $y_*\in  \big \{ \gamma_x(s), \, x\in \pa \mathsf C_{z_k}
\text{ and } s\in [0,s_{\mathsf O_{\mathsf F}}(x)]\big \}$. The set
$A$ is therefore  closed and thus,  so is  $\mathsf V_{\mathsf F}$.

Finally, let us prove that,
\begin{equation}\label{ofvf}
{\mathsf O_{\mathsf F}}\subset \mathsf V_{\mathsf F}.
\end{equation}  
To prove~\eqref{ofvf}, we consider $y\in \mathsf O_{\mathsf F}$.   
The curve  $s\in [0,t_{\mathsf C_{z_k}}(y)] \mapsto \gamma_{\varphi_y(t_{\mathsf C_{z_k}}(y))}(s)= \varphi_y(t_{\mathsf C_{z_k}}(y)-s)$  passes through $y\in \mathsf O_{\mathsf F}$ at time $s=t_{\mathsf C_{z_k}}(y)$ (see~\eqref{eq.hbb},~\eqref{eq.hbb2}, and~\eqref{eq.sinf}). Thus, by definition of $s_{\mathsf O_{\mathsf F}}$, it holds $t_{\mathsf C_{z_k}}(y)\le s_{\mathsf O_{\mathsf F}}(\varphi_y(t_{\mathsf C_{z_k}}(y)) )$  and thus, by definition of   $\mathsf V_{\mathsf F}$, 
$y=\gamma_{\varphi_y(t_{\mathsf C_{z_k}}(y))}(t_{\mathsf C_{z_k}}(y))\in \{ \gamma_{\varphi_y(t_{\mathsf C_{z_k}}(y)]}(s), s \in [0,s_{\mathsf O_{\mathsf F}}(x))\}= \mathsf V_{\mathsf F}$. 
 This proves~\eqref{ofvf} and in particular $\mathsf F\subset\mathsf V_{\mathsf F}$. 
 \medskip
 
 \noindent
The interior of $\mathsf V_{\mathsf F}$ might be  a good candidate to be $\Gamma_{\mathsf F}$ but this set is not necessarily  smooth or does not satisfy~\eqref{eq.ultrastable}. This is due to the fact that the function $s_{\mathsf O_{\mathsf F}}$ is not necessarily   smooth. For this reason, 
we approximate $s_{\mathsf O_{\mathsf F}}$    from above by a smooth function: this is made in 
 the next step, see~\eqref{eq.<>s2}. 
\medskip

\noindent
\textbf{Step 3: Construction of $\Gamma_{\mathsf F}$.} 
\medskip

\noindent
\textbf{Step 3a: Approximation of $s_{\mathsf O_{\mathsf F}}$ from above by a smooth function and definition of~$\Gamma_{\mathsf F}$.}
Since $s_{\mathsf O_{\mathsf F}}$ is upper semicontinuous (see~\eqref{ofvf0}), from~\cite[Theorem 3]{tong1952some}, there exists a decreasing sequence of continuous functions $\tilde s_n:\pa \mathsf C_{z_k} \to \mathbb R$, $n\ge 1$, such that for all $n\ge 1$, $\tilde  s_n\ge s_{\mathsf O_{\mathsf F}}$  and for all $x\in \pa \mathsf C_{z_k}$,  $\tilde s_n(x)\to s_{\mathsf O_{\mathsf F}}(x)$ when $n\to +\infty$.    
Let us prove that  there exists $n_0\ge 1$ such that:
\begin{equation}\label{eq.<>s}
\text{for all } x\in \pa \mathsf C_{z_k}, \ s_{\mathsf O_{\mathsf F}}(x)\le  \tilde s_{n_0}(x)< s_{\Gamma_{z_k}}(x).
\end{equation}
We just have to prove the second   inequality in~\eqref{eq.<>s}. For that purpose, we argue by contradiction: assume that  
\begin{equation}\label{geabs}
\text{for all $n\ge 1$, there exists 
$x_n\in \pa \mathsf C_{z_k}$ such that}\, \tilde s_{n}(x_n)\ge s_{\Gamma_{z_k}}(x_n).
\end{equation}
By compactness and up to extracting a subsequence, let  $x_{\infty}\in  \pa { \mathsf C_{z_k}}$ such that $x_n\to x_{\infty}$ when $n\to +\infty$. Let $\ve >0$. There exists $N_0\ge 1$ such that  for all $n\ge N_0$, $\tilde s_{n}(x_{\infty})-s_{\mathsf O_{\mathsf F}}(x_{\infty}) \le \ve/2$. 
 One then has for all $n\ge N_0$, using that $\tilde s_{n}\le \tilde s_{N_0}$,
\begin{align*}
  \tilde s_{n}(x_n)- s_{\mathsf O_{\mathsf F}}(x_{\infty})&=(\tilde s_{n}(x_n)-\tilde s_{N_0}(x_{\infty}))+(\tilde s_{N_0}(x_{\infty})- s_{\mathsf O_{\mathsf F}}(x_{\infty}))\le \tilde s_{N_0}(x_n)-\tilde s_{N_0}(x_{\infty}) +  \ve/2.
 \end{align*}
 Moreover, because $\tilde s_{N_0}$ is a continuous function, there exists $N_1\ge 1$ such that for all $n\ge N_1$, $\vert \tilde s_{N_0}(x_n)-\tilde s_{N_0}(x_{\infty})\vert \le \ve/2$. Thus, for $n\ge \max(N_0,N_1)$,
 $   \tilde s_{n}(x_n)- s_{\mathsf O_{\mathsf F}}(x_{\infty})\le \ve$, i.e.  $\limsup_{n\to +\infty} \tilde s_{n}(x_n)- s_{\mathsf O_{\mathsf F}}(x_{\infty})\le 0$. Now, since $s_{\Gamma_{z_k}}$ is lower semicontinuous (see~\eqref{s+}) and from~\eqref{geabs}, it holds: $s_{\mathsf O_{\mathsf F}}(x_{\infty})\ge \limsup_{n\to +\infty}  \tilde s_{n}(x_n) \ge  \liminf_{n\to +\infty} s_{\Gamma_{z_k}}(x_n)\ge s_{\Gamma_{z_k}}(x_\infty)$. This contradicts the second statement in~\eqref{vfgammaf}, and thus concludes the proof of~\eqref{eq.<>s}. \\
Since the function $s_{\Gamma_{z_k}}-\tilde s_{n_0}$ is lower semicontinuous and $\pa \mathsf C_{z_k}$ is compact, $s_{\Gamma_{z_k}}-\tilde s_{n_0}$ attains its infimum on $\pa \mathsf C_{z_k}$ and since $ s_{\Gamma_{z_k}}>\tilde s_{n_0}$ (see~\eqref{eq.<>s}), this minimum is positive. Let us then consider $0<\ve<\min_{\pa \mathsf C_{z_k}} (s_{\Gamma_{z_k}}-\tilde s_{n_0})$ so that 
\begin{equation}\label{eq.<>s2}
 \tilde s_{n_0} +\ve < s_{\Gamma_{z_k}} \ \text{ on }  \pa \mathsf C_{z_k}. 
\end{equation} 
 Since 
 $ \pa \mathsf C_{z_k}$ is compact and $ \tilde s_{n_0} +\ve$ is continuous on  $ \pa \mathsf C_{z_k}$, there exists $\beta \in \mathcal C^\infty( \pa \mathsf C_{z_k},\mathbb R)$ such that $ \tilde s_{n_0} +\ve /2\le \beta \le \tilde s_{n_0} + {3\ve}/{4}$, so that in view of \eqref{eq.<>s} and~\eqref{eq.<>s2}, 
\begin{equation}\label{eq.betas} 
 s_{\mathsf O_{\mathsf F}} < \beta <s_{\Gamma_{z_k}} \ \text{ on }  \pa \mathsf C_{z_k}.
 \end{equation}
 We now define 
\begin{equation}\label{F}
 \Gamma_{\mathsf F}:= \mathsf C_{z_k} \,  \cup  \,  \big \{ \gamma_x(s), \, x\in \pa \mathsf C_{z_k} \text{ and } s\in [0,\beta(x))\big \} .
  \end{equation}

\noindent  
\textbf{Step 3b: Properties of $ \Gamma_{\mathsf F}$.} Let us finally prove that $ \Gamma_{\mathsf F}$ satisfies all the properties listed in Proposition~\ref{pr.gammak}. First notice that   by construction, $\overline{ \Gamma_{\mathsf F}}$ is included in $\Gamma_{z_k}$, this indeed  follows from~\eqref{F} together with the second inequality in~\eqref{eq.betas}. Moreover, $ \Gamma_{\mathsf F}$ contains $\mathsf V_{\mathsf F}$ (since $s_{\mathsf O_{\mathsf F}}<\beta$, see~\eqref{eq.betas} and~\eqref{VF}) and thus,  $\mathsf F\subset  \Gamma_{\mathsf F}$. Furthermore, by construction, $\Gamma_{\mathsf F}$ is simply connected. 
Let us now prove that $ \Gamma_{\mathsf F}$ is open and satisfies~\eqref{eq.ultrastable}.
\begin{enumerate}
\item Proof of the fact that $ \Gamma_{\mathsf F}$ is open. To this end, let us first show that $ \Gamma_{\mathsf F}\setminus \overline{ \mathsf C_{z_k}}$ is open.  
Let us denote by $\mathsf d_{\partial \Omega}$ the geodesic distance in $\partial \Omega$. Let $y_1\in  \Gamma_{\mathsf F}\setminus \overline{ \mathsf C_{z_k}}$
and write $y_1=\gamma_{x_1}(s_1)$ where $x_1\in \pa \mathsf C_{z_k}$ and $s_1\in (0,\beta(x_1))$. 
Then, there exists $t_1\in (s_1, \beta(x_1))$ such that   $\gamma_{y_1}(-t_1)\in \mathsf C_{z_k}$.   
 Since  the mapping     $x\mapsto \beta(x)$  is  continuous and  $t_1<\beta(x_1)$, there exists $\ve_1>0$ such that for all $x\in \pa \mathsf C_{z_k}$,    
\begin{equation}\label{eq.fitsve}
\mathsf d_{\partial \Omega}\left(x,x_1\right)\leq \ve_1 \text{ implies } t_1<\beta(x)
\end{equation}
Let  $\ve_0=\mathsf d_{\partial \Omega}(\gamma_{y_1}(-t_1),\partial \mathsf C_{z_k}) >0$.
Since the mapping $y\mapsto \gamma_y(-t_1)$ is continuous, there exists $\ve_2>0$ such that if $\mathsf d_{\partial \Omega}\left(y,y_1\right)\leq \ve_2$ then 
$$\mathsf d_{\partial \Omega}(\gamma_y(-t_1),\gamma_{y_1}(-t_1))\leq \ve_0/2,$$ 
and thus $\gamma_y(-t_1)\in \mathsf C_{z_k}$.
Let $y\in \Gamma_{z_k}$.  
Write $y=\gamma_x(t_{\mathsf C_{z_k}}(y))$ where $x= \varphi_{y}(t_{\mathsf C_{z_k}}(y))\in \pa \mathsf C_{z_k}$, see~\eqref{eq.sinf}. Since when  $\mathsf d_{\partial \Omega}\left(y,y_1\right)\leq \ve_2$ one has $\varphi_y(t_1)=\gamma_y(-t_1)\in \mathsf C_{z_k}$, it holds:  
\begin{equation}\label{eq.fitsve2}
\text{for all } y\in \Gamma_{z_k},\  \mathsf d_{\partial \Omega}\left(y,y_1\right)\leq \ve_2 \Rightarrow t_{\mathsf C_{z_k}}(y)<t_1.
\end{equation}
  Since $y\in \Gamma_{z_k}\setminus \overline{\mathsf C_{z_k}}\mapsto \varphi_{y}(t_{\mathsf C_{z_k}}(y))\in \pa \mathsf C_{z_k}$ is smooth (see indeed~\eqref{eq.sinfregularite}), there exists $\ve_3>0$ such that if $\mathsf d_{\partial \Omega}\left(y,y_1\right)\le \ve_3$, then  $y\in \Gamma_{z_k}\setminus \overline{\mathsf C_{z_k}} $ and $\mathsf d_{\partial \Omega}(x,x_1)\le \ve_1$ where $x=\varphi_{y}(t_{\mathsf C_{z_k}}(y))$ and $x_1=\varphi_{y_1}(t_{\mathsf C_{z_k}}(y_1))$. In conclusion, from~\eqref{eq.fitsve} and~\eqref{eq.fitsve2}, if $\mathsf d_{\partial \Omega}\left(y,y_1\right)\le \min(\ve_2,\ve_3)$, then $t_{\mathsf C_{z_k}}(y)<t_1<\beta(x)$, where $x=\varphi_{y}(t_{\mathsf C_{z_k}}(y))$ and $y=\gamma_x(t_{\mathsf C_{z_k}}(y))$. Thus, from~\eqref{F}, if 
  $\mathsf d_{\partial \Omega}\left(y,y_1\right)\le \min(\ve_2,\ve_3)$, then 
  $y\in  \Gamma_{\mathsf F}\setminus \overline{ \mathsf C_{z_k}}$. The set  $\Gamma_{\mathsf F}\setminus \overline{ \mathsf C_{z_k}}$ is therefore open. In addition, since $\overline{\mathsf C_{z_k}}\subset \mathsf O_{\mathsf F}\subset \Gamma_{\mathsf F}$ and $ \mathsf O_{\mathsf F}$ is open, $\Gamma_{\mathsf F}=(\Gamma_{\mathsf F}\setminus \overline {\mathsf C_{z_k}})\cup  \overline{\mathsf C_{z_k}}\subset{\rm int \, }(\Gamma_{\mathsf F})$. Therefore the set $\Gamma_{\mathsf F}$ is open.

  \medskip

\noindent
Moreover, using the same arguments as those used to prove the second statement of~\eqref{ofvf0}, 
it holds:
\begin{equation}\label{closureF}
 \overline{ \Gamma_{\mathsf F}}= \mathsf C_{z_k} \,  \cup  \,  \big \{ \gamma_x(s), \, x\in \pa \mathsf C_{z_k} \text{ and } s\in [0,\beta(x)]\big \},
 \end{equation}
Consequently, since  $\beta< s_{\Gamma_{z_k}}$ and the trajectories of~\eqref{eq.hbb2} are continuous (see also~\eqref{eq.des+}), one has:
$$\overline{ \Gamma_{\mathsf F}}\subset \Gamma_{z_k}.$$ 
In addition,   from~\eqref{closureF} and~\eqref{F}, one has: 
\begin{equation}\label{boundaryF}
\pa \Gamma_{\mathsf F}=\big \{  \gamma_x(\beta(x)), \, x\in \pa \mathsf C_{z_k}  \big \}.  
 \end{equation}

\item Proof of the fact that the set $\Gamma_{\mathsf F}$ satisfies~\eqref{eq.ultrastable}, i.e.  that 
$\nabla f\cdot \mathsf n_{ \Gamma_{\mathsf F} }>0$  on $ \pa \Gamma_{\mathsf F}$, 
 where we recall that $\mathsf n_{ \Gamma_{\mathsf F}}\in T\partial \Omega$ is the unit outward normal to $ \Gamma_{\mathsf F}$.  Notice that by construction,  the set $ \Gamma_{\mathsf F}$  is a stable set for the dynamics~\eqref{eq.hbb} and thus $\nabla f\cdot \mathsf n_{ \Gamma_{\mathsf F} }\ge0$  on $ \pa \Gamma_{\mathsf F}$. Let us prove that this inequality is actually a strict inequality. Let us define the function
 $$\Upsilon: y\in \Gamma_{z_k}\setminus \overline{\mathsf C_{z_k}}\mapsto (x,t)\in \pa \mathsf C_{z_k}\times \mathbb R_+^* \text{ such that } \gamma_x(t)=y.$$
 Notice that if $\Upsilon(y)=(x,t)$, then $t=t_{\mathsf C_{z_k}}(y)$ (see~\eqref{eq.sinf}) and $x=\varphi_{y}(t_{\mathsf C_{z_k}}(y))$. The mapping $\Upsilon$ is a $\mathcal C^\infty$ diffeomorphism from  $\Gamma_{z_k}\setminus \overline{\mathsf C_{z_k}}$ into its range. Let us denote by $F:=\Upsilon^{-1}$ its inverse function, i.e. 
 \begin{equation}\label{F(x,t)}
F(x,t)=\gamma_x(t).  
 \end{equation}
Thus,  for all $x\in \pa \mathsf C_{z_k}$, $ (\text{Jac\,} F) (x,\beta(x))$ is a bijection bewteen $T_x\pa \mathsf C_{z_k} \times \mathbb R$ and $ T_{\gamma_x(\beta(x))} \Gamma_{z_k}$. For all $x\in \pa \mathsf C_{z_k}$ and  $v = (v_1,v_2)\in T_x\pa \mathsf C_{z_k} \times \mathbb R$, one has:
\begin{equation}\label{jac1}
 (\text{Jac\,} F) (x,\beta(x))v=  (\partial_x F)(x,\beta(x))   v_1+ ( \partial_t     F) (x,\beta(x))    \times v_2 \in T_{\gamma_x(\beta(x))} \Gamma_{z_k},
 \end{equation}
   \noindent
 where $ (\partial_t  F)(x,\beta(x))  =   \nabla f(\gamma_x(\beta(x)) $ (see~\eqref{F(x,t)} and~\eqref{eq.hbb2}). Using  the chain rule, one has:
\begin{align}
\nonumber
\text{Jac}_x \big ( F(x,\beta(x))\big )  v_1&= (\partial_xF)(x,\beta(x))  v_1 +      ( \partial_t F)(x,\beta(x)))   \big(  \nabla\beta(x)    \cdot v_1\big )\\
\label{jac2}
&= (\pa _xF)(x,\beta(x))  v_1  + \nabla f(\gamma_x(\beta(x))     \big(\nabla\beta(x)   \cdot v_1\big),
\end{align}
where $\text{Jac}_x \big ( F(x,\beta(x))\big )  v_1\in T_{\gamma_x(\beta(x))} \Gamma_{\mathsf F} $  and $\nabla\beta(x)\in T_x\pa \mathsf C_{z_k}$. 
To prove~\eqref{eq.ultrastable} we argue by contradiction:  assume that  there exists $x\in \pa \mathsf C_{z
_k}$ such that $\nabla f(\gamma_x(\beta(x)) )\cdot \mathsf n_{ \Gamma_{\mathsf F} }(\gamma_x(\beta(x)))=0$ (see~\eqref{boundaryF}) which is equivalent to $\nabla f(\gamma_x(\beta(x)) )\in T_{\gamma_x(\beta(x))} \pa \Gamma_{\mathsf F}$.  This implies, in view of~\eqref{jac1} and~\eqref{jac2} that $\text{Ran}\, (\text{Jac\,} F) (x,\beta(x))\subset T_{\gamma_x(\beta(x))} \pa \Gamma_{\mathsf F}$, which contradicts the fact that $F$ is a diffeomorphism. This concludes the proof of~\eqref{eq.ultrastable}.
\end{enumerate}
  The proof of Proposition~\ref{pr.gammak} is complete. \end{proof}

\begin{proof}[Proof of Proposition~\ref{pr.omegak}]
For all $y\in \Omega$,  recall that $\varphi_y(s)\in \Omega$ for all $s\ge 0$ (see~\eqref{eq.hbb}) and $\lim_{s\to +\infty}\varphi_y(s)=x_0$. 
Define for  $r>0$:
 $$\mathsf C_{x_0} =(f|_{\Omega})^{-1}\big((-\infty, f(x_0)+r)\big),$$
For $r>0$ small enough,    $\overline{\mathsf C_{x_0}} \subset \Omega$ and $\vert \nabla f\vert \neq 0$ on $\pa \mathsf C_{x_0}$. Thus,   ${\mathsf C_{x_0}}$ is a smooth open neighborhood  of $x_0$ and   $\nabla f \cdot \mathsf n_{\mathsf C_{x_0}}>0$ on $\pa \mathsf C_{x_0}$. 
 For all $x\in \pa \mathsf C_{x_0}$ and $s\in \mathbb R$, let $\gamma_x(s):=\varphi_x(-s)$  (see~\eqref{eq.hbb}) which  satisfy for all $s\in \mathbb R$
$$
   \frac{d}{dt}\gamma_x(s)=\nabla f(\varphi_x(s)) \text{ with } \gamma_x(0)=x.
$$
    Let us define 
for all $x\in \pa \mathsf C_{x_0}$,
$$
   s_{\Omega}(x):=\inf\{s\ge 0, \, \gamma_x(s)  \notin\Omega \}.
$$
The proof of Proposition~\ref{pr.omegak} follows exactly the same lines as the proof of Proposition~\ref{pr.gammak} if one shows that $ s_\Omega$ is lower semicontinuous. The difference here, comparing  $   s_{\Gamma_{z_k}}$ (see~\eqref{eq.des+}) and  $  s_{\Omega}$, is that $  s_{\Omega}$ can be infinite due to the existence of critical points  of $f$  on $\partial \Omega$. Let us thus prove that $ s_\Omega: \pa \mathsf C_{x_0}\to \mathbb R_+\cup \{+\infty\}$ is lower semicontinuous. To this end, let $(x_n)_{n\ge 0}$ be a sequence  in $ \pa { \mathsf C_{x_0}}$ converging to $x_\infty\in  \pa { \mathsf C_{x_0}}$  and   $s_*\in  {\mathbb R_+}\cup \{+\infty\}$ a limit of  a subsequence of $( s_\Omega(x_n))_{n\ge 0}$. For ease of notation, up to extracting a subsequence, we assume that $ s_\Omega(x_n)\to s_*$ when $n\to +\infty$. If $s_*=+\infty$ then, $s^*\ge s_\Omega(x_\infty)$. Let us now consider the case when $s^*<+\infty$. 
In particular, $ s_\Omega(x_n)$ is finite for $n$ large enough. 
In this case,  $s_\Omega(x_\infty)\le s_*$ by the same proof as  the one made to show~\eqref{s+}. 
 In conclusion $s_\Omega$  is lower semicontinuous. Then, the same arguments as those used to prove Proposition~\ref{pr.gammak} allows us to conclude the proof of  Proposition~\ref{pr.omegak}. 
\end{proof}


\end{document}